\documentclass[11pt, a4paper]{amsart}
\usepackage{amsmath,amsfonts,amsthm,bbm, amssymb}
\usepackage{standalone}
\usepackage{comment}
\usepackage{tikz-cd}
\usepackage[cmtip, all]{xy}
\usepackage[original]{imakeidx} 
\usepackage{comment}
\usepackage[english]{babel}
\usepackage{enumerate}
\usepackage{graphicx}
\usepackage{amssymb}
\usepackage{stmaryrd}
\usepackage{dsfont, mathrsfs,bold-extra}
\usepackage{setspace}
\usepackage{wasysym}

\usepackage[colorinlistoftodos, textsize=tiny]{todonotes}
\newcounter{myenumi}
\renewcommand{\themyenumi}{(\arabic{myenumi})}
\newenvironment{listi}{%
\setlength{\parindent}{0pt}
\setcounter{myenumi}{0}
\bigskip
\renewcommand{\item}{
\par
\refstepcounter{myenumi}
\makebox[2.5em][l]{\themyenumi}
}
}{
\par
\bigskip
\noindent
\ignorespacesafterend
}

\usepackage[color, notref, notcite, final]{showkeys}
\definecolor{refkey}{gray}{.5}   
\definecolor{labelkey}{gray}{.5} 
\definecolor{Red}{rgb}{1,0,0}

\usepackage{soul}

\newtheorem{thm}{Theorem}[subsection]
\newtheorem{prop}[thm]{Proposition}
\newtheorem{lem}[thm]{Lemma}
\newtheorem{cor}[thm]{Corollary}
\newtheorem{conj}[thm]{Conjecture}

\newtheorem{ques}[thm]{Question}

\theoremstyle{definition}

\newtheorem{df}[thm]{Definition}
\newtheorem{const}[thm]{Construction}
\newtheorem{rmk}[thm]{Remark}
\newtheorem{remks}[thm]{Remarks}

\newtheorem{exm}[thm]{Example}
\newtheorem{exms}[thm]{Examples}
\newtheorem{notat}[thm]{Notation}
\numberwithin{equation}{subsection}

{\hfill$\square$\end{defn}}
{\hfill$\square$\end{remk}}
{\hfill$\square$\end{remks}}
{\hfill$\square$\end{exm}}
{\hfill$\square$\end{exms}}

	\newcommand{\N}{\mathbb{N}}

	\newcommand{\Z}{\mathbb{Z}}
	\newcommand{\C}{\mathbb{C}}
    
	\newcommand{\Q}{\mathbb{Q}}
	\newcommand{\R}{\mathbb{R}}
	\renewcommand{\P}{\mathbb{P}}
	\newcommand{\E}{\mathbb{E}}
	\newcommand{\A}{\mathbb{A}}

    \newcommand{\G}{{\mathbb{G}}}
    \newcommand{\bbL}{\mathbb{L}}
    \newcommand{\T}{\mathbb{T}}

 \newcommand{\Mod}{\mathrm{Mod}}

	\newcommand{\Sch}{\mathbf{Sch}}

   \newcommand{\lSch}{\mathbf{lSch}}
   \newcommand{\Sm}{\mathbf{Sm}}
   \newcommand{\lSm}{\mathbf{lSm}}
   \newcommand{\cSm}{\mathbf{cSm}}
   \newcommand{\cSch}{\mathbf{cSch}}
   \newcommand{\SmlSm}{\mathbf{SmlSm}}

    \newcommand{\Shv}{\mathbf{Shv}}
    
    \newcommand{\Psh}{\mathbf{Psh}}
    
    \newcommand{\Pshltr}{\mathbf{Psh}^\mathrm{ltr}}

\providecommand{\frac}[1]{\operatorname{Frac}(#1)}  
\providecommand{\Spec}[1]{\operatorname{Spec}(#1)} 	%

\newcommand{\Sing}{\operatorname{Sing}}              
\newcommand{\id}{\operatorname{id}}   			
\newcommand{\Pic}{\operatorname{Pic}}		 
\renewcommand{\hom}{\operatorname{Hom}}				 
\newcommand{\coker}{\operatorname{Coker}}			 
\renewcommand{\ker}{\operatorname{Ker}}				 

\newcommand{\ltr}{{\mathrm{ltr}}}
\newcommand{\colim}{\operatorname{colim}}		 
\renewcommand{\lim}{\operatorname{lim}}			 
\newcommand{\pt}[1]{\mathrm{\pt}_{#1}}
\newcommand{\Cone}{\operatorname{Cone}}	
\newcommand{\colimit}{\mathop{\rm colim}}
\newcommand{\limit}{\mathop{\rm lim}}
\newcommand{\limitone}{\mathop{\rm lim^1}}

\newcommand{\Blow}{{\mathrm{Bl}}}
\newcommand{\Normal}{{\mathrm{N}}}
\newcommand{\Tangent}{{\mathrm{T}}}
\newcommand{\Deform}{{\mathrm{D}}}
\newcommand{\Gr}{{\mathrm{Gr}}}

\newcommand{\BGL}{{\mathrm{BGL}}}
\newcommand{\KGL}{{\mathbf{KGL}}}
\newcommand{\MGL}{{\mathrm{MGL}}}

\newcommand{\spectML}{\mathbf{M\Lambda}}

\newcommand{\Adm}{\mathbf{Adm}}
\newcommand{\SmAdm}{\mathbf{SmAdm}}

\newcommand{\HH}{\mathrm{HH}}
\newcommand{\THH}{\mathrm{THH}}

\newcommand{\Formal}{\mathrm{F}}

\newcommand{\TC}{\mathrm{TC}}

\newcommand{\TP}{\mathrm{TP}}

\newcommand{\Tr}{\mathrm{Tr}}

\newcommand{\Brep}{\operatorname{B^{rep}}}
\newcommand{\Bcy}{\operatorname{B^{cy}}}
\newcommand{\Sphere}{\mathbb{S}}
\newcommand{\Kth}{\mathrm{K}}
\newcommand{\KthNiz}{\mathrm{K}^{\mathrm{Niz}}}
\newcommand{\Bott}{\mathrm{Bott}}
\newcommand{\Thom}{\mathrm{Th}}
\newcommand{\Betti}{\mathrm{Betti}}

\newcommand{\bOmega}{\mathbf{\Omega}}

	\newcommand{\ol}{\overline}
	\renewcommand{\ul}{\underline}

\newcommand{\gp}{{\mathrm{gp}}}

\newcommand{\et}{{\acute{e}t}}
\newcommand{\Real}{\mathrm{Re}}
\newcommand{\ket}{{k\acute{e}t}}
\newcommand{\leta}{{l\acute{e}t}}
\newcommand{\set}{{s\acute{e}t}}
\newcommand{\clspace}{\mathrm{B}}
\newcommand{\Eilenberg}{H}
\newcommand{\sharpp}{\sharp}
\newcommand{\Stab}{\mathrm{Stab}}

\def\bP{\mathbf{P}}

\newcommand{\logKGL}{\mathbf{logKGL}}
\newcommand{\logMGL}{{\mathrm{logMGL}}}
\newcommand{\spectMGL}{{\mathbf{MGL}}}
\newcommand{\spectlogMGL}{{\mathbf{logMGL}}}
\newcommand{\logTHH}{\mathrm{logTHH}}
\newcommand{\spectlogTHH}{\mathbf{logTHH}}
\newcommand{\logTC}{\mathrm{logTC}}
\newcommand{\logTP}{\mathrm{logTP}}
\newcommand{\logTr}{\mathrm{logTr}}
\newcommand{\logHH}{\mathrm{logHH}}
\newcommand{\spectlogTr}{\mathbf{logTr}}
\newcommand{\spectlogTC}{\mathbf{logTC}}
\newcommand{\spectlogTP}{\mathbf{logTP}}
\newcommand{\Kthlog}{\mathrm{logK}}
\newcommand{\logPic}{\operatorname{logPic}}
\newcommand{\spectlogML}{\mathbf{logM\Lambda}}

\newcommand{\cA}{\mathcal{A}}

\newcommand{\cC}{\mathcal{C}}
\newcommand{\cD}{\mathcal{D}}
\newcommand{\cE}{\mathcal{E}}
\newcommand{\cF}{\mathcal{F}}
\newcommand{\cG}{\mathcal{G}}

\newcommand{\cL}{\mathcal{L}}
\newcommand{\cM}{\mathcal{M}}
\newcommand{\cN}{\mathcal{N}}
\newcommand{\cO}{\mathcal{O}}
\newcommand{\cP}{\mathcal{P}}
\newcommand{\cQ}{\mathcal{Q}}
\newcommand{\cR}{\mathcal{R}}

\newcommand{\cT}{\mathcal{T}}

\newcommand{\cV}{\mathcal{V}}
\newcommand{\cX}{\mathcal{X}}
\newcommand{\cY}{\mathcal{Y}}
\newcommand{\cZ}{\mathcal{Z}}
\newcommand{\cW}{\mathcal{W}}

\newcommand{\bA}{\mathbf{A}}

\newcommand{\bC}{\mathbf{C}}

\newcommand{\bE}{\mathbf{E}}

\newcommand{\bT}{\mathbf{T}}

\newcommand{\sF}{\mathscr{F}}

\newcommand{\sX}{\mathscr{X}}

\newcommand{\rR}{\mathrm{R}}
\newcommand{\rT}{\mathrm{T}}

\renewcommand{\epsilon}{\varepsilon}
\renewcommand{\phi}{\varphi}

\newcommand{\boxx}{\square}
\newcommand{\Gmm}{\mathbb{G}_m^{\mathrm{log}}}
\newcommand{\cptsph}[1]{V_{#1}}

\def\pt{\operatorname{pt}}

\newcommand{\eff}{{\rm eff}}

\newcommand{\Spc}{\mathbf{Spc}}

\newcommand{\infSpc}{\mathcal{S}}
\newcommand{\infSpcpt}{\mathcal{S}\mathrm{pc}_\ast}
\newcommand{\infTopos}{\mathcal{T}\mathrm{opos}}
\newcommand{\infPsh}{\mathcal{P}\mathrm{sh}}
\newcommand{\infShv}{\mathcal{S}\mathrm{hv}}
\newcommand{\infSpt}{\mathcal{S}\mathrm{pt}}
\newcommand{\infCycSpt}{\mathcal{C}\mathrm{yc}\mathcal{S}\mathrm{pt}}
\newcommand{\infPro}{\mathcal{P}\mathrm{ro}}
\newcommand{\infDA}{\mathcal{DA}}
\newcommand{\Fun}{\operatorname{Fun}}
\newcommand{\Map}{\operatorname{Map}}

\newcommand{\one}{\mathbf{1}}
\newcommand{\LPr}{\mathcal{P}\mathrm{r}^{\mathrm{L}}}
\newcommand{\RPr}{\mathcal{P}\mathrm{r}^{\mathrm{R}}}
\newcommand{\LPrpt}{\mathcal{P}\mathrm{r}^{\mathrm{L}}_{\ast}}

\newcommand{\cofib}{\operatorname{cofib}}
\newcommand{\tcofib}{\operatorname{tcofib}}
\newcommand{\tfib}{\operatorname{tfib}}
\newcommand{\fib}{\operatorname{fib}}

\newcommand{\unit}{\mathbf{1}}
\newcommand{\ringE}{\mathbf{E}}

\newcommand{\CAlg}{\mathrm{CAlg}}

\newcommand{\horn}{\mathrm{horn}}
\newcommand{\cube}{\mathrm{cube}}
\newcommand{\vertex}{\mathrm{vertex}}
\newcommand{\infC}{\mathcal{C}}
\newcommand{\infD}{\mathcal{D}}

\newcommand{\infSH}{\mathcal{SH}}
\newcommand{\infSHS}{\mathcal{SH}^{\eff}}

\newcommand{\infSHPS}{\mathcal{SH}^{\eff, \P^\bullet}}
\newcommand{\infH}{\mathcal{H}}
\newcommand{\infHpt}{\mathcal{H}_{\ast}}
\newcommand{\infHP}{\mathcal{H}^{\mathbb{P}^\bullet}}

\newcommand{\inflogHptAdm}{\mathrm{log}\mathcal{H}_{\ast}^\Adm}

\newcommand{\inflogH}{\mathrm{log}\mathcal{H}}

\newcommand{\inflogHpt}{\mathrm{log}\mathcal{H}_\ast}

\newcommand{\inflogSH}{\mathrm{log}\mathcal{SH}}
\newcommand{\inflogSHS}{\mathrm{log}\mathcal{SH}^{\mathrm{eff}}}
\newcommand{\inflogDA}{\mathrm{log}\mathcal{DA}}
\newcommand{\inflogFDA}{\mathrm{log}\mathcal{FDA}}
\newcommand{\inflogDAeff}{\mathrm{log}\mathcal{DA}^{\mathrm{eff}}}
\newcommand{\inflogDM}{\mathrm{log}\mathcal{DM}}
\newcommand{\inflogDMeff}{\mathrm{log}\mathcal{DM}^{\mathrm{eff}}}
\newcommand{\infDM}{\mathcal{DM}}

\newcommand{\Einfty}{{\mathbb{E}_\infty}}
\newcommand{\map}{\operatorname{map}}
\newcommand{\Striv}{\mathbb{S}^{\mathrm{triv}}}

\newcommand{\infHPone}{\mathcal{H}^{\mathbb{P}^1}}
\newcommand{\infHPonept}{\mathcal{H}_{\ast}^{\mathbb{P}^1}}
\newcommand{\infSHPoneS}{\mathcal{SH}^{\eff, \mathbb{P}^1}}

\newcommand{\infCat}{\mathcal{C}\mathrm{at}}
\newcommand{\MS}{\mathbf{MS}}

\makeatletter
\newcommand{\colim@}[2]{%
  \vtop{\m@th\ialign{##\cr
    \hfil$#1\operator@font hocolim$\hfil\cr
    \noalign{\nointerlineskip\kern1.5\ex@}#2\cr
    \noalign{\nointerlineskip\kern-\ex@}\cr}}%
}
\newcommand{\hocolim}{%
  \mathop{\mathpalette\colim@{\rightarrowfill@\textstyle}}\nmlimits@
}
\makeatother

\makeatletter
\newcommand{\ncolim@}[2]{%
  \vtop{\m@th\ialign{##\cr
    \hfil$#1\operator@font colim$\hfil\cr
    \noalign{\nointerlineskip\kern1.5\ex@}#2\cr
    \noalign{\nointerlineskip\kern-\ex@}\cr}}%
}
\newcommand{\varcolim}{%
  \mathop{\mathpalette\ncolim@{\rightarrowfill@\textstyle}}\nmlimits@
}
\makeatother

\makeatletter
\newcommand{\varvarcolim}{%
  \mathop{\mathpalette\ncolim@{}}\nmlimits@
}
\makeatother

\makeatletter
\newcommand{\holim@}[2]{%
  \vtop{\m@th\ialign{##\cr
    \hfil$#1\operator@font holim$\hfil\cr
    \noalign{\nointerlineskip\kern1.5\ex@}#2\cr
    \noalign{\nointerlineskip\kern-\ex@}\cr}}%
}
\newcommand{\holim}{%
  \mathop{\mathpalette\holim@{\leftarrowfill@\textstyle}}\nmlimits@
}
\makeatother

\definecolor{winered}{rgb}{0.8,0,0}
\definecolor{greensea}{rgb}{0.24, 0.8, 0.54}
\definecolor{indigo}{rgb}{0.3, 0, 0.5}
\usepackage[pdfauthor={FDPA}, %
				pdftitle={logmot} ,colorlinks=true,hypertexnames=false]{hyperref}
\hypersetup{
	bookmarksnumbered=true,
	linkcolor={winered},
	citecolor=blue,
}
\usepackage[normalem]{ulem}

\usepackage{float}
\newcounter{elno}

\newcounter{elno-abc}

\newif\iftoplevel
\topleveltrue

 \usepackage[ left=3cm, right=3cm, top = 2.8cm, bottom = 2.8cm ]{geometry}
 
\makeatletter
\def\@tocline#1#2#3#4#5#6#7{\relax
  \ifnum #1>\c@tocdepth 
  \else
    \par \addpenalty\@secpenalty\addvspace{#2}%
    \begingroup \hyphenpenalty\@M
    \@ifempty{#4}{%
      \@tempdima\csname r@tocindent\number#1\endcsname\relax
    }{%
      \@tempdima#4\relax
    }%
    \parindent\z@ \leftskip#3\relax \advance\leftskip\@tempdima\relax
    \rightskip\@pnumwidth plus4em \parfillskip-\@pnumwidth
    #5\leavevmode\hskip-\@tempdima
      \ifcase #1
      \or\or \hskip 2em \or \hskip 2em \else \hskip 3em \fi%
      #6\nobreak\relax
    \dotfill\hbox to\@pnumwidth{\@tocpagenum{#7}}\par
    \nobreak
    \endgroup
  \fi}
\makeatother

\makeindex
\makeindex[name=notation,title=Index of Notation]

\begin{document}

\author{Federico Binda}
\address{Dipartimento di Matematica ``Federigo Enriques'',  
Universit\`a degli Studi di Milano, Via Cesare Saldini 50, 20133 Milano, Italy}
\email{federico.binda@unimi.it}

\author{Doosung Park}
\address{Bergische Universit{\"a}t Wuppertal,
Fakult{\"a}t Mathematik und Naturwissenschaften, Gau{\ss}strasse 20, 42119 Wuppertal, Germany}
\email{dpark@uni-wuppertal.de}

\author{Paul Arne {\O}stv{\ae}r}
\address{Dipartimento di Matematica ``Federigo Enriques'', 
Universit\`a degli Studi di Milano, Via Cesare Saldini 50, 20133 Milano, Italy} 
\email{paul.oestvaer@unimi.it}
\address{Department of Mathematics, University of Oslo, 
Niels Henrik Abels hus, Moltke Moes vei 35, 0851 Oslo, Norway}
\email{paularne@math.uio.no}

\title{Logarithmic motivic homotopy theory}

\begin{abstract}
This work is dedicated to the construction of a new motivic homotopy theory for (log) schemes, generalizing Morel-Voevodsky's (un)stable $\mathbb{A}^1$-homotopy category. 
Our framework can be used to represent 
log topological Hochschild and cyclic homology, 
as well as algebraic $K$-theory of regular schemes. 
Additionally, we can realize the cyclotomic trace as a morphism between motivic spectra.
Among our applications, we provide a generalized framework of oriented cohomology theories that enables us to produce new residue sequences for (topological) Hochschild, periodic, and cyclic homology of classical schemes. We also compute $\THH$ and its variants for Grassmannians, and we define a new version of algebraic cobordism.
Finally, we give a construction of a log \'etale stable realization functor, as well as a Kato-Nakayama realization functor, which is of independent interest for applications in log geometry.
\end{abstract}

\maketitle

\setcounter{tocdepth}{1}
\tableofcontents


\section{Introduction}
\label{section:introduction}

The term "motivic homotopy theory" describes a unified framework for applying methods from algebraic topology to study objects traditionally associated with algebraic and 
arithmetic geometry, see \cite{MV}, \cite{zbMATH01194164}.
An aspect that guided this construction is the idea that the affine line $\mathbb{A}^1$ plays the same role in algebraic geometry as the unit interval $[0, 1]$ in algebraic topology. 

The main aim of this work is to shift the current perspective and develop a homotopy theory for algebraic varieties without relying on the axiom of $\mathbb{A}^1$-homotopy invariance. This research builds on the framework established in \cite{logDM}, where a cycle-theoretic approach is explored through log correspondences over a field.

In this paper, we construct analogs of the Morel-Voevodsky categories $\mathcal{H}(S)$ and $\mathcal{SH}(S)$ over a general base. We demonstrate that many generalized cohomology theories of arithmetic significance, such as topological Hochschild homology ($\THH$), can be represented in our framework, although they cannot be represented within the original motivic context.

Additionally, this paper advances the program outlined in \cite{logDMcras}, which serves to motivate and summarize several of the results presented here.

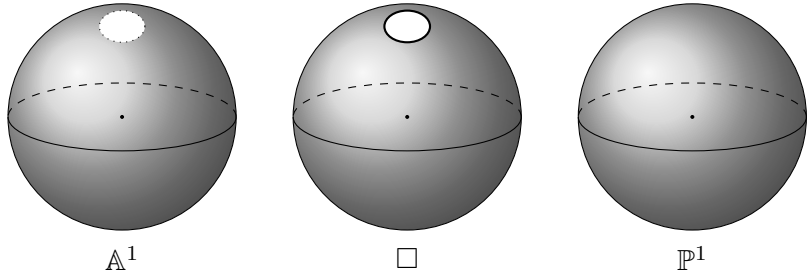
\begin{figure}[tbp]
\centering
\begin{tikzpicture}[scale = 0.75]
  \shade[ball color = gray!40, opacity = 0.4] (0,0) circle (2cm);
  \draw (0,0) circle (2cm);
  \draw (-2,0) arc (180:360:2 and 0.6);
  \draw[dashed] (2,0) arc (0:180:2 and 0.6);
  \fill[fill=black] (0,0) circle (1pt);
  \filldraw[fill=white, dotted] (0,1.6) ellipse (0.4cm and 0.28cm);
  \node at (0,-2.5) {$\mathbb{A}^1$};
\begin{scope}[shift={(5,0)}]
  \shade[ball color = gray!40, opacity = 0.4] (0,0) circle (2cm);
  \draw (0,0) circle (2cm);
  \draw (-2,0) arc (180:360:2 and 0.6);
  \draw[dashed] (2,0) arc (0:180:2 and 0.6);
  \fill[fill=black] (0,0) circle (1pt);
  \filldraw[fill=white, thick] (0,1.6) ellipse (0.4cm and 0.28cm);
    \node at (0,-2.5) {$\boxx$};
\end{scope}
\begin{scope}[shift={(10,0)}]
  \shade[ball color = gray!40, opacity = 0.4] (0,0) circle (2cm);
  \draw (0,0) circle (2cm);
  \draw (-2,0) arc (180:360:2 and 0.6);
  \draw[dashed] (2,0) arc (0:180:2 and 0.6);
  \fill[fill=black] (0,0) circle (1pt);
    \node at (0,-2.5) {$\mathbb{P}^1$};
\end{scope}
\end{tikzpicture}
\caption{The logarithmic interval $\boxx$ is contractible as $\mathbb{A}^1$ and compact as $\mathbb{P}^1$}
\label{fig:exmp}
\end{figure}

\subsection{Why avoid $\mathbb{A}^1$-invariance?} 
One of the effective tools for studying the algebraic $K$-theory of a ring is the theory of trace methods. This involves examining trace maps from $K$-theory to more computable invariants derived from Hochschild homology, such as $\THH$ and $\TC$, along with an analysis of their (homotopy) fibers.

While algebraic $K$-theory is represented in the realm of motivic homotopy theory, the scenario is strikingly different for (topological) Hochschild homology and its variants. It is relatively straightforward to show, following the work of Weibel and Geller \cite{MR0965242}, that when we apply Voevodsky's $\mathbb{A}^1$-localization functor $L_{\mathbb{A}^1}$, we have $L_{\mathbb{A}^{1}} \THH \simeq L_{\mathbb{A}^{1}} \HH \simeq 0$. This indicates that, regardless of any regularity assumptions, it is impossible to extract information about (topological) Hochschild homology within Voevodsky's framework. 

The situation does not improve for $\TC$, as a result by Elmanto \cite{EEnonA1} demonstrates that $L_{\mathbb{A}^{1}} \TC$ vanishes after profinite completion. These no-go results clearly illustrate that, at a fundamental level, $\mathbb{A}^1$-motivic homotopy theory is incompatible with trace methods.

\subsection{Why log geometry?} 
Our setting for overcoming this problem is to work in the context of logarithmic geometry
in the sense of Deligne, Faltings, Fontaine, Illusie, Kato, Tsuji and many others \cite{Ogu}. Porting over some of the ideas developed in  \cite{logDM}, our insight is to use the $1$-point compactification of the affine line $\boxx=(\mathbb{P}^{1},\infty)$ equipped with its canonical logarithmic structure, as a contracting object, instead of the affine line $\mathbb{A}^{1}$. 
The resulting condition is intermediate between $\mathbb{A}^{1}$-invariance and $\mathbb{P}^{1}$-invariance (considered, for example, by Ayoub \cite{AyoubP1loc}). 
Figure \ref{fig:exmp} gives a geometric picture of this situation.

There are several reasons to embrace log geometry as a natural framework for studying non-\(\mathbb{A}^1\)-invariant cohomology theories and for the choice of our contracting object, as discussed in \cite{logDM}. In this discussion, we will focus on a perspective that has not been highlighted in the aforementioned work, specifically the relationship between logarithmic geometry and (topological) Hochschild and cyclic homology.
This idea can be traced back to the work of Hesselholt and Madsen on the \(K\)-theory of local fields. Additionally, the fact that Hochschild homology and its variants lack a satisfactory localization property (unlike \(K\)-theory) motivated Rognes to define (topological) Hochschild homology for a logarithmic ring (or spectrum) \cite{RognesLogHH}. By representing the extension of \(\THH\) in a motivic category that shares certain properties with Morel and Voevodsky's construction \cite{MV}, we can obtain new localization sequences for \(\logTHH\) and its variants, generalizing the results from \cite{RSS}.
It turns out that \(\logTHH\) and its relatives are indeed representable within this framework, leading to the formation of motivic \(\mathbb{P}^1\)-spectra. This fact has numerous implications, which we will explore further.

\subsection{Outline and summary of main results}
We now provide an overview of the content of this work. 

Sections \ref{section:colmhc} and \ref{polmsas} focus on the basic constructions. Our discussion begins with a quasi-compact, quasi-separated, integral quasi-coherent log scheme \( S \) (essentially, any log scheme you would encounter in practice) and the symmetric monoidal \(\infty\)-category \(\mathcal{P}(\lSm/S)\) of presheaves of spaces on log smooth saturated log schemes \(\lSm/S\) of finite type over \( S \).

We start by examining \(\boxx\)-local (in the natural sense) and dividing Nisnevich local presheaves of spectra on \(\lSm/S\). Specifically, we consider presheaves \(\mathcal{F}\) that turn every strict Nisnevich distinguished square into a cartesian square and transform every dividing cover (a class of admissible blow-ups, see Section \ref{ssec:topologies}) into an equivalence. This topology was introduced in \cite{logDM}.

Through this approach, we form the category \(\inflogSH^{\eff}(S)\) of effective log motives over \( S \). The noneffective version, \(\inflogSH(S)\), is derived by considering its \(\mathbb{P}^1\)-stabilization, which involves formally inverting the tensor product with the scheme \(\mathbb{P}^1 = (\mathbb{P}^1, \emptyset)\), treated as a log scheme with trivial log structure. 

While the basic definitions are fundamentally formal, it is a significant result that \(\inflogSH(S)\) can be presented as a category of symmetric spectra. Additionally, it aligns with the formal inversion of the log Tate sphere \(\Gmm := (\mathbb{P}^1, 0 + \infty)\), as specified by the divisor \(0 + \infty\) on \(\mathbb{P}^1\). See Lemma \ref{Ori.55}. The outcome is a stable, presentable symmetric monoidal \(\infty\)-category.

The construction of $\inflogSH$ satisfies essentially by construction the following universal property, as explained in 
\S\ref{subsection:universal}: 
For $S\in \lSch$ and any pointed presentable symmetric monoidal $\infty$-category $\infD^\otimes$, 
there exists a naturally induced fully faithful functor
\[
\Fun^{\otimes,\mathrm{L}}(\inflogSH(S)^\otimes,\infD^\otimes)
\to
\Fun^{\otimes}((\lSm/S)^\times ,\infD^\otimes).
\]
Its essential image consists of all $\cF$ that satisfy $\boxx$-invariance and $dNis$-descent and such that the 
cofiber of $$\cF(i_1)\colon\cF(S)\rightarrow\cF(\P^1)$$ is an invertible object in $\infD^\otimes$, where $i_1$
denotes the unit section.
We refer to Robalo \cite{RobaloThesispublished} for a similar universal property in stable motivic homotopy theory.

In Section \ref{models}, we discuss an alternative model for  $\inflogSH$ when \( S \) is a quasi-compact and quasi-separated scheme with trivial log structure. This context allows us to compare our theory with the classical \(\A^1\)-invariant framework. 
Rather than starting with the category of log smooth fine log schemes over \( S \), we can consider the category \( \SmlSm/S \) consisting of log smooth log schemes whose underlying scheme is smooth over \( S \). Next, we consider the \(\infty\)-categories of \emph{strict} Nisnevich sheaves of spectra, denoted as \( \infShv_{sNis}(\SmlSm/S, \infSpt) \). We focus on the full subcategory spanned by those objects \( \cF \) that are local with respect to the maps \( \cF(X) \to \cF(X \times (\P^n, \P^{n-1})) \) for all integers $n\geq 1$. We call such objects \( \cF \) for \((\P^\bullet, \P^{\bullet-1})\)-invariant. After establishing this, we stabilize with respect to the log Tate circle.
While the condition of being \((\P^\bullet, \P^{\bullet-1})\)-invariant is a priori stronger than \(\boxx\)-invariance, it turns out that the resulting category is indeed equivalent to \( \inflogSH(S) \).

In fact, the following more precise statement holds:
\begin{thm}[Corollaries \ref{ProplogSH.9} and \ref{cor:fundamental_models}]
There are equivalences of symmetric monoidal $\infty$-categories
\begin{align*}
\inflogSH^\eff(S)^\otimes
\simeq &
\infShv_{sNis}(\SmlSm/S,\infSpt)[\square^{-1},\SmAdm^{-1}]^\otimes
\\
\simeq &
\infShv_{sNis}(\SmlSm/S,\infSpt)
[(\P^\bullet,\P^{\bullet-1})^{-1}]^\otimes.
\end{align*}
\end{thm}

 Here, $\SmAdm$ denotes the class of admissible blow-ups along smooth centers in the sense of Definition \ref{logH.35}.
 The previous Theorem is one of the strongest properties of our construction: it implies that one gets \emph{a posteriori} dividing descent and invariance under admissible blow-ups along smooth centers as a consequence of a very simple axiom: the sheaf property plus the invariance with respect to compactified affine spaces. In practice, this latter condition is often straightforward to check and typically follows from the projective bundle formula (note, however, that there are examples of theories satisfying $(\P^\bullet,\P^{\bullet -1})$-invariance, but not the projective bundle formula, like real topological Hochschild homology \cite{2305.04150}).

When $S\in \Sch$ is a scheme, we can also consider an \emph{unstable version} of the above construction. In this case we simply start with the category $\infShv_{sNis}(\SmlSm/S)_{(\ast)}$ of strict Nisnevich sheaves of (pointed) spaces, and we localize with respect to $\boxx$-local and $\SmAdm$-local equivalences. The resulting $\infty$-category is denoted $\inflogH_{(\ast)}(S)$: 
again, it underlies a presentable symmetric monoidal $\infty$-category (often denoted with the same symbols). 
We see it as a non $\mathbb{A}^1$-invariant analogue of $\mathcal{H}_{(\ast)}(S)$. 

With these definitions, 
we deduce an unstable equivalence 
\begin{equation}
\label{equation:logmotivicsphere}
\P^1
\simeq
S^{1}\otimes\Gmm.
\end{equation}
We view \eqref{equation:logmotivicsphere} as the log motivic analogue of the 
fundamental equivalence $\P^{1}\cong S^{1}\otimes \mathbb{G}_{m}$ in \cite{MV}.  We then obtain the following equivalence
\begin{equation}
\inflogSH(S)
\simeq\textup{lim}
\left(\cdots\xrightarrow{\Omega_{\mathbb{P}^{1}}}
\inflogH_{\ast}(S)\xrightarrow{\Omega_{\mathbb{P}^{1}}}
\inflogH_{\ast}(S)\right).
\end{equation}

\

The inevitable question now is whether our logarithmic motivic theory is at all understandable. 
Happily, the short answer is yes at this stage, mainly due to the following reasons discussed in
the main body of the paper.

\subsubsection{Comparison with $\mathbb{A}^1$-motivic homotopy theory} In \S\ref{section:mvlmt} we compare our construction with 
Morel-Voevodsky's $\mathbb{A}^{1}$-invariant motivic homotopy theory of schemes.
If $S\in \Sch$ is a scheme, 
we have natural transformations
\[
\omega_\sharpp \colon \inflogH \to \infH
\text{ and }
\omega_\sharpp\colon \inflogHpt\to \infHpt,
\]
\[
\omega_\sharpp \colon \inflogSHS \to \infSHS
\text{ and }
\omega_\sharpp
\colon
\inflogSH
\to
\infSH.
\]
It turns out that their right adjoints, denoted by $\omega^*$,
are fully faithful functors.
One can view this as providing an embedding of homotopy theories:
\begin{center}
$(\A^1\text{-local motivic homotopy theory})
\subset
(\boxx\text{-local motivic homotopy theory})$
\end{center}
In particular, 
the endomorphism ring of the motivic sphere $\mathbf{1}$ 
--- the unit for the tensor product ---
is isomorphic to the endomorphism ring of $\omega^*(\mathbf{1})$ in $\inflogSH$.
Over fields,
this is the Grothendieck-Witt ring of quadratic forms via Morel's theorem 
on the motivic $\pi_0$ of the sphere spectrum \cite{zbMATH05057435} 
(\cite{zbMATH07472296} shows a similar result for rings of integers in certain number fields).
We expect that $\omega^*(\unit)$ is equivalent to $\unit$ in $\inflogSH$ on a general base.
For the moment, we have the following results. In $\inflogDM$, 
the equivalence $\omega^*(\unit)\simeq \unit$ shown in 
\cite[Theorem 8.2.11]{logDM} uses resolution of singularities, while in \cite[Theorem 1.1]{Doosgroups} it is shown that $\omega^*(\unit) \simeq \unit$ in $\inflogSH(k)$, again for a field admitting resolution of singularities. In fact, more generally, $\omega^*$ is shown to be fully faithful and colimits-preserving for $\inflogSH(k)$.

\subsubsection{Logarithmic Picard groups and $K$-theory}Logarithmic Picard groups $\logPic$ of log schemes are part of our setup, 
see \S\ref{section:lpg}.
We show $(\P^n,\P^{n-1})$-invariance for $\logPic$ and for the log multiplicative group scheme, we have the following Theorem. 
\begin{thm}[{Theorem} \ref{Pic.9}, Corollary \ref{Pic.10}]
For every $S\in \Sch$ there is an equivalence
\[
\clspace \Gmm \simeq \P^\infty
\]
in $\inflogH(S)$,
which is natural in $S$. When $S\in \lSch$, the same equivalence holds in $\inflogSH^\eff(S)$.
\end{thm}

This is the analogue of the equivalence $\P^\infty\simeq \mathrm{B}\mathbb{G}_m$ proved by Morel and Voevodsky using $\A^1$-homotopy invariance.

In \S\ref{section:kfslogschemes} we construct a logarithmic $K$-theory of log schemes.  
While a notion of perfect complexes on log schemes remains elusive (but see \cite{BLMP} for an approach to solve this problem), 
we employ the infinite Grassmannian to geometrically represent $K$-theory of log schemes  
defined over a regular base scheme. This gives us a motivic spectrum $\logKGL$ that satisfies the following:

\begin{thm}[Theorem \ref{K-theory.7}]
Suppose $S$ is a regular scheme,
and let $d$ be an integer.
For every $X\in \lSm/S$ with compactifying log structure given by the embedding $X-\partial X \hookrightarrow X$, where $\partial X$ is a Cartier divisor on $X$,  there are canonical equivalences of spectra
\[
\Kth(X-\partial X)
\simeq
\map_{\inflogSHS(S)}(\Sigma_{S^1}^\infty X_+,\Omega_{\Gmm}^\infty \logKGL)
\]
and
\[
\Kth(X-\partial X)
\simeq
\map_{\inflogSH(S)}(\Sigma_{\P^1}^d \Sigma_{\P^1}^\infty X_+,\logKGL).
\]
\end{thm}
Note that over a base $S$ with trivial log structure, the fact that ``log $K$-theory'' agrees with the $K$-theory of the open complement of the support of the log structure was one of the expected properties.  For a general log scheme $S$, we obtain a motivic spectrum simply by the pullback of $\logKGL$ defined over the integers.

\subsubsection{Oriented cohomology theories} We develop in \S\ref{section:oct} a theory of Chern orientations and Thom classes in our setting. This allows us to prove a projective bundle formula \ref{Ori.11},
a Thom isomorphism theorem \ref{Ori.16}, a Whitney sum formula \ref{Ori.20}, an excess intersection formula \ref{Ori.21}, a self intersection formula \ref{Ori.26}  for any oriented cohomology theory represented in $\inflogSH(S)$. All the above results apply in particular for cohomology theories on classical schemes that admit a natural extension to log schemes, like Hodge, crystalline, prismatic, syntomic cohomology and so on, that are \emph{not} $\A^1$-homotopy invariant. One of the key inputs in the proofs is the following fundamental computation, obtained by adapting the argument in $\inflogDM$.
\begin{thm}[Proposition \ref{Ori.3}] For every $S\in\Sch$, there is an equivalence
\[
\P^n/\P^{n-1}
\cong
(\P^1/\pt)^{\wedge n}
\]
in $\inflogSHS(S)$.
\end{thm}

One of the main applications of our formalism is the construction of Gysin or \emph{residue} sequences for non $\A^1$-invariant cohomology theories: we adopted the expression ``residue sequence'' following \cite{BLPO}, since the graded pieces of the Gysin sequence for $\logHH$ yield the classical residue sequence for Hodge cohomology.   The source of them is the following Theorem.
\begin{thm}[Theorem \ref{Thom.1}, Remark \ref{rmk:Gysin}] Let $S\in \Sch$. Let $X\in \SmlSm/S$ and let $Z\in \SmlSm/S$ be a strict closed subscheme in $X$ that has strict normal crossing with $\partial X$ over $S$. Then there is a cofiber sequence of pointed motivic spaces
\[
(\Blow_Z X,E)_+ \to X_+ \to \Thom(\Normal_Z X)
\]
in $\inflogH_{\ast}(S)$,
functorial with respect to $(X',Z')\to (X,Z)$ such that the induced morphism $Z'\to Z\times_X X'$ is cartesian,
where $\Thom(\Normal_Z(X))$ is the Thom space of the normal bundle of $Z$ in $X$, appropriately defined. 
\end{thm}
In the presence of an orientation and in the stable setting, the
Thom space term can be trivialized, thus assuming a familiar form. This applies, for example, to prismatic cohomology, as discussed in \cite{BLMP}, which can be represented by an oriented $\P^1$-spectrum (see \S\ref{subsection:rncpc}).
In \S \ref{ssec:Grass} we also present a complete computation of the cohomology of the Grassmannians $\Gr(r, n)$. 
\begin{thm}[Theorem \ref{Ori.9}] Let $S\in \Sch$ and let $\ringE$ be an oriented homotopy commutative monoid in $\inflogSH(S)$. 
There is a naturally induced ring isomorphism
\[
\varphi_{r,n}
\colon
\ringE^{**}\otimes_{\Z}R_{r,n}
\to
\ringE^{**}(\Gr(r,n)),
\]
where $R_{r,n}$ is the $\Z$-algebra defined in \eqref{eq:ring_grass}. 
\end{thm}
The argument requires new ingredients, since we cannot apply the $\A^1$-homotopies used, for example, in \cite{zbMATH05663788}. Finally, we introduce a logarithmic algebraic cobordism spectrum $\spectlogMGL$ using our definition of Thom spaces.  This is the universally oriented homotopy commutative monoid in $\inflogSH$.

\subsubsection{Logarithmic trace methods} In \S \ref{section:lthh} we construct a $\P^1$-spectrum $\spectlogTHH\in \inflogSH(S)$
for all $S\in \Sch$. The circle action on $\spectlogTHH$ allows us to consider the variants $\spectlogTP$ and, of course, $\spectlogTC^-$, representing, respectively, periodic and negative topological cyclic homology. As seen in \cite{BLPO2}, the cyclotomic structure makes it possible to obtain from $\spectlogTP$ and $\spectlogTC^-$ yet another spectrum, that is $\spectlogTC$, representing
log topological cyclic homology. More precisely, we have the following: 

\begin{thm}[Theorem \ref{Hoch.13}, Theorem \ref{Hoch.34}, Theorem \ref{trace.8}]
   Let $S\in \Sch$. Then there are \emph{oriented} homotopy commutative monoids in $\inflogSH(S)$
   \[
\spectlogTHH,\; \spectlogTC^-,\; \spectlogTP,\; \spectlogTC
\in
\inflogSH(S),\]
and there is an equivalence of spectra
\[
\logTHH(X)
\simeq
\map_{\inflogSHS(S)}(\Sigma_{S^1}^\infty X_+,\Omega_{\Gmm}^\infty \spectlogTHH)\simeq \map_{\inflogSH(S)}(\Sigma_T^\infty X_+,\spectlogTHH)\]
for all $X\in \SmlSm/S$. The spectrum $\spectlogTHH$ is $(2,1)$-periodic, in the sense that 
we have an equivalence of spectra
\[
\spectlogTHH
\simeq
\Sigma^{2,1}\spectlogTHH.\]
Similar results hold for $\spectlogTC^-$, $\spectlogTP$, and $\spectlogTC$ also.
\end{thm}
In particular, when $X = \Spec{R}$ is just the spectrum of a ring with trivial log structure, we find that $\THH(R)$ can be described as a mapping spectrum in our motivic category.  Two main ingredients go into this: a proof that a suitably globalized version of $\logTHH$ satisfies $(\P^\bullet, \P^{\bullet-1})$-invariance and a $\P^1$-bundle formula for $\logTHH$. See \S \ref{ssec:P1bundleTHH}. 

Since the above spectra are, in fact, oriented cohomology theories in the sense of \S\ref{section:oct}, all the results mentioned above apply. For example, we immediately get a full computation of $\THH$ of the Grassmannians, new Gysin sequences, and so on.

When $k$ is a field that admits resolution of singularities, we can put together the results of \S \ref{section:kfslogschemes} and \ref{section:lthh} to obtain the following result.
\begin{thm}[Theorem \ref{trace.4}]
There is a unique homomorphism
\begin{equation}\label{eq:tr_intro}
\spectlogTr\colon \logKGL \to \spectlogTC
\end{equation}
of homotopy commutative monoids in $\inflogSH(k)$ preserving orientations.  
\end{thm}
The assumption on resolution of singularities here comes from the equivalence between $\inflogH(k)$ and a variant denoted $\inflogH^{\Adm}(k)$, in which we invert arbitrary proper birational maps which are isomorphisms on the complement of the log structure, but we do not consider it to be crucial. 
We see \eqref{eq:tr_intro} as a first important step in studying \emph{motivic filtrations} on $\TC$ and variants by using the geometric methods given by motivic homotopy theory, such as the slice filtration. See \S \ref{future} below. 

Note that 
\[
\lambda^*\Omega_{\Gmm}^\infty \spectlogTr\colon \lambda^*\Omega_{\Gmm}^\infty \logKGL \to \lambda^*\Omega_{\Gmm}^\infty \spectlogTC
\]
is equivalent to the cyclotomic trace $\Tr\colon \Kth\to \TC$ in $\infShv_{Nis}(\Sm/k,\infSpt)$.
 This relies on \cite{logSHF1}, see Remark \ref{trace.12}.

\subsubsection{Kato-Nakayama spaces, \'etale realization}In \S\ref{section:rf} we construct log \'etale and Kato-Nakayama realization functors. 
Over a  field $k$ with an embedding  $k\hookrightarrow \mathbb{C}$ into the field of complex numbers, we can consider the functor that sends $X\in \lSch/k$ its Kato-Nakayama space $X^{\log}$. It is a manifold with boundary if $X$ is log smooth, considered here as an object of the category of topological spaces. Since the Kato-Nakayama realization of $\boxx$ is nothing but the closed unit disk (seen as a manifold with boundary), this can be promoted to  an $\infty$-functor
$\mathrm{Re}_{\mathrm{KN}}$, satisfying the following property.
\begin{thm}[Proposition \ref{KNreal.5}]
There is a commutative triangle
\[
\begin{tikzcd}
\inflogH(k)\ar[rr,"\omega_\sharp"]\ar[rd,"\Real_{\mathrm{KN}}"']&
&
\infH(k)\ar[ld,"\Real_\Betti"]
\\
&
\infSpc,
\end{tikzcd}\]
where $\Real_{\Betti}$ denotes the Betti realization functor 
$\infH(k)\to \infSpc$.
There is a similar commutative triangle for $\inflogSH(k)$, $\infSH(k)$, and $\infSpt$.
\end{thm}
In light of the work of Vistoli and Talpo \cite{MR3869583}, one could consider more refined Betti realization functors, with value in a category of log analytic spaces. We leave this task to future work. 

When $k$ is an arbitrary field and $\ell$ is invertible in $k$, we construct in \S \ref{subsection:ler} a log \'etale $\ell$-adic realization functor
\[
\mathrm{Re}^{l\et}_\ell\colon \inflogSH_{l\et}(k) \to \mathcal{P}\mathrm{ro}(\infSpt)^{\Z/\ell}.
\]
This is obtained by considering the log \'etale homotopy type or, more precisely, the \emph{shape} of $X\in \lSch$, defined as $\Pi(\infShv(X_{l\et}, \infSpc))$ (see \cite[Chapter 7]{HTT} and the work of Hoyois \cite{Hoyoisetalesymmetric} and Zargar \cite{zbMATH07103864}). The basic input is given by the $\boxx$-invariance of Kummer \'etale cohomology with locally constant coefficients, proved in Theorem \ref{ketreal.1}. 
\

In summary, 
the fundamental theories considered in motivic homotopy theory are available in our logarithmic setting, but, in addition, important examples from $p$-adic geometry come into focus.

\begin{rmk}Let us offer a stacky perspective on our set-up, 
cf.\ \cite{zbMATH06552967}, \cite{zbMATH02128586}.
A rank one Deligne-Faltings log structure on a scheme $X$ is equivalent to a morphism 
$X\to [\mathbb{A}^{1}/\mathbb{G}_{m}]$; 
the quotient stack is formed with respect to the canonical $\mathbb{G}_{m}$-action by scaling.
Now, a morphism $f\colon X\to [\mathbb{A}^{1}/\mathbb{G}_{m}]$ is the datum of a 
$\mathbb{G}_m$-torsor on $X$, a.k.a., a line bundle, together with a section.  
For $X$ regular and noetherian, 
this is exactly the datum of a generalized Cartier divisor $D$ on $X$. 
In concrete terms, 
the closed subscheme $D$ can be recovered as $f$'s fiber over 
$\mathrm{B}\mathbb{G}_m \subset [\mathbb{A}^{1}/\mathbb{G}_{m}]$.
In particular, 
for the log scheme $\boxx$ and the divisor $\infty\in \mathbb{P}^{1}$ 
we obtain the cartesian square 
\begin{equation}\label{eq:whi}
\begin{tikzcd}
\{\infty\} \ar[d] \ar[r] & \mathbb{P}^1 \arrow[d] \\
\mathrm{B}\mathbb{G}_m \ar[r] & {[\mathbb{A}^{1}/\mathbb{G}_{m}].}
\end{tikzcd}    
\end{equation}
This provides a kind of analogy between the geometry of $\boxx$-homotopy invariance and 
Annala-Hoyois-Iwasa's work on weighted $\mathbb{A}^{1}$-homotopy invariance in 
$\mathbb{P}^{1}$-inverted stable homotopy theory for schemes ---
they show that the bottom horizontal map in \eqref{eq:whi} is an equivalence assuming 
elementary blow-up excision and $\P^1$-stability, see \cite{AHI}. 
Both of these properties hold by the construction of our logarithmic stable homotopy theory
and Theorem \ref{ProplogSH.5}. See \S \ref{section:mvlmt} for the construction of a comparison functor. 
\end{rmk}

\subsection{Discussions about non $\A^1$-invariant homotopy theory}
There are at least four approaches to non $\A^1$-invariant motivic homotopy theory.
Let us review their advantages and disadvantages as follows.
\begin{listi}
\item The interval $\boxx$ in $\inflogSH$ allows us to adapt various arguments in $\A^1$-homotopy theory to the non $\A^1$-invariant setting,
which is one of the main contents of this paper.
Representing a cohomology theory of schemes in $\inflogSH(k)$ demands, of course, its extension to log schemes.
This is often possible, as we have discussed, but requires extra effort, see, e.g., \cite{BLPO}, \cite{BLPO2}, and \cite{BLMP}.
The unstable categories $\inflogH(S)$ can be used as a replacement for $\infH(S)$ if one removes the $\A^1$-invariant axiom. To our knowledge this is completely new. 
\item The slight variant $\inflogSH^\Adm$ of $\inflogSH$ allows non-toroidal boundaries.
This is useful to remove the assumption of resolution of singularities,
but it is harder to represent a cohomology theory.
\item The theory of motives with modulus due to Kahn-Miyazaki-Saito-Yamazaki \cite{MR4442406} is specialized to understand the ramification behavior at the boundary,
which is not observable in $\inflogSH$ or $\inflogDM$.
For the sake of this,
representing cohomology theories would require even more effort. 
Of course, a representability result in motives with modulus is a priori a stronger statement. 
\item
A different perspective is given by the
$\P^1$-inverted homotopy theory, see Annala-Hoyois-Iwasa \cite{AHI}. 
To represent a cohomology theory in this setting,
one needs to know a priori elementary blow-up excision,
which can be obtained by another input (e.g., blow-up formula for $K$-theory) or deduced from $(\P^n,\P^{n-1})$-invariance if an extension to log schemes is known. 
This theory is stable by construction, 
while unstable computations are possible in our log setting. 
\end{listi}

\subsection{Outlook and future developments}\label{future}
We quickly review some of the consequences of the results presented here that we have either developed or will develop in upcoming works. 
\subsubsection{Residue sequences for $\logHH$} The Hochschild-Kostant-Rosenberg filtration on logarithmic $\HH$ was constructed by Binda, Lundemo, Park and \O stv\ae r in \cite{BLPO}.
The theory of orientations in \S \ref{section:oct} produces the Gysin sequences for $\logHH$.
Even when $X=\Spec{R}$ and $Z=\Spec{R/I}$ are smooth affine schemes, the Gysin sequence relating $\HH(R)$ and $\HH(R/I)$ is completely new if $\mathrm{codim}_Z(X)\geq 2$ since non-affine log schemes are involved.

The variants $\mathrm{logHC}^-$ and $\mathrm{logHP}$ are studied in \cite{BLPO2},
where logarithmic refinements of the Beilinson filtrations \cite{Ant19} are also constructed.

\subsubsection{Blow-up excision for $\mathrm{THR}$} Real topological Hochschild homology $\mathrm{THR}$ is another example of non $\A^1$-invariant cohomology theories.
The purpose of \cite{2305.04150} is to extend $\mathrm{THR}$ to log schemes and to show that it is $(\P^n,\P^{n-1})$-invariant.
Together with Theorem \ref{ProplogSH.5},
this implies that $\mathrm{THR}$ satisfies blow-up excision for closed immersions of smooth schemes.
This statement does not involve log schemes,
but the proof needs log schemes.

\subsubsection{Logarithmic prismatic cohomology via $\logTHH$} The purpose of \cite{BLPO2} is to study logarithmic refinements of the Bhatt-Morrow-Scholze filtrations \cite{BMS19} on $\THH$ and its variants.
In particular,
the definition of Nygaard completed logarithmic prismatic cohomology is obtained.
This study is further carried out in \cite{BLMP},
where
Binda, Lundemo, Merici, and Park
construct a $\P^1$-spectrum version of Nygaard completed logarithmic prismatic cohomology.
The theory of orientations in \S \ref{section:oct} again produces the Gysin maps for Nygaard completed logarithmic prismatic cohomology, and over a perfectoid base also for the non-Nygaard completed version. 
\subsubsection{Filtrations on $\Kthlog$ and $\logTC$} 
We can define the slice filtration on $\inflogSH$ that is directly analogous to Voevodsky's slice filtration \cite{MR1977582} on $\infSH$.
For a perfect field $k$ of characteristic $p$ admitting resolution of singularities,
\cite{BPO3} shows that the logarithmic cyclotomic trace morphisms
\[
\logKGL
\to
\mathbf{logHC}^-
\text{ and }
\logKGL
\to
\spectlogTC(-;\Z_p)
\text{ ($\mathrm{char}$ $p>0$)}
\]
are compatible with the slice filtration on the left-hand side and the Beilinson filtration on $\mathbf{logHC}^-$ \cite{BLPO2} and the Bhatt-Morrow-Scholze filtration on $\spectlogTC(-;\Z_p)$.
This argument breaks down in $\infSH(k)$ since $\mathrm{HC}^-$ and $\TC$ are not $\A^1$-invariant.

\subsubsection{Logarithmic cyclotomic trace}

By Theorem \ref{trace.8},
we obtain a natural log cyclotomic trace as a morphism of spectra
\(
\logTr
\colon
\Kth(X-\partial X)
\to
\logTC(X)
\)
for every log smooth fs log scheme $X$ with $\ul{X}$ smooth over a field $k$,
where $k$ is a perfect field admitting resolution of singularities.
A recent preprint \cite[Theorem F]{logSHF1} improves this and does not rely on resolution of singularities as follows: There exists a natural morphism of $\E_\infty$-rings
\(
\logTr
\colon
\Kth(X-\partial X)
\to
\logTC(X)
\)
for every log regular fs log scheme $X$ with $\ul{X}$ regular such that $\logTr$ agrees with the usual cyclotomic trace if $X$ has trivial log structure.

\subsection{Changes from the previous ArXiv version}
This paper was first published on ArXiv in March 2023 after being circulated informally since 2021. This version is a revision of the previously available ArXiv submission. For the readers' convenience, we highlight some of the most significant changes below.

First, we have minimized the use of model categories in this updated version, while the previous version employed model categories for certain constructions. As a result, we have streamlined many arguments in Section \ref{section:colmhc}.

We removed the noetherian assumption on \( S \) for \( \inflogSH(S) \). To achieve this, we introduced a logarithmic version of the noetherian approximation in Section \ref{noetherian}. This generalization is often desirable since many schemes in \( p \)-adic geometry are not noetherian. In a similar vein, some arguments are now available without the conventional “fs” assumption on the log scheme. Although we have not explored this systematically in this work, for instance, in \cite{BKV}, one must address non-fine log structures on integral perfectoid rings.

Additionally, in the previous version, \( \logTHH \) was defined only for noetherian fs log schemes. We have also eliminated this noetherian assumption by introducing an alternative equivalent definition, as outlined in Definition \ref{THHsheaf.2}.

\subsection{Notation}

We will use the following notation (see also the Index for a more comprehensive list).

\vspace{0.1in}

\begin{tabular}{l|l}
$\Sch$ &  quasi-compact and quasi-separated schemes
\\
$\Sch_{noeth}$ & noetherian schemes of finite Krull dimension
\\
$\lSch$ & quasi-compact and quasi-separated integral, saturated \\ &  quasi-coherent log schemes
\\
$\lSch_{noeth}$ & noetherian fs log schemes of finite Krull dimension
\\
$\Sm/S$ &  smooth schemes of finite type over a scheme $S$
\\
$\lSm/S$ &  log smooth fs log schemes of finite type over an fs log scheme $S\in \lSch$
\\
$\SmlSm/S$ &  full subcategory of $\lSm/S$ of log smooth fs log schemes $X$ such
\\ &
that the underlying scheme $\underline{X}$ is smooth over $\underline{S}$
\\
$\id\xrightarrow{ad} f_*f^*$ &  unit of an adjunction $(f^*,f_*)$
\\
$f^*f_*\xrightarrow{ad'}\id$ &  counit of an adjunction $(f^*,f_*)$
\\
$\Psh(\bC)$ &  presheaves (of sets) on a category $\bC$
\\
$\Shv_\tau(\bC)$ &  $\tau$-sheaves (of sets) on a category $\bC$ with respect to a topology $\tau$ 
\\
$\infPsh(\bC)$ &  $\infty$-category of presheaves of spaces on a category $\bC$
\\
$\infPsh(\bC)_*$ &  $\infty$-category of presheaves of pointed spaces on a category $\bC$
\\
$\infShv_\tau(\bC)$ &  $\infty$-category of $\tau$-sheaves of spaces on a category $\bC$
\\
$\infShv_\tau(\bC)_*$ &  $\infty$-category of $\tau$-sheaves of pointed spaces on a category $\bC$
\\
$\infSpc$ &  $\infty$-category of Kan complexes (spaces)
\\
$\infSpt$ &  $\infty$-category of spectra
\\
$\Map_{\infC}(-,-)$ & mapping space in an $\infty$-category $\infC$
\\
$\hom_{\infC}(-,-)$ & $\pi_0\Map_{\infC}(-,-)$
\\
$\map_{\infC}(-,-)$ & mapping spectra in a stable $\infty$-category $\infC$
\end{tabular}

\

We freely use some classical results from log geometry. Our main reference is Ogus' monograph on the subject \cite{Ogu}. See also Appendix A to \cite{logDM} for a short recollection. If $S$ is a (classical) scheme, $X$ is a smooth $S$-scheme and $D=D_1+\cdots +D_n$ is a strict normal crossing divisor on $X$ over $S$, we will use the ``adding boundary'' notation $(X,D)$ for the log smooth fs log scheme
with the underlying scheme $X$ and the log structure given by the Deligne-Faltings structure \cite[\S III.1.7]{Ogu}.
In this case, $\partial X:=D$ denotes the ``boundary'' of $X$.

If $Y$ is a log scheme, we will denote by $\underline{Y}$ the underlying scheme (without log structure, or considered as a log scheme with trivial log structure).

We will use the adding boundary notation for log schemes with nontrivial log structures too: If $Y$ is a log smooth fs log scheme over $S$ such that $\ul{Y}$ is smooth and $E$ is a divisor on $\ul{Y}$ such that $\partial Y+E$ is a strict normal crossing divisor on $\ul{Y}$ over $S$,
then we set
\[
(Y,E)
:=
(\ul{Y},\partial Y+E).
\]

We will often write $\infC$ instead of $\infC^\otimes$ for a symmetric monoidal $\infty$-category, abusively confusing $\infC$ with $\infC^{\otimes}_{\langle1\rangle}$.
For abbreviation, we will also write $\tau$ for the set of  \v{C}ech nerves of $\tau$-covers, where $\tau$ is a Grothendieck topology on a category $\cC$. For example, write $dNis$ for the set of  \v{C}ech nerves of $dNis$-covers.

\subsection{Acknowledgments}
We thank Ben Antieau, Joseph Ayoub, Marc Hoyois, Tommy Lundemo,  Alberto Merici and Teruhisa Koshikawa
for their encouragement and helpful comments.
The authors gratefully acknowledge the hospitality and support of the 
Centre for Advanced Study at the Norwegian Academy of Science and Letters in Oslo, Norway, 
which funded and hosted the research project “Motivic Geometry" during the 2020/21 academic year, 
and the RCN Frontier Research Group Project no. 250399 “Motivic Hopf Equations" and 
no. 312472 “Equations in Motivic Homotopy."
Binda would like to thank the Isaac Newton Institute for Mathematical Sciences for support and hospitality during the program ``$K$-theory, algebraic cycles, and motivic homotopy theory'' in 2022 when part of the work on this paper was carried out,
partially supported by the EPSRC Grant EP/R014604/1 (UK) and by the PRIN 2022 'The arithmetic of motives and L-functions' at MUR (Italy).  
Park was partially supported by the research training group GRK 2240 ``Algebro-Geometric Methods in Algebra, 
Arithmetic and Topology.''
{\O}stv{\ae}r was partially supported by a Guest Professorship awarded by The Radboud Excellence Initiative.
The work on this paper was supported by The European Commission -- Horizon-MSCA-PF-2022 
``Motivic integral $p$-adic cohomologies."

Finally, we are very grateful to the referee for the detailed and constructive comments, 
which significantly improved the exposition and simplified some proofs.

\newpage

\section{Construction of logarithmic motivic homotopy categories}
\label{section:colmhc}
The purpose of this section is to provide $\infty$-categorical constructions 
for our homotopy categories:
\[
\inflogH,
\;
\inflogHpt,
\;
\inflogSHS, 
\;
\inflogSH,
\;
\inflogDAeff,
\text{ and }
\inflogDA.
\]
The linear versions $\inflogDAeff= \inflogDAeff(S, \Lambda)$ and $\inflogDA(S, \Lambda)$ appeared before in
\cite{logDM}, \cite{BLPO}, and are constructed as localization and $\P^1$-stabilization of the category of presheaves of chain complexes of abelian groups on the category of log smooth log schemes over a fixed log scheme $S$. As in $\A^1$-homotopy theory, there are of course non-linear versions, constructed starting from presheaves of spaces or of spectra. 

The basic idea is the following. For every $S\in \Sch$, 
Morel-Voevodsky \cite{MV} constructed their unstable $\A^1$-motivic homotopy category $\infH(S)$ 
by using the following ingredients:
\[
\text{(1) $\infPsh(\Sm/S)$, (2) Nisnevich topology, (3) $\A^1$-localization}.
\]
Starting with any $S\in \lSch$, 
we construct our unstable logarithmic motivic homotopy category $\inflogH$ by using:
\[
\text{(1) $\infPsh(\lSm/S)$, (2) dividing Nisnevich topology, (3) $(\P^1, \infty)$-localization}.
\]
We will often abbreviate the log scheme $(\P^1, \infty)$ using the symbol $\boxx$. 
The dividing Nisnevich topology was introduced in \cite[Definition 3.1.4]{logDM} as a finer variant of the (strict) Nisnevich topology, obtained by including proper surjective log \'etale monomorphisms as covers.
It played an important role in the development of logarithmic motives in \cite[\S 7]{logDM}.

We will also consider an alternative but equivalent construction, practically useful in many contexts (this is used extensively in \cite{BLPO} and \cite{BLMP}), using 
\[
\text{(1) $\infPsh(\SmlSm/S)$, (2) strict Nisnevich topology, (3) $(\P^n, \P^{n-1})$-localization, $n\geq 1$}.
\]

In our treatment of the general formalism for the construction of motivic $\infty$-categories (suitably tailored to the particular situations of the $\boxx$-localization, or the $(\P^n, \P^{n-1})_{n\geq 1}$-localization), we follow the presentation of Ayoub--Gallauer--Vezzani \cite{AGV}.

\subsection{Stabilizations, spectra and symmetric spectra}\label{sec:spectra} We briefly review the process of constructing the stabilization of an $\infty$-category, and the formal inversion of a given object, see \cite[\S 2.2]{RobaloThesispublished}.

\begin{const}
\label{alter.4}
Let $\infC\in \LPr$ be a presentable $\infty$-category, and suppose we are given an adjunction
\begin{equation}
\label{alter.4.4}
G:
\infC \rightleftarrows \infC
:U.
\end{equation}
The \emph{$G$-stabilization of $\infC$} is defined to be\index[notation]{StabGC @ $\Stab_G(\infC)$}
\begin{equation}
\label{alter.4.3}
\Stab_G(\infC)
:=
\limit(\cdots \xrightarrow{U}\infC \xrightarrow{U}\infC),
\end{equation}
where the limit is taken in $\infCat_\infty$, and we view the diagram $(\cdots \xrightarrow{U}\infC \xrightarrow{U}\infC)$ as a functor of $\infty$-categories $\N^\mathrm{op}\to \infCat_\infty$.

Observe that by \cite[Corollary 3.3.3.2]{HTT}, an object of $\Stab_G(\infC)$ can be written as a sequence of objects $(X_i)_{i\in \mathbb{N}}$ of $\infC$,
together with equivalences (the \emph{bonding maps}) $X_i \xrightarrow{\simeq} U X_{i+1}$ in $\cC$ for all integers $i\geq 0$. 
Such a sequence is called a \emph{$G$-spectrum}\index{spectrum}. 
Morphisms are given by maps of sequences with obvious compatibility conditions. 
 \begin{rmk} 
The inclusion $\RPr\to \infCat_\infty$ preserves small limits by \cite[Proposition 5.5.7.6]{HTT}, and the equivalence of $\infty$-categories
\begin{equation}
\label{alter.4.1}
(\RPr)^\mathrm{op}\xrightarrow{\simeq} \LPr
\end{equation}
in \cite[Corollary 5.5.3.4]{HTT} sends a functor in $\RPr$ to its left adjoint.
Then from \eqref{alter.4.1} we obtain an equivalence of $\infty$-categories
\begin{equation}
\label{alter.4.2}
\Stab_G(\infC)
\simeq
\colim_{\LPr}(\infC\xrightarrow{G}\infC \xrightarrow{G}\cdots)
\end{equation}
where the colimit is taken in $\LPr$, and we view the diagram $(\infC\xrightarrow{G}\infC \xrightarrow{G}\cdots)$ as a functor of $\infty$-categories $\N\to \LPr$.
\end{rmk}

The identity functors from $\cC$ to the 0th $\cC$ in the colimit \eqref{alter.4.2} and from the 0th $\cC$ in the limit \eqref{alter.4.3} to $\cC$ naturally induce functors\index[notation]{SigmaGinfty @ $\Sigma_G^\infty$}\index[notation]{OmegaG @ $\Omega_G^\infty$}
\begin{equation}
G^\infty
:
\infC
\rightleftarrows
\Stab_G(\infC)
:
U^\infty.
\end{equation}
Using \eqref{alter.4.1}, this can be shown to be an adjoint pair.
The functor $G^\infty$ is called the \emph{infinite suspension functor}, and the functor $U^\infty$ is called the \emph{infinite loop space functor}.

For a $G$-spectrum $X:=(X_0,X_1,\ldots)$, we have 
\begin{equation}
U^\infty X\simeq X_0.
\end{equation}
\end{const}

We will use the construction of $\Stab_{G}(\infC)$ when $\infC = \infC_{\langle1\rangle}^{\otimes}$ for a presentably symmetric monoidal  $\infty$-category $\infC^\otimes$ (i.e.,~an object of $\mathrm{CAlg}(\Pr^{\rm L})$) and $G:=\Sigma_B:=(-)\otimes B$ for some object $B\in \infC$. We denote by $\Omega_B$ its right adjoint, and by $\Stab_B(\infC)$ the resulting $\infty$-category. 

In good cases, the $\Sigma_B$-stabilization is equipped with a particularly nice symmetric monoidal structure, using the presentation given by the $\infty$-category of symmetric $B$-spectra in $\infC$ introduced by Hovey in \cite{HoveySpectra}.  Recall  from \cite[Definition 2.16]{RobaloThesispublished} that an object $B$ in a symmetric monoidal $\infty$-category $\cC^\otimes$ is called \emph{symmetric} if the cyclic permutation $\sigma_3$ on $B\otimes B \otimes B$ is equivalent to the identity.

\begin{prop}\label{prop:def-spt}\label{alter.1} Assume that $B$ is a symmetric object of $\infC^\otimes$, and let $\infSpt_B^\Sigma(\infC)^\otimes$ denote the $\infty$-category of \emph{symmetric spectra} as in \cite[Section 2.3]{RobaloThesispublished}.
    Then $\infSpt_B^\Sigma(\infC)^\otimes$ is a  stable presentably symmetric monoidal $\infty$-category, and it is naturally equivalent to the formal inversion  of the object $B$ in $\infC$ constructed in \cite[Proposition 2.9]{RobaloThesispublished}. The underlying $\infty$-category $\infSpt_B^\Sigma(\infC)_{\langle 1 \rangle}$ is equivalent to $\Stab_B(\infC)$. 
    \begin{proof}
Combine \cite[Corollary 2.22, Theorem 2.26]{RobaloThesispublished}.
    \end{proof}
\end{prop}
The construction provides a morphism $\Sigma^\infty_B\colon \infC^\otimes \to \infSpt_B^\Sigma(\infC)^\otimes$ in $\CAlg(\LPr)$, sending $B$ to a $\otimes$-invertible object, and the category $\infSpt_B^\Sigma(\infC)^\otimes$ is initial for this property. We denote by $\Omega^\infty_B$ its right adjoint. 
We refer to \cite{RobaloThesispublished} for more details.\index[notation]{infSptB @ $\infSpt_B^\Sigma$}

\subsection{Topologies on log schemes}\label{ssec:topologies}
In this subsection, 
we will review various cd-structures and topologies on  log schemes that will be used throughout this paper, and we briefly review some useful notions from log geometry that we will use in a slightly non-standard level of generality. We refer to \cite[Appendix A]{Koshikawa} for a more detailed discussion. All schemes are assumed to be quasi-compact and quasi-separated.  

Recall that for a pre-log ring $(A,M_A)$, we write $\Spec{A,M_A}$ for the log scheme corresponding to $(A,M_A)$, that is, the log scheme whose underlying scheme is $X=\Spec{A}$ and with log structure given by $(M_A\to \mathcal{O}_X)^a$, where $(-)^a$ denotes the associated log structure to the pre-log structure $M_A\to \mathcal{O}_X$. Here, we see $M_A$ as constant sheaf of monoids on $\Spec{A}$.  In this paper we are only interested in quasi-coherent log schemes, i.e., log schemes $X=(\underline{X},\mathcal{M}_X)$ such that the sheaf of monoids $\mathcal{M}_X$ is quasi-coherent in the sense of \cite[Definition II.2.1.5]{Ogu}, i.e., such that Zariski locally on $X$ there is a chart. 

\begin{df} \label{def:logsmooth}
Let $S$ be a  quasi-coherent integral log scheme. A morphism $X\to S$ of quasi-coherent integral log schemes is \emph{log smooth}\index{log smooth morphism} (resp.\ \emph{log \'etale}\index{log \'etale morphism}) if, \'etale locally on $\underline{X}$ and $\underline{S}$ there exists a fine log structure  $\mathcal{M}_{S, 0}$ on $\underline{S}$ such that $\mathcal{M}_{S,0} \subseteq \mathcal{M}_S$, and a log smooth (resp.\ log \'etale) morphism $(\underline{X}, \mathcal{M}_{X,0}) \to (\underline{S}, \mathcal{M}_{S,0})$ in the sense of \cite[Definition IV.3.1.1]{Ogu} or \cite[\S 3.3]{KatoLog} such that $X \cong (\underline{X}, \mathcal{M}_{X,0}) \times_{(\underline{S}, \mathcal{M}_{S,0})} S$.

Note that the class of log smooth (resp.\ log \'etale) morphisms in the previous sense is closed under composition and base-change (in the category of integral log schemes). Also, note that by assumption, the underlying morphism of schemes $\underline{X}\to \underline{S}$ is locally of finite presentation. 
\end{df}
\begin{rmk} Let $S$ be a quasi-coherent integral log scheme. 
    In light of \cite[Proposition 2.4]{BKV}, one has the following result. Let $f\colon X\to Y$ be any morphism of log smooth log schemes over $S$. Then, \'etale locally on $X$, $Y$ and $S$, there is a fine log structure $\mathcal{M}_{S,0} \subseteq \mathcal{M}_{S}$ on $\underline{S}$ and a morphism $f_0\colon (\underline{X}, \mathcal{M}_{X,0})\to (\underline{Y}, \mathcal{M}_{Y,0})$ of log smooth log schemes over $(\underline{S}, \mathcal{M}_{S,0})$ such that $f$ is the pullback of $f_0$ along $(\underline{S}, \mathcal{M}_{S,0}) \to S$.  See also Proposition \ref{noetherian.3} below.
\end{rmk}

\begin{df}\label{def:chart_logsmooth}
    Let $S=\Spec{A,M_A}$ be an affine integral log scheme. A morphism $X\to S$ of log schemes is \emph{chart log smooth} (resp.\ \emph{chart log \'etale}) if, \'etale locally on $\underline{X}$, there exists a chart $P \to \Gamma(X, \mathcal{M}_X)$ over $M_A$ such that the following conditions hold.
    \begin{enumerate}
        \item $P$ is integral and relatively coherent over $M_A$, that is, $P$ is finitely presented over $M_A$;
        \item the map $M_A^\gp \to P^\gp$ is injective, and the torsion part of its cokernel is a finite group of order invertible in $A$ (resp.\ the cokernel is a finite group of order invertible in $A$);
        \item the induced morphism of schemes $X\to \Spec{A} \times_{\Spec{A[M_A]}} \Spec{A[P]}$ is \'etale.
    \end{enumerate}
\end{df}
\begin{rmk}
   (1) This definition depends a priori on the choice of the chart $M_A\to A$. However, if $M_A$ and $\mathcal{M}_X$ are fine, then the above notion of smoothness coincides with Kato's original definition of smoothness in terms of infinitesimal lifting property \cite[\S 3.3]{KatoLog}.  In particular, it is independent of the choice of the chart. See \cite[Remark A.13]{Koshikawa}.

    (2) One could insist on an additional condition, namely that the morphism $M_A\to P$ is integral. This guarantees that the underlying morphism of schemes is flat. This is automatically satisfied when $M_A$ is a valuative monoid in light of \cite[Proposition I.4.6.3(5)]{Ogu}. 
    \end{rmk}
     \begin{prop}Let $f\colon X\to S$ be a morphism of integral quasi-coherent log schemes. Then $f$ is log smooth in the sense of Definition \ref{def:logsmooth} if and only if, strict \'etale locally on $X$ and $S$, it is chart log smooth.
     \begin{proof} This is essentially discussed in \cite{Koshikawa}: we review the argument. A log smooth morphism in the sense of Definition \ref{def:logsmooth} is chart log smooth, since locally it is the base change of a log smooth morphism locally of finite presentation between fine log schemes, and Kato's chart lemma applies. Conversely, assume that $S= \Spec{A,M_A}$ and that $f$ satisfies the conditions of Definition \ref{def:chart_logsmooth}. 
     The assumption that $P$ is finitely presented (as monoid) over $M_A$ implies that $A[P]$ is of finite presentation (as a ring) over $A[M_A]$ so that, in particular, $\underline{X}$ is locally of finite presentation over $A$. By \cite[Proposition A.15]{Koshikawa}, this implies the following: there exists a \emph{fine} submonoid $M_{A,0} \subset M_A$ and a fine log structure $\mathcal{M}_{X,0}$ on $X$ with a chart $M_{A,0} \to P_0$ such that $(X, \mathcal{M}_{X,0}) \to \Spec{A, M_{A,0}}$ is log smooth in the sense of \cite[\S 3.3]{KatoLog}, and $(X, \mathcal{M}_X) = (X, \mathcal{M}_{X,0})\times_{\Spec{A, M_{A,0}}} \Spec{A, M_{A}}$. This completes the proof. \end{proof}
     \end{prop}

\begin{rmk}
    From this moment on, we will further restrict our objects of study to be quasi-coherent \emph{saturated} log schemes. This condition is commonly satisfied in all concrete applications.
\end{rmk}

We freely use some basic properties of toric varieties and fans. Suppose that $P$ is an fs monoid.
Let $\Spec{P}$ be the monoscheme associated with $P$ as in \cite[Definition II.1.2.1]{Ogu}
(denoted by $\mathrm{spec}(P)$ in loc.\ cit.).
Let $\A_P$ be the fs affine log scheme $\Spec{P\to \Z[P]}$, 
see the discussion following \cite[Definition III.1.2.3]{Ogu}.

There is a fully faithful functor from the category of toric fans to the category of monoschemes, 
see \cite[Theorem II.1.9.3]{Ogu}.
If $P$ is an fs monoid such that $P^\gp$ is torsion-free, then $\Spec{P}$ corresponds to the fan in the lattice $(P^\gp)^\vee$ with a single maximal cone $P^\vee$, 
where $(-)^\vee$ denotes the dual cone or lattice.
If $\Sigma$ is a toric fan, 
we let $\T_\Sigma$ denote the associated log scheme in the sense of \cite[p.\ 276, \S III.1.2]{Ogu}.
Observe that the underlying scheme $\ul{\T_{\Sigma}}$ is the toric variety associated with $\Sigma$.
For every fs log scheme $X$,
we set
\[
X-\partial X
:=
\{x\in X:
\ol{\cM}_{X,x}=0\}.
\]
This is an open subset of $X$ by \cite[Proposition III.1.2.8]{Ogu}.
We regard $X-\partial X$ as an open subscheme of $X$.
\begin{rmk}
In fact, the proof of \cite[Proposition III.1.2.8]{Ogu} works under the weaker assumption that $X$ is coherent, that is, if Zariski locally on $X$ there is a chart subordinate to a finitely generated monoid (in particular, without the assumption of being integral or saturated).
\end{rmk}

\begin{df}\label{df:dividing_cover}
Suppose $f\colon Y\to X$ is a morphism in $\lSch$.
We say that $f$ is a \emph{dividing cover} \index{dividing cover} if $f$ is a proper surjective log  \'etale monomorphism.

In this case, the naturally induced morphism of schemes $Y-\partial Y\to X-\partial X$ is an isomorphism by 
\cite[Remark 3.1.7(1)]{logDM}.
\end{df}

\begin{exm}
\label{logtop.8}
Every log blow-up \cite[Definition III.2.6.2]{Ogu} is a dividing cover, see \cite[Example A.11.1]{logDM}.
\end{exm}

\begin{exm}
\label{logtop.8bis} Here's an elementary way to produce dividing covers.
Suppose $X$ is an fs log scheme with a strict morphism $X\to \T_\Sigma$, 
where $\Sigma$ is a fan.
For any subdivision of fans $\Sigma'\to \Sigma$, 
the naturally induced morphism $\T_{\Sigma'}\to \T_{\Sigma}$ is Zariski locally of the form 
$\A_{\theta}\colon \A_Q\to \A_P$, 
where $\theta\colon P\to Q$ is a monomorphism of fs monoids such that $P^\gp\to Q^\gp$ is an isomorphism of 
abelian groups.
Hence the induced morphism $\mathbb{A}_{\theta}$ is log \'etale by \cite[Corollary IV.3.1.10]{Ogu}. By base change, the same holds for the projection $p\colon X\times_{\T_\Sigma}\T_{\Sigma'}\to X$.
Since $\ul{\T_{\Sigma'}}\to \ul{\T_{\Sigma}}$ is proper, the same holds for  $p$.
By \cite[Lemma A.11.4]{logDM}, $p$ is universally surjective.
In conclusion, $p$ is a dividing cover. 
Geometrically, any fan subdivision as above is realized as a blow-up of the corresponding toric variety $\mathbb{T}_\Sigma$. A typical example of a fan subdivision arises in the resolution of singularities of $\mathbb{T}_\Sigma$. 
\end{exm}

\begin{rmk}
\label{Ori.75}
A morphism $f\colon Y\to X$ is called a \emph{log modification} \index{log modification} if Zariski locally on $X$, 
there exists a log blow-up $g\colon Z\to Y$ such that $fg$ is a log blow-up too.
In \cite[Proposition A.11.9]{logDM}, it was claimed that a log modification is the same as a dividing cover.
However, this is false, see the second part of \cite[Example IV.4.3.4]{Ogu}.
To fix this, 
one replaces ``log modification'' by ``dividing cover'' in \cite{logDM}, 
so that \cite[Proposition A.11.9]{logDM} becomes obsolete.
Example A.11.8 is the only part in \cite{logDM} where the definition of log modification is used;
however, we can replace this by Example \ref{logtop.8bis}.
In summary, 
this correction does not affect the content of \cite{logDM}.
\end{rmk}

We refer to \cite[Definition 3.3.22]{logDM} and \cite[Definition 2.10]{Vcdtop} for the notions of 
quasi-bounded cd-structures and regular cd-structures. Recall also that a cd-structure $P$ is called \emph{squarrable}\index{squarrable} if for every morphism $Z\to X$, 
the fiber product $Q\times_X Z$ is representable and belongs to $P$.

\begin{df}
\label{logtop.2}
Suppose $S\in \lSch$.
\begin{enumerate}
\item[(1)]
The \emph{strict Nisnevich} cd-structure\index{strict Nisnevich cd-structure} consists of cartesian squares of the form
\[
\begin{tikzcd}
Y'\ar[r,"g'"]\ar[d,"f'"']&
Y\ar[d,"f"]
\\
X'\ar[r,"g"]&
X
\end{tikzcd}
\]
in $\lSch/S$ such that $g$ is an open immersion, $f$ is strict Nisnevich, and the morphism $f^{-1}(\ul{X'}-g(\ul{X}))\to \ul{X'}-g(\ul{X})$ with reduced scheme structures is an isomorphism.
\item[(2)]
The \emph{dividing} cd-structure \index{dividing cd-structure} consists of cartesian squares of the form
\[
\begin{tikzcd}
\emptyset\ar[d]\ar[r]&
Y\ar[d,"f"]
\\
\emptyset\ar[r]&
X
\end{tikzcd}
\]
in $\lSch/S$ such that $f$ is a dividing cover in the sense of Definition \ref{df:dividing_cover}.
\end{enumerate}
The \emph{dividing Nisnevich} cd-structure \index{dividing Nisnevich cd-structure} 
is the union of the strict Nisnevich cd-structure and dividing cd-structure.

On $\lSch$ we define the \emph{strict Nisnevich topology} to be the topology associated with the strict Nisnevich cd-structure. Similarly, we define the \emph{dividing Nisnevich topology} on $\lSch$ as the topology associated with the dividing Nisnevich cd-structure.
Analogously, one can define the \emph{dividing topology} as the topology associated with the dividing cd-structure.
\index{strict Nisnevich topology}\index{dividing topology}\index{dividing Nisnevich topology.}

We let $sNis$, $dNis$, and $div$ be shorthand for these 
topologies.\index[notation]{sNis @ $sNis$}\index[notation]{div @ $div$}
\end{df}\index[notation]{dNis @ $dNis$}

\begin{rmk}
Every proper surjective \'etale monomorphism of schemes is an isomorphism by \cite[Th\'eor\`eme IV.17.9.1]{EGA}.
On the other hand, 
Example \ref{logtop.8bis} yields oodles of 
proper surjective log  \'etale monomorphisms between fs log schemes that are not isomorphisms. 
The basic underlying idea behind the dividing topology is to eliminate this discrepancy. 
Indeed, 
every dividing cover becomes an isomorphism in the category of dividing sheaves. 

The Nisnevich topology plays a fundamental role in the theory of motives, e.g., in \cite{MV}.
The strict Nisnevich topology is our log geometric generalization of the Nisnevich topology for schemes.
To our delight, 
the dividing Nisnevich topology inherits the desirable properties of both the strict Nisnevich topology 
and the dividing topology.
\end{rmk}

\begin{rmk}
    Let us offer a different perspective on the dividing topology. Elementary examples of log schemes are constructed by considering compactifications of classical schemes, that is, embeddings into proper schemes over the base. The choice of such embedding is not canonical, as we can clearly construct new compactifications from old ones by taking blow-ups with center in the boundary, which is precisely the locus where the log structure is nontrivial. The effect of imposing dividing descent in the category of log motives is to get rid of some of these choices, by inverting a large class of blow-ups with center in the support of the log structure.
\end{rmk}

\begin{exm}
The cartesian commutative square
\[
\begin{tikzcd}
\A_{\N}\ar[d]\ar[r]&
\boxx\ar[d]
\\
\A^1\ar[r]&
\P^1
\end{tikzcd}
\]
is not a strict Nisnevich distinguished square, 
despite the claim made to the contrary in \cite[Example 5.2.11]{logDM} 
since the morphism $\boxx\to \P^1$ is not strict.
Nevertheless, it is cocartesian in the category of log motives.
We refer to Proposition \ref{logH.4} for a more general statement.
\end{exm}

We will often consider other topologies in order to construct useful motivic categories. It is always possible to ``pull back'' a topology from schemes to log schemes, as follows.

\begin{df}
\label{stricttau}
For a topology $\tau$ of schemes,
the \emph{strict $\tau$-topology on log schemes} consists of the families of strict morphisms $\{U_i\to X\}_{i\in I}$ such that $\{\ul{U_i}\to \ul{X}\}_{i\in I}$ is a $\tau$-covering.
An example of strict topology is the strict Nisnevich topology mentioned above. 
As another example, the \emph{strict \'etale topology}\index{strict \'etale topology} is generated by the families $\{u_i\colon U_i\to X\}_{i\in I}$ such that each $u_i$ is strict \'etale and $\amalg_{i\in I}U_i\to X$ is surjective. We write $\set$, for this topology. \index[notation]{set @ $\set$}
\end{df}

Recall that a morphism $\theta\colon P\to Q$ of integral monoids is called Kummer if it is injective and $\mathbb{Q}$-surjective (that is, for all $x\in Q$, there exists $n\in \Z$ such that $nx\in \theta(P)$). A morphism of integral log schemes $X\to Y$ is Kummer if the induced homomorphisms
$\ol{\cM}_{Y,f(x)}\to \ol{\cM}_{X,x}$ are Kummer for all $x\in X$.
We refer to  \cite[\S 2]{MR1922832} for a list of properties of Kummer morphisms and Kummer \'etale morphisms (that is, log \'etale and Kummer) of log schemes. 

\begin{df}
\label{logtop.3}
On $\lSch$, we have the following topologies.
\begin{enumerate}
\item[(1)]
The \emph{Kummer \'etale topology}\index{Kummer \'etale topology} is generated by the families $\{u_i\colon U_i\to X\}_{i\in I}$ such that each $u_i$ is Kummer \'etale and $\amalg_{i\in I}U_i\to X$ is surjective. 
\item[(2)]
The \emph{log \'etale topology}\index{log \'etale topology} is generated by the families $\{u_i\colon U_i\to X\}_{i\in I}$ such that each $u_i$ is log \'etale and $\amalg_{i\in I}U_i\to X$ is \emph{universally surjective}.
\item[(3)]
The \emph{dividing \'etale topology}\index{dividing \'etale topology} is the smallest topology finer than both the dividing topology and strict \'etale topology.
\end{enumerate}
Let $\ket$, $\leta$, and $d\et$ be shorthand notations for these topologies.\index[notation]{ket @ $\ket$}\index[notation]{leta @ $\leta$}\index[notation]{det @ $d\et$}
\end{df}

According to \cite[Proposition 2.2.2]{MR1457738}, any exact surjective morphism is universally surjective.
Therefore, in (1) and (2), the morphism $\amalg_{i\in I}U_i\to X$ is universally surjective.

\begin{df}
\label{logtop.4}
Suppose $S\in \Sch$ and let $\SmlSm/S$\index[notation]{SmlSmS @ $\SmlSm/S$} be the full subcategory of 
$\lSm/S$ consisting of $Y\in \lSm/S$ such that $\ul{Y}\in \Sm/S$.

Suppose $X\in \Sm/S$ and $D$ is a strict normal crossing divisor on $X$ over $S$ as in \cite[Definition 7.2.1]{logDM}.
Let $(X,D)$ denote the fs log scheme
whose underlying scheme is $X$ and whose log structure is given by the Deligne-Faltings structure \cite[\S III.1.7]{Ogu}.
Every object of $\SmlSm/S$ arises as this form, see \cite[Lemma A.5.10]{logDM} and Remark \ref{logtop.9} below.

We set $\boxx:=(\P^1,\infty)$,\index[notation]{box @ $\boxx$} where we regard the point at $\infty$ as an effective Cartier divisor.

For $Y\in \SmlSm/S$ and a divisor $E$ on $\ul{Y}$ such that $\partial Y+E$ is a strict normal crossing divisor on $\ul{Y}$ over $S$,
we will use the ``adding boundary'' notation
\[
(Y,E)
:=
(\ul{Y},E+\partial Y).
\]
\end{df}

The category $\SmlSm/S$ is simpler to deal with than $\lSm/S$, 
but it is unclear how to generalize a useful notion of $\SmlSm/S$ to an arbitrary fs log base scheme $S$.

\begin{rmk}
\label{logtop.9}Let $S$ be a scheme.
The proof of \cite[Lemma A.5.9]{logDM} needs a small modification: it is claimed in loc.\ cit.\ that for a log smooth morphism of fs log schemes $f\colon X\to S$, Zariski locally on $X$,
there exists a neat chart $P$ on $X$ such that the induced morphism $X\to S\times \A_P$ is strict smooth.
However, the claim in \cite[Theorem IV.3.3.1]{Ogu} is \'etale local,
which implies that the above claim holds only \'etale locally on $X$.
To fix this,
use \cite[Theorem III.1.2.7(4)]{Ogu} to obtain a neat chart $0\to P$ of $f$ Zariski locally on $X$, 
which is possible since $P^\gp$ is torsion free.
With this neat chart in hand,
argue as in \cite[Theorem IV.3.3.1]{Ogu} to conclude.
\end{rmk}

\begin{rmk}
\label{logtop.10}
The proof of \cite[Lemma A.11.3]{logDM} has a similar issue.
We refer to \cite[Proposition 3.4.1]{Parthesis} for the correct proof.
\end{rmk}

\begin{prop}
\label{logtop.5}
Suppose $S\in \Sch$.
Then the category $\SmlSm/S$ admits products.
\end{prop}
\begin{proof}
Suppose $(X,D),(Y,E)\in \SmlSm/S$.
The product $(X,D)\times_S (Y,E)$ in $\lSm/S$ can be written as $(X\times_S Y,X\times_S E+D\times_S Y)$, which is in $\SmlSm/S$.
\end{proof}

\begin{prop}
\label{logtop.6}
Suppose $S\in \lSch_{noeth}$.
The strict Nisnevich topology and the dividing Nisnevich topology on $\lSm/S$ and $\lSch/S$ are complete, quasi-bounded, and regular.
\end{prop}
\begin{proof}
We refer to \cite[Proposition 3.3.30]{logDM}.
\end{proof}

For further reference, we record a comparison result for Nisnevich distinguished squares.

\begin{lem}
\label{K-theory.5}
Suppose
\begin{equation}
\label{K-theory.5.1}
\begin{tikzcd}
Y'\arrow[d,"f'"']\arrow[r,"g'"]&
Y\arrow[d,"f"]
\\
X'\arrow[r,"g"]&
X
\end{tikzcd}
\end{equation}
is a strict Nisnevich distinguished square.
Then
\begin{equation}
\label{K-theory.5.2}
\begin{tikzcd}
Y'-\partial Y'\arrow[d]\arrow[r]&
Y-\partial Y\arrow[d]
\\
X'-\partial X'\arrow[r]&
X-\partial X
\end{tikzcd}
\end{equation}
is a Nisnevich distinguished square.
\end{lem}
\begin{proof}
Since $f$, $f'$, $g$, and $g'$ are strict, 
the square \eqref{K-theory.5.2} is the pullback of \eqref{K-theory.5.1} along the canonical open immersion 
$X-\partial X\to X$. The fact that $\eqref{K-theory.5.2}$ is then a Nisnevich square is obvious.
\end{proof}

\begin{rmk}
The dividing Nisnevich topology is \emph{not} subcanonical.
For example,
if $Y\to X$ is a dividing cover of fs log schemes that is not an isomorphism,
then the naturally induced map $\hom_{\lSch}(X,Y)\to \hom_{\lSch}(Y,Y)$ is not a bijection.
This means that the presheaf represented by $Y$ is not a dividing sheaf.
Nevertheless, 
for any $X\in \lSm/S$, 
we will often use the notation $X$ for the sheaf represented by $X$ to avoid lengthy notation.
We often use the notation $\pt:=S$ for the log base scheme.
\end{rmk}

\subsection{Effective logarithmic motives}
\label{alter}
 We now proceed to the definition of the effective version of our motivic categories.

\begin{notat}
For an $\infty$-category $\infC$, write $\infPsh(\infC)$ for the $\infty$-category of presheaves of spaces (or Kan complexes) on $\infC$. When $\infC$ is endowed with a Grothendieck topology $\tau$, we write $\infShv_\tau(\infC) \subseteq \infPsh(\infC)$ for the full subcategory spanned by $\tau$-sheaves.
For a presentable $\infty$-category $\infD$,
write $\infPsh(\infC,\infD)$ for the $\infty$-category of presheaves on $\infC$ with values in $\cD$ and $\infShv_\tau(\infC,\infD)\subseteq \infPsh(\infC,\infD)$ for the full subcategory spanned by $\tau$-sheaves.
For $X\in \infC$,
we use the notation $\Sigma_{S^1}^\infty X_+$ for the sheaf of spectra representable by $X$ when $\tau$ is clear from the context.
\end{notat}

\begin{rmk}[Pointed variants] Recall that for any $\infty$-category $\infC$ there is a canonical equivalence of pointed $\infty$-categories  $\infPsh(\infC,\infSpcpt)
\simeq
\infPsh(\infC)_\ast$. So we will freely identify them. Similarly, we will use the notation $\infShv_\tau(\infC)_\ast$ to denote the category of $\tau$-sheaves of pointed spaces.
\end{rmk}

\begin{df}
Let $\cC$ be a category, and let $\infD$ be a presentable $\infty$-category. Let $\cW$ be a set of morphisms in $\cC$. Let $\cF\in \infPsh(\cC, \infD)$.
We say that $\cF$ is \emph{$\cW$-local} if $\cF(f)$ is an equivalence for all $f\in \cW$. A morphism $f\colon\cF\to \cF'$ is a \emph{$\cW$-equivalence} if $\Map(\cF,\cG) \simeq \Map (\cF', \cG)$ for every $\cW$-local object $\cG$.
\end{df}

\begin{rmk} Suppose that $\infC$ admits small products. Then the Yoneda functor $y\colon \infC \to \infPsh(\infC)$ can be lifted to a symmetric monoidal functor $y\colon \infC^{\times} \to \infPsh(\infC)^{\otimes}$, where the latter is equipped with the Day convolution product by \cite[Corollary 4.8.1.12]{HA}. If $\infC$ is  equipped with a Grothendieck topology $\tau$, then $\infShv_\tau(\infC)$ underlies a symmetric monoidal $\infty$-category $\infShv_\tau(\infC)^\otimes$ such that the sheafification functor $L_\tau$ lifts to a symmetric monoidal functor.  
Indeed, it is enough to notice that for every $Y\in \infC$ and \v{C}ech nerve of the $\tau$ cover $\sX\to X$ with 
$X\in \cC$ the projection 
$\sX\times Y\to X\times Y$ is the \v{C}ech nerve of a $\tau$-cover too.
This is immediate from the definitions.
\end{rmk}

We define our basic objects as follows:
\begin{df}\label{def:logH_andvariants}
Let $S\in \lSch$.
We set 
\[\inflogSH^{\eff}(S)\] to be the full $\infty$-subcategory of $\infShv_{dNis}(\lSm/S,\infSpt)$ spanned by the objects $M$ that are local with respect to the maps $\Sigma_{S^1}^\infty(X \times \boxx)_+ \to \Sigma^\infty_{S^1}X_+$ and their desuspensions. Write $L_\boxx$ for the left adjoint to the inclusion functor.

The category $\inflogSH^{\eff}(S, \Lambda)$ will be called the \emph{effective} or \emph{$S^1$-stable} log motivic homotopy category of $S$. Objects of $\inflogSH^{\eff}(S)$ will be called (effective) log-$S$-motives.

In the previous ArXiv version, $\inflogSH^\eff(S)$ was denoted by $\inflogSH_{S^1}(S)$.
\end{df}

\begin{rmk}
To simplify the notation, write $\boxx$ for the set of maps $X\times \boxx\to X$ for $X\in \lSm/S$.
Then we have an equivalence
\[
\inflogSH^\eff(S)
:=
\infShv_{dNis}(\lSm/S,\infSpt)[\boxx^{-1}]
\simeq
\infPsh(\lSm/S,\infSpt)[(\boxx, dNis)^{-1}],
\]
where $(\boxx, dNis)$ denotes the saturation of the union of the set of $\boxx$-local equivalences and of the \v{C}ech nerves of dividing Nisnevich covers.
We write $L_{\boxx, dNis} = L_{\boxx}\circ L_{dNis}$ for the $(\boxx, dNis)$-localization functor. 
\end{rmk}

\begin{rmk}\label{alter.1.2}
Consider the category $\infShv_{dNis}(\lSm/S)_*[\boxx^{-1}]$ or, equivalently, the full subcategory of $\infShv_{dNis}(\lSm/S)_*$ of presheaves of pointed spaces spanned by $\boxx$-local  objects. Let $\Sigma_{S^1}$
be the $S^1$-suspension functor, induced by $\cF\mapsto S^1\wedge \cF$ for any pointed presheaf $\cF$: it clearly induces a functor on the localized category. Then we have an equivalence of $\infty$-categories $\inflogSHS(S)
\simeq \colim_{\LPr}(\infShv_{dNis}(\lSm/S)_*[\boxx^{-1}])$, where the colimit is given by the iteration of the functor $\Sigma_{S^1}$. This 
 exhibits $\inflogSH^\eff(S)$ as the $S^1$-stabilization  of $\infShv_{dNis}(\lSm/S)_*[\boxx^{-1}]$ (see \cite[Proposition 1.4.2.2]{HA}).
\end{rmk}

Note that clearly the $\infty$-category $\inflogSH^{\eff}(S)$ is presentable, stable, and is generated under colimits and up to desuspensions by the family of $\Sigma^\infty_{S^1}X_+$ for all $X\in \lSm/S$. Moreover, it  underlies a presentable symmetric monoidal $\infty$-category, where the monoidal structure is the Day convolution of the cartesian product in $\lSm/S$.

\begin{rmk}
Observe that the tensor product operation in  $\inflogSH^\eff(S)$ for every $S\in \lSch$ preserves small colimits  in each variable. 
\end{rmk}

\begin{prop}\label{prop:is_premotivic}
Let $\Lambda$ be a connective commutative ring spectrum as in Notation \ref{not:Lambda}. Then
The assignment 
\index[notation]{logSHeff @ $\inflogSH_{\tau}^\eff$}
\index[notation]{logSHveff @ $\inflogSH_{\tau}^{\eff, (\wedge)}$}
\[
S \mapsto \inflogSH^{\eff}(S)
\]
can be promoted into a functor 
\[
     \inflogSH^{\eff}\colon \lSch^\mathrm{op} \to \CAlg( \LPr)
\]
satisfying the following properties:
\begin{enumerate}
    \item for every morphism $f\colon X\to Y$ in $\lSch$, denote by $f^*$ the functor
    \[f^*\colon  \inflogSH^{\eff}(Y) \to \inflogSH^{\eff}(X) \]
    given by the image of $f$. Then $f^*$ is monoidal, and has a right adjoint $f_*$,
    \item for every log smooth morphism $f\colon X\to Y \in \lSch$, the functor $f^*$ has a left adjoint $f_\sharp$, and the canonical map
    \[ f_\sharp(f^* M \otimes N) \xrightarrow{\simeq} M \otimes f_\sharp N\]
    is an equivalence;
    \item for every cartesian square 
    \[
    \begin{tikzcd}
Y'\arrow[d,"f'"']\arrow[r,"g'"]&Y\arrow[d,"f"]
\\
X'\arrow[r,"g"]&X
\end{tikzcd}
    \]
    where $f$ is a log smooth morphism, the Beck-Chevalley transformation $f_\sharp'{g'}^{*}\to g^{*}f_{\sharp}$ is an equivalence.
\end{enumerate}

\begin{proof} This result is ``classical'' and formal. See e.g., \cite[\S 1.4]{Ayoubrig} for a proof using model categories. The fact that all the assignments in sight can be promoted to functors with value in $\CAlg(\LPr)$ is discussed in Robalo's thesis \cite[Chapter 9]{RobaloThesis}.
\end{proof}
\end{prop}

\subsection{The Tate circle and stable motivic homotopy theory}
We begin with discussing several motivic properties that hold in $\infShv_{dNis}(\lSm/S)[\boxx^{-1}]$.

\begin{notat}
\label{Thomdf.2} Let $S\in \lSch$. If $X_0\to X_1$ is a morphism in $\infShv_{dNis}(\lSm/S)_\ast$, 
let $X_1/X_0$\index[notation]{X1X0 @ $X_1/X_0$} be shorthand notation for $\cofib(X_0\to X_1)$.
If $X_0\to X_1$ is a morphism in $\infShv_{dNis}(\lSm/S)$, 
let $X_1/X_0$ be shorthand notation for $\cofib((X_0)_+\to (X_1)_+)$. If no confusion arises, we will use the same notation to denote the cofibers computed in the localizations  $\infShv_{dNis}(\lSm/S)[\boxx^{-1}]$ and $\infShv_{dNis}(\lSm/S)_\ast[\boxx^{-1}]$, that we systematically identify with the subcategories of the $\infty$-categories of (pointed) dividing Nisnevich sheaves spanned by $\boxx$-local objects.

We follow the same convention for the strict Nisnevich sheaves too.
\end{notat}

\begin{rmk}
In $\A^1$-homotopy theory, the notation $X_1/X_0$ is often used when $X_0$ is a subscheme of a scheme $X_1$.
The quotient $\P^1/\A^1$ is such an example.
In $\boxx$-homotopy theory we replace $\A^1$ by $\boxx$ and consequently the quotient $\P^1/\A^1$ 
gets replaced by $\P^1/\boxx$ although $\boxx$ is not, strictly speaking, a subscheme of $\P^1$.
This is the reason why we require more flexibility when using the notation $X_1/X_0$.

In $\A^1$-homotopy theory it is common to write $(X,x)$ for a scheme $X$ pointed at $x$. 
For example, the Tate circle is the pointed motivic space $(\mathbb{G}_m,1)$.
However, 
the same notation refers to a log scheme in our setting.
To avoid possible confusion, we shall use the notation $X/x$.
For example, $\P^1$ pointed at $1$ is written as $\P^1/1$.
\end{rmk}

\begin{df}
\label{logH.26}
The \emph{log multiplicative group scheme} is defined as 
\index[notation]{Gmm @ $\Gmm$}
\[
\Gmm:=(\P^1,0+\infty).
\]
The log Tate circle\index{log Tate circle} 
is\index[notation]{St1 @ $S_t^1$}
\begin{equation}
\label{logH.26.1}
S_t^1
:=
\Gmm/1.
\end{equation}
Morel-Voevodsky \cite{MV} write $T:=\A^1/\G_m$ for the motivic Tate object in the unstable motivic homotopy category $\mathcal{H}(S)$ of a scheme $S$.
In $\infShv_{sNis}(\lSm/S)_\ast[\boxx^{-1}]$,
we use the notation\index[notation]{T @ $T$}
\begin{equation}
\label{logH.26.2}
T
:=
\A^1/(\A^1,0).
\end{equation}
Note that the projection $(\A^1,0)\to \A^1$ is proper, 
while the open immersion $\G_m\to \A^1$ is not proper.
\end{df}

We combine Proposition \ref{logtop.6} and \cite[Corollary 5.10]{Vcdtop} (see also Remark \ref{logH.21}) to deduce the following two properties. 

\begin{prop}
\label{logHprop.1}
Suppose $S\in \lSch$ and let 
\[
\begin{tikzcd}
Y'\ar[d]\ar[r]&
Y\ar[d]
\\
X'\ar[r]&
X
\end{tikzcd}
\]
be a strict Nisnevich distinguished square in $\lSm/S$.
Then this is a cocartesian square in $\infShv_{sNis}(\lSm/S)[\boxx^{-1}]$.
\end{prop}

\begin{prop}
\label{logHprop.2}
Suppose $S\in \lSch$ and let $Y\to X$ be a dividing cover in $\lSm/S$.
Then $Y\to X$ is an equivalence in $\infShv_{dNis}(\lSm/S)[\boxx^{-1}]$.
\end{prop}

For the theory of cubes in an $\infty$-category, we refer to Appendix \ref{cubes}.
If $Q$ is an $n$-cube in a category of sheaves of spaces, we let $Q_+$ be the naturally induced $n$-cube in the corresponding category of sheaves of pointed spaces, obtained by applying the functor $(-)_+$ objectwise.

\begin{prop}
\label{logHprop.3}
Suppose $S\in \lSch$ and let $Q$ be the cube associated with a Zariski cover 
$\{U_i\to X\}_{1\leq i\leq n}$ in $\lSm/S$.
Then, in $\infShv_{sNis}(\lSm/S)[\boxx^{-1}]$, we have an equivalence
\[
\colimit(\horn(Q))
\simeq
X.
\]
\end{prop}
\begin{proof}
We proceed by induction on $n$.
If $n=1$, the claim is clear.
If $n>1$, let $Q_i$ be the $(n-1)$-cube $Q|_{\{i\}\times (\Delta^1)^{n-1}}$.
From \eqref{Thomdf.5.2}, we have a cocartesian square
\[
\begin{tikzcd}
\colimit(\horn(Q_0))\ar[d]\ar[r]&
\colimit(\horn(Q_1))\ar[d]
\\
U_1\ar[r]&
\colimit(\horn(Q))
\end{tikzcd}
\]
in $\inflogH(S)$.
Together with the induction hypothesis, this becomes a cocartesian square
\[
\begin{tikzcd}
U_1\cap (U_2\cup \cdots \cup U_n)\ar[d]\ar[r]&
U_2\cup \cdots \cup U_n\ar[d]
\\
U_1\ar[r]&
\colimit(\horn(Q))
\end{tikzcd}
\]
in $\inflogH(S)$.
Proposition \ref{logHprop.1} finishes the proof.
\end{proof}

\begin{df}
\label{Ori.38}
Suppose $S\in \Sch$, $X$ is a scheme in $\Sm/S$, with a strict normal crossing divisor $Z_1+\cdots+Z_m$, 
and $n$ is an integer in $[1,m]$.
For simplicity of notation, we set $Y:=(X,Z_{n+1}+\cdots+Z_m)$ and 
$$Y_I:=(X,Z_{i_1}+\cdots+Z_{i_r}+Z_{n+1}+\cdots+Z_m)$$ 
if $I=\{i_1,\ldots,i_r\}$ is a subset of $[1,n]$.
In this setting, 
let\index[notation]{CubeYZ @ $\cube(Y,Z_1+\cdots+Z_n)$}\index[notation]{HornYZ @ $\horn(Y,Z_1+\cdots+Z_n)$}
\[
\cube(Y,Z_1+\cdots+Z_n)
\text{ and }
\horn(Y,Z_1+\cdots+Z_n)
\]
denote the cube and cubical horn naturally consisting of 
$Y_I$.
For example, if $m=3$ and $n=2$, then $\cube(Y,Z_1+Z_2)$ is the diagram
\[
\begin{tikzcd}
(X,Z_1+Z_2+Z_3)\ar[r]\ar[d]&
(X,Z_1+Z_3)\ar[d]
\\
(X,Z_2+Z_3)\ar[r]&
(X,Z_3).
\end{tikzcd}
\]
We obtain the cubical horn by removing the lower right corner $(X,Z_3)$.
\end{df}

\begin{prop}
\label{Ori.5}
With the above notation, suppose $\{X-Z_i\to X\}_{1\leq i\leq n}$ is a Zariski cover.
Then in  $\infShv_{sNis}(\lSm/S)[\boxx^{-1}]$, there is an equivalence
\[
\colimit(\horn(Y,Z_1+\cdots+Z_n))\simeq Y.
\]
\end{prop}
\begin{proof}
By \v{C}ech descent for the Zariski topology,
we are reduced to consider the intersections $(X-Z_{i_1})\cap \cdots \cap (X-Z_{i_r})$ instead of $X$.
Without loss of generality, we may assume $i_1=1$.
In this intersection, $Z_1$ becomes $0$.
This means that $Y_{\{1\}\cup I}\to Y_I$ is an isomorphism of fs log schemes for every subset $I$ of $[2,n]$.
Together with \eqref{Thomdf.5.2}, this concludes the proof.
\end{proof}
We record the following elementary but important consequence of the above Proposition.
\begin{prop}
\label{logH.1}
For every $S\in \lSch$ there is a cocartesian square
\[
\begin{tikzcd}
(\P^1,0+\infty)\arrow[d]\arrow[r]&
(\P^1,0)\arrow[d]
\\
(\P^1,\infty)\arrow[r]&
\P^1
\end{tikzcd}
\]
in $\infShv_{sNis}(\lSm/S)[\boxx^{-1}]$.
\end{prop}
\begin{proof}
If $S\in \Sch$, this is a special case of Proposition \ref{Ori.5}.
For general $S\in \lSch$, apply $p^*$ to have a required cocartesian square, where $p\colon S\to \Z$ is the structure morphism.
\end{proof}

\begin{prop}
\label{logH.4}
For every $S\in \lSch$ and $n\geq 1$ there is a cocartesian square
\[
\begin{tikzcd}
(\Blow_O(\A^n),E)\ar[d]\ar[r]&
\A^n\ar[d]
\\
(\Blow_O(\P^n),E)\ar[r]&
\P^n
\end{tikzcd}
\]
in $\infShv_{sNis}(\lSm/S)[\boxx^{-1}]$
Here $O$ is the origin of $\A^n$ and $E$ is the exceptional divisor.
It follows that there is an equivalence
\[
T
\xrightarrow{\simeq}
\P^1/(\P^1,0)
\]
in $\infShv_{sNis}(\lSm/S)_\ast[\boxx^{-1}]$.
\end{prop}
\begin{proof}
Let $(a)$ and $(b)$ be the left and right squares in the commutative diagram
\[
\begin{tikzcd}
\A^n-O\ar[d]\ar[r]&
(\Blow_O(\A^n),E)\ar[d]\ar[r]&
\A^n\ar[d]
\\
\P^n-O\ar[r]&
(\Blow_O(\P^n),E)\ar[r]&
\P^n.
\end{tikzcd}
\]
Since $(a)$ and the composite square are  both strict Nisnevich distinguished squares, they are cocartesian in $\infShv_{sNis}(\lSm/S)[\boxx^{-1}]$.
By applying the pasting law, 
see e.g., \cite[Lemma 4.4.2.1]{HTT}, 
we deduce that $(b)$ is cocartesian in $\infShv_{sNis}(\lSm/S)[\boxx^{-1}]$.
If $n=1$, then we get $T \xrightarrow{\simeq} \P^1/(\P^1,0)$.
\end{proof}

\begin{prop}
\label{logH.2}
If $S\in \lSch$, 
then in  $\infShv_{sNis}(\lSm/S)_\ast[\boxx^{-1}]$ there is an equivalence
\[
\P^1/1 \simeq S^1\wedge S_t^1.
\]
\end{prop}
\begin{proof}
Follows from Proposition \ref{logH.1} since $(\P^1,0)\cong (\P^1,\infty)=\boxx$.
\end{proof}

\begin{prop}
\label{logH.5}
If $S\in \lSch$, 
then in $\infShv_{sNis}(\lSm/S)_\ast[\boxx^{-1}]$ there is an equivalence
\[
T
\simeq
S^1\wedge S_t^1
\]
\end{prop}
\begin{proof}
Combine Propositions \ref{logH.4} and \ref{logH.2}, and use the equivalence $(\P^1,0)\simeq \pt$.
\end{proof}

\begin{df}
Let $S\in \lSch$.
We define \index[notation]{logSH @ $\inflogSH$}
\[
\inflogSH(S) = \infSpt^{\Sigma}_T(\inflogSH^\eff(S)).
\]
As in the $\A^1$-setting, this construction can be promoted to
a functor
\[
\inflogSH\colon \lSch^\mathrm{op} \to \CAlg(\LPr)
\]
satisfying the properties (1)--(3) of Proposition \ref{prop:is_premotivic}. 
\end{df}
\begin{rmk} Let $S\in \lSch$.
Write $\Sigma_{\Gmm}$ for $S^1_t\wedge - $ in $\inflogSH^{\eff}(S)$. We will see in light of Proposition \ref{Ori.56} and Remark \ref{Ori.74} below that $\Gmm$ is a symmetric object of $\inflogSH^\eff(S)$. Therefore
we have an equivalence of $\infty$-categories
\begin{equation}
\label{alter.1.1}
\inflogSH(S)
\simeq
\colim_{\LPr}(\inflogSH^{\eff}(S) \xrightarrow{\Sigma_{\Gmm}}\inflogSH^{\eff}(S) \xrightarrow{\Sigma_{\Gmm}} \cdots ),
\end{equation}
thanks to Proposition \ref{prop:def-spt}. In particular, the category $\inflogSH(S)$ is a \emph{stable} $\infty$-category.
\end{rmk}

\begin{df}\label{def:sphere_suspension_logH}
Suppose $S\in \lSch$.
The functors
\[
S^1\wedge (-), T\wedge (-)
\colon 
\infShv_{sNis}(\lSm/S)_\ast[\boxx^{-1}]\to  \infShv_{sNis}(\lSm/S)_\ast[\boxx^{-1}]
\]
naturally induce equivalences\index[notation]{SigmaS1 @ $\Sigma_{S^1}$}
\[
\Sigma_{S^1},\Sigma_T\colon \inflogSH(S,\Lambda)
\to
\inflogSH(S,\Lambda).
\]
These have right adjoints $\Omega_{S^1}$ and $\Omega_T$, 
which we also write as $\Sigma_{S^1}^{-1}$ and $\Sigma_T^{-1}$.
By combining these functors, 
we define\index[notation]{Sigmapq @ $\Sigma_{p,q}$}
\[
\Sigma^{p,q}:=\Sigma_{S^1}^{p-2q}\Sigma_T^q \colon \inflogSH(S, \Lambda)\to \inflogSH(S, \Lambda)
\]
for all integers $p$ and $q$.
Similarly, we can define $\Sigma^{p,q}\colon \inflogSHS(S, \Lambda)\to \inflogSHS(S, \Lambda)$ for all integers $q\geq 0$ and $p$.
\end{df}

\subsection{Auxiliary models and unstable logarithmic motivic homotopy theory}

Next, we explain  two classes of morphisms that we often want to invert in our log motivic homotopy categories. Throughout this section, $S$ will denote a quasi-compact and quasi separated scheme, with trivial log structure. In this case we are able to introduce an unstable (resp.\ pointed) variant of the above construction, that will be denoted by $\inflogH_{(*)}(S)$. Unlike its stable analogue, its definition involves inverting an extra class of morphisms, namely admissible blow-ups. Notably, after $S^1$-stabilization this extra class becomes automatically  inverted in $\inflogSHS(S)$, see    Corollary \ref{ProplogSH.9}.

\begin{df}
\label{logH.35}
Suppose that $X\in \SmlSm/S$ (in particular, $\ul{X}$ is of finite presentation over $S$).
If $Z\in \Sm/S$ is a closed subscheme of $\ul{X}$ that has strict normal crossing
with $\partial X$ in the sense of \cite[Definition 7.2.1]{logDM}, 
the \emph{blow-up of $X$ along $Z$}, \index{blow-up} denoted 
$\Blow_Z X$,\index[notation]{BlZX @ $\Blow_Z X$} is the fs log scheme whose underlying scheme is the 
blow-up $\Blow_Z \ul{X}$, 
and whose log structure is the compactifying log structure associated with the open immersion 
$X-(\partial X\cup Z)\to \Blow_Z \ul{X}$.
Observe that there is a canonical projection $\Blow_Z X \to X$.

In the case of $Z\subset \partial X$, 
we say that $\Blow_Z X$ is an 
\emph{admissible blow-up along a smooth center}.\index{admissible blow-up along a smooth center}
The smallest class of morphisms in $\lSm/S$ containing all admissible blow-ups along smooth centers 
and closed under compositions is denoted by $\SmAdm/S$.\index[notation]{SmAdm @ $\SmAdm/S$}
\end{df}

\begin{lem}
The class $\SmAdm/S$ is closed under pullbacks along arbitrary morphisms $S'\to S$ of schemes.
\end{lem}
\begin{proof}
By working strict \'etale locally on $X$,
we reduce to the case when $Z\to \ul{X}$ is given by a pullback of
\[
Z_0:=\A^m \times \{0\}\times \ul{\A_\N^n}\times \{0\} \to \ul{X_0}:=\A^{m+m'}\times \ul{\A_{\N}^{n+n'}}
\]
with $n'>0$ along a composite morphism $\ul{X}\to \ul{X_0}\times S\to \ul{X_0}$,
where the first morphism is \'etale and the second morphism is the projection.
We set $X_0:=\A^{m+m'}\times \A_{\N}^{n+n'}$.
We have an isomorphism
\[
\Blow_Z X\times_S S'
\cong
(\Blow_{Z_0} X_0 \times S')\times_{\ul{X_0}\times S}\times \ul{X} 
\]
To conclude,
recall that $\ul{X}\to \ul{X_0}\times S$ is \'etale and the blow-up construction is closed under flat base change.
\end{proof}

\begin{df}
Let $\cSch$ be the category of pairs $X=(\ul{X},X-\partial X)$ of quasi-compact and quasi-separated
schemes
satisfying the following conditions:
\begin{enumerate}
\item[(i)] $X-\partial X$ is a dense open subscheme of $\ul{X}$,
\item[(ii)] the complement of $X-\partial X$ in $\ul{X}$ is the support of a Cartier divisor on $X$ (possibly empty).
\end{enumerate}
A morphism $f\colon Y\to X$ in $\cSch$ is a commutative square of schemes
\[
\begin{tikzcd}
Y-\partial Y\ar[d]\ar[r,"f-\partial f"]&
X-\partial X\ar[d]
\\
Y\ar[r,"f"]&
X.
\end{tikzcd}
\]

We can endow $X\in \cSch$ with the compactifying log structure associated with $X-\partial X\to \ul{X}$.
According to \cite[Proposition III.1.6.2]{Ogu},
$\cSch$ is a full subcategory of log schemes.
By \cite[Proposition III.1.6.3]{Ogu},
if $\ul{X}$ is locally noetherian,
then the said log structure is integral,
and $\ol{\cM}_{X,x}$ is nontrivial for a point $x$ of $X$ if and only if $x\in X-\partial X$.

Observe that for all morphisms $Y,Z\to X$ in $\cSch$,
the fiber product is given by the formula
\[
Y\times_X Z
\simeq
(
\ul{Y}\times_{\ul{X}}\ul{Z},
(Y-\partial Y)\times_{X-\partial X}(Z-\partial Z)
).
\]
Since $\cSch$ admits a final object $(\Spec{\Z},\Spec{\Z})$,
$\cSch$ admits finite limits.
We use the notation $\boxx:=(\P^1,\A^1)$ in $\cSch$ in analogy with $\boxx$ in $\lSch$.
The “c” in $\cSch$ refers to the compactifying log structure.
\end{df}

\begin{rmk}
Unlike the theory of modulus pairs developed in \cite{zbMATH07341096} and \cite{2106.12837},
a morphism $Y\to X$ in $\cSch$ does not require extra conditions on the boundaries $\partial Y$ and $\partial X$.
\end{rmk}

\begin{df}
Let $\cSm$ denote the class of morphisms $f\colon Y\to X$ in $\cSch$ such that the morphism $f-\partial f\colon Y-\partial Y\to X-\partial X$ is smooth.

The class $\cSm$ is closed under compositions and pullbacks.
For $S\in \Sch$,
$\cSm/S$ contains $\lSm/S$ as a full subcategory.
\end{df}

\begin{df}
\label{ProplogSH.2}
An \emph{admissible blow-up} \index{admissible blow-up} is a morphism $p\colon Y\to X$ in $\cSch$ 
such that $\ul{p}\colon \ul{Y}\to \ul{X}$ is proper and the naturally induced morphism of schemes 
$Y-\partial Y\to X-\partial X$ is an isomorphism.
The class of admissible blow-ups in $\cSch$ is denoted by $\Adm$.\index[notation]{Adm @ $\Adm$}

Observe that $\Adm$ is closed under compositions and pullbacks and $\Adm\subset \cSm$.
\end{df}
\begin{df}
We say that a morphism $f\colon Y\to X$ in $\cSch$ is \emph{strict} if the naturally induced morphism $Y-\partial Y\to (X-\partial X)\times_X Y$ is an isomorphism.
We say that $f$ is a \emph{strict Nisnevich cover} if $f$ is strict and $\ul{f}\colon \ul{Y}\to \ul{X}$ is a Nisnevich cover of schemes.
\end{df}

\begin{rmk}
When discussing a non $\A^1$-invariant cohomology theory $H^i$ of a scheme $X$ over a field $k$,  
such as Hodge cohomology, 
it is often convenient to consider a compactification $Y\in \SmlSm/k$ such that $Y-\partial Y\cong X$ 
and $Y$ is proper over $k$.
One may want to extend $H^i$ to $\SmlSm/k$ and take $H^i(Y)$ as a new cohomology group of $X$, 
but a technical problem of this approach is to show independence of the choice of $Y$.
To achieve this, 
it is necessary to show that $H^i(Y)$ is invariant under admissible blow-ups $Y'\to Y$ for $
Y,Y'\in \SmlSm/k$ with $Y$ proper.
This is the reason why we are considering admissible blow-ups.
\end{rmk}
\begin{const}\label{constr:Adm_unstable}
For every $S\in \Sch$ as above, write \index[notation]{logH @ $\inflogH$}\index[notation]{logHpt @ $\inflogHpt$}
\[
\inflogH_{(\ast)}(S)=\infShv_{sNis} (\SmlSm/S)_{(\ast)}[(\boxx\cup \SmAdm)^{-1}] \subset \infShv_{sNis} (\SmlSm/S)_{(\ast)}
\]
for the full subcategory of the $\infty$-category of (pointed) strict Nisnevich sheaves on $\SmlSm/S$ spanned by the objects $M$ that are local with respect to the maps $(X\times \boxx)_{(+)}\to X_{(+)}$ for all $X\in \SmlSm/S$ and the maps $f\in \SmAdm/S$. Note that this construction underlies a presentable symmetric monoidal $\infty$-category $\inflogH_{(\ast)}(S)^\otimes$.

In a similar fashion, we can consider for every $S\in \cSch$ the $\infty$-categories

\[\inflogH^{\Adm}_{(\ast)}(S)\subset\infShv_{sNis}(\cSm/S)_{(\ast)}\]
spanned by the objects $M$ that are $\boxx$-local, and $\Adm$-local in the obvious sense. Again, this construction underlies a presentable symmetric monoidal $\infty$-category $\inflogH^{\Adm}_{(\ast)}(S)^\otimes$.
\end{const}

The following result is a verbatim translation of the previously seen arguments. 
\begin{prop} \label{logH.9} The assignments 
\[
S\mapsto \inflogH_{(\ast)}(S)
\]
can be promoted into functors 
\[
\inflogH,
\inflogH_{\ast}
\colon
 \Sch^\mathrm{op}\to \CAlg(\LPr)
\]
that
satisfy the compatibilities as in Proposition \ref{prop:is_premotivic}.
Similarly, we can define \index[notation]{logHAdm @ $\inflogH^{\Adm}$}\index[notation]{logHptAdm @ $\inflogHpt^{\Adm}$}
\[
\inflogH^{\Adm},
\inflogH_{\ast}^{\Adm}
\colon
\cSch^\mathrm{op}
\to
\CAlg(\LPr).
\]
\end{prop}

\begin{rmk}\label{rmk:base_point}
The adding a base-point functor $\infSpc\to \infSpc_*$   induces a symmetric monoidal functor 
\[(-)_+\colon \inflogH(S) \to \inflogH_{\ast}(S).\]

Note that the $\infty$-category $\inflogH(S)$, (resp.\  $\inflogH_{\ast}(S)$) is presentable, and is generated under colimits by the family of $X$ (resp.\ $X_+$) for all $X\in \lSm/S$. Moreover, $\inflogH(S)$, (resp.\ \ $\inflogH_{\ast}(S)$) underlies a presentable symmetric monoidal $\infty$-category, where the monoidal structure is the Day convolution of the cartesian product in $\lSm/S$.
\end{rmk}

\begin{rmk}
\label{rmk:unclearequivalence}
Let $S\in \Sch$ as above. 
There is a natural localization functor
\[
\infShv_{dNis}(\lSm/S)[\boxx^{-1}]\to \inflogH(S)
\]
obtained by inverting every morphism in $\SmAdm$. Remarkably, thanks to Corollary \ref{ProplogSH.9}, this localization becomes an equivalence after stabilization, and it is unclear  whether such equivalence already holds in the unstable setting. 

When $S$ is equipped with a nontrivial log structure, it is tempting to define an unstable version of $\inflogSH^{\eff}(S)$ (which we introduced for arbitrary $S\in \lSch$) by considering a localization of $\infShv_{dNis}(\lSm/S)$ to $\boxx$-local equivalences. However, when the log structure on $S$ is not trivial it is unclear how to define the analogue of $\SmAdm/S$. Since many fundamental properties of $\inflogH(S)$ depend on the invariance under admissible blow-ups, it's also unclear whether the category of $\boxx$-local $dNis$-sheaves (without further conditions) is actually the correct setting for studying unstable phenomena.
\end{rmk}

\subsection{Log \'etale motivic homotopy categories and compact generators}\label{sec:cmpt_objects}
In this section we consider a variant of $\inflogSH$ using a topology other than the dividing Nisnevich topology  (e.g., the log \'etale topology).
We also discuss compact generators for such a variant.

\begin{notat}
\label{not:Lambda}
Let $\Lambda\in \CAlg(\infSpt_{\geq 0})$ be a commutative connective ring spectrum, and let $\Mod_\Lambda$ be the $\infty$-category of $\Lambda$-modules. We denote by $\infPsh(\infC, \Lambda)=\Fun(\infC^\mathrm{op},  \Mod_{\Lambda})$ the $\infty$-category of $\infC$-presheaves of $\Lambda$-modules and by $\infShv_\tau( \infC, \Lambda)$ the full $\infty$-subcategory of $\tau$-sheaves, if $\infC$ is endowed with a Grothendieck topology $\tau$.
For $X\in \infC$, let $\Lambda_\tau(X)$ be the $\tau$-sheaf of $\Lambda$-modules represented by $X$.
\end{notat}

\begin{rmk}
    We will mostly be interested in the case where $\Lambda = \mathrm{H} R$ is the Eilenberg-MacLane spectrum of an ordinary commutative ring $R$, or $\Lambda = \mathbb{S}$ is the sphere spectrum (or a localization of the sphere spectrum at a prime $p$). 
\end{rmk}

\begin{rmk}(Hypercomplete version and descent by squares)\label{logH.21} Suppose that $\cC$ is a category with an initial object, and let $P$ be a complete and regular cd-structure with the associated topology $\tau$. Then $\cF\in \infPsh(\cC)$ is a $\tau$-sheaf if and only if $\cF(Q)$ is cartesian for every distinguished square $Q$  by \cite[Corollary 5.10]{Vcdtop}.

If $P$ is a quasi-bounded, regular, and squarrable cd-structure,
then $\infShv_\tau(\cC)$ is hypercomplete.
If $P$ were a bounded, complete, and regular cd-structure, 
then this follows from \cite[Proposition 3.8(a),(b)]{Vcdtop}.
In our setting, use \cite[Lemma 3.4.10]{logDM} instead of \cite[Lemma 3.5]{Vcdtop}.

The situation is more subtle when $\tau$ is not the topology associated with a cd-structure. Assume that $\Lambda$ is as in Notation \ref{not:Lambda}, and assume  that $\Lambda$ is eventually coconnective (that is, $\pi_i(\Lambda)=0$ for $i$ large enough). Assume moreover that the following conditions are satisfied:
\begin{enumerate}
    \item $(\infC, \tau)$ has finite local $\Lambda$-cohomological dimension, and $\infC$ is an ordinary category  with fiber products.
    \item there exists a full subcategory $\infC_0\subseteq \infC$, stable under fiber products, spanned by quasi-compact objects (in the sense of \cite[Expos\'e VI, D\'efinitions 1.1]{SGA4}) of finite $\Lambda$-cohomological dimension, and such that every object of $\infC$ admits a $\tau$-cover by objects of $\infC_0$.\end{enumerate}
Then every $\tau$-sheaf of $\Lambda$-modules is an hypersheaf.
This is the content of \cite[Corollary 2.4.6]{AGV}.

A classical example of application of the above criterion is when $\infC=\lSm/S$ for $S$ a quasi-compact and quasi-separated fs log scheme, and $\tau$ is log \'etale or the Kummer \'etale, or the dividing \'etale, or the strict \'etale topology and $\Lambda$ is a (classical) torsion ring. This follows from \cite[Theorem 7.2]{NakayamaII}. If $\tau$ is for example the strict \'etale topology, the same holds for the case $\Lambda$ a $\mathbb{Q}$-algebra (in this case the Nisnevich and the \'etale cohomology coincide for any discrete sheaf, since the rational Galois cohomology groups vanish in positive degree).  
\end{rmk}

\begin{const}
Let $S\in \Sch$,
and let $\tau$ be a topology on $\lSch$ finer than the dividing Nisnevich topology and generated by a family of morphisms in $\lSm$.
With $\Lambda$ as in Notation \ref{not:Lambda},
we construct
\[
\inflogSH_\tau^\eff(S,\Lambda)
\text{ and }
\inflogSH_\tau(S,\Lambda)
\]
from $\infShv_\tau(\lSm/S,\Lambda)$ as we have constructed $\inflogSH^\eff(S)$ and $\inflogSH(S)$ from $\infShv_{dNis}(\lSm/S,\infSpt)$.
\end{const}

\begin{rmk}
When $\Lambda$ is the Eilenberg-MacLane spectrum of a classical (discrete) ring, the category $\inflogSH_{\tau}^\eff(S, \Lambda)$ coincides with the category $\inflogDAeff_\tau(S, \Lambda)$ considered elsewhere (notably in \cite{BLPO}, \cite{BLMP}). When $\Lambda=\mathbb{S}$ is the sphere spectrum, we will simply write $\inflogSH_{\tau}^\eff(S)$.
\end{rmk}

\begin{prop}
\label{alter.6}
Let $\infC$ be a symmetric monoidal $\infty$-category generated by a family $\sF$, and let $B$ be an object of $\infC$.
Then the family of $\Sigma_B^{-i} \Sigma_B^\infty X$ for all $X\in \sF$ and $i\geq 0$ generates $\Stab_B(\infC)$.
\end{prop}
\begin{proof}
Suppose $f\colon \cF\to \cG$ is a map in $\Stab_B(\infC)$, and we write $\cF(X)$ and $\cG(X)$ as $(\cF_0(X),\cF_1(X),\ldots)$ and $(\cG_0(X),\cG_1(X),\ldots)$ for $X\in \sF$.
Then $f$ is an equivalence if and only if $\cF_i(X)\to \cG_i(X)$ is an equivalence of spaces for all $X\in \sF$ and integer $i\geq 0$.
Since $\cF_i(X)\to \cG_i(X)$ can be written as
\[
\Map_{\Stab_B(\infC)}(\Sigma_B^{-i} \Sigma_B^\infty X,\cF)
\to
\Map_{\Stab_B(\infC)}(\Sigma_B^{-i} \Sigma_B^\infty X,\cG),
\]
we conclude.
\end{proof}
\begin{lem}
\label{alter.7}
Suppose $S\in \lSch$.
Then the family of $\Sigma_T^\infty \Lambda_\tau(X)$
 for all $X\in \lSm/S$ generates $\inflogSH(S, \Lambda)$ under colimits,  desuspensions, and negative $T$-suspensions.
A similar result holds for the $\Adm$-variants.
\end{lem}
\begin{proof}
Immediate from Proposition \ref{alter.6}.
\end{proof}

\begin{prop}
\label{logH.31}
Suppose $S\in \lSch_{noeth}$, and assume that either
\begin{enumerate}
    \item $\tau$ is  the dividing Nisnevich topology or,
    \item $\tau = d\et , l\et$, and $\Lambda$ satisfies conditions (1) and (2) of Remark \ref{logH.21}. 
\end{enumerate}
Then, for all $X\in \lSm/S$ and integer $p$ (resp.\ integers $p$ and $q$), $\Sigma_{S^1}^p  \Lambda_\tau(X)$ (resp.\ $\Sigma^{p,q}\Sigma_T^\infty \Lambda_\tau(X)$) is compact in $\inflogSH_{\tau}^\eff(S, 
\Lambda)$ (resp.\ $\inflogSH_\tau(S, \Lambda)$). In particular, the categories $\inflogSH^\eff_{\tau}(S, \Lambda)$ and $\inflogSH_\tau(S, \Lambda)$ are all compactly generated stable $\infty$-categories.
\end{prop}
\begin{proof}  Under the above assumptions, the $\infty$-category $\infShv_\tau(\lSm/S, \Lambda)$ is hypercomplete. Hence by e.g., \ \cite[Lemma 2.4.5]{AGV},  $\Lambda_\tau(X)\in \infShv_\tau(\lSm/S, \Lambda)$ is compact for every $X\in \lSm/S$.

Recall now that $\inflogSH^\eff_{\tau}(S, \Lambda)$ is defined as the full subcategory of $\infShv_\tau(\lSm/S, \Lambda)$ spanned by $\boxx$-local objects. If $\cF=\colim \cF_i$ is a filtered colimit of $\boxx$-local sheaves, we have equivalences
\begin{align*}
\Map_{\infShv_\tau(\lSm/S, \Lambda)}(\Lambda_\tau(X), \colim \cF_i) &\simeq \colim \Map_{\infShv_\tau(\lSm/S, \Lambda)}(\Lambda_\tau(X), \cF_i) \\
&\simeq  \colim \Map_{\inflogSH^\eff_{\tau}(S, \Lambda)}( \Lambda_\tau(X), \cF_i)
\end{align*}
so that $\Lambda_\tau(X)$ is compact in $\inflogSH^\eff_{\tau}(S, \Lambda)$, noting that $\colim \cF_i$ is again $\boxx$-local. The case of $\Sigma^p_{S^1} \Lambda_\tau(X)$ for arbitrary $p$ follows immediately.
We deduce that $\inflogSH^\eff_\tau(S, \Lambda)$ is compactly generated.

For $\inflogSH_{\tau}(S, \Lambda)$, we only need to consider the case when $p=q=0$ since $\Sigma^{p,q}$ preserves small colimits.
Let $\{\cF_i\}_{i\in I}$ be a filtered system in $\inflogSH_\tau(S, \Lambda)$, and express each $\cF_i$ as a $\Gmm$-spectrum in $\inflogSH_{\tau}^\eff(S, \Lambda)$
\[
\cF_i=(\cF_{i0},\cF_{i1},\ldots)
\]
For every integer $j\geq 0$, we set $\cG_j:=\colimit_{i\in I} \cF_{ij}$.
Since $\Omega_{\Gmm}$ commutes with filtered colimits (since its left adjoint $\Sigma_{\Gmm}$ preserves $\omega$-compact objects), we have an equivalence  $\cG_j\xrightarrow{\simeq} \Omega_{\Gmm}\cG_{j+1}$.
Hence we can form the $\Gmm$-spectrum in $\inflogSH_{\tau}^\eff(S, \Lambda)$
\[
\cG:=(\cG_0,\cG_1,\ldots).
\]
One can readily check that $\cG$ is indeed the filtered colimit of $\{\cF_i\}_{i\in I}$.
Since we have equivalences of spaces (we omit $\Lambda$ for brevity)
\begin{align*}
& \Map_{\inflogSH_{\tau}(S)}(\Sigma_{T}^\infty X_+,\cG)
\simeq
\Map_{\inflogSH_{\tau}^\eff(S)}(\Sigma_{S^1}^\infty X_+,\cG_0)
\\
\simeq &
\colimit_{i\in I} \Map_{\inflogSH_{\tau}^\eff(S)}(\Sigma_{S^1}^\infty X_+,\cF_{i0})
\simeq
\colimit_{i\in I} \Map_{\inflogSH_{\tau}(S)}(\Sigma_T^\infty X_+,\cF_{i}),
\end{align*}
we conclude. In light of Lemma \ref{alter.7}, we get the statement about compact generations as well.
\end{proof}
\begin{rmk}\label{rmk:cpt_gen_variants}
    A similar argument show that $\inflogH(S)$ and $\inflogH_{\ast}(S)$ are compactly generated $\infty$-categories as well, for every $S\in \Sch_{noeth}$.
\end{rmk}

\begin{rmk}Let $S$ be any scheme,  with trivial log structure. Let $\Lambda$ be a commutative connective ring spectrum. The $\infty$-categories
$\infShv_{\tau}(\lSm/S, \Lambda)$ are compactly generated  if $\tau=sNis$ or if $\tau=\set$ and $\Lambda$ is eventually coconnective. This follows using the same argument of \cite[Proposition 2.4.20]{AGV} (see also \cite[Remark 2.4.23]{AGV}), by showing that in both cases that $\Lambda_\tau(X)$ is compact for every $X\in \lSm/S$. Our argument can be used to show that the construction of $\inflogDA(S,\Lambda)$ starting from $\SmlSm/S$ produces a compactly generated stable $\infty$-category without noetherianity assumption. Note that this argument would break if $\Lambda = \Sphere$ (which is clearly not eventually coconnective).
\end{rmk}

\subsection{Abelian motivic homotopy categories}

We will briefly discuss the $\infty$-categories  $\inflogDAeff$ and $\inflogDA$ in this section;  
as explained above, these are abelian analogues of $\inflogSHS$ and $\inflogSH$.
When $k$ is a field,
the homotopy category of $\inflogDAeff$ was discussed in \cite{logDM}.

\begin{df}
Suppose $\tau$ is a topology on $\lSch$ generated by a family of morphisms in $\lSm$, and $\Lambda$ is a (discrete) commutative ring.
Let\index[notation]{logDAeff @ $\inflogDAeff$}\index[notation]{logDA @ $\inflogDA$}
\[
\inflogDAeff_\tau(-,\Lambda),\;
\inflogDA_\tau(-,\Lambda)
\colon
\lSch^\mathrm{op}\to \CAlg(\LPr_{\Lambda})
\]
be the $\Lambda$-linear
$\infty$-categories $\inflogSH_\tau^\eff(-, H \Lambda)$ and  $\inflogSH_\tau(-, H \Lambda)$ respectively, where $H\Lambda$ is the Eilenberg-MacLane spectrum of $\Lambda$. 

We often omit $\tau$ in this notation when $\tau=dNis$.
\end{df}
The base change functor along $\Sphere\to H\Lambda$ induces a 
natural transformation
\index[notation]{betasharp @ $\beta_\sharp$}
\begin{equation}
\label{logDA.3.2}
\beta_\sharp
\colon
\inflogSH_{\tau}^{\eff}(-)  
\to \inflogDAeff_\tau(-,\Lambda)
\end{equation}
By considering $\infSpt_{\Gmm}^\Sigma(-)$ on both sides, we obtain a
natural transformation
\begin{equation}
\label{logDA.3.3}
\beta_\sharp
\colon
\inflogSH_\tau(-)
\to
\inflogDA_\tau(-, \Lambda).
\end{equation}
We denote the right adjoint by $\beta^*$,
which exists since $\beta_\sharp$ preserves small colimits.

\begin{df}
\label{logDA.4}
Suppose that $\Lambda$ is a commutative ring, and $k$ is a field.
Let $\Pshltr(k,\Lambda)$\index[notation]{Pshltr @ $\Pshltr(k,\Lambda)$} be the category of presheaves with log transfers introduced in \cite[4]{logDM}
For $X\in \lSm/k$,
let $\Lambda_{\ltr}(X)$ be the representable presheaf of $\Lambda$-modules in this category. We refer to \cite[Definition 4.1.1]{logDM} for details. We write $\infPsh^{\ltr}(k, \Lambda)$  for the stable $\infty$-category of complexes of presheaves with log transfers, defined in the obvious way. 

  Let $\tau$ be one of the topologies $dNis$, $d\et$, or $l\et$. We denote by $\infShv^{\ltr}_\tau(k, \Lambda)$ the  stable $\infty$-subcategory of $\infPsh^{\ltr}(k, \Lambda)$ spanned by the objects $\mathcal{F}$ that are local with respect to the maps  $\Lambda_{\ltr}(\sX)\to \Lambda_{\ltr}(X)$ for all $\tau$-hypercovers $\sX\to X$.   Note that every $\tau$ as above is compatible with log transfers in the sense of \cite[\S 4.2]{logDM}. As such, the sheafification functor preserves the transfer structure. See \cite[\S 4.3]{logDM} for a construction using  the language of model categories. 
 
 We let \index[notation]{logDMeff @ $\inflogDM^{\mathrm{eff}}$ }\index[notation]{logDM @ $\inflogDM$}
 \[
 \inflogDM_\tau^{\mathrm{eff}}(k,\Lambda),
\text{ and }
\inflogDM_\tau(k,\Lambda)
 \] 
be, respectively, the full $\infty$-subcategory of $\infShv^{\ltr}_\tau(k, \Lambda)$ spanned by $\boxx$-local objects (in the obvious sense), and the stable $\infty$-category of $\Gmm$-symmetric spectra $\infSpt^{\Sigma}_{\Gmm}(\inflogDM^\eff_\tau(k,\Lambda))$. Both categories underlie presentable stable symmetric monoidal $\infty$-categories.

\ 

For $X\in \lSm/k$,
the \emph{motive of $X$} is the object $M(X)\in \inflogDM_{\tau}^{\mathrm{eff}}(k,\Lambda)$ associated with 
$\Lambda_{\ltr}(X)$.
As before, we have the $\infty$-suspension functor
\begin{equation}
\label{logDA.4.1}
\Sigma_T^\infty
\colon
\inflogDM_\tau^{\mathrm{eff}}(k,\Lambda)
\to
\inflogDM_\tau(k,\Lambda).
\end{equation}

We observe that the homotopy category of $\inflogDM^\mathrm{eff}_\tau(k,\Lambda)$ is equivalent to the one given in \cite[Definitions 5.2.1, 5.3.1]{logDM}.
\end{df}

\begin{rmk}
\label{logDA.5}
We keep using the notation in Definition \ref{logDA.4}.
According to \cite[\S 4.3]{logDM}, 
there is  a colimit preserving symmetric monoidal functor\index[notation]{gammasharp @ $\gamma_\sharp$}
\begin{equation}
\gamma_\sharp
\colon
\inflogDA_\tau^{\mathrm{eff}}(k,\Lambda)
\to
\inflogDM_\tau^{\mathrm{eff}}(k,\Lambda).
\end{equation}

Similarly, we have the colimit preserving symmetric monoidal functor
\begin{equation}
\label{logDA.5.2}
\gamma_\sharp
\colon
\inflogDA_\tau(k,\Lambda)
\to
\inflogDM_\tau(k,\Lambda).
\end{equation}
We denote the right adjoint by $\gamma^*$, 
which exists since $\gamma_\sharp$ preserves colimits.
\end{rmk}

\newpage

\section{Properties of logarithmic motivic spaces and spectra}
\label{polmsas}

The fundamental properties of \emph{effective} log motives were proven in the stable linear setting of \cite[\S 7]{logDM}.
We will generalize many of these properties to the unstable setting of log motivic spaces, and to the non-effective (that is, $T$-stable) setting. 
For the sake of flexibility of our theory, 
we will provide various models for $\inflogSH(S)$ (at least when $S$ is a classical scheme).

Compared to the analogous discussion in \cite{logDM}, the main results of this section are contained in \S \ref{proplogH}. In particular, we offer a proof of the fact that the Tate circle $T$ is a symmetric object: this implies that the category of $T$-symmetric spectra is indeed a stable category, as explained in the previous section. Compared to the $\mathbb{A}^1$-invariant case, the proof is more challenging and requires a new explicit geometric construction. 

\subsection{Non-noetherian fs log schemes}
\label{noetherian}

In this subsection,
we apply Grothendieck's technique \cite[\S IV.8]{EGA} of noetherian approximation to deal with some non-noetherian fs log schemes using filtered limits. This is necessary in light of the applications discussed in \cite{BLMP}, \cite{BLPO2}, where the base is typically an integral perfectoid ring. 

Throughout this subsection,
we fix the following:
Let $S:=\lim_{i\in I} S_i$ be a filtered limit of schemes such that for every morphism $i\to j$,
the morphism $S_j\to S_i$ is affine.
Assume that $I$ has a final object $0$ and $S_0$ is quasi-compact quasi-separated.

\begin{exm}
Assume that $S=\Spec{A}$ for some ring $A$.
Then $A$ is isomorphic to a colimit of its subrings $A_i$ that are $\Z$-algebras of finite type.
In this case,
we have $S\cong \lim_{i\in I} S_i$ with $S_i:=\Spec{A_i}$,
and this limit satisfies the above condition.
\end{exm}

\begin{df}
\label{ProplogSH.1}
Suppose $W\to X$ is a closed immersion in $\Sm/S$ for $S\in \Sch$, 
and $Z$ is a strict normal crossing divisor on $X$ over $S$.
We say that \emph{$W$ has strict normal crossing with $Z$ over $S$}\index{strict normal crossing} 
if Zariski locally on $X$, there exists a commutative diagram
\begin{equation}
\label{ProplogSH.1.1}
\begin{tikzcd}[column sep=small, row sep=small]
W\ar[d]\ar[r]&
S\times \Spec{\Z[x_1,\ldots,x_n]/(x_{i_1},\ldots, x_{i_s})}\ar[d]
\\
X\ar[d,leftarrow]\ar[r]&
S\times \Spec{\Z[x_1,\ldots,x_n]}\ar[d,leftarrow]
\\
Z\ar[r]&
S\times \Spec{\Z[x_1,\ldots,x_n]/(x_1\cdots x_r)}
\end{tikzcd}
\end{equation}
of cartesian squares for some $0\leq r\leq n$ and $1\leq i_1<\cdots<i_s\leq n$. 
Here the right vertical morphisms are the obvious closed immersions, 
and the horizontal morphisms are \'etale.
\end{df}

\begin{lem}
\label{noetherian.5}
Let $X_0\to S_0$ be a finitely presented morphism of schemes, and we set $X:=X_0\times_{S_0} S$.
If $f\colon U\to X$ is a Zariski (resp.\ \'etale) cover such that it is finitely presented,
then there exists $i\in I$ and a Zariski (resp.\ \'etale) cover $f_i\colon U_i\to X_i$ such that it is finitely presented and $f\cong f_i\times_{S_i} S$.
\end{lem}
\begin{proof}
This is a consequence of \cite[Th\'eor\`emes IV.8.8.2, IV.8.10.5(iii),(iv), Proposition IV.17.7.8]{EGA}.
\end{proof}

\begin{prop}
\label{noetherian.1}
For every $X\in \SmlSm/S$,
there exists $i\in I$ and $X_i\in \SmlSm/S_i$ such that $S\times_{S_i} X_i\cong X$.
If $W\to \ul{X}$ is a closed immersion in $\Sm/S$ such that $W$ has strict normal crossing with $\partial X$ over $S$,
then we may further assume that there exists a closed immersion $W_i\to \ul{X_i}$ in $\Sm/S_i$ such that $W\to \ul{X}$ is its pullback along $S\to S_i$ and $W_i$ is strict normal crossing with $\partial X_i$ over $S$.
\end{prop}
\begin{proof}
We set $Z:=\partial X$.
By \cite[Th\'eor\`eme IV.8.8.2]{EGA},
there exists $i\in I$ and morphisms $Z_i,W_i\to \ul{X_i}$ of schemes over $S_i$ such that $W\to \ul{X}\leftarrow Z$ is the pullback of $W_i\to \ul{X_i}\leftarrow  Z_i$.
We need to show that there exists a morphism $j\to i$ in $I$ such that the pullback $W_j\to \ul{X_j}\leftarrow Z_j$ along $S_j\to S_i$ satisfies the desired property.
This question is Zariski local on $X$ by Lemma \ref{noetherian.5},
so we may assume that there exists a commutative diagram $Q$ of the form \eqref{ProplogSH.1.1} for $W\to \ul{X}\leftarrow Z$.

By \cite[Th\'eor\`eme IV.8.8.2]{EGA} again,
there exists $j\to i$ and a commutative diagram 
\[
Q_j
:=
\begin{tikzcd}[column sep=small, row sep=small]
W_j\ar[d]\ar[r]&
S_j\times \Spec{\Z[x_1,\ldots,x_n]/(x_{i_1},\ldots, x_{i_s})}\ar[d]
\\
\ul{X_j}\ar[d,leftarrow]\ar[r]&
S_j\times \Spec{\Z[x_1,\ldots,x_n]}\ar[d,leftarrow]
\\
Z_j\ar[r]&
S_j\times \Spec{\Z[x_1,\ldots,x_n]/(x_1\cdots x_r)}
\end{tikzcd}
\]
of cartesian squares such that $Q\cong Q_j\times_{S_j}S$.
We may further assume that the horizontal morphisms of $Q_j$ are \'etale using \cite[Proposition IV.17.7.8]{EGA},
which finishes the proof.
\end{proof}

\begin{prop}
\label{noetherian.2}
For every vector bundle $p\colon \cE\to X$ with $X\in \SmlSm/S$,
there exists $i\in I$ and a vector bundle $p_i\colon \cE_i\to X_i$ with $X_i\in \SmlSm/S_i$ such that $p\cong p_i\times_{S_i}S$.
\end{prop}
\begin{proof}
We get $i\in I$ and $X_i\in \SmlSm/S_i$ by Proposition \ref{noetherian.1}.
Let $\cF$ be the locally free sheaf of $\cO_{\ul{X}}$-modules corresponding to the vector bundle $\ul{p}\colon \ul{\cE}\to \ul{X}$.
By \cite[Th\'eor\`eme IV.8.5.2, Proposition IV.8.5.5]{EGA}
there exists a morphism $j\to i$ in $I$ and a locally free sheaf $\cF_j$ of $\cO_{\ul{X_j}}$-modules such that $\cF$ is isomorphic to the pullback of $\cF_j$.
Let $\ul{\cE_j}\to \ul{X_j}$ be the corresponding vector bundle.
Then the pullback $\cE_j:=\ul{\cE_j}\times_{\ul{X_j}}X_j\to X_j$ satisfies the desired property.
\end{proof}

\begin{prop}
\label{noetherian.3}
For every $X\in \lSm/S$,
there exists $i\in I$ and $X_i\in \lSm/S$ such that $S\times_{S_i}X_i\cong X$.
If $f\colon Y\to X$ is a morphism in $\lSm/S$,
then there exists $i\in I$ and a morphism $f_i\colon Y_i\to X_i$ in $\lSm/S_i$ such that $f\cong f_i\times_{S_i}S$.
If $f$ satisfies one of the following properties:
\begin{enumerate}
\item[\textup{(1)}] log \'etale,
\item[\textup{(2)}] log smooth,
\item[\textup{(3)}] log \'etale monomorphism,
\item[\textup{(4)}] dividing cover,
\end{enumerate}
then we may further assume that $f_i$ has the same property.
\end{prop}
\begin{proof}
The case (4) is a consequence of the case (3) and \cite[Th\'eor\`eme IV.8.10.5,(vi),(xii)]{EGA}.

For the cases (1)--(3),
we set $U:=X-\partial X$ and $V:=Y-\partial Y$.
Consider the induced commutative square
\[
Q:=
\begin{tikzcd}
V\ar[d,"v"']\ar[r]&
U\ar[d,"u"]
\\
\ul{Y}\ar[r,"f"]&
\ul{X}.
\end{tikzcd}
\]
By \cite[Th\'eor\`eme IV.8.8.2]{EGA},
there exists $i\in I$ and a commutative diagram
\[
Q_i:=
\begin{tikzcd}
V_i\ar[d,"v_i"']\ar[r]&
U_i\ar[d,"u_i"]
\\
\ul{Y_i}\ar[r,"f_i"]&
\ul{X_i}
\end{tikzcd}
\]
such that $Q\cong Q_i\times_{S_i}S$.
Using \cite[Th\'eor\`eme IV.8.10.5(iii)]{EGA},
we may further assume that $u_i$ and $v_i$ are open immersions.
Let $X_i$ (resp.\ $Y_i$) be the log scheme with the underlying scheme $\ul{X_i}$ (resp.\ $\ul{Y_i}$) and with the compactifying log structure associated with $u_i$ (resp.\ $v_i$).

We need to show that there exists a morphism $j\to i$ in $I$ such that $X_j,Y_j\in \lSm/S$, $U_j\cong X_j-\partial X_j$, $V_i\cong Y_j-\partial Y_j$, and $f_j(V_j)\subset U_j$ since we can recover a morphism $f_j\colon Y_j\to X_j$ by \cite[Proposition III.1.6.2]{Ogu}.
If $f$ is log \'etale (resp.\ log smooth),
then we also require that $f_j$ is log \'etale (resp.\ log smooth).
This question is \'etale local on $\ul{X}$ and $\ul{Y}$ by Lemma \ref{noetherian.5}.
Hence, we may assume that there exists a chart $\theta\colon P\to Q$ of $f$.
If $f$ is log \'etale (resp.\ log smooth),
then by \cite[Theorem IV.3.3.1]{Ogu},
we may also assume that the induced morphisms $X\to S\times \A_P$ and $Y\to X\times_{\A_P}\A_Q$ are strict \'etale, $\theta$ is injective,
and $\lvert \coker(\theta^\gp)\rvert$ (resp.\ $\lvert \coker(\theta^\gp)^\mathrm{tor}\rvert$) is invertible in $\cO_X$.
If $f$ is a log \'etale monomorphism,
then we also require that $f_j$ is a log \'etale monomorphism.
This question is Zariski local on $\ul{X}$ and $\ul{Y}$ by Lemma \ref{noetherian.5}.
Hence, by \cite[Lemma A.11.3]{logDM} (see also Remark \ref{logtop.10}),
we may also assume that the induced morphism $Y\to X\times_{\A_P}\A_Q$ is an open immersion and $\theta^\gp$ is an isomorphism.

By \cite[Th\'eor\`eme IV.8.8.2]{EGA},
there exists a morphism $j\to i$ and a commutative square
\[
\begin{tikzcd}[column sep=large]
\ul{Y_j}\ar[d]\ar[r,"f_j"]&
\ul{X_j}\ar[d]
\\
S_j\times \ul{\A_Q}\ar[r,"\id\times \ul{\A_\theta}"]&
S_j\times \ul{\A_P}
\end{tikzcd}
\]
such that $U_j\cong S_j\times (\A_P-\partial \A_P)$ and $V_j\cong S_j\times (\A_Q-\partial A_Q)$.
If $f$ is log \'etale (resp.\ log smooth),
then by \cite[Proposition IV.17.7.8]{EGA},
we may also assume that the induced morphisms $\ul{X_j}\to S_j\times \ul{\A_P}$ and $\ul{Y_j}\to \ul{X_j}\times_{\ul{\A_P}}\ul{\A_Q}$ are \'etale.
If $f$ is a log \'etale monomorphism,
then by \cite[Th\'eor\`eme IV.8.10.5(3)]{EGA},
we may also assume that the induced morphism $\ul{Y_j}\to \ul{X_j}\times_{\ul{\A_P}}\ul{\A_Q}$ is an open immersion.
The induced morphism of fs log schemes $f_j\colon Y_j\to X_j$ satisfies the desired properties.
\end{proof}

\subsection{Properties of (pointed) log motivic spaces}\label{proplogH}

The purpose of this subsection is to explore several basic properties that are satisfied in the unstable motivic 
homotopy categories $\inflogH$ and $\inflogHpt$ and their variants.
Many of the results can be transferred to $\inflogSHS$ and $\inflogSH$ by applying infinite suspension functors.

\begin{rmk} To study motivic properties for $\inflogHpt(S)$ we will attempt to apply the arguments in \cite[\S 7]{logDM}.  
However, 
in our setting, 
we \emph{cannot} take the following for given:
\begin{equation}
\label{logHprop.0.4}
\text{For a map $\cG\to \cF$ in $\inflogHpt(S)$, $\cofib(\cG\to \cF)\simeq 0$ implies $\cG\simeq \cF$.}
\end{equation}
But this is the only obstruction, in the sense that if an argument in \cite[\S 7]{logDM} 
does not appeal to \eqref{logHprop.0.4}, 
then it carries over to our setting.
More precisely, 
it will be important to provide alternatives to \cite[Propositions C.1.5, C.1.7]{logDM} 
since their proofs are based upon \eqref{logHprop.0.4}.
With this in mind, 
we will proceed by adapting the arguments in \cite[\S 7]{logDM} to our setting.
\end{rmk}

Suppose $X$ is a scheme in $\Sm/S$ with a strict normal crossing divisor $Z_1+\cdots+Z_n$.
If $\cY$ is the cubical horn associated with the Zariski cover $\{X-Z_i\to X-(Z_1\cap \cdots \cap Z_r)\}_{1\leq i\leq n}$ 
in the sense of Definition \ref{Thomdf.9}, then by Zariski descent there is an equivalence
\[
\colimit(\cY)\simeq X-(Z_1\cap \cdots \cap Z_r)
\]
in $\infH(S)$. 
Our compactified version of the previous equivalence reads as follows.

\begin{prop}
\label{Ori.37}
Suppose $S\in \Sch$ and $X$ is a scheme in $\Sm/S$ with a strict normal crossing divisor $Z_1+\cdots+Z_n$.
Then there is an equivalence
\[
\colimit(\horn(X,Z_1+\cdots+Z_n))
\simeq
(\Blow_{W}X,E)
\]
in $\inflogH(S)$, where $W=Z_1\cap \cdots \cap Z_n$ and $E$ is the exceptional divisor,
and see \textup{Definition \ref{Ori.38}} for $\horn$.
\end{prop}
\begin{proof}
Let $\widetilde{Z}_i$ be the strict transform of $Z_i$ in $\Blow_W X$ for all $1\leq i\leq n$.
Then $\{\Blow_{W}X-\widetilde{Z}_i\to \Blow_{W}X\}_{1\leq i\leq n}$ is a Zariski cover.
Apply Proposition \ref{Ori.5} to the scheme $\Blow_{W}E$ with the strict normal crossing divisor $\widetilde{Z}_1+\cdots+\widetilde{Z}_n+E$ to have an equivalence
\begin{equation}
\label{Ori.37.1}
\colimit(\horn((\Blow_W X,E),\widetilde{Z}_1+\cdots+\widetilde{Z}_n))
\simeq
(\Blow_W X,E)
\end{equation}
in $\inflogH(S)$.
For all integers $1\leq i_1<\cdots<i_r\leq n$, the morphism
\[
(\Blow_W X,\widetilde{Z}_{i_1}+\cdots+\widetilde{Z}_{i_r}+E)
\xrightarrow{\simeq}
(X,Z_{i_1}+\cdots+Z_{i_r})
\]
is an equivalence in $\inflogH(S)$ since it is an admissible blow-up along a smooth center.
Hence we have an equivalence
\begin{equation}
\label{Ori.37.2}
\colimit(\horn((\Blow_W X,E),\widetilde{Z}_1+\cdots+\widetilde{Z}_n))
\simeq
\colimit(\horn(X,Z_1+\cdots+Z_n))
\end{equation}
in $\inflogH(S)$.
Combine \eqref{Ori.37.1} and \eqref{Ori.37.2} to have the desired isomorphism.
\end{proof}

Next we recall logarithmic vector bundles and their Thom spaces.

\begin{df}
If $Y\in \lSch$, 
then a \emph{vector bundle on $Y$}\index{vector bundle} is a strict morphism $p\colon \cE\to Y$ in $\lSch$ 
such that $\ul{p}\colon \ul{\cE}\to \ul{Y}$ is a vector bundle.
The \emph{rank of $p$} is the rank of the underlying vector bundle $\ul{p}$.
\end{df}

\begin{df}
\label{Thomdf.1}
Suppose $S\in \Sch$ and $X\in \SmlSm/S$.
If $p\colon \cE\to X$ is a rank $d$ vector bundle with zero section $Z\to \cE$,
then the \emph{Thom space}\index{Thom space}\index[notation]{Thom @ $\Thom(\cE/X)$} associated with $p$ is defined as
\begin{equation}
\label{Thomdf.1.1}
\Thom(\cE/X)
:=
\cE/(\Blow_Z(\cE),E)\in \inflogHpt(S),
\end{equation}
where $E$ is the exceptional divisor on $\Blow_Z(\cE)$.
We often omit $X$ in the notation when no confusion seems likely to arise.
\end{df}

\begin{rmk}
In $\A^1$-homotopy theory, 
there are at least two definitions of the Thom space of a vector bundle $\ul{p}\colon \ul{\cE}\to \ul{X}$, 
namely 
\begin{equation}
\label{Thomdf.1.2}
\Thom(\ul{\cE}/\ul{X})
:=
\ul{\cE}/(\ul{\cE}-\ul{Z})
\end{equation}
and
\begin{equation}
\label{Thomdf.1.3}
\Thom(\ul{\cE}/\ul{X})
:=
\ul{p}_\sharp \ul{i}_* \one.
\end{equation}

The localization property \cite[Theorem 2.21, p.\ 114]{MV} 
implies the equivalence between \eqref{Thomdf.1.2} and \eqref{Thomdf.1.3} in $\A^1$-homotopy theory.
The localization property 
seems to fail in $\boxx$-homotopy theory as explained in the introduction of \cite{logA1}, 
so we do not expect that \eqref{Thomdf.1.2} and \eqref{Thomdf.1.3} are equivalent in $\boxx$-homotopy theory.
However, 
it is currently unclear whether \eqref{Thomdf.1.1} and \eqref{Thomdf.1.3} are equivalent or not.
\vspace{0.1in}
\end{rmk}
\begin{df}
We write $\wedge$ for the monoidal product of $\inflogHpt(S)$ and $\one$ for the monoidal unit of $\inflogHpt(S)$.
For all $\cF\in \inflogHpt(S)$ and integers $p\geq q\geq 0$, we set
\[
\Sigma^{p,q}\cF
:=
T^{\wedge q}\wedge S^{p-q}\wedge \cF
\]
and $\Sigma^{p,q}:=\Sigma^{p,q}\one$. See Definition \ref{def:sphere_suspension_logH}.
\end{df}

\begin{prop}
\label{Ori.64}
Suppose $S\in \Sch$, $X\in \SmlSm/S$, and $p\colon \cE\to X$ is a rank $d$ trivial bundle.
Then, in $\inflogHpt(S)$, there is an equivalence
\[
\Thom(\cE/X)
\simeq
\Sigma^{2d,d}X_+.
\]
\end{prop}
\begin{proof} This is an analogue of \cite[Proposition 7.4.5]{logDM}.
By base change (since the bundle is trivial), we reduce to the case when $X=S$.
Let $H_1,\ldots,H_d$ be the standard hyperplanes in $\A^d$.
Due to Proposition \ref{Ori.37}, there is an equivalence in $\inflogH(S)$
\begin{equation}
\label{Ori.64.1}
\colimit(\horn(\A^d,H_1+\cdots+H_d))
\simeq
(\Blow_O \A^d,E).
\end{equation}
Here $O:=(0,\ldots,0)\in \A^d$ and $E$ is the exceptional divisor.
Due to Proposition \ref{Thomdf.6}, there is an equivalence
\begin{equation}
\label{Ori.64.2}
\Sigma^{2d,d}
\simeq
\tcofib(\cube((\A^d,H_1+\cdots+H_d)_+))
\end{equation}
in $\inflogHpt(S)$
ee Definition \ref{Ori.38} for $\cube(-)$.
Combine \eqref{Ori.64.1} and \eqref{Ori.64.2} to deduce the desired equivalence.
\end{proof}

\begin{exm}
\label{Ori.57}
If $S\in \Sch$, 
then Propositions \ref{logH.4} and \ref{Ori.64} yield equivalences
\begin{equation}
\label{Ori.57.2}
\P^n/(\Blow_O(\P^n),E)
\simeq
\A^n/(\Blow_O(\A^n),E)
\simeq
\Sigma^{2n,n}
\end{equation}
in $\inflogHpt(S)$ for every integer $n\geq 1$.
\end{exm}

We will frequently encounter the following situation.

\begin{lem}
\label{ProplogSH.7}
Suppose $S\in \Sch$ and let $f_0,f_1\colon X\to Y$ be morphisms in $\SmlSm/S$.
If there is a diagram
\[
\begin{tikzcd}
X\ar[r,shift left=0.75ex,"i_1'"]\ar[r,shift right=0.75ex,"i_0'"']&
X'\ar[d,"p"]\ar[r,"g"]&
Y
\\
&
X\times \boxx
\end{tikzcd}
\]
in $\SmlSm/S$ such that $p$ is a composite of admissible blow-ups along smooth centers,
$pi_0'$ (resp.\ $pi_1'$) is the zero section (resp.\ one section), $f_0=gi_0'$, and $f_1=gi_1'$, then $f_0$ and $f_1$ are
homotopic in $\inflogH(S)$.
\end{lem}
\begin{proof}
By $\square$-invariance and invariance under admissible blow-ups along smooth centers, $pi_0'$ and $pi_1'$ are homotopic in $\inflogH(S)$.
From these, we deduce that $i_0'$ and $i_1'$ are homotopic in $\inflogH(S)$.
Compose $i_0'$ and $i_1'$ with $g$ to complete the proof.
\end{proof}

\begin{prop}
\label{Ori.40}
Suppose $\cE_1 \to X_1$ and $\cE_2\to X_2$ are vector bundles in $\SmlSm/S$, where $S\in \Sch$.
Then there is a canonical equivalence
\begin{equation}
\label{Ori.40.1}
\Thom(\cE_1)\wedge \Thom(\cE_2)
\xrightarrow{\simeq}
\Thom(\cE_1\times_S \cE_2)
\end{equation}
in $\inflogHpt(S)$.
\end{prop}
\begin{proof}
When $S\in \Sch_{noeth}$, the claim
was proved in \cite[Proposition 7.4.15]{logDM} in the stable setting. The noetherianity was not actually used in the proof. 
Since the proof used \eqref{logHprop.0.4}, we will explain an alternative approach.

We set $T:=\cE_1\times_S \cE_2\times_S \cE_3$, where $\cE_3:=S$
(in \cite[Proposition 7.4.16]{logDM} we used a nontrivial $\cE_3$).
In \cite[Construction 7.4.14]{logDM}, 
we constructed 55 fs log schemes $T_I$ corresponding to all nonempty subset $I\subset\{1,\ldots,6\}$ such that 
$I\cap \{4,5,6\}\neq \{4,5\}$.
Furthermore, 
it was shown that the morphisms $T_{14}\to T_1$, $T_{24}\to T_2$, and $T_{124}\to T_{12},T_4$ 
are composites of admissible blow-ups along smooth centers.

To construct \eqref{Ori.40.1}, 
we first observe that $\Thom(\cE_1)\wedge \Thom(\cE_2)$ is equivalent to the total cofiber of the square
\[
Q
:=
\begin{tikzcd}
T_{12}\ar[d]\ar[r]&
T_1 \ar[d]
\\
T_2\ar[r]&
T.
\end{tikzcd}
\]
Use the invariance under admissible blow-ups along smooth centers to see that $\tcofib(Q)$ is equivalent 
to the total cofiber of the square
\[
Q'
:=
\begin{tikzcd}
T_{124}\ar[d]\ar[r]&
T_{14} \ar[d]
\\
T_{24}\ar[r]&
T.
\end{tikzcd}
\]
Step 3 of the proof of \cite[Proposition 7.4.15]{logDM} shows that the total cofiber of the square
\[
Q''
:=
\begin{tikzcd}
T_{124}\ar[d]\ar[r]&
T_{14}\ar[d]
\\
T_{24}\ar[r]&
T_4.
\end{tikzcd}
\]
is equivalent to $0$ in $\inflogHpt(S)$.
Since there is a cocartesian square
\[
\begin{tikzcd}
\tcofib(Q')\ar[d]\ar[r]&
\tcofib(Q'')\ar[d]
\\
T_4\ar[r]&
T,
\end{tikzcd}
\]
we have an equivalence
\[
\tcofib(Q')
\simeq
T/T_4
=
\Thom(\cE_1\times_S \cE_2).
\]
This finishes the construction of the equivalence \eqref{Ori.40.1}.
\end{proof}

\begin{prop}
\label{Ori.65}
Suppose $S\in \Sch$ and let $\cE_i\to X$ be a vector bundle in $\SmlSm/S$ for $i=1,2,3$.
Then the diagram
\[
\begin{tikzcd}
\Thom(\cE_1)\wedge \Thom(\cE_2)\wedge \Thom(\cE_3)\ar[r,"\simeq"]\ar[d,"\simeq"']&
\Thom(\cE_1)\wedge \Thom(\cE_2\times_S \cE_3)\ar[d,"\simeq"]
\\
\Thom(\cE_1\times_S \cE_2)\wedge \Thom(\cE_3)\ar[r,"\simeq"]&
\Thom(\cE_1\times_S \cE_2\times_S \cE_3)
\end{tikzcd}
\]
commutes in $\inflogHpt(S)$.
\end{prop}
\begin{proof}
As before,
we reduce to the case when $S\in \Sch_{noeth}$.
With this assumption, the claim
in the stable setting was established in \cite[Proposition 7.4.16]{logDM}.
However, 
the proof does not use \eqref{logHprop.0.4}, 
and hence applies verbatim to $\inflogHpt(S)$.
\end{proof}

We often regard $\P^{n-1}$ as the hyperplane defined by $x_n=0$ in $\P^n$, 
where we write $[x_0:\cdots:x_n]$ for the coordinate on $\P^n$.
We are interested in the fs log scheme $(\P^n,\P^{n-1})$.\index[notation]{PnPn1 @ $(\P^n,\P^{n-1})$} Clearly, this is a log compactification of $\A^n$ different from $\boxx^n$.
The contractibility of $\boxx$ trivially implies the contractibility of $\boxx^n$, but the following is not an immediate result. In fact, it is one of the main features of our theory, as seen in \cite{logDM}.

\begin{prop}
\label{ProplogSH.4}
For all $S\in \Sch$, 
$X\in \SmlSm/S$, 
and $n\geq 1$, the map
\[
X\times (\P^n,\P^{n-1})
\to
X
\]
naturally induced by the projection $X\times (\P^n,\P^{n-1})\to X$ is an equivalence in $\inflogH(S)$.
\end{prop}
\begin{proof}
The question is Zariski local on $S$.
Hence we may assume that $S$ is affine.
By Proposition \ref{noetherian.1},
we reduce to the case when $S\in \Sch_{noeth}$.
With this assumption, the claim
was proven in \cite[Proposition 7.3.1]{logDM} in the stable setting.
Since \eqref{logHprop.0.4} was not used in the proof, the same argument holds.
\end{proof}

\begin{prop}
\label{Ori.41}
Suppose $\cE\to X$ is a vector bundle in $\SmlSm/S$, where $S\in \Sch$.
Then there is a canonical equivalence
\[
\Thom(\cE)
\simeq
\P(\cE\oplus \cO)/\P(\cE)
\]
in $\inflogHpt(S)$, where we view $\P(\cE)$ as the divisor at $\infty$.
\end{prop}
\begin{proof}
As before,
we reduce to the case when $S\in \Sch_{noeth}$.
With this assumption, the claim
was proven in \cite[Proposition 7.4.5]{logDM} in the stable setting.
Again, \eqref{logHprop.0.4} was not used in the proof.
\end{proof}

The following Lemma is a key result for understanding the $\P^1$-stabilization. We are grateful to the anonymous referee for suggesting the current proof, which significantly simplifies our previous argument.
\begin{lem}
\label{Ori.55}
Suppose $S\in \Sch$ and let $f\colon \Thom(\cO^2)\to \Thom(\cO^2)$ be the map given by the matrix
\[
\rho:=
\left(
\begin{array}{cc}
1 & 1
\\
0 & 1
\end{array}
\right)
\colon
\cO^2
\to 
\cO^2.
\]
Then there is a homotopy $f\simeq \id$ in $\inflogHpt(S)$.
\end{lem}
\begin{proof}
The matrix $\rho$ extends to a morphism $g\colon \P^2\to \P^2$ given by
\[
[z_0:z_1:z_2] \mapsto [z_0+z_1:z_1:z_2],
\]
and this sends the divisor $\P^1$ at $\infty$ into $\P^1$.
Using Proposition \ref{Ori.41},
we see that $f$ is identified with the naturally induced morphism $g\colon \P^2/\P^1\to \P^2/\P^1$.

Consider the rational map $h\colon \P^2\times \P^1\dashrightarrow \P^2$ given by
\[
([z_0:z_1:z_2],[t_0:t_1])\mapsto [t_1 z_0+t_0z_1:t_1z_1:t_1z_2],
\]
whose restriction to $\P^2\times \A^1$ yields an $\A^1$-homotopy between $g$ and $\id$, and its further restriction to $\P^1\times \A^1$ is an $\A^1$-homotopy too.
To use Lemma \ref{ProplogSH.7},
let us find a suitable composite of admissible blow-ups along smooth centers $Y\to \P^2\times \boxx$ such that $h$ is defined on $\ul{Y}$.

On $\P^2\times (\P^1-O)\cong \P^2\times \A^1$, the rational map is given by
\begin{equation}
\label{Ori.55.13}
([z_0:z_1:z_2],t)\mapsto [tz_0+z_1:tz_1:tz_2].
\end{equation}
If $z_1=1$, then \eqref{Ori.55.13} is a morphism.
If $z_0=1$ or $z_2=1$, then we can write \eqref{Ori.55.13} as
\[
(z_1,z_2,t)\mapsto [t+z_1:tz_1:tz_2],
\text{ }
(z_0,z_1,t)\mapsto [tz_0+z_1:tz_1:t].
\]
Take the admissible blow-up along the line $L:=(z_1=t=0)$ and use the new coordinates $w_1:=z_1/t$ and $u:=t/z_1$ to get
\begin{gather*}
(w_1,z_2,t)\mapsto [1+w_1:tw_1:z_2],
\text{ }
(z_1,z_2,u)\mapsto [u+1:uz_1:uz_2]
\\
(z_0,w_1,t)\mapsto [z_0+w_1,tw_1:1],
\text{ }
(z_0,z_1,u)\mapsto [uz_0+1:uz_1:u].
\end{gather*}
The third and fourth rational maps are morphisms.
The first and second rational maps become morphisms after taking the admissible blow-up along the point
\[
P:=(w_1=-1)\cap (z_2=t=0)=(u=-1)\cap (z_1=z_2=0).
\]
In summary, $h$ is defined on $\Blow_P (\Blow_L (\P^2\times \P^1))$.
Let $Y$ be the fs log scheme with this underlying scheme and the compactifying log structure associated with $\P^2\times \A^1$.
To finish the proof, apply Lemma \ref{ProplogSH.7} to the composite of admissible blow-ups $Y\to \P^2\times \square$ along smooth centers.
\end{proof}
\begin{rmk}
    The analogous result of   Lemma \ref{Ori.55} in $\A^1$-homotopy theory can be easily proven by considering the $\A^1$-homotopies
\[
\A^2\times \A^1\to \A^2
\text{ and }
(\A^2-(0,0))\times \A^1\to \A^2-(0,0)
\]
given by $(x,y,t)\mapsto (x+ty,y)$. In our case we consider a compactification of the above morphisms, and this creates a nontrivial geometric problem.
\end{rmk}

\begin{prop}
\label{Ori.56}
If $S\in \Sch$, 
then the permutation $(123)\colon T^{\wedge 3} \to T^{\wedge 3}$ is equivalent to the identity on $\inflogHpt(S)$. In particular, $T$ is a symmetric object of $\inflogHpt(S)$.
\end{prop}
\begin{proof}
Owing to Proposition \ref{Ori.64}, there is an equivalence in $\inflogHpt(S)$
\[
T^{\wedge n}
\simeq
\Thom(\cO^n).
\]
Let $f\colon \Thom(\cO^3)\to \Thom(\cO^3)$ be the map given by the matrix
\[
A:=\left(
\begin{array}{ccc}
0 & 1 & 0
\\
0 & 0 & 1
\\
1 & 0 & 0
\end{array}
\right)
\colon
\cO^3 \to \cO^3.
\]
It suffices to show that $f\simeq \id$ in $\inflogHpt(S)$.
Similarly to the above,
let $g\colon \Thom(\cO^2)\to \Thom(\cO^2)$ be the map given by the matrix
\[
B:=
\left(
\begin{array}{cc}
0 & 1
\\
-1 & 0
\end{array}
\right)
\colon 
\cO^2 \to \cO^2.
\]
Since the two matrices
\[
\left(
\begin{array}{cc}
1 & 1
\\
0 & 1
\end{array}
\right)
\text{ and }
\left(
\begin{array}{cc}
1 & 0
\\
-1 & 1
\end{array}
\right)
\]
are conjugate,
we can apply Lemma \ref{Ori.55}  to the relation
\[
B
=
\left(
\begin{array}{cc}
1 & 1
\\
0 & 1
\end{array}
\right)
\left(
\begin{array}{cc}
1 & 0
\\
-1 & 1
\end{array}
\right)
\left(
\begin{array}{cc}
1 & 1
\\
0 & 1
\end{array}
\right)
\]
to obtain a homotopy
\begin{equation}
\label{Ori.56.1}
g\simeq \id
\end{equation}
in $\inflogHpt(S)$.
Apply \eqref{Ori.56.1} to the relation
\[
A
=
\left(
\begin{array}{ccc}
0 & -1 & 0
\\
1 & 0 & 0
\\
0 & 0 & 1
\end{array}
\right)
\left(
\begin{array}{ccc}
1 & 0 & 0
\\
0 & 0 & -1
\\
0 & 1 & 0
\end{array}
\right)
\]
to finish the proof.
\end{proof}

\begin{prop}
\label{Ori.66}
Suppose $S\in \Sch$, $X\in \SmlSm/S$, 
and $Z\in \SmlSm/S$ is a strict closed subscheme of $X$ having strict normal crossing with $\partial X$ over $S$ such that $Z$ is contained in $\partial X$.
If $f\colon X'\to X$ is a strict \'etale morphism such that $Z':=f^{-1}(Z)\to Z$ is an isomorphism, 
then the naturally induced square
\begin{equation}
\label{Ori.66.1}
\begin{tikzcd}
(\Blow_{Z'}X',E')\ar[d]\ar[r]&
X'\ar[d]
\\
(\Blow_Z X,E)\ar[r]&
X
\end{tikzcd}
\end{equation}
is cocartesian in $\inflogH(S)$, where $E$ and $E'$ are the exceptional divisors.
\end{prop}
\begin{proof}
This was proven in \cite[Proposition 7.4.2]{logDM} in the stable setting.
To avoid \eqref{logHprop.0.4}, consider the naturally induced commutative diagram
\[
\begin{tikzcd}
X'-Z'\ar[d]\ar[r]\ar[rd,description,phantom,"(a)"]&
(\Blow_{Z'}X',E')\ar[d]\ar[r]\ar[rd,description,phantom,"(b)"]&
X'\ar[d]
\\
X-Z\ar[r]&
(\Blow_Z X,E)\ar[r]&
X.
\end{tikzcd}
\]
Both $(a)$ and $(a)\cup (b)$ are strict Nisnevich distinguished squares.
Use the pasting law to deduce that 
\eqref{Ori.66.1} is cocartesian in $\inflogH(S)$.
\end{proof}

Suppose $S\in \Sch$, and that $X$ in $\SmlSm/S$ is of the form $(\ul{X},Z_1+\cdots+Z_r)$. Let 
$Z$ be a smooth closed subscheme of $X$ having strict normal crossing with $Z_1+\cdots+Z_r$. 
We define\index[notation]{DZX @ $\Deform_Z X$}
\[
\Deform_Z X
:=
\Blow_{Z}(X\times \boxx)-\Blow_Z X,
\]
where we regard $Z$ as the closed subscheme $Z\times \{0\}$ of $X\times \boxx$.
The underlying scheme of the blow-up $\Blow_Z (X\times \boxx)$ contains the underlying scheme of $Z\times \boxx$ as a closed subscheme.
Let $\Normal_Z X:=(\Normal_{\ul{Z}}\ul{X})\times_{\ul{X}}X$ 
be the normal bundle\index{normal bundle}\index[notation]{NZX @ $\Normal_Z X$} of $Z$ in $X$.
With these definitions, 
we obtain the commutative diagram for the deformation to the normal cone
\begin{equation}
\label{Thom.1.6}
\begin{tikzcd}[row sep=small]
\Blow_Z X\ar[r]\ar[d]&
\Blow_{Z\times \boxx}(\Deform_Z X)\ar[r,leftarrow]\ar[d]&
\Blow_Z (\Normal_Z X)\ar[d]
\\
X\ar[r]\ar[d]&
\Deform_Z X\ar[r,leftarrow]\ar[d]&
\Normal_Z X\ar[d]
\\
\pt\ar[r,"i_1"]&
\boxx\ar[r,leftarrow,"i_0"]&
\pt,
\end{tikzcd}
\end{equation}
where $i_0$ is the zero section and $i_1$ is the one section. It is a compactified version of the $\mathbb{A}^1$-deformation to the normal cone.

\begin{lem}
\label{Ori.36}
Suppose $\cE\to X$ is a rank $n$ vector bundle in $\lSm/S$, where $S\in \Sch$.
Let $Z$ be the zero section of $\cE$,
which is isomorphic to $X$.
With the identification $\cE\cong \Normal_Z \cE$, the two maps
\begin{equation}
\label{Ori.36.3}
\cE/(\Blow_Z \cE,E)
\rightrightarrows
\Deform_Z \cE/(\Blow_{Z\times \boxx}(\Deform_Z \cE),E^D)
\end{equation}
obtained by \eqref{Thom.1.4} are homotopic in $\inflogHpt(S)$, where $E$ and $E^D$ are the exceptional divisors.
\end{lem}
\begin{proof}
We assume that $X$ has trivial log structure for simplicity.
The general case follows from the same argument.
With this assumption,
let us construct a natural open immersion $\ul{\Deform_Z \cE}\to \Blow_{Z\times \infty}(\cE\times \P^1)$.
If $\cE$ is a trivial rank $n$ bundle and $X=\Spec{A}$ for some ring $A$,
then
\[
\ul{\Deform_Z \cE}
\cong
\Spec{A[x_1/t,\ldots,x_n/t,t]}
\cup
\Spec{A[x_1,\ldots,x_n,1/t]}.
\]
With $y_i:=x_i/t$ for each $i$,
we have
\[
\ul{\Deform_Z \cE}
\cong
\Spec{A[y_1,\ldots,y_n,t]}
\cup
\Spec{A[y_1t,\ldots,y_nt,1/t]}.
\]
Compare this with the expression
\[
\cE\times \P^1
\cong
\Spec{A[y_1,\ldots,y_n,t]}\cup \Spec{A[y_1,\ldots,y_n,1/t]}
\]
to construct $\ul{\Deform_Z \cE}\to \Blow_{Z\times \infty}(\cE\times \P^1)$ for the local case. We can glue these local constructions to construct it for the general case too.

Consider the naturally induced diagram
\begin{equation*}
\begin{tikzcd}[column sep=small, row sep=small]
{[\Blow_{Z\times \P^1}(\ul{\Deform_{Z}\cE}),(\cE-Z)\times \A^1]}\ar[d]\ar[r,"\alpha"]
&
{[\ul{\Deform_{Z}\cE},\cE\times \A^1]}\ar[d]
\\
{[\Blow_{Z\times \P^1}(\Blow_{Z\times \infty}(\cE\times \P^1)),(\cE-Z)\times \A^1]}\ar[d]\ar[r]
&
{[\Blow_{Z\times \infty}(\cE\times \P^1),\cE\times \A^1]}\ar[d]
\\
{[\Blow_{Z\times \P^1}(\cE\times \P^1),(\cE-Z)\times \A^1]}\ar[r,"\beta"]
&
{[\cE\times \P^1,\cE\times \A^1]},
\end{tikzcd}
\end{equation*}
where the notation $[Y,V]$ for an open immersion of scheme $V\to Y$ means $(Y,Y-V)$.
The upper square is cartesian by Proposition \ref{Ori.66},
and the lower vertical morphisms are equivalences in $\inflogHpt(S)$ due to invariance under admissible blow-ups along smooth centers.
With this in hand,
the two morphisms
\begin{equation}
\label{Ori.36.6}
\cE/(\Blow_Z \cE,E)
\rightrightarrows
\Deform_Z \cE/(\Blow_{Z\times \boxx}(\Deform_Z \cE),E^D)
=
\cofib(\alpha)
\end{equation}
are identified with the zero and one sections
\begin{equation}
\label{Ori.36.7}
\cE/(\Blow_Z \cE,E)
\rightrightarrows
(\cE\times \square)/((\Blow_Z \cE,E)\times \square)=\cofib(\beta)
\end{equation}
We finish the proof by $\square$-invariance.
\end{proof}

\begin{thm}
\label{Thom.1}
Suppose $S\in \Sch$ and $X\in \SmlSm/S$. Let $Z$ be a smooth strict closed subscheme of $X$ having strict normal crossing with $\partial X$.
There are canonical equivalences in $\inflogHpt(S)$
\begin{equation}
\label{Thom.1.4}
X/(\Blow_Z X,E)
\xrightarrow{\simeq}
\Deform_Z X/(\Blow_{Z\times \boxx}(\Deform_Z X),E^D)
\xleftarrow{\simeq}
\Normal_Z X/(\Blow_Z (\Normal_Z X),E^N).
\end{equation}
 
Here, $E$, $E_D$, and $E_N$ denote the exceptional divisors on $\Blow_Z X$, 
$\Blow_{Z\times \boxx}(\Deform_Z X)$, 
and $\Blow_Z (\Normal_Z X)$, 
respectively.
\end{thm}
\begin{proof}
As before,
use Proposition \ref{noetherian.1} to reduce to the case where $S\in \Sch_{noeth}$.
With this assumption,
this was proven in \cite[Theorem 7.5.4]{logDM} in the stable setting.
Our generalization to the unstable case modifies parts of the proof detailed below.

In Step 3 of the proof of \cite[Theorem 7.5.4]{logDM}, 
we use Proposition \ref{Ori.66} instead of \cite[Proposition 7.4.2]{logDM}.

Using Steps 1--4 in \cite[Theorem 7.5.4]{logDM}, we reduce to the case when $X=\A_S^p$ and $Z=\{0\}\times S$.
With $\cE:=X$,
we need to show that the two morphisms in \eqref{Ori.36.6} are equivalences in $\inflogHpt(S)$ and equivalently the two morphisms in \eqref{Ori.36.7} are equivalences in $\inflogHpt(S)$.
We finish the proof by $\square$-invariance.
\end{proof}

As a consequence of Theorem \ref{Thom.1}, 
we obtain the Gysin equivalence\index{Gysin equivalence}\index[notation]{pxZ @ $\mathfrak{p}_{X,Z}$}
\begin{equation}
\label{Thom.1.5}
\mathfrak{p}_{X,Z}
\colon
 X/(\Blow_Z X,E)
\xrightarrow{\simeq}
\Thom(\Normal_Z X)
\end{equation}
in $\inflogHpt(S)$.
Moreover, we have the commutative diagram
\begin{equation}
\label{THom.1.10}
\begin{tikzcd}
\Blow_Z X\ar[r]\ar[d]&
\Blow_{Z\times \boxx}(\Blow_Z (X\times \boxx))\ar[d]\ar[r,leftarrow]&
\Blow_Z(\P(\Normal_Z X\oplus \cO))\ar[d]
\\
X\ar[r]&
\Blow_Z (X\times \boxx)\ar[r,leftarrow]&
\P(\Normal_Z X\oplus \cO).
\end{tikzcd}
\end{equation}
Combining Proposition \ref{Ori.41} and \eqref{Thom.1.4} yields canonical equivalences in $\inflogHpt(S)$
\begin{equation}
\label{Thom.1.9}
\begin{split}
 X/(\Blow_Z X,E)
&\xrightarrow{\simeq}
 \Blow_Z(X\times \boxx)/(\Blow_{Z\times \boxx}(\Blow_Z (X\times \boxx)),E^D)
\\
&\xleftarrow{\simeq}
 \P(\Normal_Z X\oplus \cO)/(\Blow_Z(\P(\N_Z X\oplus \cO)),E^N).
\end{split}
\end{equation}

\begin{rmk}\label{rmk:Gysin}The Gysin equivalence \eqref{Thom.1.5} provides a cofiber sequence in $\inflogHpt(S)$
\[
(\Blow_Z X,E)_+ \to X_+ \to \Thom(\Normal_Z X)
\]
that we call the \emph{Gysin sequence} or the \emph{residue sequence} \index{residue sequence} for the pair $(X,Z)$. This is functorial in $(X,Z)$ in the following sense:
suppose that
\[
\begin{tikzcd}
Z'\ar[r]\ar[d]&
X'\ar[d]
\\
Z\ar[r]&
X
\end{tikzcd}
\]
is a cartesian square in $\SmlSm/S$ such that $Z$ and $Z'$ have strict normal crossing with $\partial X$ and $\partial X'$ respectively.
Then there is a commutative diagram
\begin{equation}
\label{Thom.1.8}
\begin{tikzcd}[row sep=tiny, column sep=tiny]
&
\Blow_{Z'} X'\ar[rr]\ar[dd]\ar[ld]&&
\Blow_{Z'\times \boxx}(\Deform_{Z'} X')\ar[rr,leftarrow]\ar[dd]\ar[ld]&&
\Blow_{Z'} (\Normal_{Z'} X')\ar[dd]\ar[ld]
\\
\Blow_Z X\ar[rr,crossing over]\ar[dd]&&
\Blow_{Z\times \boxx}(\Deform_Z X)\ar[rr,leftarrow,crossing over]&&
\Blow_Z (\Normal_Z X)
\\
&
X'\ar[rr]\ar[ld]&&
\Deform_{Z'} X'\ar[rr,leftarrow]\ar[ld]&&
\Normal_{Z'} X'\ar[ld]
\\
X\ar[rr]&&
\Deform_Z X\ar[rr,leftarrow]\ar[uu,crossing over,leftarrow]&&
\Normal_Z X.\ar[uu,crossing over,leftarrow]
\end{tikzcd}
\end{equation}

For the above we obtain the naturally induced commutative diagram
\begin{equation}
\label{Thom.1.7}
\begin{tikzcd}
X/(\Blow_Z X,E)\ar[d]\ar[r,"\mathfrak{p}_{X,Z}"]&
\Thom(\Normal_Z X)\ar[d]
\\
X'/(\Blow_{Z'} X',E')\ar[r,"\mathfrak{p}_{X',Z'}"]&
\Thom(\Normal_{Z'}X').
\end{tikzcd}
\end{equation}
\end{rmk}

\begin{thm}
\label{ProplogSH.3}
Suppose $S=\Spec{k}$,
where $k$ is a perfect field that admits resolution of singularities.
If $f\colon Y\to X$ is an admissible blow-up in $\SmlSm/S$,
then $f$ is an equivalence in $\inflogH(k)$.
\end{thm}
\begin{proof}
By resolution of singularities,
there exists a sequence of admissible blow-ups $Y_n\to \cdots \to Y_1\to X$ along smooth centers such that the composite morphism $Y_n\to X$ factors through $f$.
As in \cite[Propositions 7.6.4,  7.6.6]{logDM},
one shows that the classes of admissible blow-ups along smooth centers and 
admissible blow-ups in $\SmlSm/S$ satisfy a calculus of right fractions.
To finish the proof, 
we appeal to \cite[Lemma C.2.1]{logDM}.
\end{proof}

\subsection{A universal property of \texorpdfstring{$\inflogSH$}{logSH}}
\label{subsection:universal}
It is a well-known fact that the category of presheaves of sets on an ordinary category $\bC$ is the free cocompletion of $\bC$. As such, it enjoys an obvious universal property. Passing from presheaves of sets to a homotopical context,
Lurie showed that for every $\infty$-category $\infC$,
the Yoneda embedding
\[
\infC\to \infPsh(\infC)
\]
exhibits $\infPsh(\infC)$
as a universal presentable $\infty$-category 
built from $\infC$ (this is a rephrasing of a classical result of Dugger,  \cite[Proposition 2.1]{Dugger}).
Based on Lurie's result, 
Robalo \cite[Corollary 2.39]{RobaloThesispublished} proved that $\infSH$ enjoys an intuitive universal property among the stable presentable $\infty$-categories with a functor from $\Sch$.
We follow Robalo's argument and deduce an analogous statement for $\inflogSH$. The proofs are  formal, and we leave the details to the interested reader.

\begin{prop}
\label{univlogSH.1}
For every $S\in \Sch$ and $\infty$-category $\infD$ with small colimits, 
the Yoneda functor $\lSm/S\to \inflogH(S)$ naturally induces a fully faithful functor
\[
\Fun^{\mathrm{L}}(\inflogH(S),\infD)
\to
\Fun(\SmlSm/S,\infD).
\]
Here the left-hand side is the full subcategory of
$\Fun(\inflogH(S),\infD)$ spanned by functors preserving colimits.
Its essential image consists of functors satisfying $\boxx$-invariance, strict Nisnevich descent, and $\SmAdm$-invariance. 
\end{prop}
\begin{proof}
As in \cite[Theorem 2.30]{RobaloThesispublished}, 
this follows from the universal property of $\infPsh(-,\infSpc)$ 
and the universal property of localizations given in \cite[Proposition 5.5.4.20]{HTT}.
\end{proof}

Equipping $\SmlSm/S$ with the cartesian symmetric monoidal structure,
we have the naturally induced symmetric monoidal functor
\begin{equation}
\label{univlogSH.4.1}
\SmlSm/S^\times
\to
\inflogH(S)^\otimes.
\end{equation}
where the target is equipped with the Day convolution product. We have the following monoidal enhancement of Proposition \ref{univlogSH.1}.

\begin{prop}
\label{univlogSH.4}
For every $S\in \Sch$ and symmetric monoidal $\infty$-category $\infD^\otimes$ with small colimits, 
the functor \eqref{univlogSH.4.1} naturally induces a fully faithful functor
\[
\Fun^{\otimes,\mathrm{L}}(\inflogH(S)^{\otimes},\infD^\otimes)
\to
\Fun^{\otimes}(\SmlSm/S^\times,\infD^\otimes).
\]
Here the left-hand side is the full subcategory of $\Fun^{\otimes}(\inflogH (S)^{\otimes},\infD^\otimes)$ 
spanned by functors preserving colimits.
Its essential image consists of the functors satisfying $\boxx$-invariance, strict Nisnevich descent, and invariance with respect to $\SmAdm$.
\end{prop}

We can pass to the pointed versions.
Let $\LPrpt$ be the full subcategory $\LPr$ spanned by pointed presentable $\infty$-categories 
and the functors that preserve pointed objects.
As observed in the paragraph prior to \cite[Corollary 2.32]{RobaloThesispublished}, 
the inclusion functor $\CAlg(\LPrpt)\to \CAlg(\LPr)$ admits a left adjoint
\begin{equation}
\label{univlogSH.3.2}
\CAlg(\LPr)\to \CAlg(\LPrpt)
\end{equation}
sending $\infC\in \LPr$ to $\infC_\ast$.
Note that we have a canonical equivalence of $\infty$-categories
\begin{equation}
\label{univlogSH.3.1}
\inflogH_{\ast}(S)
\simeq
\inflogH(S)_\ast
\end{equation}
justifying our choice of notation.

\begin{thm}
\label{Ori.68}
Let $f\colon Y\to X$ be an admissible blow-up along a smooth center (in the sense of Definition \ref{logH.35}) in $\SmlSm/S$,
where $S\in \Sch$.
Then the morphism
\[
\Sigma_{S^1}^\infty f_+
\colon
\Sigma_{S^1}^\infty Y_+
\to
\Sigma_{S^1}^\infty X_+
\]
in $\inflogSH^\eff(S)$ is an equivalence.
\end{thm}
\begin{proof}
The proof of \cite[Theorem 7.2.10]{logDM} applies.
\end{proof}

\begin{cor}
Let $S\in \Sch$.
Then there exists a colimit preserving symmetric monoidal functor
\[
\Sigma_{S^1}^\infty 
\colon
\inflogHpt(S)
\to
\inflogSH^\eff(S)
\]
sending $X_+$ to $\Sigma_{S^1}^\infty X_+$ for all $X\in \SmlSm/S$.
\end{cor}
\begin{proof}
This is a consequence of Theorem \ref{Ori.68} and the universality of $\inflogHpt$ in Proposition \ref{univlogSH.4}.
\end{proof}

\begin{rmk}
\label{Ori.74}
Hence for $S\in \Sch$,
the properties proved for $\inflogHpt(S)$ in \S \ref{proplogH} can be transported to $\inflogSH^\eff(S)$ and further to $\inflogSH(S)$. In particular, $T\in \inflogSH^\eff(S)$ is a symmetric object by Theorem \ref{Ori.56}.
\end{rmk}

In the stable case, we have the following result.
\begin{prop}
\label{univlogSH.5}
For every $S\in \lSch$ and pointed presentable symmetric monoidal stable $\infty$-category $\infD^\otimes$, 
there is a naturally induced fully faithful functor
\[
\Fun^{\otimes,\mathrm{L}}(\inflogSH(S)^\otimes,\infD^\otimes)
\to
\Fun^{\otimes}(\lSm/S^\times ,\infD^\otimes).
\]
Its essential image consists of the functors $\cF$ satisfying $\boxx$-invariance and dividing Nisnevich descent 
such that the cofiber of $\cF(1)\xrightarrow{\cF(i_1)} \cF(\P^1)$ is an invertible object,
where $i_1$ is the one section.
\end{prop}
\begin{proof}
We can use the analogue of \eqref{univlogSH.3.1} and \eqref{univlogSH.3.2} replacing $\inflogH(S)^{\otimes}$ with the category of $\boxx$-local sheaves $\infShv_{dNis}(\lSm/S)[\boxx^{-1}]^\otimes$ and its pointed variant to have an equivalence of $\infty$-categories
\[
\Fun^{\otimes,\mathrm{L}}(\infShv_{dNis}(\lSm/S)_\ast[\boxx^{-1}]^\otimes,\infD^\otimes)
\xrightarrow{\simeq}
\Fun^{\otimes,\mathrm{L}}(\infShv_{dNis}(\lSm/S)[\boxx^{-1}]^\otimes,\infD^\otimes).
\]
Since $\cD$ is stable,
use \cite[Proposition 2.9]{RobaloThesispublished} for $S^1$ to have an equivalence of $\infty$-categories
\[
\Fun^{\otimes,\mathrm{L}}(\inflogSH^\eff(S)^\otimes,\infD^\otimes)
\xrightarrow{\simeq}
\Fun^{\otimes,\mathrm{L}}(\infShv_{dNis}(\lSm/S)[\boxx^{-1}]^\otimes,\infD^\otimes).
\]
By \cite[Proposition 2.9]{RobaloThesispublished} again and Remark \ref{Ori.74}, 
we have a fully faithful functor of $\infty$-categories
\[
\Fun^{\otimes,\mathrm{L}}(\inflogSH(S)^\otimes,\infD^\otimes)
\to
\Fun^{\otimes,\mathrm{L}}(\inflogSH^\eff(S)^\otimes,\infD^\otimes)
\]
whose essential image consists of the functors mapping $\P^1/1$ to an invertible object.
The argument given in the proof of  Proposition \ref{univlogSH.4} allows to conclude.
\end{proof}

We now discuss some properties that hold in
$\inflogSH^\eff$,
which can be also transported to  $\inflogSH$ using the functor $\Sigma_{\Gmm}^\infty$.

\begin{thm}
\label{ProplogSH.5}
Suppose $S\in \Sch$ and $Z\to X$ is a closed immersion in $\Sm/S$. 
Then the naturally induced square
\[
\begin{tikzcd}
\Sigma_{S^1}^\infty(Z\times_X \Blow_Z X)_+\ar[d]\ar[r]&
\Sigma_{S^1}^\infty\Blow_Z X_+\ar[d]
\\
\Sigma_{S^1}^\infty Z_+\ar[r]&
\Sigma_{S^1}^\infty X_+
\end{tikzcd}
\]
is cocartesian in $\inflogSH^\eff(S)$.
\end{thm}
\begin{proof}
As before,
use Proposition \ref{noetherian.1} to reduce to the case where $S\in \Sch_{noeth}$.
With this assumption,
this is \cite[Theorem 7.3.3]{logDM}.
The outline of the proof is as follows.
We first reduce to the case where $X=Z\times \A^n$ for some integer $n\geq 0$ using the deformation to the normal cone.
After compactification,
we work with the case where $X=Z\times (\P^n,\P^{n-1})$.
Then the morphism $\Sigma_{S^1}^\infty Z_+\to \Sigma_{S^1}^\infty X_+$ is an equivalence in $\inflogSH^\eff(S)$ by Proposition \ref{ProplogSH.4}.
Since $\Blow_Z X$ is a $\boxx$-bundle on $\P_Z^{n-1}$,
the morphism $\Sigma_{S^1}^\infty(Z\times_X \Blow_Z X)_+\to \Sigma_{S^1}^\infty\Blow_Z X_+$ is an equivalence in $\inflogSH^\eff(S)$.
\end{proof}

We end this subsection by providing the $(\P^n,\P^{n-1})$-invariance.

\begin{thm}\label{thm:PnPn-1_invariance}
For all $S\in \lSch$, 
$X\in \lSm/S$ and integer $n\geq 1$, 
the map
\[
\Sigma_{S^1}^\infty (X\times (\P^n,\P^{n-1}))_+
\to
\Sigma_{S^1}^\infty X_+
\]
naturally induced by the projection is an equivalence in $\inflogSHS(S)$.
\end{thm}
\begin{proof}
If $S$ is a scheme, this is done in Proposition \ref{ProplogSH.4}.
For general $S$, using the monoidal product, we reduce to the case when $X=S$.
Let $p\colon S\to \Spec{\Z}$ be the structure morphism.
The map
\[
\Sigma_{S^1}^\infty (\P^n,\P^{n-1})_+\to  \Sigma_{S^1}^\infty \pt_+
\]
is an equivalence in $\inflogSHS(\Spec{\Z})$.
Applying $p^*$ to this map finishes the proof.
\end{proof}

\subsection{Various models for \texorpdfstring{$\inflogSH$}{logSH}}
\label{models}Let $S$ be a scheme equipped with trivial log structure.
The standard construction of $\inflogSH^\eff(S)$ and $\inflogSH(S)$ is obtained by considering the category of dividing Nisnevich (or \'etale) sheaves on $\lSm/S$, localized to $\boxx$-local equivalences (followed by Tate stabilization in the non-effective case). This model has a major disadvantage in applications: any cohomology theory representable in $\inflogSH(S)$ has, by design, dividing descent. This condition is not easy to verify in practice: concretely, it boils down to requiring invariance under admissible blow-ups. 

As we have recalled in Theorem \ref{thm:PnPn-1_invariance}, dividing descent and $\boxx$-invariance together imply $(\P^n, \P^{n-1})$-invariance. It turns out that, after restriction to the subcategory $\SmlSm/S$, the converse is true: $(\P^n, \P^{n-1})$-invariance and \emph{strict} Nisnevich descent imply dividing descent. This provides an alternative, but equivalent, motivic category, that is much easier to use in applications. This is one of the key results in \cite{logDM} for $\inflogDMeff(k)$ and $\inflogDA^\eff(S)$. In what follows, we explain the modifications necessary to adapt the arguments to our more general framework.

For every $S\in \Sch$, 
the inclusion of categories $\iota\colon \SmlSm/S\to \lSm/S$ naturally induces adjunctions
\begin{equation}
\label{logSH.12.1}
\iota_\sharpp : \infPsh(\SmlSm/S) \rightleftarrows \infPsh(\lSm/S): \iota^*.
\end{equation}

\begin{lem}
\label{logSH.12}
For every $S\in \Sch$ the functor $\iota_\sharpp$ naturally induces an equivalence of $\infty$-categories
\[
\infShv_{dNis}(\SmlSm/S)
\xrightarrow{\simeq}
\infShv_{dNis}(\lSm/S).
\]
\end{lem}
\begin{proof}
The question is Zariski local on $S$. Hence we can assume that $S=\Spec{A}$ is affine.
Suppose $\{U_i\to X\}_{i\in I}$ is a dividing Nisnevich cover.
It has a refinement $\{V_i\to X\}_{i\in J}$ such that each $V_i$ has a chart $\A_{P_i}$.
There is a dividing cover of $\A_{P_i}$ that belongs to $\SmlSm/S$ by toric resolution of singularities
\cite[Theorem 11.1.9]{CLStoric}.
Hence $\{U_i\to X\}_{i\in I}$ has a refinement $\{W_i\to X\}_{i\in K}$ such that $W_i$ belongs to $\SmlSm/S$.
By the implication (i)$\Rightarrow$(ii) in \cite[Th\'eor\`eme III.4.1]{SGA4}, we conclude in the case of sheaves of sets. For sheaves of spaces, 
we can use \cite[Lemma 2.1.4]{AGV}.
\end{proof}
An immediate consequence of Lemma \ref{logSH.12} is the following result.
\begin{prop}
\label{logSH.14}
For $S\in \Sch$,
there is an equivalence of symmetric monoidal $\infty$-categories
\[
\inflogSH^\eff(S)
\simeq
\infShv_{dNis}(\SmlSm/S,\infSpt)[\square^{-1}].
\]
\end{prop}

\begin{cor}
\label{ProplogSH.9}
There is an equivalence of symmetric monoidal $\infty$-categories
\[
\inflogSH^\eff(S)
\simeq
\infShv_{dNis}(\SmlSm/S,\infSpt)[\square^{-1},\SmAdm^{-1}].
\]
\end{cor}
\begin{proof}
By Proposition \ref{logSH.14},
the right-hand side is a localization of the left-hand side.
Theorem \ref{Ori.68} finishes the proof.
\end{proof}
\begin{df}
Let $S\in \Sch$.
Write $(\P^\bullet,\P^{\bullet-1})$ for the set of projections $X\times (\P^n,\P^{n-1})\to X$ for all $X\in \SmlSm/S$ (or $X\in \lSm/S$) and integer $n\geq 1$.
\end{df}

\begin{thm}
\label{logSH.7}
For $S\in \Sch$, there are equivalences of symmetric monoidal $\infty$-categories
\begin{align*}
\inflogSH^\eff(S)
\simeq &
\infShv_{dNis}(\lSm/S, \infSpt)[(\P^\bullet,\P^{\bullet-1})^{-1}]
\\
\simeq &
\infShv_{dNis}(\SmlSm/S, \infSpt)[(\P^\bullet,\P^{\bullet-1})^{-1}].
\end{align*}
\end{thm}
\begin{proof}
By definition (resp.\ Proposition \ref{logSH.14}), the second (resp.\ third) $\infty$-category is a localization of $\inflogSH^\eff(S)$.
Theorem \ref{thm:PnPn-1_invariance} finishes the proof.
\end{proof}

The following theorem is crucial for us.

\begin{thm}
\label{logSH.15}
For $S\in \Sch$, there is an equivalence of symmetric monoidal $\infty$-categories 
\[
\infShv_{sNis}(\SmlSm/S, \infSpt)
[(\P^\bullet,\P^{\bullet-1})^{-1}]
\simeq
\infShv_{dNis}(\SmlSm/S, \infSpt)
[(\P^\bullet,\P^{\bullet-1})^{-1}].
\]
\end{thm}
\begin{proof}
According to Remark \ref{logH.21}, it is enough to show that every morphism of the form
$
\Sigma_{S^1}^\infty Y_+
\to
\Sigma_{S^1}^\infty X_+
$
for all dividing covers $Y\to X$ is an equivalence in $\infShv_{sNis}(\SmlSm/S, \infSpt)
[(\P^\bullet,\P^{\bullet-1})^{-1}]$.

This question is Zariski local on $S$.
Hence we may assume that $S$ is affine.
By Proposition \ref{noetherian.3}(4),
we reduce to the case when $S\in \Sch_{noeth}$.
This is done in
\cite[Theorem 7.7.4]{logDM}.
Indeed, the homotopy category of $\infShv_{sNis}(\SmlSm/S, \infSpt)
[(\P^\bullet,\P^{\bullet-1})^{-1}]$ satisfies the assumption in \cite[\S 7.7]{logDM}, and then set $M:=\Sigma_{S^1}^\infty$,
see also Remark \ref{Ori.75} for the terminology change.
\end{proof}

\begin{cor}\label{cor:fundamental_models}
For $S\in \Sch$,
there are equivalences of symmetric monoidal $\infty$-categories
  \begin{align*}
 \inflogSH^\eff(S)
 \simeq &
\infShv_{sNis}(\SmlSm/S,\infSpt)
[(\P^\bullet,\P^{\bullet-1})^{-1}],
\\
 \inflogSH(S) \simeq &\infSpt^\Sigma_{T}(\infShv_{sNis}(\SmlSm/S, \infSpt)
[(\P^\bullet,\P^{\bullet-1})^{-1}])
\\
\simeq & \infSpt^\Sigma_{T}(\infShv_{sNis}(\SmlSm/S)_\ast
[(\P^\bullet,\P^{\bullet-1})^{-1}]).
 \end{align*}
\end{cor}
\begin{proof}
The first two equivalences are consequences of Theorem \ref{logSH.15}.
The last equivalence is a consequence of Proposition \ref{logH.2}.
\end{proof}

In particular,
for $S\in \Sch$,
we have an equivalence of $\infty$-categories
\begin{equation}
\label{alter.1.5}
\inflogSH(S)
\simeq
\limit(\cdots \xrightarrow{\Omega_T}\inflogHpt(S)\xrightarrow{\Omega_T}\inflogHpt(S)).
\end{equation}

\begin{rmk} 
Proposition \ref{logSH.14} and Theorems \ref{logSH.7} and \ref{logSH.15} hold for $T$-spectra and $\P^1$-spectra too. Note that the arguments work \emph{verbatim} if the strict Nisnevich topology is replaced by the strict \'etale topology, and the dividing Nisnevich topology is replaced by the dividing \'etale topology.
 \end{rmk}

In summary, 
when $S\in \Sch$, 
we have the following models for $\inflogSH^\eff(S)$ and $\inflogSH(S)$:
\begin{center}
\begin{tabular}{l | lll}
&
Category&
Topology&
Localization
\\
\hline
Model 1 & $\lSm/S$ & $dNis$ & $(\P^\bullet,\P^{\bullet-1})$
\\
Model 2 & $\lSm/S$ & $dNis$ & $\boxx$
\\
Model 3 & $\SmlSm/S$ & $dNis$ & $\boxx$
\\
Model 4 & $\SmlSm/S$ & $dNis$ & $(\P^\bullet,\P^{\bullet-1})$
\\
Model 5 & $\SmlSm/S$ & $sNis$ & $(\P^\bullet,\P^{\bullet-1})$
\end{tabular}
\end{center}
Note that, in the last three models, every map in $\SmAdm$ is moreover automatically inverted.

\newpage

\section{Motivic versus logarithmic motivic theories}
\label{section:mvlmt}

We begin by fixing some notation and introducing the functors relating the categories of log schemes with the ordinary categories of schemes.

\begin{df}
\label{compth.1}
Suppose $S\in \Sch$.
We define functors\index[notation]{rho @ $\rho$}\index[notation]{lambda @ $\lambda$}
\[
\rho\colon \SmlSm/S \to \Sm/S,
\;
\lambda\colon \Sm/S\to \SmlSm/S,\;\lSm/S
\]
and\index[notation]{omega @ $\omega$}
\[
\omega\colon \SmlSm/S,\; \lSm/S\to \Sm/S
\]
as follows:
$\rho(X):=\ul{X}$, $\lambda(X):=X$, and $\omega(X):=X-\partial X$.
One can readily check that 
\begin{equation}
\label{compth.1.1}
\begin{tikzcd}
\SmlSm/S
\ar[r,shift left=1.5ex,"\rho"]
\ar[r,"\lambda" description,leftarrow]
\ar[r,shift right=1.5ex,"\omega"']&
\Sm/S
\end{tikzcd}
\end{equation}
and
\begin{equation}
\label{compth.1.2}
\begin{tikzcd}
\lSm/S
\ar[r,"\lambda",shift left=0.75ex,leftarrow]
\ar[r,shift right=0.75ex,"\omega"']&
\Sm/S
\end{tikzcd}
\end{equation}
are adjoint functors.
Our convention for such diagrams is that $\rho$ is left adjoint to $\lambda$, and $\lambda$ is left adjoint to $\omega$.
\end{df}
The diagrams \eqref{compth.1.1} and \eqref{compth.1.2} naturally induces sequences of adjunctions of presheaves of (pointed) spaces
\begin{equation}
\label{compth.1.3}
\begin{tikzcd}[column sep=large]
\infPsh(\SmlSm/S)_{(*)}
\ar[rr,shift left=4ex,"\rho_\sharpp"]
\ar[rr,shift left=2ex,"\rho^*\cong \lambda_\sharpp" description,leftarrow]
\ar[rr,"\rho_*\cong \lambda^*\cong \omega_\sharpp" description]
\ar[rr,shift right=2ex,"\lambda_*\cong \omega^*" description,leftarrow]
\ar[rr,shift right=4ex,"\omega_*"']&&
\infPsh(\Sm/S)_{(*)}
\end{tikzcd}
\end{equation}
and
\begin{equation}
\label{compth.1.4}
\begin{tikzcd}[column sep=large]
\infPsh(\lSm/S)_{(*)}
\ar[rr,shift left=3ex,"\rho^*\cong \lambda_\sharpp",leftarrow]
\ar[rr,shift left=1ex, "\rho_*\cong \lambda^*\cong \omega_\sharpp" description]
\ar[rr,shift right=1ex,"\lambda_*\cong \omega^*" description,leftarrow]
\ar[rr,shift right=3ex,"\omega_*"']&&
\infPsh(\Sm/S)_{(*)}.
\end{tikzcd}
\end{equation}

We now recall the construction of the motivic homotopy categories due to Morel and Voevodsky \cite{MV} in our context.

\begin{df}
\label{compth.2}
Let\index[notation]{H @ $\infH$}\index[notation]{Hpt @ $\infH_{\ast}$}
\[
 \infH_{(\ast)}\colon  \Sch^\mathrm{op}\to \CAlg(\LPr)
\]
be the
functor
associated with $\infShv_{Nis}(\Sm/-)[(\A^1)^{-1}]_{(\ast)}$.
Let\index[notation]{SH @ $\infSH$}\index[notation]{SHS @ $\infSH_{S^1}$}
\[
\infSH^{\eff}, \infSH \colon \Sch^\mathrm{op}\to \CAlg(\LPr)
\]
be the $S^1$ (resp.\ $T$) stabilization.

By e.g.,  Proposition \ref{alter.1} we have equivalences of $\infty$-categories
\begin{equation}
\label{compth.2.1}
\infSH(S)
\simeq
\colim_{\LPr}(\infHpt(S) \xrightarrow{\Sigma_{\P^1}}\infHpt(S) \xrightarrow{\Sigma_{\P^1}}\cdots )
\end{equation}
and
\begin{equation}
\infSH(S)
\simeq
\limit(\cdots \xrightarrow{\Omega_{\P^1}}\infHpt(S) \xrightarrow{\Omega_{\P^1}}\infHpt(S) ).
\end{equation}
\end{df}

The idea of contracting the projective line $\mathbb{P}^1$ instead of the affine line $\mathbb{A}^1$ was considered \index{P1-localization @ $\P^1$-localization} in \cite{AyoubP1loc}.
It is not clear how to compare this theory directly with our construction, but we offer a variant of Ayoub's construction, in which we further localize with respect to all projections $X\times \P^n\to X$ for all $X\in \Sm/S$ and integers $n\geq 1$. We refer to this localization as the \emph{$\P^\bullet$-localization}.
The definition of the $\P^\bullet$-localized motivic homotopy category is completely straightforward.

\begin{df}
\label{compth.3}
Let\index[notation]{HP @ $\infHP$}
\[
\infHP_{(\ast)} \colon  \Sch^\mathrm{op}\to \CAlg(\LPr)
\]
be the functor
associated with  $\infShv_{Nis}(\Sm/-)[(\P^\bullet)^{-1}]_{(\ast)}$, where  for each $S$, the $\infty$-category $\infShv_{Nis}(\Sm/S)[(\P^\bullet)^{-1}]_{(\ast)}$ denotes the full subcategory of $\infShv_{Nis}(\Sm/S)_{(\ast)}$ spanned by objects which are local with respect to all projections $X\times \P^n\to X$.
Let\index[notation]{SHPS @ $\infSHPS$}
\[
\infSH^{\eff, \P^\bullet} \colon \Sch^\mathrm{op}\to \CAlg(\LPr)
\]
be its $S^1$-stabilization.
We can similarly define $\infHPone$, $\infHPonept$, and $\infSH^{\eff, \P^1}$.
\index[notation]{HPone @ $\infHPone$}\index[notation]{HPonept @ $\infHPonept$}\index[notation]{HPoneS @ $\infSH^{\eff, \P^1}$}
\end{df}

\begin{ques}
\label{compth.5}
Is there a natural equivalence of functors
\[
\infSHPoneS
\simeq
\infSHPS
?
\]
\end{ques}

\begin{exm}
\label{compth.6}
Let $X$ be a scheme.
If $\cF$ be a quasi-coherent sheaf on $X$, we regard $\cF$ as a Nisnevich sheaf on $\Sm/X$ associated with the presheaf given by
\[
\underline{\cF}(Y):=\Gamma(Y,(Y\to X)^*\cF).
\]
If $X=\Spec{A}$ is affine and $\cF$ corresponds to an $A$-module $M$, then $R\Gamma(Y, \underline{\cF}) = R\Gamma(Y, \cO_Y)\otimes_A M$; 
see \cite[\S 3]{AyoubP1loc}. 
The sheaf $\underline{\cF}$ is $\P^n$-local for every $n\geq 1$  
by the argument for \cite[Proposition 3.3]{AyoubP1loc}.
\end{exm}

\begin{const}
\label{compth.4}
The functor $\rho_\sharpp$ maps $(\P^n,\P^{n-1})$ to $\P^n$ for every integer $n\geq 1$ and every strict Nisnevich distinguished square to a Nisnevich distinguished square. By Corollary \ref{cor:fundamental_models},
we obtain a
natural transformation
\[
\rho_\sharpp \colon \inflogSHS \to \infSHPS.
\]
We denote its objectwise right adjoint by $\rho^*$.
\end{const}

By using the functor $\rho_\sharpp$, various properties of motives in $\inflogSHS$ can be translated into properties of motives in $\infSHPS$.
For example, we have the following result.

\begin{prop}
\label{compth.11}
Suppose $S\in \Sch$ and $i\colon Z\to X$ is a closed immersion in $\Sm/S$.
Then the projection $p\colon \Blow_Z X\to X$ is an equivalence in $\infSHPS(S)$.
\end{prop}
\begin{proof}
We may assume that $i$ has codimension $d$.
Theorem \ref{ProplogSH.5} gives a cocartesian square
\begin{equation}
\label{compth.11.1}
\begin{tikzcd}
\P(\Normal_Z X)\ar[r]\ar[d]&
\Blow_Z X\ar[d,"p"]
\\
Z\ar[r,"i"]&
X
\end{tikzcd}
\end{equation}
in $\inflogSHS(S)$.
Apply $\rho_\sharpp$ to \eqref{compth.11.1} to see that \eqref{compth.11.1} is a cocartesian square in $\infSHPS(S)$ too.
Hence it remains to check that the projection
\[
\P(\Normal_Z X)\to Z
\]
is an equivalence in $\infSHPS(S)$.

By induction on the number of open subsets in a trivialization of the projective bundle $\P(\Normal_Z X)$, we reduce to showing that the projection
\[
Z\times \P^n\to Z
\]
is an equivalence in $\infSHPS(S)$ for every integer $n\geq 1$.
This follows from the definition.
\end{proof}

\begin{const}
\label{compth.7}
The functor $\omega_\sharpp$ maps $\boxx$ to $\A^1$, $\P^1$ to $\P^1$, and every strict Nisnevich distinguished square to a Nisnevich distinguished square.
This naturally induces a natural transformation
\[
\omega_\sharpp
\colon
\inflogSH
\to
\infSH.
\]
We similarly have natural transformations
\[
\omega_\sharpp \colon \inflogH \to \infH,
\omega_\sharpp\colon \inflogHpt\to \infHpt,
\text{ and }
\omega_\sharpp \colon \inflogSHS \to \infSHS.
\]
Their right adjoints, denoted by $\omega^*$,
are both fully faithful by \cite[Proposition 2.5.7]{logA1}.
This means there is an embedding of motivic homotopy theories:
\[
(\A^1\text{-local homotopy theory})
\subset
(\boxx\text{-local homotopy theory})
\]

\begin{rmk} For the convenience of the reader,
we outline the proof for $\omega^*\colon \infH(S)\to \inflogH(S)$.
To begin with, 
the left adjoint functor $\omega_\sharp\colon \inflogH(S)\to \infH(S)$ admits a factorization
\[
\inflogH(S)
\xrightarrow{L_{\A^1}}
\inflogH(S)[(\A^1)^{-1}]
\xrightarrow{\omega_\sharp}
\infH(S),
\]
where $L_{\A^1}$ denotes the localization functor.
Let us explain why the latter $\omega_\sharp$ is an equivalence of $\infty$-categories, 
which shows the claim since a right adjoint of $L_{\A^1}$ is fully faithful.
There is a functor
\[
\lambda_\sharp
\colon
\infH(S)
\to
\inflogH(S)[(\A^1)^{-1}]
\]
sending $X\in \Sm/S$ to $X$.
We have an isomorphism of schemes $\omega\lambda(X)\cong X$,
and this implies $\omega_\sharp \lambda_\sharp\simeq \id$.
On the other hand,
we have an isomorphism of schemes $\lambda\omega(Y)\cong Y-\partial Y$ for all $Y\in \lSm/S$.
One can show $Y-\partial Y\simeq Y$ in $\inflogH(S)[(\A^1)^{-1}]$ as in \cite[Proposition 2.4.3]{logA1},
and this implies $\lambda_\sharp \omega_\sharp \simeq \id$.
Hence $\omega_\sharp$ is an equivalence of $\infty$-categories.\end{rmk}
\end{const}

\begin{rmk}
If $S\in \lSch$ has a nontrivial log structure and $f\colon X\to S$ is a log smooth morphism of fs log schemes,
then taking the complement $X-\partial X$ is a drastic operation since it removes everything over the boundary 
$\partial S$.
Instead,
we can remove the ``vertical boundary'' $\partial_S X$,
which is defined to be the set of points $x\in X$ such that
$\cM_{S,f(x)}\to \cM_{X,x}$ is not vertical 
in the sense of \cite[Definition I.4.3.1]{Ogu}.

However,
there is a technical difficulty of removing vertical boundaries because there is no functor 
$\omega\colon \lSm/S\to \lSm/S$ sending $X\in \lSm/S$ to $X-\partial_S X$ 
(unless $S$ has trivial log structure).
For example,
consider the morphisms
\[
Z:=\A_\N\times \A_\N
\xrightarrow{h}
Y:=\A_\N\times \A_\N
\xrightarrow{g}
X:=\A_\N\times \A^1
\xrightarrow{f}
S:=\A_\N,
\]
where $f$ is the first projection,
$g$ removes the log structure on the second factor,
and $h$ is naturally induced by the homomorphism $\N^2\to \N^2$ sending $(x,y)$ to $(x,x+y)$.
Then we have isomorphisms of fs log schemes $X-\partial_S X\cong X$ and
$Y-\partial_S Y\cong \A_\N\times \G_m$.
Since the diagonal homomorphism $\N\to \N^2$ is vertical,
we also have an isomorphism $Z-\partial_S Z\cong Z$.
Taking the existence of $\omega$ for granted, 
then $f$ agrees with $\omega(Z)\to \omega(X)$ up to isomorphism.
However, unlike $\omega(Z)\to \omega(X)$, the map $f$ does not factor through $w(Y)$.
\end{rmk}

\begin{const}
\label{compth.13}
For $S\in \Sch$,
let $\omega\colon \cSm/S\to \Sm/S$ be the functor sending $X\in \cSm/S$ to $X-\partial X$.
As in Construction \ref{compth.7},
we have natural transformations
\begin{gather*}
\omega_\sharp
\colon
\inflogH^\Adm
\to
\infH,
\text{ }
\omega_\sharp
\colon
\inflogHpt^\Adm
\to
\infHpt,
\\
\omega_\sharp
\colon
\inflogSH^{\eff,\Adm}
\to
\infSHS,
\text{ and }
\omega_\sharp
\colon
\inflogSH^\Adm
\to
\infSH,
\end{gather*}
where ${\inflogSH}^{\eff,\Adm}$ and $\inflogSH^\Adm$ are the $S^1$ and $T$-stabilization of the $\Adm$-categories considered in Construction \ref{constr:Adm_unstable}. 
Let $\omega^*$ denote the corresponding right adjoints.
\end{const}

\begin{const}
\label{compth.14}
Let $S\in \Sch$, and let $\psi\colon \SmlSm/S\to \cSm/S$ be the natural functor.
This maps $\boxx$ to $\boxx$, $\P^1$ to $\P^1$, every strict Nisnevich cover to strict Nisnevich cover, 
and every admissible blow-up along a smooth center to an admissible blow-up.

Using Corollary \ref{ProplogSH.9}
we have a natural transformation
\[
\psi_\sharpp
\colon
\inflogSH
\to
\inflogSH^\Adm
\]
We similarly have a natural transformation
\[
\psi_\sharpp \colon \inflogSHS \to {\inflogSH}^{\eff,\Adm}.
\]
The corresponding right adjoints are denoted by $\psi^*$.

Observe also that there is a commutative triangle
\begin{equation}
\label{compth.14.1}
\begin{tikzcd}
\inflogSH\ar[rd,"w_\sharp"']\ar[rr,"\psi_\sharp"]&
&
\inflogSH^\Adm\ar[ld,"w_\sharp"]
\\
&
\infSH.
\end{tikzcd}
\end{equation}
We also have similar commutative triangles for $\inflogSHS$.
\end{const}

\begin{rmk}
For a normal $S\in \Sch$,
the log structure on $X\in \SmlSm/S$, which was given by the Deligne-Faltings structure, agrees with the compactifying log structure associated with the open immersion $X-\partial X\to X$ by \cite[Proposition III.1.7.3(4)]{Ogu}.
Together with \cite[Proposition III.1.6.2]{Ogu},
we see that $\psi\colon \SmlSm/S\to \cSm/S$ is fully faithful.
\end{rmk}

\begin{prop}
\label{ProplogSH.8}
Suppose that $k$ is a perfect field admitting resolution of singularities.
Then the functors
\[
\psi_\sharp
\colon
\inflogSHS(k)
\to
{\inflogSH}^{\eff,\Adm}(k)
\text{ and }
\psi_\sharp
\colon
\inflogSH(k)
\to
\inflogSH^\Adm(k)
\]
are equivalences. In particular, every map in $\Adm$ is invertible in $\inflogSH(k)$.
\end{prop}
\begin{proof}
We define the functor
\[
\alpha
\colon
\inflogSH^{\eff}(k)
\to
\infPsh(\cSm/k,\infSpt)
\]
by setting
\[
\alpha(\cF)(X)
:=
\colimit_{Y\in (\Adm\downarrow X)\cap \lSm/k} \cF(Y)
\]
for all $\cF\in \inflogSH^{\eff}(k)$ and $X\in \cSm/k$.
Here, 
the colimit runs over the category $(\Adm\downarrow X)\cap \lSm/k$, 
whose objects are admissible blow-ups $Y\to X$ with $Y\in \lSm/k$ and whose morphisms are 
admissible blow-ups $Z\to Y$ over $X$.

Suppose $Y,Y'\in (\Adm\downarrow X)\cap \lSm/k$.
By resolution of singularities,
there exists a blow-up $\ul{Y''}\to \ul{Y}\times_{\ul{X}} \ul{X'}$ such that $\ul{Y''}\in \Sm/k$ and the preimage $Y''-\partial Y''$ of $X-\partial X$ in $\ul{Y''}$ is the complement of a strict normal crossing divisor on $\ul{Y''}$.
We set $Y'':=(\ul{Y''},Y''-\partial Y'')$,
which is an object of $\SmlSm/k$.
Then the naturally induced morphisms $Y''\to Y,Y'$ are admissible blow-ups,
so the category $(\Adm\downarrow X)\cap \lSm/k$ is connected.

Suppose $f,g\colon Y'\rightrightarrows Y$ are two morphisms in $(\Adm\downarrow X)\cap \lSm/k$.
Since $f$ and $g$ agree on the dense open subset $Y'-\partial Y'$ on $Y'$,
we have $f=g$.
Hence the category $(\Adm\downarrow X)\cap \lSm/k$ is filtered.
Together with Theorem \ref{ProplogSH.3},
we have $\alpha(\cF)(X)\simeq \cF(Y)$ whenever $Y\in (\Adm\downarrow X)\cap \lSm/k$.

If $X\in \cSm/k$,
then there exists an object $Y$ of $(\Adm\downarrow X)\cap \lSm/k$ by resolution of singularities.
Since $Y\times \boxx\in (\Adm\downarrow X\times \boxx)\cap \lSm/k$,
we have equivalences of spectra
\[
\alpha(\cF)(X)
\simeq
\cF(Y)
\simeq
\cF(Y\times \boxx)
\simeq
\alpha(\cF)(X\times \boxx).
\]
Hence $\alpha(\cF)$ is $\boxx$-invariant,
and we obtain the factorization  
\[
\alpha
\colon
\inflogSH^{\eff}(k)
\to
\inflogSH^{\eff,\Adm}(k).
\]

Conversely,
we have the functor
\[
\psi^*
\colon
\inflogSH^{\eff,\Adm}(k)
\to
\inflogSH^{\eff}(k)
\]
given by
\[
\psi^*(\cG)(Y)
:=
\cG(Y)
\]
for all $\cG\in \inflogSH(k)$ and $Y\in \lSm/k$.
There is a natural equivalence $\alpha\psi^*\simeq \id$.
Due to Theorem \ref{ProplogSH.3},
we also have a natural equivalence $\psi^*\alpha\simeq\id$.
Hence $\psi^*$ is an equivalence of $\infty$-categories.
This implies that $\psi_\sharp$ is an equivalence of $\infty$-categories too.

The claim about $\inflogSH$ holds true by passing to the $\Gmm$-stabilization.
\end{proof}

\begin{rmk}
\label{compth.15}
Suppose $k$ is a perfect field admitting resolution of singularities.
As above,
one can similarly show that there are equivalences of symmetric monoidal $\infty$-categories
\[
\psi_\sharp
\colon
\inflogH(k)
\xrightarrow{\simeq}
\inflogH^\Adm(k)
\text{ and }
\psi_\sharp
\colon
\inflogHpt(k)
\xrightarrow{\simeq}
\inflogHpt^\Adm(k).
\]
\end{rmk}

\begin{const}
\label{compth.9} The functor $\lambda$ maps $\P^1$ to $\P^1$ and every Nisnevich distinguished square to a strict 
Nisnevich distinguished square. We obtain
natural transformations:
\[
\lambda_\sharpp\colon \infShv_{Nis}(\Sm/-)\to \inflogH
\]
\[
\lambda_\sharpp\colon \infShv_{Nis}(\Sm/-,\infSpt)\to \inflogSHS.
\]
We denote the right adjoints by $\lambda^*$.
\end{const}

\begin{ques}
\label{compth.10}
Is the functor
\[
\lambda^*\colon \inflogSHS(S) \to \infShv_{Nis}(\Sm/S,\infSpt)
\]
fully faithful for $S\in \Sch$?
\end{ques}

\begin{rmk}\label{rmk:non_A1_stable}
We can regard the essential image of $\lambda^*$ as a \emph{homotopy theory of schemes without $\A^1$-invariance}. 
We find it an appealing question to describe the said image using exclusively the language of schemes with no mention of logarithmic structures. An analogous (but less structured) problem was considered in \cite[3.3.5]{BThesisPubl}.
\end{rmk}

\begin{const}
\label{compth.16}
Suppose $S\in \Sch$.
Let us compare $\inflogSH(S)$ with the $\infty$-category $\MS_S$ (following the notation in \cite[\S 4]{AHI}) of $\P^1$-spectra due to Annala-Iwasa \cite{AI}.
Consider the full subcategory of $\infPsh(\Sm/S,\infSpt)$ spanned by the presheaves satisfying Zariski descent and elementary blow-up excision,
see \cite[Definition 2.1(i)]{AHI} for the latter notion.
Then $\MS_S$ is defined to be the $\infty$-category of $\P^1$-spectra in this $\infty$-category.

The functor $\lambda_\sharp\colon \infPsh(\Sm/S,\infSpt)\to \infPsh(\lSm/S,\infSpt)$ maps $\P^1$ to $\P^1$ and every Zariski distinguished square to a strict Nisnevich distinguished square and every elementary blow-up square to an elementary blow-up square.
Furthermore, an elementary blow-up square is cocartesian in $\inflogSH^{\SmAdm}(S)$ and hence in $\inflogSH(S)$ by Proposition \ref{Ori.66}.
It follows that we have the induced natural transformation
\[
\lambda_\sharp\colon\MS\to \inflogSH.
\]
We denote its right adjoint by $\lambda^*$. We expect this functor to be fully faithful for $S\in \Sch$. 
\end{const}

\begin{rmk} Let us further elaborate on 
Construction \ref{compth.16}.
    The construction of $\MS_S$ is inherently \emph{stable} (since it is built as the category of $\P^1$-spectra). Our approach allows us to consider an \emph{unstable} motivic homotopy theory without $\A^1$-invariance, given by the essential image of \[\lambda^*\colon \inflogH(S) \to \infShv_{Nis}(\Sm/S)\]
    for any $S\in \Sch$. 
    Let us provisionally denote the latter category by $\mathcal{H}^{\SmAdm}(S)$.
   It is an interesting question to study explicitly $\mathcal{H}^{\SmAdm}(S)$, without referring to log structures.  We make a couple of immediate remarks. First, every elementary blow-up square is cocartesian in $\inflogH(S)$, thus the same is true in $\mathcal{H}^{\SmAdm}(S)$; that is, $\mathcal{H}^{\SmAdm}(S)$ is a subcategory of the category of $\infShv_{Nis}(\Sm/S, \infSpc)$ spanned by those sheaves $\cF$ satisfying elementary blow-up excision.
Second, the Gysin sequence in Remark \ref{rmk:Gysin} suggests that the objects of $\mathcal{H}^{\SmAdm}(S)$ are local with respect to the inclusions $\P^{n-1} \to \P^n$.
    We leave the study of this category for future work. 
\end{rmk}

\newpage

\section{Logarithmic Picard groups and \texorpdfstring{$\clspace \Gmm$}{BGm}}
\label{section:lpg}
The goal of this section is to discuss the cohomology of different sheaves of \emph{units} associated with a log scheme, namely $\mathcal{M}^{\gp}$ and $\Gmm$. Our main result is Theorem \ref{Pic.9}, proving that for every $S\in \Sch$ there is an equivalence $\clspace \Gmm \simeq \P^\infty$ in $\inflogH(S)$. By considering the pullback along $\ul{S}\to S$, the same equivalence will hold in  $\inflogSH(S)$ for $S\in \lSch$ thanks to Corollary \ref{ProplogSH.9}.

\subsection{\texorpdfstring{$(\P^n,\P^{n-1})$}{(Pn,Pn-1)}-invariance of logarithmic Picard groups}
Let us start with a definition.

\begin{df}
If $X$ is a log scheme, 
we define the \emph{logarithmic Picard group}\index{logarithmic Picard group}\index[notation]{logPic @ $\logPic$} 
as 
\begin{equation}
\label{logPic.4.3}
\logPic(X)
:=
H_{sNis}^1(X,\cM^{\gp}).
\end{equation}
If $X$ is a scheme, we observe that there is an isomorphism of abelian groups $\logPic(X)\cong \Pic(X)$ with the Picard group.
\end{df}
\begin{rmk}
Our $\logPic(X)$ is different from Kato's original definition in \cite{zbMATH00721812} as 
\[
H_{\set}^1(X,\cM^{\gp}),
\]
which we denote by $\logPic_{\set}(X)$.
It is natural to ask whether the canonical homomorphism of abelian groups
\begin{equation}
\label{logPic.4.4}
\logPic(X)
\to
\logPic_{\set}(X)
\end{equation}
is an isomorphism or not.
This question is a log version of Hilbert's theorem 90.
\end{rmk}

In this subsection we study the $(\P^n,\P^{n-1})$-invariance of $\cM^{\gp}$, $\logPic$, and $\logPic_{\set}$, 
i.e., whether the naturally induced homomorphisms of abelian groups
\begin{equation}
\label{logPic.4.5}
\begin{split}
\cM^{\gp}(X) \to \cM^{\gp}(X\times  (\P^n,\P^{n-1}))
&,\;
\logPic(X)\to \logPic(X\times  (\P^n,\P^{n-1}))
,
\\
\logPic_{\set}(X)\to &\logPic_{\set}(X\times  (\P^n,\P^{n-1}))
\end{split}
\end{equation}
are isomorphisms or not for, say,  fs log schemes $X$.

We will show that the homomorphisms in \eqref{logPic.4.5} are isomorphisms in the following two cases:
\begin{enumerate}
\item[(1)] $S$ is a geometrically unibranch scheme, and $X\in \SmlSm/S$.
\item[(2)] $X$ is a saturated log scheme such that $\ul{X}$ is the spectrum of a field.
\end{enumerate}
We will also show that the homomorphism in \eqref{logPic.4.4} is an isomorphism in the case (1).

\begin{prop}
\label{logPic.7}
Let $X$ be a locally noetherian connected scheme.
Then there is an isomorphism of abelian groups
\[
\Pic(X)\times \Z
\xrightarrow{\cong}
\Pic(X\times \P^n)
\]
sending $(\cL,d)\in \Pic(X)\times \Z$ to $q^*\cL(d)$, where $q\colon X\times \P^n\to X$ is the projection.
\end{prop}
\begin{proof}
We refer to \cite[Remarque II.4.2.7]{EGA} and \cite[Proposition 3, Lecture 13]{zbMATH03299277}.
\end{proof}

\begin{prop}
\label{logPic.9}
Let $X$ be an irreducibile unibranch (resp.\ geometrically unibranch) scheme.
Then we have the vanishing
\[
H_{Nis}^1(X,\Z)=0
\text{ (resp.\ }
H_{\et}^1(X,\Z)=0
\text{)}.
\]
\end{prop}
\begin{proof}
For the \'etale case, we refer to \cite[Proposition IX.3.6(2)]{SGA4}.
The Nisnevich case can be proven similarly, 
with reference to \cite[Proposition IV.18.12.6]{EGA} instead of \cite[Proposition IV.18.14.15]{EGA}.
\end{proof}

\begin{prop}
\label{logPic.3}
Let $X$ be a geometrically unibranch irreducible scheme.
Then the homomorphisms in \eqref{logPic.4.5} are isomorphisms.
\end{prop}
\begin{proof}
Let $i\colon X\times \P^{n-1} \to X\times \P^n$ be the closed immersion at $\infty$.
There is an exact sequence of sheaves of abelian groups
\[
0
\to
\cM_{X\times \P^n}^{\gp}
\to
\cM_{X\times (\P^n,\P^{n-1})}^{\gp}
\to
i_*\Z
\to
0.
\]
Hence we have a long exact sequence
\begin{equation}
\label{logPic.3.1}
\begin{split}
0&
\to
\cM^{\gp}(X\times \P^n)
\to
\cM^{\gp}(X\times (\P^n,\P^{n-1}))
\to
H_{\et}^0(X\times \P^n,i_*\Z)
\\
&\xrightarrow{\delta}
\Pic(X\times \P^n)
\to
\Pic(X\times  (\P^n,\P^{n-1}))
\to
H_{\et}^1(X\times \P^n,i_*\Z)\to \cdots.
\end{split}
\end{equation}

Since $X\times \P^{n-1}$ is irreducible, there are isomorphisms of abelian groups
\begin{equation}
\label{logPic.3.3}
\Gamma(X\times \P^n,i_*\Z)
\cong
\Gamma(X\times \P^{n-1},\Z)
\cong
\Z.
\end{equation}
Furthermore, since $X\times \P^{n-1}$ is geometrically unibranch, 
Proposition \ref{logPic.9} implies the vanishing
\begin{equation}
\label{logPic.3.2}
H_{Nis}^1(X\times \P^n,i_*\Z)
=
H_{Nis}^1(X\times \P^{n-1},\Z)
=
H_{\et}^1(X\times \P^n,i_*\Z)
=
H_{\et}^1(X\times \P^{n-1},\Z)
=0.
\end{equation}
The element $1\in \Z$ corresponds to a section $s\in \Gamma(X\times \P^n,i_*\Z)$.
There is a Zariski cover of $(\P^n,\P^{n-1})$ consisting of
\[
U_i:=
\Spec{\Z[x_0/x_i,\ldots,x_n/x_i],\N(x_n/x_i)}
\]
for $0\leq i\leq n-1$ and
\[
U_n:=\Spec{\Z[x_0/x_n,\ldots,x_n/x_n]}.
\]
The section $s$ can be locally lifted to
\[
x_n/x_i\in \Gamma(X\times U_i,\cM^\gp)
\]
for $0\leq i\leq n$.
The Zariski cover $\{X\times U_i\}$ of $X\times (\P^n,\P^{n-1})$ defines a sequence
\[
0 \to \bigoplus_i \Gamma(X\times U_i,\cM^{\gp}) \to \bigoplus_{i\neq j} \Gamma(X\times (U_i\cap U_j),\cM^{\gp}).
\]
The section $(x_n/x_0,\ldots,x_n/x_n)$ maps to a section whose components have the form 
\[
x_j/x_i\in \Gamma(X\times (U_i\cap U_j),\cO^*)\subset \Gamma(X\times (U_i\cap U_j),\cM^{\gp}).
\]
This is exactly the element in the \v{C}ech cohomology associated with $\cO(1)\in \Pic(X\times \P^n)$.
Hence the boundary map $\delta$ in \eqref{logPic.3.1} sends $s$ to $\cO(1)\in \Pic(X\times \P^n)$.
Together with Proposition \ref{logPic.7}, \eqref{logPic.3.3}, and \eqref{logPic.3.2} we conclude.
\end{proof}

\begin{prop}
\label{logPic.8}
Suppose $S$ is a geometrically unibranch scheme and $X\in \SmlSm/S$.
Then the homomorphisms in \eqref{logPic.4.4} and \eqref{logPic.4.5} are isomorphisms.
\end{prop}
\begin{proof}
By \cite[Proposition 17.5.7]{EGA}, $X$ is geometrically unibranch.
Hence we only need to consider the case when $X$ is irreducible since every connected component of $X$ is an irreducible component.

We proceed by induction on the number of irreducible components $r$ of $\partial X$.
The case $r=0$ is done by arguing as in  \cite[\href{https://stacks.math.columbia.edu/tag/03P7}{Tag 03P7}]{stacks-project} for \eqref{logPic.4.4} and by Proposition \ref{logPic.3} for \eqref{logPic.4.5}.
If $r>0$, let $D$ be an irreducible component of $\partial X$.
Let $Y$ be the fs log scheme whose underlying scheme is $\ul{X}$ and whose log structure is the compactifying log structure associated with the open immersion
\[
X-\overline{\partial X-D}\to X.
\]
For simplicity of notation, we set
\[
X':=X\times (\P^n,\P^{n-1}),
\text{ }
Y':=Y\times (\P^n,\P^{n-1}),
\text{ and }
D':=D\times (\P^n,\P^{n-1}).
\]
Let $i\colon D\to X$ and $i'\colon D'\to X'$ denote the closed immersions.
Since $D$ and $D'$ are geometrically unibranch by \cite[Proposition 17.5.7]{EGA}, 
as in \eqref{logPic.3.3} and \eqref{logPic.3.2} there are isomorphisms of abelian groups
\begin{equation}
\label{logPic.8.1}
\Gamma(X,i_*\Z) \cong \Gamma(X',i_*'\Z)\cong \Z.
\end{equation}
and the vanishing
\begin{equation}
\label{logPic.8.2}
H_{sNis}^1(X,i_*\Z)
=
H_{sNis}^1(X',i_*'\Z)
=
H_{\set}^1(X,i_*\Z)
=
H_{\set}^1(X',i_*'\Z)
=
0.
\end{equation}
There are exact sequences
\[
0\to \cM_Y^{\gp} \to \cM_X^{\gp} \to i_*\Z\to 0
\text{ and }
0\to \cM_{Y'}^{\gp} \to \cM_{X'}^{\gp} \to i_*'\Z\to 0.
\]
Together with \eqref{logPic.8.1} and \eqref{logPic.8.2}, these give a commutative diagram 
\begin{equation}
\label{logPic.8.3}
\begin{tikzcd}[column sep=small]
0\ar[r]&
\cM^{\gp}(Y)\ar[r]\ar[d,"a"']&
\cM^{\gp}(X)\ar[r]\ar[d,"b"]&
\Z\ar[r]\ar[d,"\cong"]&
\logPic(Y)\ar[r]\ar[d,"c"]&
\logPic(X)\ar[r]\ar[d,"d"]&
0
\\
0\ar[r]&
\cM^{\gp}(Y')\ar[r]&
\cM^{\gp}(X')\ar[r]&
\Z\ar[r]&
\logPic(Y')\ar[r]&
\logPic(X')\ar[r]&
0
\end{tikzcd}
\end{equation}
with exact rows.
We have a similar commutative diagram for $\logPic_{\set}$.
By induction, $a$ and $c$ are isomorphisms.
An application of the five lemma shows that $b$ and $d$ are isomorphisms too.

To show that \eqref{logPic.4.4} is an isomorphism, replace the lower row of \eqref{logPic.8.3} by the corresponding sequence for $\logPic_{\set}$, and argue similarly as above.
\end{proof}

We do not know whether there is an isomorphism of abelian groups
\[
\logPic(X)
\xrightarrow{\cong}
\logPic(X\times (\P^n,\P^{n-1}))
\]
for arbitrary fs log schemes $X$ and $n\geq 1$.
The assumption imposed in Proposition \ref{logPic.8} seems too restrictive.
In the following, we investigate another case.

\begin{prop}
\label{logPic.6}
Let $X$ be a saturated log scheme such that $\ul{X}$ is the spectrum of a field.
Then the homomorphisms in \eqref{logPic.4.5} are isomorphisms.
\end{prop}
\begin{proof}
We set $G:=\ol{\cM}^{\gp}(X)$ as abbreviation.
The assumption that $X$ is saturated implies that $G$ is torsion-free.
Hence, 
\cite[Proposition XI.3.6(2)]{SGA4} is applicable to $G$ due to \cite[Remarques XI.3.7]{SGA4}.
Since $\ul{X}$ and $\ul{X}\times \P^{n-1}$ are irreducible and geometrically unibranch, as in \eqref{logPic.3.3} and \eqref{logPic.3.2} there are isomorphisms of abelian groups
\begin{equation}
\label{logPic.6.1}
\Gamma(X,G) \cong \Gamma(X\times (\P^n,\P^{n-1}),G)\cong G
\end{equation}
and the vanishing
\begin{equation}
\label{logPic.6.2}
H_{sNis}^1(X,G)
=
H_{sNis}^1(X\times \P^{n-1},G)
=
H_{\set}^1(X,G)
=
H_{\set}^1(X\times \P^{n-1},G)
=
0.
\end{equation}

Let $i\colon \ul{X}\times \P^{n-1}\to \ul{X}\times \P^n$ be the closed immersion at $\infty$.
There are exact sequences
\[
0\to \cM_{\ul{X}}^\gp\to \cM^\gp_{X}\to G \to 0
\text{ and }
0\to \cM_{\ul{X}\times (\P^n,\P^{n-1})}^{\gp}\to \cM^\gp_{X\times (\P^n,\P^{n-1})}\to i_*G \to 0.
\]
As abbreviations, we set $Y:=\ul{X}$, $X'=X\times (\P^n,\P^{n-1})$, and $Y':=\ul{X}\times (\P^n,\P^{n-1})$.
Together with \eqref{logPic.6.1} and \eqref{logPic.6.2}, these give a commutative diagram with exact rows:
\[
\begin{tikzcd}[column sep=small]
0\ar[r]&
\cM^{\gp}(Y)\ar[r]\ar[d,"a"']&
\cM^{\gp}(X)\ar[r]\ar[d,"b"]&
G\ar[r]\ar[d,"\cong"]&
\logPic(Y)\ar[r]\ar[d,"c"]&
\logPic(X)\ar[r]\ar[d,"d"]&
0
\\
0\ar[r]&
\cM^{\gp}(Y')\ar[r]&
\cM^{\gp}(X')\ar[r]&
G\ar[r]&
\logPic(Y')\ar[r]&
\logPic(X')\ar[r]&
0.
\end{tikzcd}
\]
We have a similar commutative diagram for $\logPic_{\set}$.
By Proposition \ref{logPic.3}, $a$ and $c$ are isomorphisms.
The five lemma lets us conclude that also $b$ and $d$ are isomorphisms.
\end{proof}

\subsection{The infinite projective space}
Morel-Voevodsky showed the equivalence
\begin{equation}
\label{Pic.0.1}
\P^\infty \xrightarrow{\simeq} \clspace\G_m
\end{equation}
in $\infH(S)$, where $S\in \Sch$, see \cite[Proposition 3.7, p.\ 138]{MV}.
As argued in this text, 
$\Gmm$ replaces $\G_m$ in $\boxx$-motivic homotopy theory.
The purpose of this section is to show that there is an equivalence
\begin{equation}
\label{Pic.0.2}
\P^\infty \to \clspace \Gmm
\end{equation}
in $\inflogH(S)$ for every $S\in \Sch$.

\

There is an explicit description of $\A^1$ -homotopies that yields the equivalence
\eqref{Pic.0.1} in \cite[\S 2]{zbMATH05635033}.
Our method of establishing \eqref{Pic.0.2} is to ``compactify'' their argument.

\begin{df}
In this subsection, we use the abbreviation 
\[
\cptsph{n}:=(\Blow_O(\P^n),H+E)
\]
for every integer $n\geq 1$, 
where $O$ is the point $[0:\cdots:0:1]$ in $\P^n$,
$H$ is the
projective hyperplane given by the equation $z_n=0$ if the coordinates of $\P^n$ are written as $[z_0:\cdots:z_n]$,
and $E$ is the exceptional divisor.
Observe that the log structure on $\cptsph{n}$ is the compactifying log structure 
associated with the open immersion $\A^n-0\to \Blow_O(\P^n)$.
\end{df}

There is a closed immersion
\[
\P^n\to \P^{n+1}
\]
sending $[a_1:\cdots:a_{n+1}]$ to $[a_1:\cdots:a_n:0:a_{n+1}]$.
Since the preimage of $O\in \P^{n+1}$ is $O\in \P^n$, there is a naturally induced  morphism
\begin{equation}
\label{Pic.1.3}
\Blow_O(\P^n)\to \Blow_O(\P^{n+1}).
\end{equation}
This is an extension of the closed immersion
\begin{equation}
\label{Pic.1.2}
\A^n-0\to \A^{n+1}-0
\end{equation}
sending $(a_1,\ldots,a_{n})$ to $(a_1,\ldots,a_{n},0)$.
Combine \eqref{Pic.1.3} and \eqref{Pic.1.2} to produce a strict closed immersion
\begin{equation}
\label{Pic.1.1}
i\colon \cptsph{n}\to \cptsph{n+1}.
\end{equation}
Finally, 
we form a colimit in the category of presheaves 
\[
\cptsph{\infty}
:=
\colimit_n \cptsph{n}.
\]

For every integer $n\geq 1$, let $v_1,\ldots,v_n$ be the points $(1,0,\ldots,0)$, $\ldots$, $(0,\ldots,0,1)$ in $\A^n-0$.

\begin{lem}
\label{Pic.3}
Suppose $S\in \Sch$.
For every integer $n\geq 1$, the inclusion
\[
i\colon \cptsph{n} \rightarrow \cptsph{n+1}
\]
in \eqref{Pic.1.1} is homotopic to the constant map to $v_{n+1}\in V_{n+1}$ in $\inflogH(S)$.
\end{lem}
\begin{proof}
Let $u\colon S\to \Spec{\Z}$ be the structure morphism.
If we show the claim for $\Spec{\Z}$, then we can use $u^*$ to deduce the claim for $S$.
Hence we may assume $S=\Spec{\Z}$.

The closed immersion $v$ in \eqref{Pic.1.2} is $\A^1$-homotopic to a constant map via the morphism
\[
h\colon (\A^n-0)\times \A^1\to \A^{n+1}-0
\]
sending $(a_1,\ldots,a_n,t)$ to $(ta_1,\ldots,ta_{n},1-t)$.
We have the corresponding rational map $\P^n\times \P^1\dashrightarrow \P^{n+1}$ given by
\[
([z_0:\cdots:z_n],[t_0:t_1])
\mapsto 
[z_0 t_0:\cdots:z_{n-1} t_0:z_n(t_1-t_0):z_n t_1].
\]
One can check that this rational map becomes a morphism of schemes after the blow-up at $H\times [0:1]$,
where $H$ is the divisor $z_n=0$.
Hence we have the morphism
\[
\Blow_{H\times [0:1]}(\P^n \times \P^1)
\to
\P^{n+1}.
\]
The inverse image of $O\in \P^{n+1}$ is $(O,[1:1])$.
Hence we have the induced morphism
\[
\Blow_{(O,[1:1])}(\Blow_{H\times [0:1]}(\P^n \times \P^1))
\to
\Blow_O (\P^{n+1}).
\]
One can also check that there is a proper birational morphism
\[
\Blow_{E\times [1:1]+\widetilde{H}\times [0:1]}(\Blow_O(\P^n)\times \P^1)
\to
\Blow_{(O,[1:1])}(\Blow_{H\times [0:1]}(\P^n \times \P^1)),
\]
where $\widetilde{H}$ is the strict transform of $H$,
and $E$ is the exceptional divisor on $\Blow_O(\P^n)$.
Hence we have the composite
\[
g\colon
\Blow_{E\times [1:1]+\widetilde{H}\times [0:1]}(\Blow_O(\P^n)\times \P^1)
\to
\Blow_O(\P^{n+1}).
\]

Let $Y$ be the fs log scheme with the underlying scheme 
$\Blow_{E\times [1:1]+\widetilde{H}\times [0:1]}(\Blow_O(\P^n)\times \P^1)$, 
and whose log structure is the compactifying log structure associated with $(\A^n-0)\times \A^1$.

Next, $g$ gives a morphism
\[
f\colon Y\to \cptsph{n+1}
\]
extending $h$, and there is a projection $p\colon Y\to  \cptsph{n}\times \boxx$.
The zero section and one section $i_0,i_1\colon  \cptsph{n}\to  \cptsph{n}\times \boxx$ admit lifts
\[
i_0',i_1'\colon  \cptsph{n}\to Y.
\]
Since $p$ is a composite of admissible blow-ups along smooth centers, 
we can apply Lemma \ref{ProplogSH.7} to deduce that $i=fi_1'$ is homotopic to the constant map 
$fi_0'$ in $\inflogH(S)$.
\end{proof}

\begin{rmk}
We note that $h$ appearing in Lemma \ref{Pic.3} cannot be extended to a morphism
\[
\Blow_O(\P^n)\times \P^1\to \Blow_O(\P^{n+1})
\]
since $g$ does not factor through the blow-down $\Blow_O(\P^n)\times \P^1$.
\end{rmk}

\begin{lem}
\label{Pic.5}
For every $S\in \lSch$ and $n\geq 2$, the two constant maps
\[
\pt \rightrightarrows V_n
\]
to $v_1$ and $v_n$ are homotopic in $\inflogH(S)$.
\end{lem}
\begin{proof}
Let $f\colon \P^1\to \P^{n}$ be the morphism given by
\[
[a:b]
\mapsto 
[a:0\cdots : 0:b-a:b].
\]
Since $f(\A^1)\subset \A^n$, 
$f$ naturally induces a morphism $g\colon \boxx\to (\P^n,H)$, where $H$ is the divisor $\P^n-\A^n$.
The image of $f$ does not intersect with the point $O:=[0:\cdots:0:1]$, so $g$ can be lifted to a morphism
\[
h\colon \boxx \to V_n.
\]
To conclude, observe that $h([0:1])=v_n$ and $h([1:1])=v_1$.
\end{proof}

\begin{prop}
\label{Pic.4}
For every $S\in \Sch$, 
the object $\cptsph{\infty}$ is equivalent to $\pt$ in $\inflogH(S)$.
\end{prop}
\begin{proof}
By Lemmas \ref{Pic.3} and \ref{Pic.5}, 
the morphism $V_n\to V_{n+1}$ is equivalent to the the constant map to $v_1$ in $\inflogH(S)$,
which finishes the proof.
\end{proof}

\begin{df}
\label{Pic.11}
Suppose $S\in \lSch$ and $G\in \infShv_{dNis}(\lSm/S)$ is a group object acting on $X\in \infShv_{dNis}(\lSm/S)$.
Let $X/^h G$\index[notation]{Xhg @ $X/^h G$} denote the homotopy colimit of the simplicial object
\begin{equation}
\label{Pic.11.1}
\begin{tikzcd}
\cdots\ar[r,shift left=0.5ex]\ar[r,shift right=0.5ex]\ar[r,shift left=1.5ex]\ar[r,shift right=1.5ex]&
X\times G^2\ar[r,shift left=1ex]\ar[r,shift right=1ex]\ar[r]&
X\times G\ar[r,shift left=0.5ex]\ar[r,shift right=0.5ex]&
X.
\end{tikzcd}
\end{equation}
If $X\to Y$ is a $G$-equivariant morphism in $\infShv_{dNis}(\lSm/S)$ with trivial $G$-action on $Y$, 
there is a naturally induced morphism $X/^h G\to Y$.
\end{df}

\begin{exm}
\label{Ori.73}
For every $S\in \lSch$ there is no morphism
\[
m \colon \Gmm\times \Gmm \to \Gmm
\]
in $\lSm/S$ that extends the multiplication $\G_m\times \G_m\to \G_m$ sending $(x,y)$ to $xy$
since it is impossible to define $m$ at $(0,\infty)$ and $(\infty,0)$.
However, $m$ exists in $\infShv_{dNis}(\lSm/S)$.
To show this, consider the morphism
\[
p\colon Y:= 
(\Blow_{(0,\infty)\cup 
(\infty,0)}(\P^1\times \P^1),\widetilde{H_1}+\widetilde{H_2}+\widetilde{H_3}+\widetilde{H_4}+E_1+E_2) 
\to 
\Gmm\times \Gmm,
\]
where $\widetilde{H_i}$ is the strict transforms of $0\times \P^1$, $\P^1\times 0$, $\infty\times \P^1$, 
or $\P^1\times \infty$, and $E_1$ and $E_2$ are the exceptional divisors.
Since $p$ is a dividing cover,
$p$ is an equivalence as a morphism of sheaves.
We now claim that there  exists a morphism $n\colon Y\to \Gmm$ extending the multiplication morphism $\G_m\times \G_m\to \G_m$.
To show this,
note that this admits an extension $\Blow_{(0,\infty)+ (\infty 0)}(\P^1\times \P^1)\to \P^1$,
and check that
this induces $p$.
We set $m$ to be the composite $np^{-1}$.

The unit of $\Gmm$ is defined to be the one section $u\colon 1\to \Gmm$.
The morphism $\P^1\to \P^1$ sending $[a:b]$ to $[b:a]$ sends $\G_m$ onto $\G_m$.
By \cite[Proposition III.1.6.2]{Ogu}, this naturally induces a morphism $i\colon \Gmm\to \Gmm$.

\

With the construction described above, the fact that $\G_m$ is a group object readily implies that $\Gmm$ is a group object in $\infShv_{dNis}(\lSm/S)$.
\end{exm}

\begin{lem}
\label{Pic.8}
Suppose $S\in \lSch$.
There is an equivalence
\begin{equation}
\label{Pic.8.1}
\cptsph{n+1}/^h \Gmm
\xrightarrow{\simeq}
\P^n
\end{equation}
in $\inflogH(S)$.
\end{lem}
\begin{proof}
Consider the rational map $\P^{n+1}\dashrightarrow \P^n$ given by
\[
[z_0:\cdots:z_{n+1}]
\mapsto
[z_0:\cdots:z_n].
\]
This becomes a morphism of schemes $g\colon \Blow_O(\P^{n+1})\to \P^n$ after the blow-up at $O\in \P^{n+1}$.
Observe that $g$ is a $\P^1$-bundle associated with a line bundle.
The $0$-section (resp.\ $\infty$-section)
is the exceptional divisor $E$ (resp.\ hyperplane $H$ given by $z_{n+1}=0$).
Hence the induced morphism $f\colon V_{n+1}\to \P^n$ is a $\Gmm$-bundle,
and then $f$ is $\Gmm$-equivariant with trivial $\Gmm$-action on $\P^n$.
This constructs the morphism in \eqref{Pic.8.1}.
We show that it is an equivalence.
By considering the \v{C}ech nerve of $\P^n$ associated with a Zariski covering $\{U_i\}_{1\leq i\leq n}$ that trivializes the $\Gmm$-bundle $f$,
it suffices to show
\[
(U_I\times \Gmm)/^h \Gmm \simeq U_I
\]
for all nonempty subsets $I$ of $\{0,\ldots,n\}$,
where $U_I:=\cap_{i\in I}U_i$.
This follows since the simplicial diagram \eqref{Pic.11.1} defining $(U_I\times \Gmm)/^h \Gmm$ 
admits an evident retraction.
\end{proof}

\begin{thm}
\label{Pic.9}
For every $S\in \Sch$ there is an equivalence
\[
\clspace \Gmm \simeq \P^\infty
\]
in $\inflogH(S)$,
which is natural in $S$.
\end{thm}
\begin{proof}
Colimits preserve equivalences, so Lemma \ref{Pic.8} yields an equivalence
\begin{equation}
\label{Pic.9.1}
\cptsph{\infty}/^h \Gmm
\simeq
\P^\infty
\end{equation}
in $\inflogH(S)$.
Proposition \ref{Pic.4} and \cite[Proposition 2.14, p.\ 74]{MV} finish the proof.
\end{proof}

\begin{cor}\label{Pic.10}
For every $S\in \lSch$,
there is an equivalence
\[
\clspace \Gmm \simeq \P^\infty
\]
in $\inflogSH^\eff(S)$ and $\inflogSH(S)$,
which is natural in $S$.
\end{cor}
\begin{proof}
The claim holds for $\ul{S}$ by Proposition \ref{ProplogSH.9} and Theorem \ref{Pic.9}.
We deduce the claim for $S$ by pulling back.
\end{proof}

\newpage

\section{\texorpdfstring{$K$}{K}-theory of fs log schemes}
\label{section:kfslogschemes}
\subsection{Setting and summary of the results}
Morel and Voevodsky \cite[Propositions 3.7 and 3.9]{MV} proved that the algebraic $K$-theory of regular schemes is represented 
geometrically by the infinite Grassmannian $\Z\times \Gr$ in the unstable motivic homotopy category $\mathcal{H}(S)_\ast$ for any noetherian scheme $S$ of finite Krull dimension. This result was later refined by Voevodsky in \cite{zbMATH01194164}, where the Bott periodic $\P^1$-spectrum $\KGL$ was constructed. In \cite{zbMATH01194164} Voevodsky announced that $\KGL$
 represents homotopy invariant algebraic $K$-theory for  schemes in the \emph{stable} motivic homotopy category.
Cisinski completed the proof in \cite[Th\'eor\`eme 2.20]{zbMATH06156613}. See also \cite{PPR} for a detailed discussion.

\

The main goal of this section is to construct a logarithmic $K$-theory spectrum $\logKGL$ in $\inflogSH^\Adm(S)$ for any 
regular scheme $S$  (see Definition \ref{K-theory.16}), and 
show  that there is an equivalence
\begin{equation}
\label{K-theory.0.2}
\logKGL
\simeq
(\bbL_{sNis,\boxx\cup \Adm}\Z\times \Gr,\bbL_{sNis,\boxx\cup \Adm}\Z\times \Gr,\ldots)
\end{equation}
in $\inflogSH^\Adm(S)$.
Here the pointed log motivic space $\bbL_{sNis,\boxx\cup \Adm}\Z\times \Gr$ can be viewed as a 
logarithmic replacement for the infinite Grassmannian employed by Morel-Voevodsky \cite{MV}.
Our proof of the equivalence \eqref{K-theory.0.2} uses the singular functor $\Sing$ discussed 
in \S\ref{subsection:applicationsSing}. 
This is the main reason for assuming that $S$ is a regular scheme
in \eqref{K-theory.0.2}.  At the end of this  section,  
we will compare our logarithmic $K$-theory with Niziol's $K$-theories of fs log schemes defined in \cite{MR2452875}.
Elementary computations 
show that the two theories are different.

The main application of the above-mentioned result will be
discussed in \S\ref{subsection:logtrace}, where we will use the geometric representability of logarithmic $K$-theory 
for constructing a logarithmic cyclotomic trace map, as a morphism internal to our motivic homotopy theory. Note that the corresponding construction would dramatically fail in $\mathcal{SH}(S)$.

\begin{rmk}Let $\Kth\in \infPsh(\Sch_{noeth},\infSpt)$\index[notation]{Kth @ $\Kth$} 
denote the Bass $K$-theory presheaf of spectra.
If $X\in \Sch$, then
\[
\pi_{n}\Kth(X)
\cong
\Kth_n(X)
\]
for all integers $n$, where $\Kth_n(X)$ is the $n$th Bass $K$-group of perfect complexes on $X$.

Due to \cite[Propositions 3.9]{MV}, 
there is a canonical equivalence of sheaves of pointed spaces
\begin{equation}
\label{K-theory.14.1}
\Omega_{S^1}^\infty \Kth
\simeq
\bbL_{Nis} \Omega \clspace \big( \coprod_{n\geq 0} \BGL_n \big),
\end{equation}
where $\bbL_{Nis}$ is the Nisnevich localization functor, and $\Omega_{S^1}^\infty$ is the infinite suspension functor.

If $X\in \Sch_{noeth}$ is regular, 
then the Bass $K$-theory spectrum $\Kth(X)$ and $\Omega_{S^1}^\infty$ participates in an equivalence of spectra
\begin{equation}
\label{K-theory.14.2}
\Kth(X)
\simeq
(\Omega_{S^1}^\infty \Kth(X),\Omega_{\mathbb{G}_m}\Omega_{S^1}^\infty \Kth(X),\Omega_{\mathbb{G}_m}^2\Omega_{S^1}^\infty \Kth(X),\ldots),
\end{equation}
where the bonding maps
\[
\Sigma_{S^1}\Omega_{\mathbb{G}_m}^n\Omega_{S^1}^\infty \Kth(X)
\to
\Omega_{\mathbb{G}_m}^{n+1}\Omega_{S^1}^\infty \Kth(X)
\]
are naturally induced by the map
\[
\Omega_{S^1}^\infty \Kth\to \Omega_{\P^1}\Omega_{S^1}^\infty \Kth
\]
appearing in the projective bundle formula \cite[Theorem 7.3]{TT} for $K$-theory.
Restricting $\Kth$ to $\Sm/S^\mathrm{op}$ defines an object of $\infSHS(S)$, 
since $\Kth(-)$ is $\A^1$-invariant for regular schemes and satisfies Nisnevich descent.
\end{rmk}
\subsection{The singular functor}
\label{subsection:tsf}

In this subsection,
we begin by reviewing the singular functor in \cite[\S 6]{logDM}.

Throughout, 
$\bT$ is a small ordinary category with finite limits, 
$\pt$ denotes the initial object, and $\cA$ is a set of morphisms in $\bT$ that admits a calculus of right fractions.
Moreover, 
$I$ is an interval object of $\bT$ in the sense of \cite[\S 2.3]{MV},
i.e., an object of $\bT$ together with morphisms
\[
i_0,i_1\colon \pt\to I,
\;
p\colon I\to \pt
\]
in $\bT$ and with a morphism
\[
\mu\colon I\times I\to I
\]
in $\bT[\cA^{-1}]$ satisfying the identities
\begin{gather*}
\mu\circ (i_0\times \id)=(\id\times i_0)\circ \mu=i_0p,
\\
\mu\circ (i_1\times \id)=(\id\times i_1)\circ \mu=\id.
\end{gather*}
We also assume that $i_0\amalg i_1\colon \pt\amalg \pt \to I$ is a monomorphism.

There is a localization functor
\[
v\colon \bT\rightarrow \bT[\cA^{-1}]
\]
inducing an adjunction
\[
\begin{tikzcd}
\infPsh(\bT)\arrow[rr,shift left=1.5ex,"v_\sharp "]\arrow[rr,"v^*" description,leftarrow]\arrow[rr,shift right=1.5ex,"v_*"']&&\infPsh(\bT[\mathcal{A}^{-1}]).
\end{tikzcd}
\] 
Here $v_\sharp$ is left adjoint to $v^*$, 
and $v^*$ is left adjoint to $v_*$.
Let $\cA\downarrow X$ denote the over category of $X$, 
whose objects are the morphisms $f\colon Y\to X$ in $\cA$,  with the obvious convention for maps.
The functors $v_\sharp$ and $v^*$ are then given objectwise by  
\begin{equation}
\label{Singboxx.0.1}
v_\sharp \cF(X)
\cong
\colimit_{Y\in \cA\downarrow X}\cF(Y), \quad v^*\cG(X):=\cG(X)
\end{equation}
for any presheaf of Kan complexes $\cF\in \infPsh(\bT)$, $\cG\in \infPsh(\bT[\cA^{-1}])$, and $X\in \bT$.

For abbreviation, let $I$ be the set of projections $X\times I\to X$ for all $X\in \bT$.

\begin{prop}
\label{Singboxx.5}
The unit $\id \to v^*v_\sharp$ is an $\cA$-local equivalence.
\end{prop}
\begin{proof}
We refer to \cite[Proposition 6.1.11]{logDM}.
\end{proof}

By assumption, $I$ is an interval object of $\bT[\cA^{-1}]$ in the sense of \cite[\S 2.3]{MV}. Recall that there is a functor
\[
\Sing^I\colon \infPsh(\bT[\cA^{-1}])\to \infPsh(\bT[\cA^{-1}]).
\]
that is defined as follows: let $\Delta_I^\bullet$ be the simplicial object of $\bT[\cA^{-1}]$ defined objectwise by $[n]\mapsto I^n$ (see \cite[p.88]{MV}, or \cite[6.2]{logDM}). Then we have the $\infty$-functor
\[\cF \mapsto \underline{\Map}(y(\Delta_I^\bullet), \cF) \in \Fun(\Delta^\mathrm{op}, \infPsh(\bT[\cA^{-1}]). \]
We let $\Sing^I(\cF)$ be the geometric realization of $\underline{\Map}(y(\Delta_I^\bullet), \cF)$. Using e.g., \ \cite[A.2.9.30]{HTT} one can readily verify that it agrees with \cite[Definition 6.2.1]{logDM}.
\begin{df}
\label{Singboxx.1}
The singular functor\index[notation]{SingIA @ $\Sing^{I\cup \cA}$}
\[
\Sing^{I\cup \cA}\colon \infPsh(\bT)\to \infPsh(\bT)
\]
is given by $v^*\Sing^I v_\sharp$.
\end{df}

\begin{df}
\label{Singboxx.2}
Let $f,g\colon \cF\to \cG$ be two morphisms in $\infPsh(\bT)$.
An \emph{elementary $I\cup \cA$-homotopy between $f$ and $g$}\index{homotopy} is a morphism
\[
h\colon v_\sharp \cF\times I \to v_\sharp \cG
\]
in $\infPsh(\bT[\cA^{-1}])$ such that
\[
h\circ i_0=f
\text{ and }
h\circ i_1=g.
\]
Two morphisms are called \emph{$I\cup \cA$-homotopic} if they can be connected by 
a finite number of elementary $I\cup \cA$-homotopies.

A morphism $f\colon \cF\to \cG$ in $\infPsh(\bT)$ is called a 
\emph{strict $I\cup \cA$-homotopy equivalence}\index{strict homotopy equivalence} if there exists 
a morphism $g\colon \cG\to \cF$ in $\infPsh(\bT)$ such that the compositions $f\circ g$ and $g\circ f$ are both $I\cup \cA$-homotopic to the respective identities.
\end{df}

\begin{prop}
\label{Singboxx.3}
The following statements hold true.
\begin{enumerate}
\item [{\rm (1)}]
The morphism
\[
i_0\colon \pt \to I
\]
is a strict $I\cup \cA$-homotopy equivalence.
\item [{\rm (2)}]
If $f$ and $g$ are $I\cup \cA$-homotopic morphisms in $\infPsh(\bT)$, 
then $\Sing^{I\cup \cA}(f)$ and $\Sing^{I\cup \cA}(g)$ are  homotopic.
\item [{\rm (3)}]
For every $X\in \bT$, the morphism
\[
\Sing^{I\cup \cA}(X) \to \Sing^{I\cup \cA}(X\times \boxx) 
\]
naturally induced by $i_0$ is an equivalence equivalence of presheaves of Kan complexes.
\item [{\rm (4)}]
For every $\cF\in \infPsh(\bT)$, the canonical morphism
\[
\cF\to \Sing^{I\cup \cA}(\cF)
\]
is an $I\cup \cA$-local equivalence.
\end{enumerate}
\end{prop}
\begin{proof}
We refer to \cite[Propositions 6.2.3, 6.2.4, 6.2.6, Corollary 6.2.5]{logDM}.
\end{proof}

\subsection{Applications of \texorpdfstring{$\Sing$}{Sing}}
\label{subsection:applicationsSing}

Our theory of $\Sing$ is mainly applied to the case when $\bT=\cSm/S$, 
where $S$ is a scheme,
$I=\boxx$, and $\cA$ is the set of admissible blow-ups $\Adm$ in $\cSm/S$.

\begin{prop}
\label{Singboxx.6}
For every $S\in \Sch$, 
the class of admissible blow-ups in $\cSm/S$ admits a calculus of right fractions.
\end{prop}
\begin{proof}
It is obvious that the class of admissible blow-ups is closed under compositions.

The class of admissible blow-ups is closed under pullbacks,
so the class of admissible blow-ups satisfies the right Ore condition.

Suppose $f,g\colon Y\to X$ and $h\colon X\to X'$ be morphisms in $\cSm/S$ such that $hf=hg$.
Then $f$ and $g$ agree on the dense open subscheme $Y-\partial Y$ of $Y$,
so we have $f=g$.
This shows that the class of admissible blow-ups satisfies the right cancellability condition.
\end{proof}

\begin{prop}
\label{Singboxx.4}
For every $S\in\Sch$ and  presheaf $\cF\in \infPsh(\cSm/S)$, 
the canonical morphism
\[
\cF\to \Sing^{\boxx\cup\Adm}(\cF)
\]
is a $\boxx\cup \Adm$-local  equivalence.
\end{prop}
\begin{proof}
Immediate from Proposition \ref{Singboxx.3}(4).
\end{proof}

\begin{prop}
\label{K-theory.1}
Suppose $S\in \Sch$ and $\cF\in \infPsh(\cSm/S)$ is of the form $\colimit \cF_i$,  
where each $\cF_i$ is representable by a smooth and proper scheme over $S$.
Then the unit map
\[
\cF\rightarrow  \omega^* \omega_\sharp \cF
\]
is a $\boxx\cup\Adm$-local  equivalence.
\end{prop}
\begin{proof}
By Proposition \ref{Singboxx.5}, there is a $\boxx\cup \Adm$-local  equivalence
\[
\cF\to v^*v_\sharp \cF.
\]
Use \eqref{Singboxx.0.1} to have an equivalence of Kan complexes
\[
v^*v_\sharp \cF(Y)\cong \colimit_{Y'\in \Adm\downarrow Y} \cF(Y').
\]
On the other hand, there is an equivalence
\[
\omega^*\omega_\sharp \cF(Y)
\cong
\cF(Y-\partial Y).
\]
Hence we only need to show that the map
\[
\colimit_{Y'\in \Adm \downarrow Y}\cF(Y')
\to
\cF(Y-\partial Y)
\]
is an equivalence.
We only need to show this for each $\cF_i$.
Hence we only need to show that the map
\begin{equation}
\label{K-theory.1.1}
\colimit_{Y'\in \Adm\downarrow Y} \hom(Y',X)
\to
\hom(Y-\partial Y,X)
\end{equation}
is an isomorphism for every smooth and proper scheme $X$ over $S$.

Suppose $f\colon Y-\partial Y\to X$ is a morphism of schemes.
Let $\ul{Y'}$ be the closure of the image of the graph morphism 
$Y-\partial Y\to (Y-\partial Y)\times \ul{X}$ in $\ul{Y}\times \ul{X}$.
By setting $Y':=(\ul{Y'},Y'-\partial Y')$ and $Y'-\partial Y'$ as the preimage of $Y-\partial Y$ in $Y'$,
we obtain an admissible blow-up $p$ participating 
in a commutative diagram of fs log schemes
\[
\begin{tikzcd}
Y'\arrow[d,"p"']\arrow[rd,"g"]
\\
Y\arrow[r,"f"] & X.
\end{tikzcd}
\]
This shows that \eqref{K-theory.1.1} is surjective.

It remains to show that \eqref{K-theory.1.1} is injective.
Suppose $Y'$ and $Y''$ are two admissible blow-ups of $Y$ with two morphisms $f'\colon Y'\to X$ and 
$f''\colon Y''\to X$ such that the restrictions of $f'$ and $f''$ to $Y-\partial Y$ are equal.
We need to show that there exists a common admissible blow-up $Y'''$ of $Y'$ and $Y''$ such that 
we have the same morphism from $f'$ and $f''$ by composing $Y'''\to Y'$ and $Y'''\to Y''$.
A common admissible blow-up always exists since $\Adm\downarrow Y$ admits a calculus of 
right fractions by Proposition \ref{Singboxx.6}, so we may assume that $Y'=Y''$.
To conclude, 
observe that $Y-\partial Y$ is dense in $Y'$ and $X$ is separated.
\end{proof}

\begin{lem}
\label{K-theory.3}
For $\cF\in \infPsh(\cSm/S)$, where $S\in \Sch$, 
there is a natural isomorphism 
\[
\Sing^{\boxx\cup \Adm}\omega^*\cF\cong  \omega^*\Sing^{\A^1}\cF.
\]
\end{lem}
\begin{proof}Let $\cG\in \infPsh(\cSm/S)$. Recall that $\Sing^{\boxx\cup \Adm}\cG$ is defined as the geometric realization of the simplicial presheaf \[[n]\mapsto \Sing_n^{\boxx\cup \Adm}\cG=v^*\Sing_n^{\boxx}v_\sharp\cG= v^*\underline{\Map}(\boxx^n, v_\sharp \cG).\]
For every $X\in \cSm/S$, 
we have then
\begin{equation}
\label{K-theory.3.1}
\Sing_n^{\boxx\cup\Adm}\omega^*\cF(X)
=
\colimit_{Y\in \Adm\downarrow (X\times \boxx^n)}\omega^*\cF(Y)
=
\colimit_{Y\in \Adm\downarrow (X\times \boxx^n)}\cF(Y-\partial Y).
\end{equation}
If $Y\to X\times \boxx^n$ is an admissible blow-up, then the morphism
\[
Y-\partial Y\to (X-\partial X)\times \A^n
\]
is an isomorphism of schemes.
Hence we have
\begin{equation}
\label{K-theory.3.2}
\colimit_{Y\in \Adm\downarrow X\times \boxx^n}\cF(Y-\partial Y)
\cong
\cF((X-\partial X)\times \A^n)
\cong
\Sing_n^{\A^1}\cF(X-\partial X)
=
\omega^*\Sing_n^{\A^1}\cF(X).
\end{equation}
Combine \eqref{K-theory.3.1} and \eqref{K-theory.3.2} to have an isomorphism of presheaves
\[
\Sing_n^{\boxx\cup\Adm}\omega^*\cF
\cong
\omega^*\Sing_n^{\A^1}\cF
\]
To conclude, 
observe that the degeneracy and face maps for $\Sing_n^{\boxx\cup\Adm}\omega^*\cF$ and $\omega^*\Sing_n^{\A^1}\cF$ 
are compatible with this isomorphism.
\end{proof}

\subsection{Infinite Grassmannians and \texorpdfstring{$\Omega_{S^1}^\infty \Kth$}{K}}

In the following,
we review how to obtain the infinite dimensional Grassmannians from the finite dimensional ones.

\begin{df} Fix integers $0<r<n$. Let $\Gr(r,n)$\index[notation]{Grrn @ $\Gr(r,n)$} denote the $(r, n)$ Grassmannian\index{Grassmannian}, that is, the scheme representing the functor $S\mapsto \Gr(r,n)(S)$, where $\Gr(r,n)(S)$ is the set of isomorphism classes of surjections $\cO^{\oplus n}_S \to \mathcal{Q}$, where $\mathcal{Q}$ is a locally free $\cO_S$-module of rank $n-r$, see \cite[\href{https://stacks.math.columbia.edu/tag/089R}{Tag 089R}]{stacks-project}.

We equip $\Gr(r,n)$ with the trivial log structure.
In particular, for every  log scheme $X$, there is a bijection 
\[
\hom(X,\Gr(r,n))\cong \hom(\ul{X},\Gr(r,n)).
\]
There are canonical morphisms
\[
\Gr(r,n)\to \Gr(r,n+1)
\text{ and }
\Gr(r,n)\to \Gr(r+1,n+1)
\]
taking $V\subset \cO^n$ to $V\oplus \{0\}\subset \cO^n\oplus \cO$ and $V\oplus \cO\subset \cO^n\oplus \cO$, 
where $\cO$ denotes the standard linear space of dimension $1$.
From these, 
we can form colimits
\[
\Gr(r,\infty):=\colimit_{n\geq r}\Gr(r,n)
\text{ and }
\Gr:=\colimit_{r} \Gr(r,\infty).
\]
Observe that $\Gr(r,\infty)$ and $\Gr$ are filtered colimits of smooth schemes.
There is a natural morphism
\[
\Gr(r,n)\times \Gr(r',n')\to \Gr(r+r',n+n')
\]
sending $(V,W)$ to $V\oplus W$.
This gives a (non-commutative) monoid structure on the coproduct $\coprod_{r=1}^\infty \Gr(r,\infty)$
in the category of  log schemes.
\end{df}

\begin{prop}
\label{K-theory.15}
If $X$ is a regular affine noetherian scheme, then the canonical morphism
\begin{equation}
\label{K-theory.15.1}
\Sing^{\A^1}(\Z\times \Gr)(X)
\to
\Omega_{S^1}^\infty \Kth(X)
\end{equation}
is an equivalence of Kan complexes.
In particular, if $Y$ is a regular noetherian scheme, then the canonical morphism
\begin{equation}
\label{K-theory.15.2}
L_{Nis}\Sing^{\A^1}(\Z\times \Gr)(Y)
\to
\Omega_{S^1}^\infty \Kth(Y)
\end{equation}
is an equivalence of Kan complexes.
\end{prop}
\begin{proof}
For \eqref{K-theory.15.1}, see the proof of \cite[Proposition 3.16, p.\ 141]{MV}.
This implies that \eqref{K-theory.15.2} is a 
weak equivalence since $\Omega_{S^1}^\infty \Kth$ satisfies Nisnevich descent.
\end{proof}

\begin{prop}
\label{K-theory.18}
Suppose $S$ is a regular noetherian scheme.
Then for all $X\in \cSm/S$, there is a canonical equivalence of spaces
\[
\Omega_{S^1}^\infty \Kth(X-\partial X)
\simeq
\Map_{\inflogHpt^\Adm(S)}(X_+,\omega^*\Omega_{S^1}^\infty \Kth).
\]
\end{prop}
\begin{proof}
For every $X\in \cSm/S$, we have an equivalence
\[
\omega^*\Omega_{S^1}^\infty \Kth(X)
\simeq
\Omega_{S^1}^\infty \Kth(X-\partial X).
\]
The $\A^1$-invariance of $K$-theory implies that $\omega^*\Omega_{S^1}^\infty \Kth$ is $\boxx$-local.
If $Y\to X$ is an admissible blow-up, then $Y-\partial Y\to X-\partial X$ is an isomorphism of schemes.
It follows that $\omega^*\Omega_{S^1}^\infty \Kth$ is $\Adm$-local.
Furthermore, 
together with Lemma \ref{K-theory.5}, 
we deduce that $\omega^*\Omega_{S^1}^\infty \Kth$ is strict Nisnevich local since $\Kth$ is Nisnevich local.
In conclusion, $\omega^*\Omega_{S^1}^\infty \Kth$ is $(sNis,\boxx\cup \Adm)$-local.
This gives the desired equivalence.
\end{proof}

The following is the most important result in this subsection, 
which shows that the unstable $K$-theory spectrum $\omega^*\Omega_{S^1}^\infty \Kth$ admits a geometric model 
in our logarithmic setting.

\begin{thm}
\label{K-theory.2}
Suppose $S$ is a regular scheme.
Then there is a canonical equivalence
\begin{equation}
\label{K-theory.2.9}
\Z\times \Gr
\to
\omega^*\Omega_{S^1}^\infty \Kth
\end{equation}
in $\inflogHptAdm(S)$.
\end{thm}
\begin{proof}
Due to Proposition \ref{K-theory.1}, there is a $\boxx\cup\Adm$-local  equivalence 
\begin{equation}
\label{K-theory.2.1}
\Z\times \Gr
\simeq
\omega^*(\Z\times \Gr).
\end{equation}
Apply Proposition \ref{Singboxx.4} to have a $\boxx\cup\Adm$-local  equivalence 
\begin{equation}
\label{K-theory.2.2}
\omega^*(\Z\times \Gr)
\simeq
\Sing^{\boxx\cup\Adm} \omega^*(\Z \times  \Gr).
\end{equation}
Furthermore, 
Lemma \ref{K-theory.3} gives an equivalence
\begin{equation}
\label{K-theory.2.3}
\Sing^{\boxx\cup\Adm} \omega^*(\Z \times  \Gr)
\cong
\omega^* \Sing^{\A^1}(\Z \times \Gr).
\end{equation}
Next we combine \eqref{K-theory.2.1}, \eqref{K-theory.2.2}, and \eqref{K-theory.2.3} 
to deduce the $(sNis,\boxx\cup\Adm)$-local  equivalence
\begin{equation}
\label{K-theory.2.10}
\Z\times \Gr
\simeq
L_{sNis}\omega^*\Sing^{\A^1}(\Z\times \Gr).
\end{equation}

There is a canonical morphism of  presheaves
\begin{equation}
\label{K-theory.2.8}
\omega^* \Sing^{\A^1}(\Z\times \Gr)
\to
\omega^* L_{Nis}\Sing^{\A^1}(\Z\times \Gr).
\end{equation}
If $x$ is a point of $X\in \cSm/S$,
then there exists an open affine neighborhood $U$ of $x$ in $X$ such that the restriction $\partial U$ 
of the Cartier divisor $\partial X$ to $U$ is principal.
We note that $U-\partial U$ is affine.
By Proposition \ref{K-theory.15}, the morphism of Kan complexes
\[
\omega^* \Sing^{\A^1}(\Z\times \Gr)(U)
\to
\omega^* L_{Nis}\Sing^{\A^1}(\Z\times \Gr)(U)
\]
obtained from \eqref{K-theory.2.8} becomes $\Omega_{S^1}^\infty \Kth(U-\partial U)\xrightarrow{\simeq} \Omega_{S^1}^\infty \Kth(U-\partial U)$.
This means that $L_{sNis}$ turns \eqref{K-theory.2.8} into a pointwise weak equivalence.
Together with \eqref{K-theory.15.2} and \eqref{K-theory.2.10}, we obtain the desired equivalence.
\end{proof}

\subsection{The logarithmic \texorpdfstring{$K$}{K}-theory spectrum}

Next we explain how to construct a $(2,1)$-periodic log motivic spectrum from a 
pointed log space satisfying the following condition.

\begin{df}
\label{K-theory.8}
A pointed log motivic space $\cF\in \inflogHpt(S)$ satisfies 
the \emph{loops $\P^1$-bundle formula} \index{loops P1-bundle formula @ loops $\P^1$-bundle formula} 
if there exists an equivalence
\[
\alpha_{\cF}\colon \cF \to \Omega_{\P^1}\cF
\]
in $\inflogHpt(S)$.
\end{df}

\begin{const}
\label{K-theory.9}
If a pointed log motivic space $\cF\in \inflogHpt(S)$ satisfies the loops $\P^1$-bundle formula, 
then by adjunction we have a morphism
\[
\beta_{\cF}\colon \Sigma_{\P^1} \cF \to \cF.
\]
We use $\beta_{\cF}$ as the bonding maps in the log $\P^1$-spectrum 
\[
\Bott(\cF):=(\cF,\cF,\ldots)\in \inflogSH(S).
\]
This is called \emph{the logarithmic motivic Bott spectrum of 
$\cF$}.\index{logarithmic motivic Bott spectrum} \index[notation]{Bott @ $\Bott$}
By construction, 
there are canonical equivalences in $\inflogSH(S)$
\begin{equation}
\label{K-theory.9.1}
\Bott(\cF) \simeq  \Sigma_{\P^1} \Bott(\cF)
\end{equation}
and
\begin{equation}
\label{K-theory.9.2}
\Omega_{\Gmm}^\infty \Bott(\cF)
\simeq
(\cF,\Omega_{\Gmm}\cF,\Omega_{\Gmm}^2\cF,\ldots).
\end{equation}
The bonding maps in \eqref{K-theory.9.2} are induced by $\cF\to \Omega_{\P^1}\cF\simeq \Omega_{S^1}\Omega_{\Gmm}\cF$.
\end{const}

\begin{const}
\label{K-theory.19}
If a pointed log motivic space $\cF\in \inflogHpt(S)$ satisfies the loops $\P^1$-bundle formula,
then we can view $\Bott(\cF)$ as a $(\P^1)^{\wedge 2}$-spectrum
\begin{equation}
\label{K-theory.19.1}
(\cF,\cF,\ldots)
\end{equation}
with bonding maps given by the composite $\Sigma_{\P^1}^2\cF\to \Sigma_{\P^1}\cF\to \cF$.
We can also view $\Bott(\cF)\wedge \Bott(\cF)$ as a $(\P^1)^{\wedge 2}$-spectrum
\begin{equation}
\label{K-theory.19.2}
(\cF\wedge \cF,\cF\wedge \cF,\ldots),
\end{equation}
where the bonding maps are given by
\[
\cF\wedge \cF\wedge \P^1\wedge \P^1
\xrightarrow{\simeq}
\cF\wedge \P^1\wedge \cF\wedge \P^1
\xrightarrow{\beta_{\cF}\wedge \beta_{\cF}}
\cF\wedge \cF.
\]

Suppose $\cF$ is a homotopy commutative monoid in $\inflogHpt(S)$ with multiplication 
$\mu\colon \cF\wedge \cF\to \cF$ such that the diagram
\begin{equation}
\label{K-theory.19.3}
\begin{tikzcd}
\cF\wedge \cF\wedge \P^1\wedge \P^1\ar[r,"\simeq"]\ar[d,"\mu"']&
\cF\wedge \P^1\wedge \cF\wedge \P^1\ar[r,"\beta_{\cF}\wedge \beta_{\cF}"]&
\cF\wedge \cF\ar[d,"\mu"]
\\
\cF\wedge \P^1\wedge \P^1\ar[rr,"\beta_{\cF}^2"]&
&
\cF
\end{tikzcd}
\end{equation}
commutes. 
Then, since the bonding maps are compatible, $\mu$ induces a map
\[
\Bott(\cF)\wedge \Bott(\cF)\to \Bott(\cF)
\]
Furthermore, 
the assumptions on $\cF$ ensures $\Bott(\cF)$ is a homotopy commutative monoid in $\inflogSH(S)$.

The composite of the two maps in the upper horizontal row of \eqref{K-theory.19.3} is
equivalent to the composite 
\begin{equation}
\label{K-theory.19.4}
\cF\wedge \cF\wedge \P^1\wedge \P^1
\xrightarrow{\alpha_{\cF}\wedge \alpha_{\cF}}
\Omega_{\P^1}\cF\wedge \Omega_{\P^1}\cF\wedge \P^1\wedge \P^1
\xrightarrow{\simeq}
\Omega_{\P^1}\cF\wedge \P^1\wedge \Omega_{\P^1}\cF\wedge \P^1
\xrightarrow{\gamma\wedge \gamma}
\cF\wedge \cF,
\end{equation}
where $\gamma\colon \Omega_{\P^1}\cF\wedge \P^1\to \cF$ is the canonical map.
The adjoint of the composite of the two right-most morphisms in \eqref{K-theory.19.4} takes the form
\[
\nu
\colon
\Omega_{\P^1}\cF\wedge \Omega_{\P^1}\cF
\to
\Omega_{(\P^1)^{\wedge 2}}(\cF\wedge \cF).
\]
To check that \eqref{K-theory.19.3} commutes, it suffices to check that the diagram
\begin{equation}
\label{K-theory.19.5}
\begin{tikzcd}
\cF\wedge \cF\ar[d,"\mu"']\ar[r,"\alpha_{\cF}\wedge \alpha_{\cF}"]&
\Omega_{\P^1}\cF\wedge \Omega_{\P^1}\cF \ar[r,"\nu"]&
\Omega_{(\P^1)^{\wedge 2}}(\cF\wedge \cF)\ar[d,"\mu"]
\\
\cF\ar[rr,"\alpha_{\cF}^2"]&
&
\Omega_{(\P^1)^{\wedge 2}}(\cF)
\end{tikzcd}
\end{equation}
commutes by adjunction.
\end{const}

\begin{rmk}
The above discussions on the loops $\P^1$-bundle formula and logarithmic motivic Bott spectrum can be adapted in the setting of $\inflogHpt^\Adm(S)$ and $\inflogSH^\Adm(S)$.
\end{rmk}

\begin{prop}
\label{K-theory.10}
Suppose $S$ is a regular scheme. 
Then $\omega^*\Omega_{S^1}^\infty \Kth \in \inflogHpt^\Adm(S)$ satisfies the loops $\P^1$-bundle formula.
A similar result holds for $\omega^*\Omega_{S^1}^\infty \Kth \in \inflogHpt(S)$ too.
\end{prop}
\begin{proof}
We focus  on the case of $\omega^*\Omega_{S^1}^\infty \Kth \in \inflogHpt^\Adm(S)$ since the proofs are similar.
By the projective bundle formula for $K$-theory of schemes, we have an equivalence of presheaves of spaces
\begin{equation}
\label{K-theory.10.1}
\Omega_{S^1}^\infty \Kth
\simeq
\Omega_{\P^1}\Omega_{S^1}^\infty \Kth.
\end{equation}
For every $X\in \lSm/S$, we have a canonical isomorphism of schemes
\[
(X\times \P^1)-\partial(X\times \P^1)
\cong
(X-\partial X)\times \P^1.
\]
This implies that there is a natural equivalence
\begin{equation}
\label{K-theory.10.2}
\omega^*\Omega_{\P^1}
\simeq
\Omega_{\P^1}\omega^*.
\end{equation}
Apply $\omega^*$ to \eqref{K-theory.10.1}, and use \eqref{K-theory.10.2} to obtain the desired equivalence.
\end{proof}

\begin{df}
\label{K-theory.16}
Let $S$ be a regular scheme. 
Proposition \ref{K-theory.10} shows that 
$\omega^*\Omega_{S^1}^\infty \Kth \in \inflogHpt^\Adm(S)$ satisfies the loops $\P^1$-bundle formula.
With reference to Construction \ref{K-theory.9}, 
we define the \emph{logarithmic $K$-theory spectrum of $S$} \index{logarithmic $K$-theory spectrum} to be the 
Bott periodic $\P^1$-spectrum
\begin{equation}
\label{K-theory.16.1}
\logKGL:=\Bott(\omega^*\Omega_{S^1}^\infty \Kth)=(\omega^*\Omega_{S^1}^\infty \Kth,\omega^*\Omega_{S^1}^\infty 
\Kth,\ldots)\in \inflogSH^\Adm(S).
\end{equation}
\index[notation]{logKGL @ $\logKGL$}
By Theorem \ref{K-theory.2},
we have an equivalence in $\inflogSH^{\Adm}(S)$
\begin{equation}
\label{K-theory.16.2}
\logKGL
\simeq
(\bbL_{sNis,\boxx\cup \Adm}\Z\times \Gr,\bbL_{sNis,\boxx\cup \Adm}\Z\times \Gr,\ldots).
\end{equation}

We also have the logarithmic $K$-theory spectrum
\begin{equation}
\label{K-theory.16.3}
\logKGL
:=
\Bott(\omega^*\Omega_{S^1}^\infty \Kth)\in \inflogSH(S).
\end{equation}

Using the functor $\psi^*\colon \inflogSH^\Adm(S)\to \inflogSH(S)$ in \ref{compth.14},
we can relate \eqref{K-theory.16.1} and \eqref{K-theory.16.3} by the equivalence
\begin{equation}
\psi^*\logKGL
\simeq
\logKGL
\end{equation}
since the triangle \eqref{compth.14.1} commutes.
\end{df}

\begin{thm}
\label{K-theory.7}
Suppose $S$ is a regular scheme.
For all $X\in \cSm/S$ and integer $d\in\mathbb{Z}$, there are canonical equivalences of spectra
\begin{equation}
\label{K-theory.7.3}
\Kth(X-\partial X)
\simeq
\map_{{\inflogSH}^{\eff, \Adm}(S)}(\Sigma_{S^1}^\infty X_+,\Omega_{\Gmm}^\infty \logKGL)
\end{equation}
and
\begin{equation}
\label{K-theory.7.4}
\Kth(X-\partial X)
\simeq
\map_{\inflogSH^\Adm(S)}(\Sigma_{\P^1}^d \Sigma_{\P^1}^\infty X_+,\logKGL).
\end{equation}
Similar results hold for $\inflogSHS(S)$ and $\inflogSH(S)$ too.
\end{thm}
\begin{proof}
We focus on the $\Adm$-cases ${\inflogSH}^{\eff,\Adm}(S)$ and ${\inflogSH}^{\Adm}(S)$, since the proofs without $\Adm$-localization are similar.
Due to \eqref{K-theory.9.1}, we may assume $d=0$.
Then, by adjunction, \eqref{K-theory.7.4} is a consequence of \eqref{K-theory.7.3}.

By \eqref{K-theory.9.2}, there is an equivalence in ${\inflogSH}^{\eff,\Adm}(S)$
\begin{equation}
\label{K-theory.7.5}
\Omega_{\Gmm}^\infty \logKGL
\simeq
(\omega^*\Omega_{S^1}^\infty \Kth,\Omega_{\Gmm}\omega^*\Omega_{S^1}^\infty \Kth,\ldots).
\end{equation}
This gives an equivalence of spectra
\begin{equation}
\label{K-theory.7.1}
\begin{split}
& \map_{{\inflogSH}^{\eff,\Adm}(S)}(\Sigma_{S^1}^\infty X_+,\Omega_{\Gmm}^\infty\logKGL)
\\
\simeq &
(\Omega_{S^1}^\infty \Kth(X-\partial X),\Omega_{\mathbb{G}_m}\Omega_{S^1}^\infty \Kth(X-\partial X),\ldots).
\end{split}
\end{equation}
Compare with \eqref{K-theory.14.2} to conclude.
\end{proof}

\begin{prop}
\label{K-theory.20}
The motivic space $\Z\times \Gr$ is a homotopy commutative monoid in $\inflogHpt^\Adm(S)$. If $S$ is a regular scheme,
 $\logKGL$ is a homotopy commutative monoid in $\inflogSH^\Adm(S)$. 
\end{prop}
\begin{proof}
The morphism $\Gr(r,n)\times \Gr(s,l)\to \Gr(r+s,n+l)$ sending any pair $(V\in \cO^n,W\in \cO^l)$ to $V\oplus W$ naturally induces a morphism $\mu\colon (\Z\times \Gr)\wedge (\Z\times \Gr)\to \Z\times \Gr$.
Together with Theorem \ref{K-theory.2}, we have a map $\mu'\colon (\Z\times \Gr)\wedge (\Z \times \Gr)\to \omega^*\Omega_{S^1}^\infty \Kth$.
The square
\[
\begin{tikzcd}
(\Z\times \Gr)\wedge (\Z\times \Gr)\wedge (\Z\times \Gr)\ar[d,"\mu\wedge \id"']\ar[r,"\id\wedge \mu"]&
(\Z\times \Gr)\wedge (\Z\times \Gr)\ar[d,"\mu"]
\\
(\Z\times \Gr)\wedge (\Z\times \Gr)\ar[r,"\mu"]&
\omega^*\Omega_{S^1}^\infty \Kth
\end{tikzcd}
\]
commutes since the corresponding square removing $\omega^*$ in the lower right corner commutes in $\infHpt(k)$.
This shows that $\mu$ satisfies the associative axiom.
One can similarly show that $\mu$ satisfies the unital and commutative axioms, and hence $\Z\times \Gr$ is a homotopy commutative monoid in $\inflogHpt^\Adm(S)$.

It remains to check that \eqref{K-theory.19.3} commutes for $\omega^*\Omega_{S^1}^\infty \Kth$.
In other words, we need to show that the diagram
\[
\begin{tikzcd}[column sep=tiny]
(\Z\times \Gr)\wedge (\Z\times \Gr)\wedge \P^1\wedge \P^1\ar[r]\ar[d]&
(\Z\times \Gr)\wedge \P^1\wedge (\Z\times \Gr)\wedge \P^1\ar[r]&
(\Z\times \Gr)\wedge (\Z\times \Gr)\ar[d]
\\
(\Z\times \Gr)\wedge \P^1\wedge \P^1\ar[rr]&
&
\omega^*\Omega_{S^1}^\infty \Kth
\end{tikzcd}
\]
commutes.
This follows from the commutativity of the corresponding square in $\infHpt(k)$, where $\omega^*$ in the lower right corner is removed.
\end{proof}

\begin{df}
\label{K-theory.11}
Suppose $S$ is a regular scheme and $f\colon X\to S$ is a morphism in $\lSch$.
The \emph{logarithmic $K$-theory of $S$} \index{logarithmic $K$-theory} \index[notation]{Klog @ $\Kthlog$} 
is the mapping spectrum
\[
\Kthlog(X):=\map_{\inflogSH(X)}(\Sigma_{\P^1}^\infty X_+,f^*\logKGL).
\]
For every integer $n\in\mathbb{Z}$, we set
\[
\Kthlog_n(X):=\pi_n\Kthlog(X).
\]
\end{df}

If $f$ is log smooth, 
then we have equivalences of spectra
\begin{equation}
\label{K-theory.11.1}
\Kthlog(X)
\simeq
\map_{\inflogSH(S)}(\Sigma_{\P^1}^\infty X_+,\logKGL)
\simeq
\Kth(X-\partial X),
\end{equation}
where the right equivalence is due to 
Theorem \ref{K-theory.7}.
In particular, if $S\in \Sm/k$, then we have an equivalence of spectra
\[
\Kthlog(S)
\simeq
\Kth(S).
\]
One may ask if this is true in a more general situation.

\begin{conj}
\label{K-theory.12}
For every $S\in \Sch$ there exists an equivalence of spectra
\[
\Kthlog(S)
\simeq
\Kth(S).
\]
\end{conj}

\begin{rmk}
There exist examples of non-regular schemes $S$ and $X\in \lSm/S$ such that
\[
\Kthlog(X) \not \simeq \Kth(X-\partial X).
\]
Indeed, 
the left-hand side is $\boxx$-invariant, but the right-hand side is not.
\end{rmk}

\begin{rmk}
\label{K-theory.13}
Niziol defined $K$-groups $\KthNiz_n(X_{\tau})$ for various topologies $\tau$ 
and fs log schemes $X$ in \cite{MR2452875}.
Some of these $K$-groups are computed when $X$ is regular and $\partial X$ 
is a strict normal crossing divisor in \cite[Corollary 4.17]{MR2452875}.
In particular, if $k$ is a field, 
then for every integer $n\geq 0$ there are isomorphisms of abelian groups
\[
\KthNiz_n(\boxx_{k\et}) 
\cong
\Kth_n(\P^1)\oplus \Kth(k)\otimes \Z^{(\Q/\Z)'-0}
\text{ and }
\KthNiz_n(\boxx_{kfl})
\cong
\Kth_n(\P^1)\oplus \Kth(k)\otimes \Z^{(\Q/\Z)-0},
\]
where, 
if $k$ has characteristic $p$, we set 
$$(\Q/\Z)':=\oplus_{\ell\neq p} \Q_\ell/\Z_\ell.$$ 
Hence the groups $\KthNiz_n(\boxx_{k\et})$ and $\KthNiz_n(\boxx_{kfl})$ are non-isomorphic to $\Kth_n(k)$.

On the other hand, 
suppose $X$ is an fs log scheme log smooth over a regular scheme.
Due to \eqref{K-theory.11.1} and $\A^1$-invariance of $K$-theory for regular schemes, 
there are isomorphisms of abelian groups
\[
\Kthlog_n(X\times \boxx)\cong \Kth_n(X\times \A^1)\cong \Kth_n(X).
\]
Hence our logarithmic $K$-group $\Kth_n^{log}(X)$ is non-isomorphic to Niziol's $K$-groups 
$\KthNiz_n(X_{k\et})$ and $\KthNiz_n(X_{kfl})$.
\end{rmk}

\newpage

\section{Oriented cohomology theories}
\label{section:oct}

Orientations in $\A^1$-homotopy theory can be used to give a systematic proof of the 
projective bundle formula for 
various cohomology theories including homotopy algebraic $K$-theory and motivic cohomology.
Furthermore, one can associate Chern and Thom classes as in the classical intersection theory.
We refer to \cite{zbMATH02216798}, \cite{zbMATH05367298}, and \cite{Deg08} for these developments.

The purpose of this section is to
prove various analogous results in our logarithmic setting, without $\mathbb{A}^1$-invariance. As discussed in the introduction, the scope of applications of these results goes beyond the logarithmic world, since it gives automatically results for virtually all known non-$\mathbb{A}^1$-homotopy invariant cohomology theories (for example: crystalline, Hodge, prismatic and syntomic cohomology). In the presence of an orientation, the Thom space term in the Gysin sequence of Remark \ref{rmk:Gysin} can be trivialized in light of the Thom isomorphism theorem \ref{Ori.16} below.

In this section, we also present a complete computation of the cohomology the Grassmannians $\Gr(r,n)$: our proof uses in essential way the $(\P^\bullet, \P^{\bullet-1})$-invariance of any cohomology theory representable in $\inflogSH$, and once again does not follow from the analogous computation in the $\mathbb{A}^1$-setting. 

\

We begin by defining precisely what we mean by cohomology theory associated with a homotopy commutative monoid $\ringE$.

\begin{df}
\label{Ori.12}
Suppose $S\in \lSch$ and let $\ringE$ be a homotopy commutative monoid in $\inflogSH(S)$.
For integers $p$ and $q$, 
we define the \emph{$\ringE$-cohomology}\index{E-cohomology @ $\ringE$-cohomology} 
of $X\in \lSm/S$ by\index[notation]{Epq @ $\ringE^{p,q}$}
\begin{equation}
\label{Ori.12.1}
\ringE^{p,q}(X)
:=
\hom_{\inflogSH(S)}(\Sigma_T^\infty X_+,\Sigma^{p,q}\ringE).
\end{equation}
For any morphism $Y\to X$ in $\lSm/S$, we define
\[
\ringE^{p,q}(X/Y)
:=
\hom_{\inflogSH(S)}(\Sigma_T^{\infty} X/Y,\Sigma^{p,q}\ringE).
\]
We set $\ringE^{p,q}:=\ringE^{p,q}(\pt)$ and $\ringE^{**}(X):=\bigoplus_{p,q\in \Z}\ringE^{p,q}(X)$ for abbreviation.
\end{df}
\begin{rmk}
For $X,X'\in \lSm/S$ and integers $p$, $q$, $r$, and $s$, any two maps $u\colon \Sigma_T^\infty X_+\to \Sigma^{p,q}\ringE$ and $v\colon \Sigma_T^\infty X_+'\to \Sigma^{r,s}\ringE$
give
\[
\Sigma_T^\infty (X\times X')_+
\xrightarrow{\simeq}
\Sigma_T^\infty X_+ \wedge \Sigma_T^\infty X_+'
\xrightarrow{u\wedge v}
\Sigma^{p,q}\ringE \wedge \Sigma^{r,s}\ringE
\to
\Sigma^{p+r,q+s}\ringE.
\]
From this, we have a map
\begin{equation}
\label{Ori.12.2}
\ringE^{p,q}(X)\times \ringE^{r,s}(X')
\to
\ringE^{p+r,q+s}(X\times X').
\end{equation}
When a morphism $\alpha\colon X''\to X\times X'$ is given, compose \eqref{Ori.12.2} with $\alpha^*$ to obtain the \emph{cup product} \index{cup product}
\begin{equation}
\label{Ori.12.3}
\smile\colon \ringE^{p,q}(X)\times \ringE^{r,s}(X')\to \ringE^{p+r,q+s}(X'').
\end{equation}
We are particularly interested in the case when $X=X'=X''$ and $\alpha$ is the diagonal morphism.
We similarly have the cup product for $\ringE^{**}(X/Y)$ for every morphism $Y\to X$ in $\lSm/S$ by considering the diagonal map $X/Y\to X/Y \wedge X/Y$.

This cup product satisfies the axioms of rings since $\ringE$ satisfies similar axioms, so $\ringE^{**}(X)$
is a graded ring.
Let $\unit$ denote the unit in these rings.
Since \eqref{Ori.12.2} is functorial in $X$, there is a naturally induced graded ring homomorphism
\[
f^*\colon \ringE^{**}(X)\to \ringE^{**}(Y)
\]
for every morphism $f\colon Y\to X$ in $\inflogH(S)$.
We have a similar graded abelian group homomorphism for a map $X'/Y'\to X/Y$ in $\inflogHpt(S)$.
\end{rmk}

\subsection{Orientations}
\label{ori}
Throughout this subsection, we fix $S\in \Sch$ and a homotopy commutative monoid $\ringE$ in $\inflogSH(S)$.

\begin{df}
\label{Ori.7}
There are closed immersions
\[
\P^1\to \P^2\to \cdots \to \P^n \to \cdots,
\]
where $\P^n\to \P^{n+1}$ maps $[x_0:\cdots:x_n]$ to $[x_0:\cdots:x_n:0]$.
Furthermore, we can let the distinguished points of $\P^n$ and $\P^\infty$ be $1\in \P^1$.
A \emph{Chern orientation}\index{orientation} (or simply \emph{orientation}) of $\ringE$ is a class $c_\infty \in \ringE^{2,1}(\P^\infty/\pt)$ whose restriction to $\P^1/\pt$ is the class of $\ringE^{2,1}(\P^1/\pt)$ given by
\[
(\P^1/\pt)\wedge \unit
\colon
\Sigma_T^\infty \P^1/\pt \to (\P^1/\pt) \wedge \ringE=\Sigma^{2,1}\ringE.
\]
For abbreviation, we set $\Sigma^{2,1}\unit := (\P^1/\pt)\wedge \unit$.
\end{df}

\begin{rmk}
One may also define a Chern orientation of $\ringE$ as a class $c_\infty$ whose restriction to $\P^1/\pt$ is $-\Sigma^{2,1}\unit$, see \cite[Definition 1.2]{zbMATH05367298} for example.
There is a one-to-one correspondence between this definition and ours, which sends $c_\infty$ to $-c_\infty$.
\end{rmk}

\begin{df}
A Chern orientation $c_\infty$ can be expressed as a morphism $\Sigma_T^\infty \P^\infty/\pt \to \Sigma^{2,1}\ringE$ in $\inflogSH(S)$.
Hence for every $X\in \lSm/S$, there are canonical maps
\begin{equation}
\label{Ori.7.2}
\begin{split}
\hom_{\inflogSHS(S)}(\Sigma_{S^1}^\infty X_+,\Sigma_{S^1}^\infty \P^\infty/\pt)
&\to
\hom_{\inflogSH(S)}(\Sigma_T^\infty X_+,\Sigma_T^\infty \P^\infty/\pt)
\\
&\to
\hom_{\inflogSH(S)}(\Sigma_T^\infty X_+,\Sigma^{2,1}\ringE)
=
\ringE^{2,1}(X).
\end{split}
\end{equation}
We have maps
\begin{equation}
\label{Ori.7.3}
\begin{split}
\logPic(X)
=
H_{sNis}^1(X,\Gmm)
\to &
\hom_{\inflogSHS(S)}(\Sigma_{S^1}^\infty X_+,\Sigma_{S^1}^\infty \clspace\Gmm/\pt)
\\
\cong &
\hom_{\inflogSHS(S)}(\Sigma_{S^1}^\infty X_+,\Sigma_{S^1}^\infty \P^\infty/\pt),
\end{split}
\end{equation}
where the last isomorphism is due to Theorem \ref{Pic.9}.
Combine with \eqref{Ori.7.2} to have a canonical map
\[
c_1\colon \logPic(X)\to \ringE^{2,1}(X),
\]
which we call the \emph{first Chern class}. \index{first Chern class} \index[notation]{c1 @ $c_1(\cE)$}
There is a canonical homomorphism $\Pic(\ul{X})=\logPic(\ul{X})\to \logPic(X)$.
By composition we obtain the canonical map
\[
c_1\colon \Pic(\ul{X})\to \ringE^{2,1}(X).
\]
In particular, since every line bundle $\cL$ determines a class of $\Pic(\ul{X})$, we can define $c_1(\cL)\in \ringE^{2,1}(X)$.

The map \eqref{Ori.7.3} sends the class $\cO(-1)\in\Pic(\P^1)$ to the class
\[
\alpha\in \hom_{\inflogSHS(S)}(\Sigma_{S^1}^\infty (\P^1)_+,\Sigma_{S^1}^\infty \P^\infty/\pt)
\]
induced by the inclusion $\P^1\to \P^\infty$.
The map \eqref{Ori.7.2} sends the class $\beta$ to the restriction of $c_\infty$ to $\P^1/\pt$.
Hence we have the formula
\begin{equation}
\label{Ori.7.11}
c_1(\cO(-1))=\Sigma^{2,1}\unit.
\end{equation}

Let $\cT_1$ be the tautological bundle on $\P^\infty$.
The map \eqref{Ori.7.3} sends the class $\cT_1\in \Pic(\P^\infty)$ to the class
\[
\beta\in \hom_{\inflogSHS(S)}(\Sigma_{S^1}^\infty (\P^\infty)_+,\Sigma_{S^1}^\infty \P^\infty/\pt)
\]
induced by the quotient map $(\P^\infty)_+\to \P^\infty/\pt$.
The map \eqref{Ori.7.2} sends the class $\alpha$ to $c_\infty$.
Hence we have the formula
\begin{equation}
\label{Ori.7.12}
c_1(\cT_1)=c_\infty.
\end{equation}
\end{df}

If $f\colon Y\to X$ is a morphism in $\lSm/S$, then there is a commutative diagram
\[
\begin{tikzcd}
\logPic(X)\ar[d,"f^*"']\ar[r,"c_1"]&
\ringE^{2,1}(X)\ar[d,"f^*"]
\\
\logPic(Y)\ar[r,"c_1"]&
\ringE^{2,1}(Y).
\end{tikzcd}
\]
In particular, for every line bundle $\cL$ over $X$, we have the formula
\begin{equation}
\label{Ori.7.4}
f^*(c_1(\cL))=c_1(f^*(\cL)).
\end{equation}

\begin{lem}
\label{Ori.48}
Let $\cT_{1,n}$ be the tautological bundle on $\P^n$.
Then $\P(\cT_{1,n}\oplus \cO)$ is isomorphic to the blow-up of $\P^{n+1}$ along a point, and $\P(\cT_{1,n})$ is the exceptional divisor on $\P(\cT_{1,n}\oplus \cO)$.
\end{lem}
\begin{proof}
The scheme $\P(\cT_{1,n}\oplus \cO)$ admits a Zariski cover consisting of the spectra of 
\[
\Z[x_0/x_i,\ldots,x_n/x_i,x_i]
\text{ and }
\Z[x_0/x_i,\ldots,x_n/x_i,1/x_i]
\]
for $0\leq i\leq n$.
The divisor $\P(\cT_{1,n})$ is given locally by the divisor $x_0/x_i=0$ for $1\leq i\leq n$.
Let $O$ be the point $[0:\cdots:0:1]$ in $\P^{n+1}$.
The scheme $\Blow_O \P^{n+1}$ admits a Zariski cover consisting of the spectra of
\[
\Z[x_1/x_i,\ldots,x_{n+1}/x_i,x_i]
\text{ and }
\Z[x_1/x_i,\ldots,x_{n+1}/x_i,1/x_i]
\]
for $1\leq i\leq n+1$.
By identifying $x_{n+1}$ with $x_0$, we have the required claim.
\end{proof}

\begin{lem}
\label{Ori.49}
Let $\cT_{1}$ be the tautological bundle on $\P^\infty$, and let $z\colon \P^\infty\to \cT_1$ be the zero section.
Then there is an exact sequence
\begin{equation}
\label{Ori.49.1}
0
\to
\ringE^{**}(\Thom(\cT_1))
\xrightarrow{z^*}
\ringE^{**}(\P^\infty)
\to
\ringE^{**}(\pt)
\to
0.
\end{equation}
\end{lem}
\begin{proof}
Let $\cT_{1,n}$ be the tautological bundle on $\P^n$.
From Theorem \ref{ProplogSH.5} and Lemma \ref{Ori.48}, we have an equivalence
\[
\P(\cT_{1,n}\oplus \cO)/\P(\cT_{1,n})
\xrightarrow{\simeq}
\P^{n+1}/\pt
\]
in $\inflogSH(S)$.
If we compose with the zero section $(\P^n)_+\to \P(\cT_{1,n}\oplus \cO)/\P(\cT_{1,n})$, then $(\P^n)_+\to \P^{n+1}/\pt$ factors through $(\P^{n+1})_+$.
Take colimits to have morphisms
\[
(\P^\infty)_+ \to  \P(\cT_{1}\oplus \cO)/\P(\cT_{1}) \xrightarrow{\simeq} \P^\infty/\pt
\]
whose composite morphism is homotopic to the quotient map in $\inflogSH(S)$.
Proposition \ref{Ori.41} finishes the proof.
\end{proof}

\begin{lem}
\label{Ori.6}
Let $\ol{i}_1,\ldots,\ol{i}_n\colon \P^1\to \P^n$ be the closed immersions sending $[a_0:a_1]$ to $[a_0:0:\cdots :0:a_1]$, $\ldots$, $[0:\cdots:0:a_0:a_1]$.
Then $\ol{i}_1,\ldots,\ol{i}_n$ are all homotopic in $\inflogSHS(S)$.
\end{lem}
\begin{proof}
Without loss of generality it suffices to show that $\ol{i}_1$ and $\ol{i}_2$ are homotopic in $\inflogSHS(S)$.
Both $\ol{i}_1$ and $\ol{i}_2$ factors through the closed immersion $\P^2\to \P^n$ sending $[a_0:a_1:a_2]$ to $[a_0:a_1:0:\cdots:0:a_2]$, so we only need to consider the case $n=2$.
Let $\ol{i}\colon \P^1\to \P^2$ be the closed immersion sending $[a_0:a_1]$ to $[a_0:a_0:a_1]$.
We will show that $\ol{i}_1$ and $\ol{i}$ are homotopic in $\inflogSHS(S)$.

\

Consider the rational map $h\colon \P^1\times \P^1 \dashrightarrow \P^2$ given by
\[
([a_0:a_1],[t_0:t_1])
\mapsto
[a_0t_1:a_0t_0:a_1t_1].
\]
This is a morphism on $\P^1\times \P^1-(0,\infty)$. With $a_1=1$ and $t_0=1$, $h$ is given by $([a_0:1],[1,t_1])\mapsto [a_0t_1:a_0:t_1]$.
Using this description,
we see that $h$ yields a morphism
\[
g\colon \Blow_{(0,\infty)}(\P^1\times \P^1)\to \P^2.
\]

Let $X$ be the fs log scheme whose underlying scheme is $\Blow_{0\times \infty}(\P^1\times \P^1)$ and whose log structure is the compactifying log structure associated with the open immersion $\P^1\times \A^1\to \Blow_{0\times \infty}(\P^1\times \P^1)$.
From $g$ and $h$, we can form a morphism
\[
f\colon X\to \P^2.
\]
There is also an admissible blow-up $p\colon X\to \P^1\times \boxx$ along a smooth center.
The zero section and one section $\P^1\rightrightarrows \P^1\times \boxx$ can be lifted to closed immersions $s_0,s_1\colon \P^1\rightrightarrows X$.
From Lemma \ref{ProplogSH.7}, we deduce $\ol{i}_1=\ol{i}_2$ in $\inflogSHS(S)$.

By permuting coordinates, we can deduce $\ol{i}_2=\ol{i}$ too.
Hence we have $\ol{i}_1=\ol{i}_2$.
\end{proof}

\begin{lem}
\label{Ori.44}
We set $X:=\Blow_{[0:0:1]}(\P^2)$.
Let $i_1,i_2\colon \P^1\rightrightarrows X$ be the strict transforms of the closed immersions $\P^1\rightrightarrows \P^2$ sending $[x_0:x_1]$ to $[0:x_0:x_1]$ and $[x_0:x_1:0]$.
Let $i_3\colon \P^1\to X$ be the closed immersion from the exceptional divisor.
If a class $a\in \ringE^{**}(X)$ satisfies $i_2^*a=0$, then we have the formula
\[
i_1^*a=-i_3^*a.
\]
\end{lem}
\begin{proof}
There is a commutative diagram
\begin{equation}
\label{Ori.44.3}
\begin{tikzcd}
\P^1\ar[d,"i_1"']\ar[r]&
\pt\ar[d]\ar[r]&
\P^1\ar[d,"i_2"]
\\
X\ar[r,"q"]&
\P^1\ar[r,"i_3"]&
X,
\end{tikzcd}
\end{equation}
where $q$ is
obtained by the rational map $\P^2\dashrightarrow \P^1$ given by $[x_0:x_1:x_2]\mapsto [x_0:x_1]$.
Let $p\colon X\to \P^2$ be the projection.
Since $qi_3=\id$, we have $i_3^*(a-q^*i_3^*a)=0$.
From Proposition \ref{ProplogSH.5}, we have
\begin{equation}
\label{Ori.44.1}
a-q^*i_3^*a=p^*b
\end{equation}
for some $b\in \ringE^{**}(\P^2)$.
Permute coordinates in Lemma \ref{Ori.6} to have $i_1^*p^*b=i_2^*p^*b$.
Together with \eqref{Ori.44.1}, we have the formula
\begin{equation}
\label{Ori.44.2}
i_1^*a-i_1^*q^*i_3^*a=i_2^*a-i_2^*q^*i_3^*a.
\end{equation}
Since $i_2^*a=0$, we have $i_1^*q^*i_3^*a=0$ from \eqref{Ori.44.3}.
Combined with \eqref{Ori.44.2}, we have
\[
i_1^*a=-i_2^*q^*i_3^*a.
\]
This implies the claimed formula since $qi_2=\id$.
\end{proof}

\begin{df}
Let $\cT_1$ be the tautological bundle on $\P^\infty$.
A \emph{Thom orientation} \index{Thom orientation} of $\ringE$ is a class $t_\infty \in \ringE^{2,1}(\Thom(\cT_1))$ such that the restriction of $t_\infty$ to the distinguished point of $\P^\infty$ is equal to the class of $\ringE^{2,1}(T)$ given by
\[
-T\wedge \unit
\colon
T \to T\wedge \ringE = \Sigma^{2,1}\ringE.
\]
\end{df}

From \eqref{Ori.49.1}, we have an isomorphism of graded rings
\begin{equation}
\label{Ori.49.2}
z^*
\colon
\ringE^{**}(\P^\infty/\pt)
\xrightarrow{\cong}
\ringE^{**}(\Thom(\cT_1)).
\end{equation}

\begin{prop}
\label{Ori.45}
Let $\cT_1$ be the tautological bundle on $\P^\infty$.
If $t_\infty\in \ringE^{**}(\Thom(\cT_1))$ is a class, then $t_\infty$ is a Thom orientation of $\ringE$ if and only if $z^*t_\infty\in \ringE^{**}(\P^\infty/\pt)$ is a Chern orientation of $\ringE$.
\end{prop}
\begin{proof}
Let $\cT_{1,1}$ be the tautological bundle on $\P^1$.
There is a commutative diagram with cartesian squares
\[
\begin{tikzcd}[column sep=small, row sep=small]
\pt\ar[d]\ar[r,"i_0"]&
\P^1\ar[d,"v"]\ar[r]&
\P^\infty\ar[d,"z"]
\\
\P^1\ar[d]\ar[r,"u"]&
\P(\cT_{1,1}\oplus \cO)\ar[d,"p"]\ar[r]&
\P(\cT_{1}\oplus \cO)\ar[d,"q"]
\\
\pt\ar[r,"i_0"]&
\P^1\ar[r]&
\P^\infty,
\end{tikzcd}
\]
where $i_0$ is the zero section, and $p$ and $q$ are the projections.
The restriction of $t_\infty$ to $\P^1$ gives a class $t_1\in \ringE^{2,1}(\Thom(\cT_{1,1}))$.
Let $\ol{t}_1$ be the image of $t_1$ in $\ringE^{2,1}(\P(\cT_{1,1}\oplus \cO))$, whose restriction to $\P(\cT_{1,1})$ is zero.
The blow-down of $\P(\cT_{1,1}\oplus \cO)$ along $\P(\cT_{1,1})$ is isomorphic to $\P^2$ due to Lemma \ref{Ori.48}.
Hence we can apply Lemma \ref{Ori.44} to our situation, and deduce the formula
\begin{equation}
\label{Ori.45.1}
u^*\ol{t}_1=-v^*\ol{t}_1.
\end{equation}
The class $t_\infty$ is a Thom orientation if and only if the restriction of $-v^*\ol{t}_1$ to $\P^1/\pt$ is equal to $\P^1/\pt\wedge \unit$.
The class $z^*t_\infty$ is a Chern orientation if and only if the restriction of $u^*\ol{t}_1$ to $\P^1/\pt$ is equal to $\P^1/\pt\wedge \unit$.
From \eqref{Ori.45.1}, we deduce that these are equivalent.
\end{proof}

\begin{df}
Let $t_\infty$ be a Thom orientation.
The above $z^*t_\infty$ is called the \emph{Chern orientation associated with $t_\infty$}.
\end{df}

\begin{prop}
\label{Ori.53}
There is a one-to-one correspondence between Chern orientations of $\ringE$ and Thom orientations of $\ringE$.
\end{prop}
\begin{proof}
Combine Lemma \ref{Ori.49} and Proposition \ref{Ori.45}.
\end{proof}

\subsection{Cohomology of projective spaces}
Throughout this subsection, we fix $S\in \Sch$ and a homotopy commutative monoid $\ringE$ in $\inflogSH(S)$.
In this subsection, we will present several results about projective spaces in $\inflogSHS(S)$.
Our approach culminates with a proof of the projective bundle formula.
The following theorem is a fundamental property of our motivic category.
\begin{prop}
\label{Ori.3}
In $\inflogSHS(S)$, there is an equivalence
\begin{equation}
\label{Ori.3.1}
\P^n/\P^{n-1}
\cong
(\P^1/\pt)^{\wedge n}.
\end{equation}
\end{prop}
\begin{proof}
In $\infSHS(k)$, there are equivalences
\[
\P^n/\P^{n-1}
\xleftarrow{\cong}
\P^n/(\P^n-O)
\xrightarrow{\cong}
(\P^1/\A^1)^{\wedge n}
\xleftarrow{\cong}
(\P^1/\pt)^{\wedge n},
\]
where $O:=[0:\cdots:0:1]\in \P^n$.
We need to ``compactify'' these.
We can replace $\A^1$ by $\boxx$ and $\P^n-O$ by $(\Blow_O(\P^n),E)$, where $E$ is the exceptional divisor.
Since $(\Blow_O(\P^n),E)$ is a $\boxx$-bundle on $\P^{n-1}$ by \cite[Lemma 7.1.7]{logDM}, the projection $(\Blow_O(\P^n),E)\to \P^{n-1}$ is an equivalence in $\inflogSHS(S)$.
Hence it remains to construct the second morphism in
\[
\P^n/\P^{n-1}
\xleftarrow{\cong}
\P^n/(\Blow_O(\P^n),E)
\to
(\P^1/\boxx)^{\wedge n}
\xleftarrow{\cong}
(\P^1/\pt)^{\wedge n}
\]
and to show that it is an equivalence in $\inflogSHS(S)$.

\

Let $\ol{H}_i$ be the divisor $x_i=0$ on $\P^n$ for $1\leq i\leq n$, 
and set
\[
\ol{\cV}
:=
\horn(\P^n,\ol{H}_1+\cdots+\ol{H}_n).
\]
Proposition \ref{Ori.37} gives an equivalence
\begin{equation}
\label{Ori.3.2}
(\Blow_O\P^n,E)
\xleftarrow{\simeq}
\colimit(\ol{\cV})
\end{equation}
in $\inflogSHS(S)$, which induces an equivalence
\begin{equation}
\label{Ori.3.3}
\P^n/(\Blow_O\P^n,E)
\xleftarrow{\simeq}
\P^n/\colimit(\ol{\cV})
\end{equation}
in $\inflogSHS(S)$.
From \eqref{Thomdf.10.1} and Proposition \ref{Thomdf.6}, we have an equivalence
\begin{equation}
\label{Ori.3.6}
\P^n/\colimit(\ol{\cV})
\xrightarrow{\simeq}
(\P^1/\boxx)^{\wedge n}
\end{equation}
in $\inflogSHS(S)$.
Combine \eqref{Ori.3.3} and \eqref{Ori.3.6} to finish the construction.
\end{proof}

Morel proved the following result in \cite{Morel_properties} for $\A^1$-homotopy theory.

\begin{lem}
\label{Ori.4}
There is a commutative diagram
\[
\begin{tikzcd}
\P^n/\pt\ar[r,"\Delta"]\ar[d]&
(\P^n/\pt)^{\wedge n}\ar[d,leftarrow,"\ol{i}_1\wedge \cdots \wedge \ol{i}_1"]
\\
\P^n/\P^{n-1}\ar[r]&
(\P^1/\pt)^{\wedge n}
\end{tikzcd}
\]
in $\inflogSHS(S)$, where the left vertical map is the obvious one, 
the lower horizontal map is \eqref{Ori.3.1}, $\Delta$ is the diagonal morphism, 
and $\ol{i}_1\colon \P^1\to \P^n$ is the closed immersion sending $[a_0:a_1]$ to $[a_0:0:\cdots:0:a_1]$.
\end{lem}
\begin{proof}
We keep using the same notation as in the proof of Proposition \ref{Ori.3}.
Let $\ol{i}_1,\ldots,\ol{i}_n\colon \P^1\to \P^n$ be the closed immersions in Lemma \ref{Ori.6}, which are all homotopic in $\inflogSHS(S)$.
It suffices to show that the diagram
\[
\begin{tikzcd}
\P^n/\pt\ar[r,"\Delta"]\ar[d]&
(\P^n/\pt)^{\wedge n}\ar[d,leftarrow,"\ol{i}_1\wedge \cdots \wedge \ol{i}_n"]
\\
\P^n/\P^{n-1}\ar[r]&
(\P^1/\pt)^{\wedge n},
\end{tikzcd}
\]
in $\inflogSHS(S)$ commutes.
For every $i$, there is an isomorphism of fs log schemes $\ol{V}_i:=(\P^n,\ol{H}_i)\cong (\P^n,\P^{n-1})$.
This means $\P^n/\pt\cong \P^n/\ol{V_i}$ in $\inflogSHS(S)$. 
Hence the obvious map
\[
(\P^n/\pt)^{\wedge n}\to (\P^n/\ol{V}_1)\wedge \cdots \wedge (\P^n/\ol{V}_n)
\]
is an equivalence in $\inflogSHS(S)$. 
Moreover, 
by strict Nisnevich descent, there is an equivalence
\[
\P^1/\boxx
\cong
\A^1/\A_\N
\]
in $\inflogSHS(S)$.
From the proof of Proposition \ref{Ori.3}, we see that the diagram
\[
\begin{tikzcd}
&
\P^n/\pt\ar[rd]\ar[d]\ar[ld]
\\
\P^n/\P^{n-1}\ar[r,"\cong",leftarrow]&
\P^n/(\Blow_O(\P^n),E)\ar[r,"\cong",leftarrow]&
\P^n/\colimit(\ol{\cV})
\end{tikzcd}
\]
in $\inflogSHS(S)$ commutes, where the maps from $\P^n/\pt$ are the obvious ones.
Hence it suffices to show that the diagram
\begin{equation}
\label{Ori.4.1}
\begin{tikzcd}
\P^n/\pt\ar[r,"\Delta"]\ar[d]&
(\P^n/\ol{V}_1)\wedge \cdots \wedge (\P^n/\ol{V}_n)
\\
\P^n/\colimit(\ol{\cV})\ar[r,"\cong"]&
(\P^1/\boxx)^{\wedge n}\ar[r,"\cong"]&
(\A^1/\A_\N)^{\wedge n}\ar[lu,"i_1\wedge \cdots \wedge i_n"']
\end{tikzcd}
\end{equation}
in $\inflogSHS(S)$ commutes, where $i_1,\ldots,i_n$ are the restrictions of $\ol{i}_1,\ldots,\ol{i}_n$.

\

Let $q\colon \A^n\to (\P^n)^n$ be the morphism given by 
$$
(a_1,\ldots,a_n)\mapsto ([a_1:0:\cdots:0:1],\ldots,[0:\cdots:0:a_n:1]).
$$
Form the cubical horn $\cZ$ from $(\P^n/\ol{V}_1)\wedge \cdots \wedge (\P^n/\ol{V}_n)$ using Proposition \ref{Thomdf.6}.
Then there is a commutative square
\[
\begin{tikzcd}
(\P^n)^n/\colimit(\cZ)\ar[d,"q"',leftarrow]\ar[r,"\cong"]&
(\P^n/\ol{V}_1)\wedge \cdots \wedge (\P^n/\ol{V}_n)\ar[d,leftarrow,"i_1\wedge \cdots \wedge i_n"]
\\
\A^n/\colimit(\cV)\ar[r,"\cong"]&
(\A^1/\A_\N)^{\wedge n}.
\end{tikzcd}
\]

\

There is a morphism
\[
\A^n \times \A^1 \to (\P^n)^n
\]
sending $(a_1,\ldots,a_n,t)$ to
\[
([a_1: ta_2: \ldots:ta_n:1],\ldots,[ta_1:\ldots:ta_{n-1}:a_n:1]).
\]
Let $\Delta$ be the fan associated with the toric variety $(\P^n)^n$.
This morphism is a morphism associated with a morphism of fans $\N^n\times \N^1\to \Delta$ whose morphism of lattices $h\colon \Z^{n}\times \Z\to (\Z^n)^r$ sends $(a_1,\ldots,a_n,t)$ to
\[
((a_1,t+a_2,\ldots,t+a_n),\ldots,(a_1+t,\ldots,a_{n-1}+t,a_n)).
\]
Let $\Theta$ be the fan associated with the toric variety $\P^1$.
According to toric resolution of singularities \cite[Theorem 11.1.9]{CLStoric}, 
there is a morphism of fans $u\colon \Sigma\to \Delta$ such that $\Sigma$ is a refinement of $\N^n\times \Theta$ 
and the morphism of lattices associated with $u$ is $h$.
Hence we have a commutative diagram
\begin{equation}
\label{Ori.4.2}
\begin{tikzcd}
&
\ul{\A}_\Sigma\ar[r]&
(\P^n)^n
\\
\A^n\times \A^1\ar[r]\ar[rru]\ar[ru]&
\A^n\times \P^1,\ar[u,crossing over,leftarrow]
\end{tikzcd}
\end{equation}
where $\ul{\A}_\Sigma$ denotes the toric variety associated with $\Sigma$.

\

We set $V_i:=(\A^n,H_i)\simeq\ol{V}_i\times_{\P^n} \A^n$ for all $1\leq i\leq n$.
Let $X'$ (resp.\ $V_i'$) be the fs log scheme whose underlying scheme is $\ul{\A}_\Sigma$, 
and whose log structure is associated with the open immersion $\A^n\times \A^1\to \ul{\A}_\Sigma$ 
(resp.\ $(V_i-\partial V_i)\times \A^1\to \ul{\A}_\Sigma$), 
and let $H_i'$ be the strict transform of $H_i$ in $\ul{\A}_{\Sigma}$ for all $1\leq i\leq n$.
The morphisms $X'\to \A^n\times \boxx$ and $V_i'\to V_i\times \boxx$ for all $1\leq i\leq n$ 
are admissible blow-ups along smooth centers.
Hence the maps
\[
\colimit(\cV')\to \colimit(\cV)\times \boxx
\text{ and }
X'\to \A^n\times \boxx
\]
are equivalences in $\inflogSHS(S)$,
where
\[
\cV:=\horn(\A^n,H_1+\cdots+H_n)
\text{ and }
\cV':=\horn(X,H_1'+\cdots+H_n').
\]
It follows that the map
\begin{equation}
\label{Ori.4.3}
X'/\colimit(\cV')\to (\A^n/\colimit(\cV))\times \boxx
\end{equation}
is an equivalence in $\inflogSHS(S)$.

\

By considering the suitable compactifying log structures for the schemes in \eqref{Ori.4.2}, we obtain the map
\[
X'/\colimit(\cV')\to (\P^n)^n/\colimit(\cZ).
\]
Let $q'\colon \A^n\to (\P^n)^n$ be the restriction of the diagonal map.
The restriction of $X'\to (\P^n)^n$ at $t=0$ and $t=1$ are exactly $q$ and $q'$.
It follows that the two naturally induced maps
\[
q,q'\colon \A^n/\colimit(\cV)\rightrightarrows (\P^n)^n/\colimit(\cZ)
\]
are homotopic in $\inflogSHS(S)$ since we can apply the argument in Lemma \ref{ProplogSH.7} to the diagram
\[
\begin{tikzcd}
\A^n/\colimit(\cV)\ar[r,shift left=0.75ex,"u_1'"]\ar[r,shift right=0.75ex,"u_0'"']&
X'/\colimit(\cV')\ar[r]\ar[d]&
(\P^n)^n/\colimit(\cZ)
\\
&
(\A^n/\colimit(\cV))\times \boxx,
\end{tikzcd}
\]
where $u_0'$ and $u_1'$ are the lifts of the zero section and one section.

\

Finally, 
by combining the commutative diagrams \eqref{Ori.4.2} and 
\[
\begin{tikzcd}
\P^n/\pt\ar[r,"\Delta"]\ar[d]&
(\P^n)^n/\colimit(\cZ)\ar[d,"q'",leftarrow]
\\
\P^n/\colimit(\ol{\cV})\ar[r,"\cong"]&
\A^n/\colimit(\cV)
\end{tikzcd}
\]
in $\inflogSHS(S)$, we deduce that \eqref{Ori.4.1} commutes.
\end{proof}

\begin{lem}
\label{Ori.58}
Suppose $\ringE$ is oriented.
Let $i\colon \P^1\to \P^n$ be the closed immersion sending $[x_0:x_1]$ to $[x_0:x_1,0,\ldots,0]$, let $p\colon \P^n\to \pt$ and $q\colon \P^1\to \pt$ be the structure morphisms, and let $\cT$ be the tautological bundle on $\P^n$.
Then the diagram
\[
\begin{tikzcd}
&
\ringE^{*-2,*-1}\ar[d,"\cong"]\ar[ld,"\psi_n"']
\\
\ringE^{**}(\P^n)\ar[r,"i^*"]&
\ringE^{**}(\P^1)
\end{tikzcd}
\]
commutes, where the right vertical isomorphism sends $x\in \ringE^{*-2,*-1}$ to $q^*(x)\smile \Sigma^{2,1}\unit$, and $\psi_n$ sends $x\in \ringE^{*-2,*-1}$ to $p^*(x) \smile c_1(\cT)$.
\end{lem}
\begin{proof}
Let $\cL$ be the tautological bundle on $\P^1$.
From \eqref{Ori.7.4}, we have the formula
\[
i^* \psi_n(x)
=
q^*(x) \smile c_1(\cL).
\]
To conclude the proof, we appeal to \eqref{Ori.7.11}.
\end{proof}

\begin{const}
Suppose $\ringE$ is oriented and $\cE\to X$ is a rank $d+1$ vector bundle with $X\in \lSm/S$.
Let $\cL$ be the tautological line bundle on $\cE$.
There is a map
\begin{equation}
\label{Ori.11.1}
\rho_{\cE}
\colon
\bigoplus_{i=0}^d \ringE^{*-i,*-2i}(X)
\to
\ringE^{**}(\P(\cE))
\end{equation}
sending $(x_0,\ldots,x_d)$ to $\sum_{i=0}^d p^*(x_i) \smile c_1(\cL) \smile \cdots \smile c_1(\cL)$ 
($i$ copies of $c_1(\cL)$).
For brevity, we often omit the symbol $\smile$, and then write the expression above as 
$$
\sum_{i=0}^d p^*(x_i)c_1(\cL)^i.
$$

\

If $U$ is an open subscheme of $X$ with $\cE_U:=\cE\times_X U$, 
we consider the projection $q\colon \cE/\cE_U\to X/U$ and the morphism
\[
\Delta_i
\colon
\P(\cE)/\P(\cE_U)
\to
\P(\cE)^{i+1}/(\P(\cE_U)\times \P(\cE)^i)
\cong
\P(\cE)/\P(\cE_U) \times \P(\cE)^i
\]
naturally induced by the diagonal morphism $\P(\cE)\to \P(\cE)^{i+1}$ for every integer $i\geq 1$.
We have homomorphisms
\begin{equation}
\label{Ori.11.5}
\begin{split}
\ringE^{*-2i,*-i}(X/U)
&\xrightarrow{q^*}
\ringE^{*-2i,*-i}(\P(\cE)/\P(\cE_U))
\\
&\to
\ringE^{**}(\P(\cE)/\P(\cE_U)\times \P(\cE)^i)
\xrightarrow{\Delta_i^*}
\ringE^{**}(\P(\cE)/\P(\cE_U))
\end{split}
\end{equation}
where the second homomorphism is constructed from \eqref{Ori.12.2} and $c_1(\cL)\in \ringE^{2,1}(\P(\cE))$.
Collect \eqref{Ori.11.5} for every $i=0,\ldots,d$ to form the homomorphism
\[
\rho_{\cE,U}
\colon
\bigoplus_{i=0}^d \ringE^{*-2i,*-i}(X/U)
\to
\ringE^{**}(\P(\cE)/\P(\cE_U)).
\]
\end{const}

\begin{lem}
\label{Ori.52}
Suppose $\ringE$ is oriented.
With the above notation, the diagram
\begin{equation}
\label{Ori.52.1}
\begin{tikzcd}[column sep=tiny]
\oplus_{i=0}^d \ringE^{*-2i,*-i}(X/U)\ar[r]\ar[d,"\rho_{\cE,U}"']&
\oplus_{i=0}^d \ringE^{*-2i,*-i}(X)\ar[r]\ar[d,"\rho_{\cE}"]&
\oplus_{i=0}^d \ringE^{*-2i,*-i}(U)\ar[r,"\delta"]\ar[d,"\rho_{\cE_U}"]&
\oplus_{i=0}^d \ringE^{*-2i+1,*-i}(X/U)\ar[d,"\rho_{\cE,U}"]
\\
\ringE^{**}(\P(\cE)/\P(\cE_U))\ar[r]&
\ringE^{**}(\P(\cE))\ar[r]&
\ringE^{**}(\P(\cE_U))\ar[r,"\delta"]&
\ringE^{*+1,*}(\P(\cE)/\P(\cE_U))
\end{tikzcd}
\end{equation}
commutes.
\end{lem}
\begin{proof}
The left and middle squares of \eqref{Ori.52.1} commute by the functoriality of the construction.
Let $j\colon U\to X$ be the open immersion, and $\Delta_{i,U}\colon \P(\cE_U)\to \P(\cE_U)\times \P(\cE)^i$ be the restriction of the diagonal map for every integer $i\geq 1$.
There is a commutative diagram
\begin{equation}
\label{Ori.52.2}
\begin{tikzcd}
\ringE^{*-2i,*-i}(U)\ar[d,"\delta"']\ar[r]&
\ringE^{**}(\P(\cE_U)\times \P(\cE)^i)\ar[r,"\Delta_{i,U}^*"]\ar[d,"\delta"]&
\ringE^{**}(\P(\cE_U))\ar[d,"\delta"]
\\
\ringE^{*-2i+1,*-i}(X/U)\ar[r]&
\ringE^{*+1,*}(\P(\cE)/\P(\cE_U)\times \P(\cE)^i)\ar[r,"\Delta_i^*"]&
\ringE^{*+1*}(\P(\cE)/\P(\cE_U)),
\end{tikzcd}
\end{equation}
where the lower row of \eqref{Ori.52.2} comes from \eqref{Ori.11.5}, and the upper left horizontal homomorphism of \eqref{Ori.52.2} is constructed from $j^*\colon \ringE^{*-2i,*-i}(U)\to \ringE^{*-2i,*-i}(X)$, \eqref{Ori.12.2}, and $c_1(\cL)\in \ringE^{2,1}(\P(\cE))$.
We set $\cL_U:=j^*(\cL)$, which is the tautological bundle on $\P(\cE_U)$.
Since $j^*c_1(\cL)=c_1(\cL_U)$ by \eqref{Ori.7.4},
the collection of the composite of the homomorphisms in the upper row of \eqref{Ori.52.2} for $i=0,\ldots,d$ constructs $\rho_{\cE_U}$.
This shows that the right square of \eqref{Ori.52.1} commutes.
\end{proof}

\begin{thm}
[Projective bundle theorem]
\label{Ori.11}
Suppose that $\ringE$ is oriented.
With the above notation, the map \eqref{Ori.11.1} is an isomorphism.
\end{thm}
\begin{proof}
There exists a Zariski cover $\{U_i\}_{1\leq i\leq n}$ of $X$ such that $\cE\times_X U_i$ is a trivial bundle on $U_i$ for every $1\leq i\leq n$.
We proceed by induction on $n$.
The case $n=1$ is done in Proposition \ref{Ori.51} below.
Suppose that the claim holds for $n-1$.
We set $U:=U_1\cup \cdots \cup U_{n-1}$.
By induction, the map $\rho_{\cE_U}$ in Lemma \ref{Ori.52} is an isomorphism.
Therefore, by the five lemma, it suffices to show that the map $\rho_{\cE,U}$ in 
Lemma \ref{Ori.52} is an isomorphism.
By excision, it suffices to show that the map
\[
\oplus_{i=0}^d \ringE^{*-2i,*-i}(U_n/U_n\cap U)
\to
\ringE^{**}(\P(\cE_{U_n})/\P(\cE_{U_n\cap U}))
\]
is an isomorphism, where $\cE_{U_n}:=\cE\times_X U_n$ and $\cE_{U_n\cap U}:=\cE\times_X (U_n\cap U)$.
Again by the five lemma and Lemma \ref{Ori.52}, it suffices to show that the maps
\[
\oplus_{i=0}^d \ringE^{*-2i,*-i}(V)
\to
\ringE^{**}(\P(\cE_V))
\]
are isomorphisms for $V=U_n$ and $V=U_n\cap U$.
This follows from the induction hypothesis.
\end{proof}

\begin{prop}
\label{Ori.51}
Suppose that $\ringE$ is oriented.
Then the map
\[
\rho_d
\colon
\bigoplus_{i=0}^d
\ringE^{*-2i,*-i}(X)
\to
\ringE^{**}(X\times \P^d)
\]
is an isomorphism for every $X\in \lSm/S$, where $\rho_d:=\rho_{\cE}$, and $\cE$ is the rank $d+1$ trivial bundle on $X$.
\end{prop}
\begin{proof}
We only deal with the case where $X=\pt$ since the general case follows in a similar way.
It suffices to show that the map
$
\widetilde{\rho}_d
\colon
\bigoplus_{i=1}^d
\ringE^{*-2i,*-i}
\to
\ringE^{**}(\P^d/\pt)
$
naturally induced by $\rho_d$ is an isomorphism.
Let $i\colon \P^{d-1}\to \P^d$ be the closed immersion sending $[a_0:\cdots:a_{d-1}]$ to $[a_0:\cdots:a_{d-1}:0]$.
There is a naturally induced diagram
\begin{equation}
\label{Ori.11.3}
\begin{tikzcd}[column sep=small]
0\ar[r]&
\ringE^{*-2d,*-d}\ar[d,"\cong"']\ar[r,"q"]\ar[rd,phantom,"(a)" description]&
\oplus_{i=1}^d \ringE^{*-2i,*-i}\ar[d,"\widetilde{\rho}_d"]\ar[r,"r"]\ar[rd,phantom,"(b)" description]&
\oplus_{i=1}^{d-1} \ringE^{*-2i,*-i}\ar[d,"\widetilde{\rho}_{d-1}"]\ar[r]&
0
\\
&
\ringE^{**}(\P^d/\P^{d-1})\ar[r]&
\ringE^{**}(\P^d/\pt)\ar[r,"i^*"]&
\ringE^{**}(\P^{d-1}/\pt)
\end{tikzcd}
\end{equation}
with exact rows, where $q$ and $r$ are the obvious inclusion and projection maps.

\

We will show that \eqref{Ori.11.3} commutes.
Let $\cL_d$ be the tautological line bundle on $\P^d$.
To show the square (b) commutes, let $(x_1,\ldots,x_d)$ be an element of $\oplus_{i=1}^d \ringE^{*-2i,*-i}$.
We need to show the formula
\[
i^*\big(\sum_{i=1}^d x_i c_1(\cL_d)^i \big)
=
\sum_{i=1}^{d-1} x_i c_1(\cL_{d-1})^i.
\]
Since $i^*(\cL_d)\cong \cL_{d-1}$, we are reduced to showing the vanishing
\begin{equation}
\label{Ori.11.4}
c_1(\cL_{d-1})^d=0.
\end{equation}
For every integer $n\geq 1$, let $\psi_n\colon \ringE^{*-2,*-1}\to \ringE^{**}(\P^n/\pt)$ be the homomorphism in Lemma \ref{Ori.58}.
There is a diagram
\begin{equation}
\label{Ori.11.2}
\begin{tikzcd}[column sep=small, row sep=small]
&
\ringE^{*-2d,*-d}\ar[r,"\id"]\ar[d,"\psi_d^{\wedge d}"]\ar[ld,"\cong"',bend right]\ar[ld,phantom,"(d)"]\ar[rd,phantom,"(e)"]&
\ringE^{*-2d,*-d}\ar[d,"\psi_{d-1}^{\wedge d}"] 
\\
\ringE^{**}((\P^1/\pt)^{\wedge d})\ar[r,leftarrow]\ar[d]\ar[rd,phantom,"(c)"]&
\ringE^{**}((\P^d/\pt)^{\wedge d})\ar[r]\ar[d,"\Delta^*"]\ar[rd,phantom,"(f)"]&
\ringE^{**}((\P^{d-1}/\pt)^{\wedge d})\ar[d,"\Delta^*"]
\\
\ringE^{**}(\P^d/\P^{d-1})\ar[r]&
\ringE^{**}(\P^d/\pt)\ar[r]&
\ringE^{**}(\P^{d-1}/\pt),
\end{tikzcd}
\end{equation}
where the homomorphisms denoted by $\Delta^*$ are naturally induced by the diagonal morphisms.
By Lemmas \ref{Ori.4} and \ref{Ori.58}, $(c)$ and $(d)$ commutes.
Apply the relation \eqref{Ori.7.4} to the inclusion $i\colon \P^{d-1}\to \P^d$ to deduce that $(e)$ commutes too.
It is trivial that the square $(f)$ commutes.
The composite $\ringE^{**}(\P^d/\P^{d-1})\to \ringE^{**}(\P^{d-1}/\pt)$ is zero, 
so from the commutativity of \eqref{Ori.11.2} we deduce $c_1(\cL_{d-1})^d=0$.
It follows that $(b)$ commutes.
The commutativity of $(c)$ and $(d)$ implies the commutativity of $(a)$.

\

We proceed by induction on $d$ to show that $\widetilde{\rho}_d$ is an isomorphism for every integer $d\geq 0$.
The claim is trivial if $d=0$.
In the following we assume $d>0$.
If $\widetilde{\rho}_{d-1}$ is an isomorphism, then $i^*$ is surjective.
From the long exact sequence
\begin{align*}
\ringE^{*-1,*}(\P^d/\pt)
&\xrightarrow{i^*}
\ringE^{*-1,*}(\P^{d-1}/\pt)
\to
\ringE^{**}(\P^d/\P^{d-1})
\\
&\to
\ringE^{**}(\P^d/\pt)
\xrightarrow{i^*}
\ringE^{**}(\P^{d-1}/\pt),
\end{align*}
we see that the sequence
\[
0
\to
\ringE^{**}(\P^d/\P^{d-1})
\to
\ringE^{**}(\P^d/\pt)
\xrightarrow{i^*}
\ringE^{**}(\P^{d-1}/\pt)
\to
0
\]
is exact.
Combine this with \eqref{Ori.11.3} and apply the five lemma to deduce that $\widetilde{\rho}_d$ is an isomorphism too.
\end{proof}

\begin{prop}
\label{Ori.43}
Suppose $\ringE$ is oriented and $\cE\to X$ is a vector bundle in $\SmlSm/S$.
Then there is an exact sequence
\begin{equation}
\label{Ori.43.1}
0
\to
\ringE^{**}(\Thom(\cE))
\to
\ringE^{**}(\P(\cE\oplus \cO))
\to
\ringE^{**}(\P(\cE))
\to
0.
\end{equation}
\end{prop}
\begin{proof}
By Proposition \ref{Ori.41}, it suffices to show that $\ringE^{**}(\P(\cE\oplus \cO)) \to \ringE^{**}(\P(\cE))$ is surjective.
This follows from Theorem \ref{Ori.11}.
\end{proof}

As in $\A^1$-homotopy theory, we have the following consequence of the projective bundle theorem.

\begin{prop}
\label{Ori.42}
Suppose $\ringE$ is oriented.
For all $x\in \ringE^{m,p}(X)$ and $y\in \ringE^{n,q}(X)$ with $X\in \lSm/S$ and $m,n,p,q\in \Z$, there is the formula
\[
x\smile y
=
(-1)^{mn}
y \smile x
\]
in $\ringE^{m+n,p+q}(X)$.
\end{prop}
\begin{proof}
Let $\eta\colon T\wedge T\to T\wedge T$ be the permutation map, which naturally induces a map
\[
\eta\colon \ringE^{**}(X)\to \ringE^{**}(X)
\]
for all $X\in \lSm/S$.
In general, we have the formula
\begin{equation}
\label{Ori.42.1}
x\smile y
=
(-1)^{mn}\eta^{pq} y\smile x.
\end{equation}
Hence, it remains to show $\eta=1$.

Let $\cT$ be the tautological bundle on $\P^2$, and we set $\lambda:=c_1(\cT)$.
When $x=y=\lambda$ in \eqref{Ori.42.1}, we have $\lambda^2=\eta\lambda^2$.
Together with Theorem \ref{Ori.11}, we have $1=\eta$.
\end{proof}

\subsection{Chern classes}
Throughout this subsection, we fix $S\in \Sch$ and an oriented homotopy commutative monoid $\ringE$ in $\inflogSH(S)$.

\begin{df}
\label{Ori.13}
Suppose $\cE\to X$ is a rank $d$ vector bundle with $X\in \lSm/S$.
To define higher Chern classes of $\cE$, we imitate Grothendieck's approach as follows.
Let $p\colon \P(\cE)\to X$ be the canonical projection, 
and let $\cT$ be the tautological line bundle on $\P(\cE)$.
For simplicity of notation, we set $\lambda:=c_1(\cT)$.
Owing to Theorem \ref{Ori.11}, 
there exist unique classes $c_i(\cE)\in \ringE^{2i,i}(X)$ for all $0\leq i\leq d$ such that $c_0(\cE)=1$ and
\begin{equation}
\label{Ori.13.1}
\sum_{i=0}^d  p^* (c_i(\cE))(-\lambda)^{d-i}=0.
\end{equation}
The class $c_i(\cE)$ is called the \emph{$i$-th Chern class} \index{Chern class}\index[notation]{ci @ $c_i(\cE)$} 
of $\cE$.
We set $c_i(\cE):=0$ if $i\geq d+1$.
\end{df}

\begin{rmk}
\label{Ori.14}
Suppose $\cE\to X$ is a line bundle in $\lSm/S$.
Then $\bP(\cE)\cong X$, and \eqref{Ori.13.1} becomes
\[
c_1(\cE)-c_1(\cE)=0.
\]
Here, the left $c_1(\cE)$ is defined in Definition \ref{Ori.13}, and the right $c_1(\cE)$ is defined in Definition \ref{Ori.7}.
Hence, these two definitions of the first Chern class of $\cE$ agree.
\end{rmk}

\begin{prop}
\label{Ori.50}
Suppose $\cE\to X$ is a trivial line bundle in $\lSm/S$.
Then $c_i(\cE)=0$ for all integers $i\geq 1$.
\end{prop}
\begin{proof}
Combine \eqref{Ori.11.4} and \eqref{Ori.13.1}.
\end{proof}

\begin{prop}
\label{Ori.18}
Suppose $f\colon X'\to X$ is a morphism in $\lSm/S$, 
and $\cE\to X$ is a rank $d$ vector bundle.
Then, for every integer $i\geq 0$, we have the formula
\[
f^*(c_i(\cE))=c_i(f^*\cE).
\]
\end{prop}
\begin{proof}
Let $\cT'$ be the tautological line bundle of $\P(f^*\cE)$, and we set $\lambda':=c_1(\cT'^*)$.
Apply $f^*$ to \eqref{Ori.13.1} and use \eqref{Ori.7.4} to have the formula
\[
\sum_{i=0}^d p'^* (f^*(c_i(\cE)))(-\lambda')^{d-i}=0,
\]
where $p'\colon \P(f^*\cE)\to X'$ is the projection.
To conclude, compare this with \eqref{Ori.13.1} for the vector bundle $f^*\cE\to X'$.
\end{proof}

\begin{prop}
[Splitting principle]
\label{Ori.19}
If $\cE\to X$ is a rank $d$ vector bundle with $X\in \lSm/S$, then there exists a strict proper smooth morphism $f\colon Y\to X$ such that $f^*\colon \ringE^{**}(X)\to \ringE^{**}(Y)$ is a monomorphism and $f^*\cE$ admits a filtration of subbundles
\[
0=\cE_0\subset \cE_1 \subset \cdots \subset \cE_d=\cE
\]
such that the quotient bundle $\cE_i/\cE_{i-1}$ is a line bundle for every $1\leq i\leq d$.
\end{prop}
\begin{proof}
As in intersection theory or $\A^1$-homotopy theory, let $Y$ be the complete flag variety associated with $\cE$.
If $f$ is the projection $Y\to X$, then it is well known that $f^*\cE$ satisfies the required condition.
Since $f$ can be factored as a tower of projective bundles, Theorem \ref{Ori.11} shows that $f^*$ is a monomorphism.
\end{proof}

\begin{df}
\label{Ori.15}
Suppose $\cE\to X$ is a rank $d$ vector bundle with $X\in \lSm/S$.
Let $p\colon \P(\cE)\to X$ and $p'\colon \P(\cE\oplus \cO)\to X$ be the canonical projections, and let $\cT$ and $\cT'$ be the tautological line bundles on $\P(\cE)$ and $\P(\cE\oplus \cO)$ for simplicity of notation.
We set $\lambda:=c_1(\cT)$ and $\lambda':=c_1(\cT')$.
Consider the class
\begin{equation}
\label{Ori.15.2}
\ol{t}(\cE)
:=
\sum_{i=0}^d p'^* (c_i(\cE))(-\lambda')^{d-i} \in \ringE^{2d,d}(\P(\cE\oplus \cO)).
\end{equation}
Due to \eqref{Ori.7.4} and \eqref{Ori.13.1}, the restriction of $\ol{t}(\cE)$ to $\P(\cE)$ is $0$.
From \eqref{Ori.43.1}, we have a unique class
\[
t(\cE)\in \ringE^{2d,d}(\Thom(\cE))
\]
mapping to $\overline{t}(\cE)$, which we call the \emph{Thom class}\index{Thom class}\index[notation]{tE @ $t(\cE)$} of $\cE$.
\end{df}

\begin{prop}
\label{Ori.54}
Let $t_\infty \in \ringE^{2,1}(\Thom(\cT_1))$ be a Thom orientation, where $\cT_1$ is the tautological line bundle on $\P^\infty$.
If $c_\infty$ is the Chern orientation associated with $t_\infty$, then the Thom class $t(\cT_1)$ obtained from $c_\infty$ coincides with $t_\infty$.
\end{prop}
\begin{proof}
Consider the isomorphism $z^*$ given in \eqref{Ori.49.2}.
We need to show $z^*t_\infty=z^*t(\cT_1)$.
The left-hand side equals $c_\infty$.
On the other hand, 
from \eqref{Ori.15.2}, 
we have the formula $z^*t(\cT_1)=c_1(\cT_1)$.
An appeal to \eqref{Ori.7.12} finishes the proof.
\end{proof}

Suppose $X\in \lSm/S$ and $\cE\to X$ is a rank $d$ vector bundle, and let $E$ be the exceptional divisor on $\Blow_X \cE$.
Using the graph morphisms of the two projections $\cE\to X$ and $(\Blow_X\cE,E)\to X$, 
we obtain a morphism
\begin{equation}
\label{Ori.16.5}
\Thom(\cE)\to \Thom(\cE)\wedge X_+.
\end{equation}
We use the latter map to define the cup product
\[
\smile \colon \ringE^{**}(X)\times \ringE^{**}(\Thom(\cE))\to \ringE^{**}(\Thom(\cE))
\]
as in \eqref{Ori.12.3}.

\begin{thm}[Thom isomorphism]
\label{Ori.16}
For every rank $d$ vector bundle $p\colon \cE\to X$ with $X\in \lSm/S$, 
the cup product homomorphism
\begin{equation}
\label{Ori.16.4}
-\smile t(\cE)
\colon
\ringE^{**}(X)
\to
\ringE^{*+2d,*+d}(\Thom(\cE))
\end{equation}
is an isomorphism.
\end{thm}
\begin{proof}
Combine Theorem \ref{Ori.11} and \eqref{Ori.43.1} to have a commutative diagram with exact rows
\[
\begin{tikzcd}[column sep=small]
0\ar[r]&
\ringE^{**}(X)\ar[d,"-\smile t(\cE)"']\ar[r]&
\oplus_{i=0}^d \ringE^{*+2i,*+i}(X)\ar[d,"\cong"]\ar[r]&
\oplus_{i=1}^d \ringE^{*+2i,*+i}(X)\ar[d,"\cong"]\ar[r]&
0
\\
0\ar[r]&
\ringE^{*+2d,*+d}(\Thom(\cE))\ar[r]&
\ringE^{*+2d,*+d}(\P(\cE\oplus \cO))\ar[r]&
\ringE^{*+2d,*+d}(\P(\cE))\ar[r]&
0.
\end{tikzcd}
\]
The five lemma finishes the proof.
\end{proof}

\begin{prop}
\label{Ori.70}
Suppose $X\in \lSm/S$ and $p\colon \cE\to X$ is a rank $d$ vector bundle.
For all integers $p$ and $q$, there is a canonical equivalence of spectra
\begin{equation}
\label{Ori.70.1}
\map(\Sigma_T^\infty X_+,\Sigma^{p,q}\ringE)
\xrightarrow{\simeq}
\map(\Sigma_T^\infty \Thom(\cE),\Sigma^{p+2d,q+d}\ringE).
\end{equation}
\end{prop}
\begin{proof}
The Thom class $t(\cE)$ gives a map $\mathbf{1}\to \map(\Sigma_T^\infty \Thom(\cE),\Sigma^{2d,d}\ringE)$, 
where $\mathbf{1}$ denotes the sphere spectrum.
We have the maps of spectra
\begin{align*}
&\map(\Sigma_T^\infty X_+,\Sigma^{p,q}\ringE)\wedge \map(\Sigma_T^\infty \Thom(\cE),\Sigma^{2d,d}\ringE)
\\
\to &
\map(\Sigma_T^\infty X_+\wedge \Thom(\cE),\Sigma^{p+2d,q+d}\ringE\wedge \ringE)
\\
\to &
\map(\Sigma_T^\infty X_+\wedge \Thom(\cE),\Sigma^{p+2d,q+d}\ringE)
\to
\map(\Sigma_T^\infty \Thom(\cE),\Sigma^{p+2d,q+d}\ringE),
\end{align*}
where the second map is naturally induced by the ring structure map $\ringE \wedge \ringE \to \ringE$, 
and the third map is naturally induced by \eqref{Ori.16.5}.
This constructs the map in \eqref{Ori.70.1}.
If we apply $\pi_r$ to \eqref{Ori.70.1} for every integer $r$, 
then we obtain an isomorphism due to Theorem \ref{Ori.16}.
Hence \eqref{Ori.70.1} is an equivalence.
\end{proof}

\begin{prop}
\label{Ori.33}
Suppose $f\colon X'\to X$ is a morphism in $\lSm/S$ and $\cE\to X$ is a vector bundle.
If $g\colon \Thom(f^*\cE)\to \Thom(\cE)$ is the naturally induced morphism, 
then for every integer $i\geq 0$ we have the formula 
\[
g^*(t(\cE))=t(f^*\cE).
\]
\end{prop}
\begin{proof}
Apply Proposition \ref{Ori.18} to \eqref{Ori.15.2}.
\end{proof}

Suppose $i\colon Z\to X$ is a codimension $d$ closed immersion in $\Sm/S$.
Combine \eqref{Thom.1.5} with Theorem \ref{Ori.16} to have an isomorphism
\begin{equation}
\label{Ori.16.1}
\ringE^{*-2d,*-d}(Z)
\cong
\ringE^{**}(X/(\Blow_Z X,E)),
\end{equation}
where $E$ is the exceptional divisor of the blow-up $\Blow_Z X$.
Compose with the canonical homomorphism $\ringE^{**}(X/(\Blow_Z X,E))\to \ringE^{**}(X)$ to have a homomorphism
\begin{equation}
\label{Ori.16.2}
i_*\colon \ringE^{*-2d,*-d}(Z)\to \ringE^{**}(X),
\end{equation}
which we call the \emph{Gysin map}. \index{Gysin map}
From \eqref{Ori.16.1}, we have the long exact sequence
\begin{equation}
\label{Ori.16.3}
\cdots
\to
\ringE^{*-2d,*-d}(Z)
\xrightarrow{i_*}
\ringE^{**}(X)
\xrightarrow{j^*}
\ringE^{**}(\Blow_Z X,E)
\to
\ringE^{*-2d+1,*-d}(Z)
\to
\cdots,
\end{equation}
where $j\colon (\Blow_Z X,E)\to X$ is the projection.

\begin{prop}
\label{Ori.63}
Suppose $\cE\to Z$ is a rank $d$ vector bundle in $\lSm/S$.
If $i\colon Z\to \cE$ is the zero section, then the isomorphism \eqref{Ori.16.1} with $X:=\cE$ coincides with the Thom isomorphism \eqref{Ori.16.4}.
\end{prop}
\begin{proof}
Lemma \ref{Ori.36} implies that the map
\[
\mathfrak{p}_{\cE,Z}
\colon
\cE/(\Blow_Z \cE,E)\to \Thom(\Normal_Z \cE) = \cE/(\Blow_Z \cE,E)
\]
is homotopic to the identity map $\id$ in $\inflogSHS(S)$, where $E$ is the exceptional divisor.
Composing $\mathfrak{p}_{\cE,Z}$ and $\id$ with the Thom isomorphism in \eqref{Ori.16.4} 
yields the claim.
\end{proof}

Suppose
\begin{equation}
\label{Ori.21.2}
\begin{tikzcd}
Z'\ar[r,"i'"]\ar[d,"g"']&
X'\ar[d,"f"]
\\
Z\ar[r,"i"]&
X
\end{tikzcd}
\end{equation}
is a cartesian square in $\Sm/S$ such that $i$ and $i'$ are closed immersions of codimensions $d$ and $d'$.
Then there is a naturally induced diagram
\begin{equation}
\label{Ori.21.1}
\begin{tikzcd}
\ringE^{*-2d,*-d}(Z)\ar[d,"c_{d-d'}(\cF) \smile g^*(-)"']\ar[r,"\cong"]&
\ringE^{**}(X/(\Blow_Z X,E))\ar[d]
\\
\ringE^{*-2d',*-d'}(Z')\ar[r,"\cong"]&
\ringE^{**}(X'/(\Blow_{Z'}X',E')),
\end{tikzcd}
\end{equation}
where $E$ and $E'$ are the exceptional divisors and $\cF:=g^*\Normal_Z X / \Normal_{Z'}X'$.

What we are concerned with is the commutativity of \eqref{Ori.21.1}.
Use the diagram \eqref{Thom.1.8} to reduce to the case where $X=\Normal_Z X$, $X'=\Normal_{Z'} X'$, and $i$ and $i'$ are zero sections.
Hence to show that \eqref{Ori.21.1} commutes, it suffices to show that the naturally induced diagram
\begin{equation}
\label{Ori.21.4}
\begin{tikzcd}
\ringE^{*-2d,*-d}(Z)\ar[d,"c_{d-d'}(\cF)\smile g^*(-)"']\ar[r,"\cong"]&
\ringE^{**}(\Thom(\Normal_Z X))\ar[d]
\\
\ringE^{*-2d',*-d'}(Z')\ar[r,"\cong"]&
\ringE^{**}(\Thom(\Normal_{Z'}X'))
\end{tikzcd}
\end{equation}
commutes.
This holds if we have the formula
\begin{equation}
\label{Ori.21.3}
c_{d-d'}(\cF)\smile \ol{t}(\Normal_{Z'}(X'))
=
g'^*(\ol{t}(\Normal_Z X)),
\end{equation}
where $g'\colon \P(\Normal_{Z'} X'\oplus \cO)\to \P(\Normal_Z X\oplus \cO)$ is the naturally induced morphism.

\begin{lem}
\label{Ori.22}
The diagram \eqref{Ori.21.1} commutes in the special case where $i$ is a closed immersion of codimension $1$ and $f=i$.
\end{lem}
\begin{proof}
We need to show \eqref{Ori.21.3}.
Let $W$ be the divisor on $X\cong \P(\Normal_Z X\oplus \cO)$ at $\infty$, 
which is isomorphic to $Z$.
We have $g'^* \ol{t}(\Normal_Z X) = c_1(\Normal_Z X)-g'^*(c_1(\cT))$, where $g'$ is the zero section, and $\cT$ is the tautological line bundle on $\P(\N_Z X\oplus \cO)$.
Hence, it remains to show the line bundle $g'^*(\cT)$ is trivial.
According to \cite[Paragraph B.5.6]{Fulton}, there is a section of the dual bundle $\cT^*$ whose zero scheme is $W$.
This implies that $g'^*(\cT^*)$ is trivial since the intersection of $W$ and the zero section is empty.
It follows that $g'^*(\cT)$ is also trivial.
\end{proof}

\begin{lem}
\label{Ori.27}
The diagram \eqref{Ori.21.1} commutes in the special case where $\Normal_{Z'}X'\to g^*\Normal_Z X$ is an isomorphism.
\end{lem}
\begin{proof}
We need to show \eqref{Ori.21.3}.
This follows from Proposition \ref{Ori.18} and the formulation \eqref{Ori.15.2}.
\end{proof}

\begin{lem}
\label{Ori.23}
Let $X$ be a scheme with a rank $n$ vector bundle $\cE\to X$.
If $s\colon X\to \cE$ is a global section of $\cE$ that intersects transversally with the zero section in a 
smooth closed subscheme $Z$ of $X$, 
then there is a morphism
\[
\Blow_{Z\times \infty}(X\times \P^1)\to \P(\cE\oplus \cO)
\]
extending the morphism $f\colon X\times \A^1\to \cE$ sending $(x,a)$ to $a s(x)$.
\end{lem}
\begin{proof}
Such an extension is unique if exists, so we can work Zariski locally on $X$.
Hence we may assume that $X$ is an affine scheme $\Spec{A}$ and $\cE$ is a trivial vector bundle.
Then $f$ is given by $A[x_1,\ldots,x_n]\to A[x]$ sending $a\mapsto a$ for all $a\in A$ and $x_i\mapsto s_i x_i$ for some $s_i\in A$.

Since $s$ is transversal with the zero section,
$\Blow_{Z\times \infty}(X\times \P^1)$ has a Zariski cover
\[
\{\Spec{A[1/x,s_1x,\ldots,s_nx]}\}
\cup
\{\Spec{A[1/(s_ix),s_1/s_i,\ldots,s_n/s_i,s_i]}:1\leq i\leq n\}.
\]
Glue the spectra of the homomorphisms
\[
A[x_1,\ldots,x_n]
\xrightarrow{x_i\mapsto s_i x}
A[1/x,s_1x,\ldots,s_n x]
\]
and
\[
A[x_1/x_i,\ldots,x_n/x_i,x_i]
\xrightarrow{x_j/x_i\mapsto s_j/s_i,x_i\mapsto s_i x}
A[x,s_1/s_i,\ldots,s_n/s_i,s_i]
\]
to produce a morphism 
$\Blow_{Z\times \infty}(X\times \P^1)\to \P(\cE\oplus \cO)$ extending $f\colon X\times \A^1\to \cE$.
\end{proof}

\begin{lem}
\label{Ori.24}
Let $\cE\to X$ be a line bundle in $\Sm/S$.
If $s\colon X\to \cE$ is a global section of $\cE$ that intersects transversally 
with the zero section in a smooth closed subscheme $Z$ of $X$, 
then we have the formula
\[
i_*i^*a=c_1(\cE)\smile a
\]
for all $a\in \ringE^{**}(X)$, where $i\colon Z\to X$ is the obvious closed immersion.
\end{lem}
\begin{proof}
Let $s_0\colon X\to \P(\cE\oplus \cO)$ be the zero section.
There are cartesian squares
\[
\begin{tikzcd}
X\ar[r,"\id"]\ar[d,"\id"']&
X\ar[d,"s"]
\\
X\ar[r,"s"]&\P(\cE\oplus \cO),
\end{tikzcd}
\quad
\begin{tikzcd}
Z\ar[r,"i"]\ar[d,"i"']&
X\ar[d,"s"]
\\
X\ar[r,"s_0"]&\P(\cE\oplus \cO).
\end{tikzcd}
\]
Lemma \ref{Ori.22} implies the formula
\[
s^*s_* a = c_1(\cE)\smile a.
\]
On the other hand, Lemma \ref{Ori.27} implies that we have
\[
s_0^*s_* a= i_*i^*a.
\]
Hence we are reduced to show the formula $s^*=s_0^*$.
In effect, 
let $Y$ be the fs log scheme whose underlying scheme is $\Blow_{Z\times \infty}(X\times \P^1)$ and whose log structure is associated with the open immersion $X\times \A^1\to \Blow_{Z\times \infty}(X\times \P^1)$.
Then $Y$ is an admissible blow-up of $X\times \boxx$ along a smooth center.
From Lemma \ref{Ori.23}, we have a morphism
\[
Y\to \P(\cE\oplus \cO),
\]
and the zero section and one section $X\rightrightarrows X\times \boxx$ have lifts to $Y$.
Lemma \ref{ProplogSH.7} finishes the proof.
\end{proof}

\begin{prop}
[Whitney sum formula]
\label{Ori.20}
Suppose there is an exact sequence of vector bundles
\[
0 \to \cE'\to \cE\to \cE''\to 0
\]
over $X\in \lSm/S$.
Then, for every integer $i\geq 0$, we have the formula
\[
c_i(\cE)
=
\sum_{j=0}^i c_j(\cE')c_{i-j}(\cE'').
\]
\end{prop}
\begin{proof}
Owing to Proposition \ref{Ori.19}, we only need to consider the case when $\cE''$ is a line bundle.
Let $\cT$ and $\cT'$ be the tautological line bundles on $\P(\cE)$ and $\P(\cE')$.
We set
\[
a:=\sum_{i=0}^{n-1}p^*c_i(\cE')(-\lambda)^{n-1-i}
\text{ and }
b:=-\lambda+p^*c_1(\cE''),
\]
where $p\colon \P(\cE)\to X$ is the projection, and $\lambda:=c_1(\cT)$.
Due to Theorem \ref{Ori.11}, it suffices to show $a\smile b=0$.

\

We set $\cF:=\cE\otimes \cE''^*$ and $\cF':=\cE'\otimes \cE''^*$ with tautological line bundles $\cL$ and $\cL'$.
Observe that we have an exact sequence
\[
0 \to \cF' \to \cF \to \cO \to 0.
\]
We identify $\P(\cE)$ with $\P(\cF)$ and $\P(\cE')$ with $\P(\cF')$,
and then there are canonical isomorphisms $\cT\cong \cL\otimes p^*\cE''$ and $\cT'\cong \cL\otimes p'^*\cE''$, 
where $p'\colon \P(\cE')\to X$ is the projection.
According to \cite[Paragraph B.5.6]{Fulton}, there exists a global section of $\cL^*$ whose zero scheme is $\P(\cF')$.
Together with Lemma \ref{Ori.24}, we have the formula
\begin{equation}
\label{Ori.20.1}
i_*i^*=\tau\smile (-),
\end{equation}
where $i\colon \P(\cE')\to \P(\cE)$ is the closed immersion, and $\tau:=c_1(\cL)$.

\

We set
\begin{equation}
\label{Ori.20.2}
c:=\sum_{i=0}^{n-1} p^*c_i(\cF')(-\tau)^{n-1-i}
\text{ and }
d:=-\tau.
\end{equation}
Since $i^*\cL\cong \cL'$, we have $i^*c=0$.
Together with \eqref{Ori.20.1}, we have
\begin{equation}
\label{Ori.20.3}
c\smile d
=
\sum_{i=0}^{n-1} p^*c_i(\cF')(-\tau)^{n-i}=0.
\end{equation}
This implies the equality $c_i(\cF)=c_i(\cF')$ for all $0\leq i\leq n-1$ and the vanishing $c_n(\cF)=0$.
In particular, 
the homomorphism $i^*\colon \ringE^{**}(\P(\cF)) \to \ringE^{**}(\P(\cF'))$ is isomorphic to the quotient homomorphism
\begin{equation}
\label{Ori.20.4}
\ringE^{**}(X)[\tau]/(\tau c)
\to
\ringE^{**}(X)[\tau]/(c).
\end{equation}
Since $i^*a=0$, we deduce that $a$ is a multiple of $c$.

\

Let $j\colon (\P(\cE),\P(\cE'))\to \P(\cE)$ be the morphism removing the log structure.
There is a canonical exact sequence
\[
0\to \cM_{\P(\cE)}^\gp \to \cM_{(\P(\cE),\P(\cE'))}^\gp \to i_*\Z\to 0.
\]
This naturally induces an exact sequence
\[
H_{\set}^0(\P(\cE'),\Z)
\xrightarrow{\delta}
\Pic(\P(\cE))
\xrightarrow{j^*}
\Pic(\P(\cE),\P(\cE')).
\]
Since there is a global section of $\cL$ whose zero scheme is $\P(\cE')$, the image of $\delta$ contains $\cL$.
This implies $j^*(\cL)=0$, and hence $j^*\cT\cong j^*p^*\cE''$.
Together with Proposition \ref{Ori.18}, we have $j^*c_1(\cT)=j^*p^*c_1(\cE'')$.
In particular, we have $j^*b=0$.

\

Due to \eqref{Ori.16.3}, there is an exact sequence
\[
\ringE^{*-2,*-1}(\P(\cE'))
\xrightarrow{i_*}
\ringE^{**}(\P(\cE))
\xrightarrow{j^*}
\ringE^{**}(\P(\cE),\P(\cE')).
\]
Since $j^*b=0$, $b$ is in the image of $i_*$.
From the explicit description \eqref{Ori.20.4}, we see that $i^*$ is surjective.
Together with \eqref{Ori.20.1}, we deduce that everything in the image of $i_*$ is a multiple of $\tau$.
In particular, $b$ is a multiple of $d$.
We have verified before that $a$ is a multiple of $c$.
Therefore, \eqref{Ori.20.3} implies $a\smile b=0$.
\end{proof}

\begin{prop}
\label{Ori.25}
Suppose $p\colon \P(\cE\oplus \cO)\to X$ is a projective bundle, where $\cE\to X$ is a rank $d$ vector bundle with $X\in \Sm/S$.
If $\cT$ is the tautological line bundle on $\P(\cE\oplus \cO)$ and $\cQ$ is the universal quotient bundle on $\P(\cE\oplus \cO)$ that fits in an exact sequence
\[
0\to \cT\to p^*(\cE\oplus \cO)\to \cQ\to 0,
\]
then we have $\ol{t}(\cE)=c_d(\cQ)$.
\end{prop}
\begin{proof}
If $t$ is an indeterminate variable, then Proposition \ref{Ori.20} gives the formula
\[
(1+c_1(\cT)t)(1+c_1(\cQ)t+\cdots +c_d(\cQ)t^d)
=
1+c_1(p^*(\cE))t+\cdots+c_{d}(p^*(\cE))t^{d}.
\]
It follows that $c_d(\cQ)$ is equal to the coefficient of $t^d$ in
\[
(1+c_1(p^*(\cE))t+\cdots+c_{d}(p^*(\cE))t^{d})/(1+c_1(\cT)t),
\]
which is precisely $\ol{t}(\cE)$.
\end{proof}

\begin{prop}
[Excessive intersection formula]
\label{Ori.21}
Suppose that \eqref{Ori.21.2} is a cartesian square in $\Sm/S$ such that $i$ and $i'$ 
are closed immersions of codimensions $d$ and $d'$.
Then the square \eqref{Ori.21.1} commutes.
\end{prop}
\begin{proof}
We need to show \eqref{Ori.21.3}.
Let $\cQ$ and $\cQ'$ be the universal quotient bundles on $\P(\cE)$ and $\P(\cE')$.
There is a canonical exact sequence
\[
0 \to \cQ' \to g'^* \cQ \to p'^*\cF \to 0,
\]
where $p'\colon \P(\cE'\oplus \cO)\to Z'$ is the projection.
To conclude, apply Propositions \ref{Ori.20} and \ref{Ori.25}.
\end{proof}

\begin{cor}
[Self-intersection formula]
\label{Ori.26}
Suppose that $i\colon Z\to X$ is a closed immersion in $\Sm/S$ of codimension $d$.
Then for all $a\in \ringE^{**}(X)$ we have the formula
\[
i^*i_*a
=
c_d(\Normal_Z X)\smile a.
\]
\end{cor}
\begin{proof}
This is a special case of Proposition \ref{Ori.21} when $f\colon X'\to X$ is equal to $i\colon Z\to X$.
\end{proof}

\begin{lem}
\label{Ori.39}
Suppose
\[
\begin{tikzcd}
X\ar[rd,"i"]\ar[d,"s_0"']
\\
\cE\ar[r,"j"]&
P
\end{tikzcd}
\]
is a diagram in $\lSm/S$, where $\cE$ is a rank $d$ vector bundle on $X$, $s_0$ is the zero section, $i$ is a closed immersion, and $j$ is an open immersion.
Then the naturally induced diagram
\[
\begin{tikzcd}[column sep=huge]
\ringE^{*-2d,*-d}(X)\ar[d,"i_*"']\ar[r,"(-)\smile t(\cE)"]&
\ringE^{**}(\cE/(\Blow_X \cE,E))\ar[d,"\cong"]
\\
\ringE^{**}(P)\ar[r,leftarrow]&
\ringE^{**}(P/(\Blow_X P,E))
\end{tikzcd}
\]
commutes.
\end{lem}
\begin{proof}
There is a diagram
\[
\begin{tikzcd}
&
\ringE^{*-2d,*-d}(X)\ar[d,"(-)\smile t(\cE)"']\ar[rdd,"(-)\smile t(\cE)",bend left]
\ar[ldd,"i_*"',bend right]\ar[ldd,"(a)",phantom]\ar[rdd,"(b)",phantom,near start]
\\
&
\ringE^{**}(\Thom(\cE))\ar[d,"\mathfrak{p}_{X,P}^*"']\ar[rd,"\mathfrak{p}_{X,\cE}^*"]
\ar[rd,"(c)",phantom,bend right=18,near start]
\\
\ringE^{**}(P)\ar[r,leftarrow]&
\ringE^{**}(P/(\Blow_X P,E))\ar[r,"\cong",leftarrow]&
\ringE^{**}(\cE/(\Blow_X \cE,E)).
\end{tikzcd}
\]
The diagram $(a)$ commutes by the definition of $i_*$, the diagram (b) commutes by Proposition \ref{Ori.63}, and the diagram (c) commutes by \eqref{Thom.1.6}.
Combine these to conclude.
\end{proof}

\begin{prop}
\label{Ori.35}
Suppose $X,Y\in \Sm/S$ and $\cE\to X$ and $\cF\to Y$ are vector bundles.
We consider $\cE\times_S \cF$ as a vector bundle on $X\times_S Y$.
If $\eta\colon \Thom(\cE)\wedge \Thom(\cF)\to \Thom(\cE\times_S \cF)$ is the equivalence \eqref{Ori.40.1} in $\inflogSH(S)$, then we have the formula
\[
t(\cE\times_S \cF)
=
\eta^*(t(\cE)\smile t(\cF)).
\]
\end{prop}
\begin{proof}
For simplicity of notation, we set $P:=\P(\cE\oplus \cO)$, $Q:=\P(\cF\oplus \cO)$, and $R:=P\times_S Q$.
Let $i\colon Z:=X\times_S Y\to R$ be the zero section.
We need to show that the top triangle in the naturally induced diagram
\[
\begin{tikzcd}[column sep=tiny, row sep=small]
&
\ringE^{**}(Z)\ar[ld,"(-)\smile t(\cE)\smile t(\cF)"'] \ar[rd,"(-)\smile t(\cE\times_S \cF)"] \ar[dd,"i_*",near end,shift left=1ex] \ar[dd,"(-)\smile \ol{t}(\cE)\smile \ol{t}(\cF)"',shift right=1ex,near end]
\\
\ringE^{**}(\Thom(\cE)\wedge \Thom(\cF))\ar[rr,crossing over,leftarrow]\ar[dd,"\cong"']&
&
\ringE^{**}(\Thom(\cE\times_S \cF))\ar[dd,"\cong"]
\\
&
\ringE^{**}(R)\ar[rd,leftarrow]\ar[ld,leftarrow]
\\
\ringE^{**}(P/(\Blow_X P,E)\wedge Q/(\Blow_Y Q,F))\ar[rr,leftarrow]&
&
\ringE^{**}(R/(\Blow_Z R,G))
\end{tikzcd}
\]
commutes, where $E$, $F$, and $G$ are the exceptional divisors.
The bottom triangle and the front square clearly commute.
The left parallelogram commutes by the constructions of $t(\cE)$ and $t(\cF)$ from $\ol{t}(\cE)$ and $\ol{t}(\cF)$.
The right parallelogram commutes owing to Lemma \ref{Ori.39}.
Therefore, it remains to show $i_*(-)=(-)\smile \ol{t}(\cE)\smile \ol{t}(\cF)$.
Since $i^*\colon \ringE^{**}(R)\to \ringE^{**}(X)$ is surjective by Theorem \ref{Ori.11}, 
it suffices to show the formula
\begin{equation}
\label{Ori.35.3}
i_*i^*(-)
=
i^*(-)\smile \ol{t}(\cE)\smile \ol{t}(\cF).
\end{equation}

\

Let $\cP$ and $\cQ$ be the universal quotient bundles on $P$ and $Q$.
We set $\cR:=\cP\times_S \cQ$.
The composite $p^* \cO \to p^*(\cE\oplus \cO)\to \cP$ gives a section of $\cP$ whose zero scheme is the zero section of $P$.
Similarly, 
we have a section of $\cQ$ whose zero scheme is the zero section of $Q$.
Combining these, we obtain a section $s$ of $\cR$.
This gives a cartesian square
\[
\begin{tikzcd}
Z\ar[r,"i"]\ar[d,"i"']&
R\ar[d,"s_0"]
\\
R\ar[r,"s"]&
\cR.
\end{tikzcd}
\]
Here $s_0$ denotes the zero section.
Proposition \ref{Ori.21} yields the formula
\begin{equation}
\label{Ori.35.4}
s^*s_{0*}=i_*i^*.
\end{equation}

\

On the other hand, 
applying Proposition \ref{Ori.21} to the cartesian square
\[
\begin{tikzcd}
R\ar[r,"\id"]\ar[d,"\id"']&
R\ar[d,"s"]&
\\
R\ar[r,"s"]&
\cR
\end{tikzcd}
\]
we deduce the formula
\[
s^*s_*=c_{r+s}(\cR).
\]
Together with Propositions \ref{Ori.20} and \ref{Ori.25}, we have 
\begin{equation}
\label{Ori.35.5}
s^*s_*=(-)\smile \ol{t}(\cE)\smile \ol{t}(\cF).
\end{equation}
Therefore, to deduce \eqref{Ori.35.3} from \eqref{Ori.35.4} and \eqref{Ori.35.5}, 
it suffices to show $s_0^*=s^*$.
This follows from Lemma \ref{Ori.23} as in the proof of Lemma \ref{Ori.24}.
\end{proof}

\subsection{Cohomology of Grassmannians}\label{ssec:Grass}
Throughout this subsection, 
we fix a scheme $S\in \Sch$ and an oriented homotopy commutative monoid $\ringE$ in $\inflogSH(S)$.
Let $\cT_{r,n}$ be the tautological bundle on $\Gr(r,n)$, 
and let $\cQ_{r,n}$ be the universal quotient bundle on $\Gr(r,n)$.
Recall that there is a canonical exact sequence
\[
0 \to \cT_{r,n} \to \cO^r \to \cQ_{r,n} \to 0.
\]
We define the ring
\begin{equation} \label{eq:ring_grass}
R_{r,n}
:=
\Z[x_1,\ldots,x_r,y_1,\ldots,y_{n-r}]/I_{r,n},
\end{equation}
where $I_{r,n}$ is the ideal generated by the polynomials $z_1,\ldots,z_n$ satisfying
\[
1+tz_1+\cdots+t^nz_n
=
(1+tx_1+\cdots + t^r x_r)(1+ty_1+\cdots +t^{n-r}y_{n-r}).
\]
It is well known that the singular cohomology ring $H^*(\Gr(r,\C^n))$ of the complex 
Grassmannian manifold $\Gr(r,\C^n)$ is isomorphic to $R_{r,n}$.

By Proposition \ref{Ori.20}, there is a canonical homomorphism
\[
R_{r,n}
\to
\ringE^{**}(\Gr(r,n))
\]
sending $x_i$ to $c_i(\cT_{r,n})$ and $y_i$ to $c_i(\cQ_{r,n})$.
Thus, there is a naturally induced graded ring homomorphism
\[
\varphi_{r,n}
\colon
\ringE^{**}\otimes_{\Z} R_{r,n}
\to
\ringE^{**}(\Gr(r,n)).
\]
Here we assign the degree $(2i,i)$ to $x_i$ and $y_i$.

\ 

We claim that $\varphi_{r,n}$ is an isomorphism.
In $\A^1$-homotopy theory, 
this was shown by Naumann-Spitzweck-{\O}stv{\ae}r \cite[Proposition 6.1]{zbMATH05663788}, using explicit $\A^1$-homotopies. Without them, we need to produce a suitable logarithmic version of the equation in
\cite[(18)]{zbMATH05663788}. Here is where the proof significantly differs from the classical one: our new key input is Lemma \ref{Ori.8} below.

\

Recall that the scheme $\Gr(r,n)$ parameterizes the rank $r$ subbundles of the rank $n$ trivial bundle $\cO^n$.
The formula $V\mapsto V\oplus \cO$ gives a closed immersion
\begin{equation}
\label{Ori.7.1}
i\colon \Gr(r-1,n-1)\to \Gr(r,n),
\end{equation}
and the formula $V\mapsto V\oplus 0$ gives a closed immersion
\begin{equation}
\label{Ori.7.7}
u\colon \Gr(r,n-1)\to \Gr(r,n).
\end{equation}
There are canonical isomorphisms of vector bundles
\begin{equation}
\label{Ori.7.10}
\begin{split}
&i^*\cT_{r,n} \cong \cT_{r-1,n-1}\oplus \cO,
\;
i^*\cQ_{r,n} \cong \cQ_{r-1,n-1},
\\
&u^*\cT_{r,n} \cong \cT_{r,n-1},
\text{ and }
u^*\cQ_{r,n} \cong \cQ_{r,n-1}\oplus \cO.
\end{split}
\end{equation}
Hence there are commutative diagrams
\begin{equation}
\label{Ori.7.5}
\begin{tikzcd}
\ringE^{**}\otimes_{\Z} R_{r,n}\ar[r,"1\otimes \alpha"]\ar[d,"\varphi_{r,n}"']&
\ringE^{**}\otimes_{\Z} R_{r-1,n-1}\ar[d,"\varphi_{r-1,n-1}"]
\\
\ringE^{**}(\Gr(r,n))\ar[r,"i^*"]&
\ringE^{**}(\Gr(r-1,n-1))
\end{tikzcd}
\end{equation}
and
\begin{equation}
\label{Ori.7.6}
\begin{tikzcd}
\ringE^{**}\otimes_{\Z} R_{r,n}\ar[r,"1\otimes \beta"]\ar[d,"\varphi_{r,n}"']&
\ringE^{**}\otimes_{\Z} R_{r,n-1}\ar[d,"\varphi_{r,n-1}"]
\\
\ringE^{**}(\Gr(r,n))\ar[r,"u^*"]&
\ringE^{**}(\Gr(r,n-1)),
\end{tikzcd}
\end{equation}
where $\alpha\colon R_{r,n}\to R_{r-1,n-1}$ sends $y_i$ to $y_i$ for all $1\leq i\leq n-r$ and $\beta\colon R_{r,n}\to R_{r,n-1}$ sends $x_i$ to $x_i$ for all $1\leq i\leq r$.

\begin{lem}
\label{Ori.8}
Let $X$ be the blow-up of $\Gr(r,n)$ along the closed immersion $i\colon \Gr(r-1,n-1)\to \Gr(r,n)$ in 
\eqref{Ori.7.1} with exceptional divisor $E$.
Then $X$ is a $\P^r$-bundle on $\Gr(r,n-1)$, and $(X,E)$ is a $(\P^r,\P^{r-1})$-bundle on $\Gr(r,n-1)$.
\end{lem}
\begin{proof}
We will use the Pl\"ucker embedding
\[
\Gr(r,n)\to \P^{\binom{n}{r}-1}
\]
sending $\mathrm{span}(v_1,\ldots,v_n)$ to $v_1\wedge \cdots \wedge v_n$.
For simplicity of notation, we set $p=\binom{n-1}{r}$ and $q=\binom{n-1}{r-1}$.
We express the coordinates of $\P^{\binom{n}{r}-1}=\P^{p+q-1}$ by $w_I$ for subsets 
$I\subset \{1,\ldots,n\}$ such that $|I|=r$.

\

Let $Y$ denote the blow-up of $\P^{p+q-1}$ along the closed subscheme defined 
by the equations $w_I=0$ for all $I$ containing $n$.
Then $X$ is the strict transform of $\Gr(r,n)$ in $Y$.
Let $e_1,\ldots,e_{p+q-1}$ be the standard coordinates in $\N^{p+q-1}$.
Note that projective space $\P^{p+q-1}$ is the toric variety associated with
\begin{gather*}
\sigma:=\Cone(e_1,\ldots,e_{p+q-1}),
\\
\Cone(e_1,\cdots,e_{i-1},e_{i+1},\ldots,e_{p+q-1},-f)\;\; (1\leq i\leq p+q),
\end{gather*}
where $f:=e_1+\cdots+e_{p+q-1}$.
By arranging the coordinates,
we may assume that $Y$ is the star subdivision relative to 
$\Cone(e_{p+1},\ldots,e_{p+q-1},-f)$.
Then $Y$ is the toric variety associated with a fan $\Sigma$ whose maximal cones are
\begin{gather*}
\sigma,
\;
\tau_{ij}:=\Cone(e_1,\cdots,e_{i-1},e_{i+1},\ldots,e_{j-1},e_{j+1},\ldots,e_{p+q-1},-e_j,-f)
\\
(1\leq i\leq p,\; p+1\leq j\leq p+q),
\\
\rho_i:=\Cone(e_1,\cdots,e_{i-1},e_{i+1},\ldots,e_{p+q-1},-e_1-\cdots-e_p)
\;\;
(p+1\leq i\leq p+q).
\end{gather*}

\

On the other hand, $\P^{p-1}$ is the toric variety associated with a fan $\Delta$ whose maximal cones are
\begin{gather*}
\sigma':=\Cone(e_1,\ldots,e_{p-1}),
\\
\tau_{i}':=\Cone(e_1,\ldots,e_{i-1},e_{i+1},\ldots,e_{p-1},-e_1-\cdots -e_p)\;\;(1\leq i\leq p).
\end{gather*}
Let $h\colon \N^{p+q-1}\to \N^{p-1}$ denote the homomorphism sending $e_i$ to $e_i$ if 
$1\leq i\leq p-1$ and to $0$ if $i\geq p$.
Then $h(\sigma)\subset \sigma'$ and $h(\tau_{ij}),h(\rho_i)\subset \tau_i'$.
This means that $h$ defines a morphism of fans $\Sigma\to \Delta$, 
which gives a morphism of scheme $p\colon Y\to \P^{p-1}$.

\

There is a dense open subscheme $U:=\Spec{A}$ of $\Gr(r,n-1)$, 
where
\begin{equation}
\label{Ori.8.1}
A:=\Z[a_{ij} \vert i\in [1,r],\;j\in [r+1,n-1]].
\end{equation}
Similarly, 
there is a dense open subscheme $V_0:=\Spec{B_0}$ of $\Gr(r,n)$, where
\begin{equation}
\label{Ori.8.2}
B_0:=\Z[a_{ij}\vert i\in [1,r],\; j\in [r+1,n]].
\end{equation}
The morphism $p$ maps $V_0$ onto $U$.
Since $p$ is proper, this implies that $p$ maps $X$ onto $\Gr(r,n-1)$.
Let $q\colon X\to \Gr(r,n-1)$ be the restriction of $p$.

\ 

We claim that $q$ is a $\P^r$-bundle.
Without loss of generality, it suffices to show that $q^{-1}(U)\to U$ is a $\P^r$-bundle.
There are dense open subschemes $V_s:=\Spec{B_s}$ of $\Gr(r,n)$ for $1\leq s\leq r$, where
\[
B_s:=\Z[b_{ij} \vert i\in [1,r],\; j\in \{s\}\cup [r+1,n-1]].
\]
We view $B_s$ as a subring of $\Z(a_{ij}\vert i\in [1,r],j\in [r+1,n])$.
By applying elementary row operations, we have identifications
\begin{equation}
\label{Ori.8.3}
b_{ij}
=
\left\{
\begin{array}{ll}
1/a_{sn} & \text{if }(i,j)=(s,s),
\\
-a_{in}/a_{sn} &\text{if }i\neq s\text{ and }j=s,
\\
a_{sj}/a_{sn} &\text{if }i=s\text{ and }j\neq s,
\\
a_{ij}-a_{in}a_{sj}/a_{sn} & \text{if }i,j\neq s.
\end{array}
\right.
\end{equation}
There is a dense open subscheme $V_s':=\Spec{B_s'}$ of $V_s\times_{\Gr(r,n)}X$, where
\[
B_s':=\Z[b_{ss}]\otimes \Z[b_{sj}/b_{ss}\vert j\in [r+1,n-1]] \otimes \Z[b_{ij}\vert i\in [1,r]-\{s\},\;j\in [r+1,n-1]].
\]
With the identifications \eqref{Ori.8.3} we have
\begin{equation}
\label{Ori.8.4}
B_s'=\Z[1/a_{sn}]\otimes \Z[a_{in}/a_{sn}\vert i\in [1,r]-\{s\}]\otimes \Z[a_{ij}\vert i\in [1,r],\; j\in [r+1,n-1]].
\end{equation}
Combine \eqref{Ori.8.1}, \eqref{Ori.8.2}, and \eqref{Ori.8.4} to deduce that the union $W$ of the open subschemes $V_0$ and $V_s'$ for $1\leq s\leq r$ of $X$ is isomorphic to $\P^r\times U$.
Hence $W$ is proper over $U$, so $W=q^{-1}(U)$.
This completes the proof that $q$ is a $\P^r$-bundle.

\ 

The log structure on $q^{-1}(U)\times_X (X,E)$ is the compactifying log structure associated 
with the open immersion $V_0\to q^{-1}(U)$.
Hence $(X,E)\to \Gr(r,n-1)$ is a $(\P^r,\P^{r-1})$-bundle since $V_0\cong U\times \A^r$.
\end{proof}

The images of $i\colon \Gr(r-1,n-1)\to \Gr(r,n)$ and $u\colon \Gr(r,n-1)\to \Gr(r,n)$ are disjoint.
Hence there is a naturally induced commutative diagram
\begin{equation}
\label{Ori.7.13}
\begin{tikzcd}[column sep=small, row sep=small]
&& \Gr(r,n-1)\ar[d,"v"]\ar[ld,"u"']
\\
\Gr(r-1,n-1)\ar[r,"i"]&
\Gr(r,n)\ar[r,leftarrow,"j"']&
(\Blow_{\Gr(r-1,n-1)}\Gr(r,n),E)\ar[d,"p"]
\\
&& \Gr(r,n-1),
\end{tikzcd}
\end{equation}
where $p$ and $j$ are the projections, and $E$ is the exceptional divisor.
From \eqref{Ori.16.3}, we have an exact sequence
\begin{equation}
\label{Ori.7.8}
\begin{split}
&\ringE^{*-2n+2r,*-n+r}(\Gr(r-1,n-1))
\xrightarrow{i_*}
\ringE^{**}(\Gr(r,n))
\\
\xrightarrow{j^*}
&\ringE^{**}(\Blow_{\Gr(r-1,n-1)}\Gr(r,n),E)
\to
\ringE^{*-2n+2r+1,*-n+r}(\Gr(r-1,n-1)).
\end{split}
\end{equation}
Since $p$ is a $(\P^r,\P^{r-1})$-bundle, $p^*$ is an isomorphism owing to Proposition \ref{ProplogSH.4}.
From $pv=\id$, we deduce that $v^*$ is an isomorphism.
Hence \eqref{Ori.7.8} becomes an exact sequence
\begin{equation}
\label{Ori.7.9}
\begin{split}
&\ringE^{*-2n+2r,*-n+r}(\Gr(r-1,n-1))
\xrightarrow{i_*}
\ringE^{**}(\Gr(r,n))
\\
\xrightarrow{u^*}
&\ringE^{**}(\Gr(r,n-1))
\xrightarrow{\delta}
\ringE^{*-2n+2r,*-n+r+1}(\Gr(r-1,n-1))
\end{split}
\end{equation}

\begin{lem}
\label{Ori.28}
Let $u,u'\colon \Gr(r-1,n)\rightrightarrows \Gr(r,n)$ be the two closed immersions given by the inclusions 
$\cO^{n-1}\rightrightarrows \cO^n$ sending $(x_1,\ldots,x_{n-1})$ to $(x_1,\ldots,x_{n-1},0)$ and 
$(x_1,\ldots,x_{n-2},0,x_{n-1})$.
Then $u^*=u'^*\colon \ringE^{**}(\Gr(r,n))\to \ringE^{**}(\Gr(r,n-1))$.
\end{lem}
\begin{proof}
Let $u''\colon \Gr(r,n-1)\rightarrow \Gr(r,n)$ be the closed immersion given by 
the inclusion $\cO^{n-1}\to \cO^n$ 
sending $(x_1,\ldots,x_{n-1})$ to $(x_1,\ldots,x_{n-1},x_{n-1})$.
As in \eqref{Ori.7.13}, 
there is a morphism $v''\colon \Gr(r,n-1)\to (\Blow_{\Gr(r-1.n-1)}\Gr(r,n),E)$ such that $pv''=\id$ and $u''=jv''$.
We also have $pv=\id$ and $u=jv$.

Since $p$ is a $(\P^r,\P^{r-1})$-bundle, $p^*$ is an isomorphism.
Hence, we get $v^*=v''^*$.
Compose with $j^*$ to deduce $u^*=u''^*$.
A similar argument shows that $u'^*=u''^*$.
\end{proof}

\begin{lem}
\label{Ori.31}
Let $i\colon \Gr(r-1,n-1)\to \Gr(r,n)$ be the closed immersion in \eqref{Ori.7.1}.
If $\eta\colon \Gr(r,n)\to \Gr(r,n)$ is the automorphism that permutes the last two coordinates, 
then $i^*=(\eta i)^*\colon \ringE^{**}(\Gr(r,n))\to \ringE^{**}(\Gr(r-1,n-1))$.
\end{lem}
\begin{proof}
Use the duality isomorphisms $\Gr(r,n)\cong \Gr(n-r,n)$, $\Gr(r-1,n-1)\cong \Gr(n-r,n-1)$, 
and apply Lemma \ref{Ori.28}.
\end{proof}

\begin{lem}
\label{Ori.29}
Let $i\colon \Gr(r-1,n-1)\to \Gr(r,n)$ be the closed immersion in \eqref{Ori.7.1}.
Then we have the formula
\[
i_*i^*
=
(-)\smile c_{n-r}(\cQ_{r,n}).
\]
\end{lem}
\begin{proof}
There is a cartesian square
\[
\begin{tikzcd}
\Gr(r-1,n-1)\ar[r,"i"]\ar[d,"i"']&
\Gr(r,n)\ar[d,"a"]
\\
\Gr(r,n)\ar[r,"a'"]&
\Gr(r+1,n+1).
\end{tikzcd}
\]
Here $a$ and $a'$ are induced by $V\mapsto V\oplus \cO$ and $V\mapsto \eta(V\oplus \cO)$, 
respectively, 
where $\eta$ is the automorphism that permutes the last two coordinates.
Proposition \ref{Ori.21} gives the formula
\begin{equation}
\label{Ori.29.1}
a'^*a_* = i_*i^*.
\end{equation}

\

The tangent bundle $\Tangent_{\Gr(r,n)}$ over $\Gr(r,n)$ participates in a canonical isomorphism of vector bundles
\[
\Tangent_{\Gr(r,n)}
\cong
\cT_{r,n}^*\otimes \cQ_{r,n}.
\]
Combine the exact sequence
\[
0
\to
\Tangent_{\Gr(r,n)}
\to
a^*\Tangent_{\Gr(r+1,n+1)}
\to
\Normal_{\Gr(r,n)}(\Gr(r+1,n+1))
\to
0
\]
with \eqref{Ori.7.10} to have
\begin{equation}
\label{Ori.29.3}
\Normal_{\Gr(r,n)}(\Gr(r+1,n+1))
\cong
\cO\otimes \cQ_{r,n}
\cong
\cQ_{r,n}.
\end{equation}
There is a cartesian square
\[
\begin{tikzcd}
\Gr(r,n)\ar[r,"\id"]\ar[d,"\id"']&
\Gr(r,n)\ar[d,"a"]
\\
\Gr(r,n)\ar[r,"a"]&
\Gr(r+1,n+1),
\end{tikzcd}
\]

\

Combine Proposition \ref{Ori.21} and \eqref{Ori.29.3} to conclude the formula
\begin{equation}
\label{Ori.29.2}
a^*a_*
=
(-)\smile c_{n-r}(\cQ_{r,n}).
\end{equation}
Since $a^*=a'^*$ owing to Lemma \ref{Ori.31}, 
we conclude by comparing \eqref{Ori.29.1} and \eqref{Ori.29.2}.
\end{proof}

For any bigraded ring $\bA^{**}$,
we write
\[
\bA^{**}(m,n)
:=
\bA^{*+m,*+n}
\]
to indicate the shift by $(m,n)\in \Z^2$.

\begin{lem}
\label{Ori.30}
There is a commutative square
\[
\begin{tikzcd}
\ringE^{**}\otimes_{\Z} R_{r-1,n-1}(-2n+2r,-n+r)\ar[r,"1\otimes \gamma"]\ar[d,"\varphi_{r-1,n-1}"']&
\ringE^{**}\otimes_{\Z} R_{r,n}\ar[d,"\varphi_{r,n}"]
\\
\ringE^{*-2n+2r,*-n+r}(\Gr(r-1,n-1))\ar[r,"i_*"]&
\ringE^{**}(\Gr(r,n)).
\end{tikzcd}
\]
Here $\gamma\colon R_{r-1,n-1}\to R_{r,n}$ is the multiplication map with $y_{n-r}$. 
\end{lem}
\begin{proof}
Combined with the commutative square \eqref{Ori.7.5}, we obtain 
\begin{equation}
\label{Ori.30.1}
\begin{tikzcd}
\ringE^{**}\otimes_{\Z} R_{r,n}(-2n+2r,-n+r)\ar[r,"1\otimes (\gamma\circ \alpha)"]\ar[d,"\varphi_{r,n}"']&
\ringE^{*-2n+2r,*-n+r}\otimes_{\Z} R_{r,n}\ar[d,"\varphi_{r,n}"]
\\
\ringE^{*-2n+2r,*-n+r}(\Gr(r,n))\ar[r,"i_*i^*"]&
\ringE^{**}(\Gr(r,n)).
\end{tikzcd}
\end{equation}
Since the homomorphism $\alpha \colon R_{r,n}\to R_{r-1,n-1}$ is surjective, it suffices to show that the square \eqref{Ori.30.1} commutes.
This follows from Lemma \ref{Ori.29}.
\end{proof}

\begin{thm}
\label{Ori.9}
There is a naturally induced ring isomorphism
\[
\varphi_{r,n}
\colon
\ringE^{**}\otimes_{\Z}R_{r,n}
\to
\ringE^{**}(\Gr(r,n)).
\]
\end{thm}
\begin{proof}
Theorem \ref{Ori.11} shows the claim for $n=1$.
We proceed by induction on $n>1$. 
For simplicity of notation, we set $s:=n-r$.
Combine \eqref{Ori.7.6} and Lemma \ref{Ori.30} to deduce the commutative diagram
\begin{equation}
\label{Ori.9.1}
\begin{tikzcd}[column sep=small]
0\ar[r]&
\ringE^{**}\otimes_{\Z} R_{r-1,n-1}(-2s,-s)\ar[d,"\varphi_{r-1,n-1}"']\ar[r,"1\otimes \gamma"]&
\ringE^{**}\otimes_{\Z} R_{r,n}\ar[d,"\varphi_{r,n}"]\ar[r,"1\otimes \beta"]&
\ringE^{**}\otimes_{\Z} R_{r,n-1}\ar[d,"\varphi_{r,n-1}"]\ar[r]&
0
\\
0\ar[r]&
\ringE^{*-2s,*-s}(\Gr(r-1,n-1))\ar[r,"i_*"]&
\ringE^{**}(\Gr(r,n))\ar[r,"u^*"]&
\ringE^{**}(\Gr(r,n-1))\ar[r]&
0.
\end{tikzcd}
\end{equation}
There is an exact sequence of singular cohomology groups of complex Grassmannian manifolds
\[
0
\to
H^{*-s}(\Gr(r-1,\C^{n-1}))
\to
H^*(\Gr(r,\C^n))
\to
H^*(\Gr(r,\C^{n-1}))
\to
0,
\]
which gives an exact sequence
\[
0
\to
R_{r-1,n-1}
\xrightarrow{\gamma}
R_{r,n}
\xrightarrow{\beta}
R_{r,n-1}
\to
0.
\]
Hence the upper row in \eqref{Ori.9.1} is exact.
By induction, $\varphi_{r-1,n-1}$ and $\varphi_{r,n-1}$ are isomorphisms.
Since $1\otimes \beta$ is surjective, $u^*$ is surjective.
Together with the exactness of \eqref{Ori.7.9}, we deduce that the lower row in \eqref{Ori.9.1} is exact.
The five lemma finishes the proof.
\end{proof}

\begin{rmk}
\label{Ori.34}
By a similar argument, one can show that the naturally induced homomorphism
\begin{equation}
\label{Ori.34.1}
\ringE^{**}(X)\otimes_{\Z} R_{r,n}\to \ringE^{**}(X\times \Gr(r,n))
\end{equation}
is an isomorphism for every $X\in \lSm/k$.
\end{rmk}

\begin{lem}
\label{Ori.32}
If $\colimit_{i\in \N} \cF_i$ is a filtered colimit in $\inflogSH(S)$, then there is a canonical exact sequence
\[
0
\to
\limitone_{i} \ringE^{*-1,*}(\cF_i)
\to
\ringE^{**}(\colimit_i \cF_i)
\to
\limit_i \ringE^{**}(\cF_i)
\to
0.
\]
\end{lem}
\begin{proof}
In $\inflogSH(S)$,
there is a canonical cofiber sequence
\[
\bigoplus_{j=0}^{i-1} \cF_j
\xrightarrow{f_i}
\bigoplus_{j=0}^i \cF_j
\to
\cF_i
\]
for all integer $i\geq 1$,
where $f_i=\id-\mathrm{shift}$.
Take $\colimit_i$ to obtain a canonical cofiber sequence
\[
\bigoplus_{i\in \N} \cF_i
\to
\bigoplus_{i\in \N} \cF_i
\to
\colimit_{i} \cF_i.
\]
By applying $\ringE^{**}$, we obtain the desired exact sequence
\begin{align*}
0
&\to
\coker\big(\prod_{i\in \N} \ringE^{*-1,*}(\cF_i) \to \prod_{i\in \N} \ringE^{*-1,*}(\cF_i)\big)
\to
\ringE^{**}(\colimit_i \cF_i)
\\
&\to
\ker \big(\prod_{i\in \N} \ringE^{**}(\cF_i) \to \prod_{i\in \N} \ringE^{**}(\cF_i)\big)
\to
0.
\end{align*}
\end{proof}

Let $\cT_r$ and $\cT$ be the tautological bundles on $\Gr(r)$ and $\Gr$, 
respectively.

\begin{cor}
\label{Ori.10}
There are canonical isomorphisms
\begin{equation}
\label{Ori.10.1}
\ringE^{**}\otimes_{\Z} \Z[[c_1,\ldots,c_r]]
\xrightarrow{\cong}
\ringE^{**}(\Gr(r))
\end{equation}
and
\begin{equation}
\label{Ori.10.2}
\ringE^{**}\otimes_{\Z} \Z[[c_1,c_2,\ldots]]
\xrightarrow{\cong}
\ringE^{**}(\Gr)
\end{equation}
sending $c_i$ to $c_i(\cT_r)$ and $c_i(\cT)$, 
respectively.
\end{cor}
\begin{proof}
Use the map $\beta$ in \eqref{Ori.7.6} to form an infinite sequence
\[
\cdots \to R_{r,n+1}\to R_{r,n}\to \cdots \to R_{r,r}.
\]
This is isomorphic to the infinite sequence of singular cohomology rings of complex Grassmannians
\[
\cdots \to H^*(\Gr(r,\C^{n+1}))\to H^*(\Gr(r,\C^{n}))\to \cdots \to H^*(\Gr(r,\C^{r})).
\]
The classical computation of the singular cohomology rings of complex Grassmannians shows that 
\[
\limit_n H^*(\Gr(r,\C^n)) \cong \Z[[c_1,\ldots,c_r]]
\text{ and }
\limitone_n H^*(\Gr(r,\C^n))=0.
\]
Hence we have
\begin{equation}
\label{Ori.10.3}
\limit_n R_{r,n} \cong \Z[[c_1,\ldots,c_r]]
\text{ and }
\limitone_n R_{r,n}=0.
\end{equation}
Applying Theorem \ref{Ori.9} and Lemma \ref{Ori.32} we deduce \eqref{Ori.10.1} 
and hence \eqref{Ori.10.2}.
\end{proof}

\begin{rmk}
Panin, Pimenov, and R\"ondigs proved Corollary \ref{Ori.10} in the setting of 
$\A^1$-homotopy theory, 
see \cite[Theorem 2.2]{zbMATH05367298}.
Their proof uses certain affine bundles of flag varieties, 
whose compactifications can be quite complicated.
This is the reason why we have followed the argument in \cite{zbMATH05663788}.
\end{rmk}

\subsection{Formal group laws and multiplicative orientations}
\label{formal}
Throughout this subsection, we fix $S\in \Sch$ and an oriented homotopy commutative monoid $\ringE$ in $\inflogSH(S)$.
In this section, 
we will define the associated formal group law of $\ringE$ and specialize in the example of $\logKGL$.

\begin{df}
\label{Ori.61}
Consider the Segre map
\begin{equation}
\label{Ori.61.1}
m\colon \P^\infty \times \P^\infty
\to
\P^\infty
\end{equation}
sending $([x_0:x_1:\cdots],[y_0:y_1:\cdots])$ to $[x_0y_0:x_0y_1:x_1y_0:\cdots]$.
Together with Remark \ref{Ori.34} and \eqref{Ori.10.1}, we have a homomorphism
\[
m
\colon
\ringE^{**}[[x]]
\to
\ringE^{**}[[x,y]].
\]
We set $\Formal_{\ringE}(x,y):=m(x)$,\index[notation]{FormalE @ $\Formal_{\ringE}(x,y)$} which we call the \emph{formal group law of $\ringE$}.\index{formal group law}
\end{df}

\begin{df}
\label{trace.5}
We say that $\ringE$ has a \emph{multiplicative orientation} \index{multiplicative orientation} if there exists 
an invertible element $u\in \ringE^{-2,-1}$ such that 
\[
\Formal_{\ringE}(x,y)=x+y-uxy.
\]
\end{df}

\begin{prop}
\label{trace.6}
Suppose $S\in \Sch$.
Let $\alpha\colon \ringE\to \ringE'$ be a homomorphism of homotopy commutative monoids in $\inflogSH(S)$.
If $\ringE$ has a multiplicative orientation, 
then $f$ induces a multiplicative orientation on $\ringE'$.
\end{prop}
\begin{proof}
The diagram
\[
\begin{tikzcd}
\ringE^{**}(\P^\infty)\ar[d,"\alpha"']\ar[r]&
\ringE^{**}(\P^\infty\times \P^\infty)\ar[d,"\alpha"]
\\
\ringE'^{**}(\P^\infty)\ar[r]&
\ringE'^{**}(\P^\infty \times \P^\infty)
\end{tikzcd}
\]
commutes.
Hence $\alpha \Formal_{\ringE} = \Formal_{\ringE'}\alpha$.
It follows that we have $\Formal_{\ringE'}(x,y)=x+y-\alpha(u)xy$.
To conclude, observe that $\alpha(u)$ is invertible since $u$ is invertible.
\end{proof}

\begin{exm}
\label{Ori.62}
Suppose $k$ is a perfect field admitting resolution of singularities.
Theorem \ref{K-theory.7} and \cite[Corollary 1.2.7]{PPR} give natural isomorphisms
\begin{equation}
\label{Ori.60}
\logKGL^{p,q}(X)
\cong
\Kth_{2q-p}(X-\partial X)
\cong
\KGL^{p,q}(X-\partial X)
\end{equation}
for all $X\in \lSm/k$ and integers $p$, $q$.
According to \cite[\S 2.7]{PPR}, $\KGL$ is oriented in the sense of \cite[Definition 1.2]{zbMATH05367298}.
Via \eqref{Ori.60}, the orientation on $\KGL$ induces an orientation on $\logKGL$.

Furthermore, there is a commutative square
\[
\begin{tikzcd}
\KGL^{**}(\clspace \G_m)\ar[d,"\cong"']\ar[r]&
\KGL^{**}(\clspace \G_m\times \clspace \G_m)\ar[d,"\cong"]
\\
\logKGL^{**}(\clspace \Gmm)\ar[r]&
\logKGL^{**}(\clspace \Gmm\times \clspace \Gmm).
\end{tikzcd}
\]
Together with \cite[Example 2.2]{zbMATH05549229}, 
we see that the formal group law on $\logKGL$ is given by
\begin{equation}
\Formal_{\logKGL}(x,y)=x+y-\beta xy,
\end{equation}
where $\beta\in \logKGL^{-2,-1}$ is the element corresponding to $\KGL$'s Bott element.
This shows that $\logKGL$ has a multiplicative orientation.
\end{exm}

\begin{const}
\label{ori.72}
Spitzweck and \O stv\ae r constructed in
In \cite[(4)]{zbMATH05549229}  a motivic lift of the Bott element, namely a morphism in $\infSH(S)$
\[
\beta\colon \Sigma_T^\infty (\clspace \G_m)_+
\to
\Sigma^{-2,-1}\Sigma_T^\infty (\clspace \G_m)_+
\]
naturally induced by
\[
\Sigma_T^\infty \P^1_+ \otimes \Sigma_T^\infty (\clspace \G_m)_+
\to
\Sigma_T^\infty (\clspace \G_m)_+ \otimes \Sigma_T^\infty (\clspace \G_m)_+
\to
\Sigma_T^\infty (\clspace \G_m)_+,
\]
where the second morphism is naturally induced by the multiplication $\G_m\times \G_m\to \G_m$.
Likewise,
we have a morphism in $\inflogSH(S)$
\[
\beta\colon \Sigma_T^\infty (\clspace \Gmm)_+
\to
\Sigma^{-2,-1}\Sigma_T^\infty (\clspace \Gmm)_+
\]
naturally induced by
\[
\Sigma_T^\infty \P^1_+ \otimes \Sigma_T^\infty (\clspace \Gmm)_+
\to
\Sigma_T^\infty (\clspace \Gmm)_+ \otimes \Sigma_T^\infty (\clspace \Gmm)_+
\to
\Sigma_T^\infty (\clspace \Gmm)_+,
\]
where the second morphism is obtained by Example \ref{Ori.73}.
Using $\beta$,
we define
\[
\Sigma^\infty_T (\clspace \Gmm)_+ [\beta^{-1}]
:=
\colimit(
\Sigma_T^\infty (\clspace \Gmm)_+
\xrightarrow{\beta}
\Sigma^{-2,-1}\Sigma_T^\infty (\clspace \Gmm)_+
\xrightarrow{\beta}
\cdots
).
\]
\end{const}

\begin{thm}
\label{Ori.71}
Suppose $k$ is a perfect field admitting resolution of singularities.
Then for every multiplicatively oriented homotopy commutative monoid $\bE$ in $\inflogSH(k)$,
there exists a unique homomorphism of homotopy commutative monoids
\[
\logKGL \to \bE
\]
that preserves orientations.
\end{thm}
\begin{proof}
We have natural equivalences in $\inflogSH(k)$
\[
\logKGL
\cong
\omega^*\KGL
\cong
\omega^*\Sigma_T^\infty (\P^\infty)_+[\beta^{-1}]
\cong
\Sigma_T^\infty (\P^\infty)_+[\beta^{-1}]
\cong
\Sigma_T^\infty (\clspace\Gmm)_+[\beta^{-1}],
\]
where the first (resp.\ second, resp.\ third, resp.\ fourth) equivalence is due to \eqref{K-theory.16.3} (resp.\ \cite[Lemmas 3.5--3.7]{zbMATH05549229}, resp.\ \cite[Theorem 1.1(2)]{Doosgroups}, resp.\ Theorem \ref{Pic.9}).
Argue as in \cite[Lemmas 3.5--3.7]{zbMATH05549229} to conclude.
\end{proof}

\subsection{Logarithmic cobordism and orientations}
Throughout this subsection, we fix $S\in \Sch$.
First, 
we construct the logarithmic cobordism spectrum $\spectlogMGL$ over $S$ following Voevodsky's approach to 
algebraic cobordism in \cite[\S 6.3]{zbMATH01194164}, 
which in turn follows that for complex cobordism.
Theorem \ref{MGL.5} establishes a  one-to-one correspondence between ring homomorphisms 
$\spectlogMGL\to \ringE$ and orientations on any homotopy commutative monoid $\ringE$ in $\inflogSH(S)$ 
--- the logarithmic version of an analogous result for algebraic cobordism shown in 
\cite{zbMATH01648745}, \cite{zbMATH05367298}, \cite{zbMATH02216798}.

\

Let $\cT_{r,n}$ and $\cT_r$ denote the tautological bundle on the Grassmannian $\Gr(r,n)$ and $\Gr(r)$, 
respectively.
The standard embedding $\cO^n \to \cO^n \oplus \cO$ gives a closed immersion $\Gr(r,n)\to \Gr(r,n+1)$, 
and thus a morphism of Thom spaces
\begin{equation}
\label{MGL.0.1}
\Thom(\cT_{r,n})\to \Thom(\cT_{r,n+1})
\end{equation}
in $\inflogH_{\ast}(S)$.
From this, we define \index[notation]{logMGLr @ $\logMGL_r$}
\[
\logMGL_r:=\Thom(\cT_r):=\colimit_n \Thom(\cT_{r,n}).
\]

There is a closed immersion
\begin{equation}
\label{MGL.0.11}
\mu_{r,s}
\colon
\Gr(r,rn) \times \Gr(s,sn)
\to
\Gr(r+s,(r+s)n)
\end{equation}
sending a pair of $V\subset \cO^{rn}$ and $W\subset \cO^{sn}$ to $\alpha(V\oplus W)\subset \cO^{(r+s)n}$,
where $\alpha\colon \cO^{rn}\oplus \cO^{sn}\to \cO^{(r+s)n}$ is the isomorphism shuffling the coordinates
in the sense of
\[
(V_1\oplus \cdots \oplus V_n) \oplus (W_1\oplus \cdots \oplus W_n)
\xrightarrow{\simeq}
V_1\oplus W_1\oplus \cdots \oplus V_n \oplus W_n
\]
with $V_i\simeq \cO^r$ and $W_i\simeq \cO^s$ for $i=1,\ldots,n$.
Since $\mu_{r,s}^*(\cT_{r+s,(r+s)n})\cong \cT_{r,rn}\times \cT_{s,sn}$, Proposition \ref{Ori.40} gives a map
\begin{equation}
\label{MGL.0.2}
\Thom(\cT_{r,rn})\wedge \Thom(\cT_{s,sn}) \to \Thom(\cT_{r+s,(r+s)n})
\end{equation}
in $\inflogH_{\ast}(S)$.
The induced square
\[
\begin{tikzcd}
\Thom(\cT_{r,rn})\wedge \Thom(\cT_{s,sn}) 
\ar[d]\ar[r]&
\Thom(\cT_{r+s,(r+s)n})\ar[d]
\\
\Thom(\cT_{r,r(n+1)})\wedge \Thom(\cT_{s,s(n+1)}) 
\ar[r]&
\Thom(\cT_{r+s,(r+s)(n+1)})
\end{tikzcd}
\]
commutes,
where the vertical morphisms are obtained from \eqref{MGL.0.1}.
Taking colimits in \eqref{MGL.0.2} yields another map in $\inflogH_{\ast}(S)$, 
namely
\begin{equation}
\label{MGL.0.7}
\logMGL_r\wedge \logMGL_s \to \logMGL_{r+s}.
\end{equation}

The closed immersion $\Gr(s,s)\to \Gr(s,sn)$ gives a map
\[
\Thom(\A^s)\simeq \Thom(\cT_{s,s}) \to \Thom(\Gr(s,sn)).
\]
Proposition \ref{Ori.64} gives an equivalence $T^s\simeq \Thom(\A^s)$ in $\inflogH_{\ast}(S)$.
Together with \eqref{MGL.0.2}, we have a map
\[
\Thom(\cT_{r,rn}) \wedge T^s
\to
\Thom(\cT_{r+s,(r+s)n})
\]
in $\inflogH_{\ast}(S)$.
Taking colimits yields another map in $\inflogH_{\ast}(S)$, 
namely
\begin{equation}
\label{MGL.0.8}
\logMGL_r\wedge T^{\wedge s}\to \logMGL_{r+s}.
\end{equation}

\begin{df}
Thanks to \eqref{alter.1.5}, we can define \index[notation]{logMGLspect @ $\spectlogMGL$}\index{logarithmic cobordism spectrum}
\[
\spectlogMGL:=(\logMGL_0,\logMGL_1,\ldots)\in \inflogSH(S)
\]
whose bonding maps are given by \eqref{MGL.0.8} when $s=1$.
This is called the \emph{logarithmic cobordism spectrum} over $S$.
In $\inflogSH(S)$, there is a canonical equivalence
\begin{equation}
\label{MGL.1.6}
\spectlogMGL
\simeq
\colimit_r \Sigma^{-2r,-r}\Sigma_T^\infty\logMGL_r.
\end{equation}
\end{df}

Recall the natural transformation
$\lambda_\sharp \colon \MS\to \inflogSH$ in Construction \ref{compth.16}.
\begin{prop}
There is an equivalence in $\inflogSH(S)$
$$
\lambda_\sharp(\spectMGL)\simeq \spectlogMGL.
$$
In particular,
we can impose a commutative monoid structure on $\spectlogMGL$.
\end{prop}
\begin{rmk} The existence of $\lambda_\sharp$ gives us a quick way to equip our spectrum $\spectlogMGL$ with the structure of a commutative monoid. 
    We remark, however, that this fact and Theorem \ref{MGL.5} below
\emph{are not necessary} to apply the results of \S \ref{ori}-\ref{formal}. In particular, the results in \cite{BLMP}, \cite{BLPO}, \cite{BLPO2} are independent from \cite{AHI}.
\end{rmk}
\begin{proof}
Since $\lambda_\sharp$ preserves colimits and $\Gr(n)$ is a colimit of smooth schemes,
the equivalence $\lambda_\sharp (\Sigma^\infty X_+)\simeq \Sigma^\infty X_+$ for $X\in \Sm/S$ yields an equivalence
\[
\lambda_\sharp (\Sigma^\infty \Gr(n))\simeq \Sigma^\infty \Gr(n).
\]
Furthermore, compare Proposition \ref{Ori.41} with the definition of Thom spaces in \cite[\S 3]{AHI} to show that $\lambda_\sharp$ preserves Thom spaces.
Together with \cite[Proposition 7.1]{AHI},
we obtain the desired equivalence.
This imposes a commutative monoid structure on $\spectlogMGL$ since $\MGL\in \MS_S$ is a commutative monoid and $\lambda_\sharp$ is symmetric monoidal.
\end{proof}

From now on $\ringE$ is a homotopy commutative monoid in $\inflogSH(S)$.

\begin{prop}
\label{MGL.2}
If $\ringE$ is oriented, 
then there is an isomorphism
\[
\ringE^{*+2r,*+r}(\logMGL_r)
\cong
\ringE^{**}[[c_1,\ldots,c_r]].
\]
\end{prop}
\begin{proof}
Combine Theorem \ref{Ori.16} and Lemma \ref{Ori.32} to have an exact sequence
\[
0
\to
\ringE^{**}\otimes \limitone_{n} R_{r,n}
\to
\ringE^{*+2r,*+r}(\logMGL_r)
\to
\ringE^{**}\otimes \limit_n R_{r,n}
\to
0.
\]
Apply \eqref{Ori.10.3} to conclude.
\end{proof}

\begin{prop}
\label{MGL.3}
If $\ringE$ is oriented, 
then there are isomorphisms
\[
\ringE^{**}(\spectlogMGL)
\cong
\limit_r \ringE^{*+2r,*+r}(\logMGL_r)
\cong
\limit_r \ringE^{**}(\Gr(r))
\cong
\ringE^{**}[[c_1,c_2,\ldots]].
\]
\end{prop}
\begin{proof}
There is a commutative diagram with vertical isomorphisms
\begin{equation}
\label{MGL.3.1}
\begin{tikzcd}[column sep=small]
\ringE^{**}(\Gr(r+1))\ar[r]\ar[d,"(-)\smile t(\cT_{r+1})"']&
\ringE^{**}(\Gr(r))\ar[d,"(-)\smile t(\cT_r\oplus \cO)"]\ar[r,"\id"]&
\ringE^{**}(\Gr(r))\ar[d,"(-)\smile t(\cT_r)"]
\\
\ringE^{*+2r+2,*+r+1}(\logMGL_{r+1})\ar[r]&
\ringE^{*+2r+2,*+r+1}(\Sigma^{2,1}\logMGL_r)\ar[r,"\cong"]&
\ringE^{*+2r,*+r}(\logMGL_r).
\end{tikzcd}
\end{equation}
This naturally induces isomorphisms
\[
\limit_r \ringE^{**}(\Gr(r))
\xrightarrow{\cong}
\limit_r \ringE^{*+2r,*+r}(\logMGL_r)
\]
and
\[
\limitone_r \ringE^{**}(\Gr(r))
\xrightarrow{\cong}
\limitone_r \ringE^{*+2r,*+r}(\logMGL_r).
\]
To finish the proof, apply Lemma \ref{Ori.32} and Corollary \ref{Ori.10}.
\end{proof}

\begin{prop}
\label{MGL.4}
If $\ringE$ is oriented, 
then there are isomorphisms
\[
\ringE^{**}(\logMGL_r\wedge \logMGL_s)
\cong
\ringE^{**}(\Gr(r)\wedge \Gr(s))
\cong
\ringE^{**}[[c_1',\ldots,c_r',c_1'',\ldots,c_s'']]
\]
and
\begin{align*}
\ringE^{**}(\spectlogMGL\wedge \spectlogMGL)
&\cong
\limit_{r,s}\ringE^{**}(\logMGL_r\wedge \logMGL_s)
\\
&\cong
\limit_{r,s}\ringE^{**}(\Gr(r)\wedge \Gr(s))
\cong
\ringE^{**}[[c_1',c_1'',c_2',c_2'',\ldots]].
\end{align*}
\end{prop}
\begin{proof}
Argue as in Propositions \ref{MGL.2} and \ref{MGL.3}, but use \eqref{Ori.34.1}.
\end{proof}

\begin{prop}
\label{MGL.8}
If $\ringE$ is oriented, 
then the diagram
\[
\begin{tikzcd}
\ringE^{**}(\Gr(r+s))\ar[r]\ar[d,"(-)\smile t(\cT_{r+s})"']&
\ringE^{**}(\Gr(r)\wedge \Gr(s))\ar[d,"(-)\smile (t(\cT_r)\wedge t(\cT_s))"]
\\
\ringE^{**}(\Thom(\cT_{r+s}))\ar[r]&
\ringE^{**}(\Thom(\cT_r)\wedge \Thom(\cT_s))
\end{tikzcd}
\]
commutes, where the horizontal maps are naturally induced by \eqref{MGL.0.11} and \eqref{MGL.0.7}.
\end{prop}
\begin{proof}
Follows from Proposition \ref{Ori.35}.
\end{proof}

\begin{prop}
\label{MGL.9}
Suppose that $\ringE$ is oriented.
With respect to the isomorphisms in {\rm Propositions \ref{MGL.3} and \ref{MGL.4}}, 
the map
\[
\ringE^{**}(\spectlogMGL)
\to
\ringE^{**}(\spectlogMGL\wedge \spectlogMGL), 
\]
naturally induced by the multiplication map on $\spectlogMGL$,  
sends any power series $f(c_1,c_2,\ldots)$ to $f(d_1,d_2,\ldots)$ for
\[
d_m:=\sum_{i=0}^m c_i c_{m-i}'
\]
for all integer $m\geq 1$.
\end{prop}
\begin{proof}
By Proposition \ref{MGL.8}, it suffices to show that the map
\[
\ringE^{**}(\Gr(r+s,n+l))
\to
\ringE^{**}(\Gr(r,n)\wedge \Gr(s,l))
\]
sends $c_m(\cT_{r+s,n+l})$ to $\sum_{i=0}^m c_i(\cT_{r,n})c_{n-i}(\cT_{s,l})$.
This follows from Proposition \ref{Ori.20}.
\end{proof}

\begin{prop}
\label{MGL.6}
The logarithmic cobordism spectrum $\spectlogMGL$ is oriented.
\end{prop}
\begin{proof}
By \eqref{MGL.1.6}, we have a morphism
\begin{equation}
\label{MGL.6.1}
\Sigma^{-2,-1}\Sigma_T^\infty \Thom(\cT_1)=\Sigma^{-2,-1}\Sigma_T^\infty \logMGL_1\to \spectlogMGL.
\end{equation}
This gives a class $t_\infty \in \spectlogMGL^{2,1}(\Thom(\cT_1))$.
When $r=0$ and $s=1$, 
the morphism \eqref{MGL.0.8} becomes $g\colon T\to \Thom(\cT_1)$, 
which is obtained by restricting $\cT_1$ to the distinguished point.
The unit map $\unit\to \spectlogMGL$ is the composite of $\Sigma^{-2,-1}\Sigma_T^\infty g$ and \eqref{MGL.6.1}.
It follows that the restriction of $t_\infty$ to $T$ is $T\wedge \unit$, and thus $t_\infty$ is a Thom orientation.
\end{proof}

\begin{prop}
\label{MGL.7}
Suppose that $f\colon \ringE\to \ringE'$ is a map of homotopy commutative monoids in $\inflogSH(S)$.
If $\ringE$ is oriented, then $f$ induces an orientation on $\ringE'$.
\end{prop}
\begin{proof}
Suppose $t_\infty\in \ringE^{2,1}(\Thom(\cT_1))$ is a Thom orientation.
The naturally induced homomorphism $\ringE^{2,1}(\Thom(\cT_1))\to \ringE'^{2,1}(\Thom(\cT_1))$ 
sends $t_\infty$ to some class $f(t_\infty)\in \ringE'^{2,1}(\Thom(\cT_1))$.
The restriction of $t_\infty$ to $T$ is $T\wedge \unit$.
Since $f$ is a ring homomorphism, the restriction of $f(t_\infty)$ to $T$ is $T\wedge \unit$.
Hence $f(t_\infty)$ is a Thom orientation.
\end{proof}

\begin{thm}
\label{MGL.5}
There is a one-to-one correspondence between the set of ring maps $\spectlogMGL\to \ringE$ 
and the set of orientations of $\ringE$.
\end{thm}
\begin{proof}
Due to Proposition \ref{Ori.53}, we only need to deal with Thom orientations.
Let $f\colon \spectlogMGL\to \ringE$ be a ring map in $\inflogSH(S)$.
By Propositions \ref{MGL.6} and \ref{MGL.7}, 
there is a naturally induced Thom orientation $\psi(f)\in \ringE^{2,1}(\Thom(\cT_1))$.
Observe that $\psi(f)$ is the composite of
\begin{equation}
\label{MGL.5.1}
\Sigma^{-2,-1}\Sigma_T^\infty \Thom(\cT_1)\to \spectlogMGL\xrightarrow{f}\ringE,
\end{equation}
where the first map is \eqref{MGL.6.1}.

\

On the other hand, 
suppose that $t_\infty\in \ringE^{2,1}(\Thom(\cT_1))$ is a Thom orientation.
Then the associated Thom classes $t(\cT_r)\in \ringE^{2r,r}(\logMGL_r)$ are 
compatible with the maps 
$\ringE^{2r+2,r+1}(\logMGL_{r+1})\to \ringE^{2r,r}(\logMGL_r)$ since 
\eqref{MGL.3.1} commutes.
Proposition \ref{MGL.3} furnishes
a class of $\ringE^{00}(\spectlogMGL)$,
which yields a map $\varphi(t_\infty)\colon \spectlogMGL\to \ringE$.
Let us show that $\varphi(t_\infty)$ is a ring map.
In effect, we need to show the diagram in $\inflogSH(S)$
\[
\begin{tikzcd}
\spectlogMGL\wedge \spectlogMGL\ar[r]\ar[d,"\varphi(t_\infty)\wedge\varphi(t_\infty)"']&
\spectlogMGL\ar[d,"\varphi(t_\infty)"]
\\
\ringE \wedge \ringE \ar[r]&
\ringE
\end{tikzcd}
\]
commutes, 
or equivalently that $\varphi(c_\infty)$ maps to $\varphi(c_\infty)\wedge \varphi(c_\infty)$ under 
\[
\ringE^{00}(\spectlogMGL)\to \ringE^{00}(\spectlogMGL\wedge \spectlogMGL).
\]
Using Propositions \ref{MGL.2} and \ref{MGL.3}, we reduce to showing 
\[
\ringE^{2r+2s,r+s}(\logMGL_{r+s})\to \ringE^{2r+2s,r+s}(\logMGL_r\wedge \logMGL_s)
\]
maps $t(\cT_{r+s})$ to $t(\cT_r)\wedge t(\cT_{s})$.
This follows from Proposition \ref{Ori.35}.

\

It remains to check $\varphi(\psi(f))=f$ and $\psi(\varphi(t_\infty))=t_\infty$.
When $t_\infty$ is given, the Thom classes $t(\cT_r)$ naturally induce $\varphi(t_\infty)$.
From the formulation \eqref{MGL.5.1}, we have $\psi(\varphi(t_\infty))=t(\cT_1)$.
Together with Proposition \ref{Ori.54}, we have $\psi(\varphi(t_\infty))=t_\infty$.

Given a map $f$ we may compose with
\[
\Sigma^{-2r,-r}\Sigma_T^\infty \Thom(\cT_r)\to \spectlogMGL
\]
to define $a_r\in \ringE^{2r,r}(\Thom(\cT_r))$.
On the other hand, $a_1=\psi(f)$ defines Thom classes $t(\cT_r)$.
By Proposition \ref{MGL.3},
the classes $a_r$ and $t(\cT_r)$ naturally induce $f$ and $\varphi(\psi(f))$, 
respectively.
Therefore, it suffices to show $a_r=t(\cT_r)$ for every integer $r\geq 1$.
Since $f$ is assumed to be a ring map, the diagram
\[
\begin{tikzcd}
\spectlogMGL^{\wedge r}\ar[r]\ar[d,"f^{\wedge r}"']&
\spectlogMGL\ar[d,"f"]
\\
\ringE^{\wedge r}\ar[r]&
\ringE
\end{tikzcd}
\]
commutes.
Thus $a_r$ maps to $a_1^{\wedge r}$ under 
\begin{equation}
\label{MGL.5.2}
\ringE^{2r,r}(\Thom(\cT_r))
\to
\ringE^{2r,r}(\Thom(\cT_1)^{\wedge r}).
\end{equation}

\

Owing to Proposition \ref{Ori.35}, the diagram
\[
\begin{tikzcd}
(\Sigma^{-2,-1}\Sigma_T^\infty \Thom(\cT_1))^{\wedge r}\ar[r]\ar[rd]&
\Sigma^{-2r,-r}\Sigma_T^\infty \Thom(\cT_r)\ar[d]
\\
&
\ringE
\end{tikzcd}
\]
commutes, and \eqref{MGL.5.2} sends $t(\cT_r)$ to $t(\cT_1)^{\wedge r}$.
Since $a_1=t(\cT_1)$, 
in order to conclude $a_r=t(\cT_r)$ it suffices to show that \eqref{MGL.5.2} is injective.

As in Proposition \ref{MGL.9}, \eqref{MGL.5.2} is isomorphic to the map
\[
\ringE^{**}[[c_1,\ldots,c_r]]
\to
\ringE^{**}[[x_1,\ldots,x_r]]
\]
sending any power series $f(c_1,\ldots,c_r)$ to $f(d_1,\ldots,d_r)$, where
\[
d_i=\sum_{j_1< \cdots < j_i}x_{j_1}\cdots x_{j_i}
\]
for $i=1,\ldots,r$.
This description implies that \eqref{MGL.5.2} is injective.
\end{proof}

\subsection{The log motivic Eilenberg-MacLane spectrum}

Throughout this subsection, $k$ is a field and $\Lambda$ is a commutative ring.
We set out to define the logarithmic analogue $\spectlogML$ of the motivic Eilenberg-MacLane spectrum 
introduced in \cite{zbMATH01194164}.
We show that $\spectlogML$ is an $\E_\infty$-ring in $\inflogSH(k)$.
Assuming that $k$ is a perfect field admitting resolution of singularities, 
we will show that $\spectlogML$ is oriented.

\

Recall from \eqref{logDA.4.1} the infinite suspension functor
\[
\Sigma_{\Gmm}^\infty
\colon
\inflogDMeff(k,\Lambda)
\to
\inflogDM(k,\Lambda).
\]
For $X\in \lSm/k$, we define
\[
M(X)
:=
\Sigma_{\Gmm}^\infty M(X)\in \inflogDM(k,\Lambda),
\]
where on the right-hand side, 
$M(X)$ denotes the log motive of $X$ in $\inflogDMeff(k,\Lambda)$, 
see Definition \ref{logDA.4}.
The permutation
$$
(123)\colon M(k)(1)^{\otimes 3}\to M(k)(1)^{\otimes 3}
$$ 
in $\inflogDMeff(k,\Lambda)$ is equivalent to the identity by Proposition \ref{Ori.56} and Theorem \ref{Ori.68}.
If $X\in \lSm/k$, then $M(X)$ is compact in $\inflogDMeff(k,\Lambda)$ \cite[Remark 5.2.6]{logDM}.
Thus by \cite[Th\'eor\`emes 4.3.61, 4.3.79]{Ayo07}, 
we obtain a canonical equivalence of spaces
\begin{equation}
\label{ML.0.1}
\Map_{\inflogDM(k,\Lambda)}(M(X)(m),M(Y))
\simeq
\colimit_{n\to \infty}\Map_{\inflogDMeff(k,\Lambda)}(M(X)(m+n),M(Y)(n))
\end{equation}
for all $X,Y\in \lSm/k$ and $m\in\mathbb{Z}$.
The functor $\omega\colon \lSm/k\to \Sm/k$ naturally induces an adjunction
\[
\omega_\sharp
:
\inflogDM(k,\Lambda)
\rightleftarrows
\infDM(k,\Lambda)
:
\omega^*.
\]
The functor $\omega_\sharp$ is monoidal and sends $M(X)$ to $M(X-\partial X)$ for all $X\in \lSm/k$.

\begin{exm}
\label{ML.6}
For every integer $j$, let $\Omega^j$ be the presheaf on $\lSm/k$ given by $\Omega^j(X):=\Omega_{X/k}^j(X)$.
We set $\Omega^j(X):=0$ if $j<0$.
According to \cite[Theorem 9.7.1]{logDM}, we can regard $\Omega^j$ as an object of $\inflogDMeff(k)$ such that
\[
\hom_{\inflogDMeff(k)}(M(X),\Omega^j[n])
\cong
H_{Zar}^n(X,\Omega^j)
\cong
H_{sNis}^n(X,\Omega^j)
\]
for all $X\in \lSm/k$ and $n\in\mathbb{Z}$.
By \cite[Properties 2.3(b)]{EVbuch}, there is an exact sequence of sheaves
\[
0
\to
\Omega_{X\times \P^1/k}^j
\to
\Omega_{X\times \boxx/k}^j
\to
a_*\Omega_{X/k}^{j-1}
\to
0
\]
for all $X\in \SmlSm/k$, where $a\colon X\to X\times \P^1$ is the closed immersion at $\infty$.
Apply $\R_{Zar}\Gamma$ to this sequence and use the strict $\boxx$-invariance of $\Omega^j$ in 
\cite[Corollary 9.2.2]{logDM} to obtain a cofiber sequence
\begin{equation}
\label{ML.6.2}
\R_{Zar}\Gamma(X\times \P^1,\Omega^j)
\to
\R_{Zar}\Gamma(X,\Omega^j)
\to
\R_{Zar}\Gamma(X,\Omega^{j-1})
\end{equation}
for all $X\in \SmlSm/k$.
Due to \cite[Corollary 9.2.5, Proposition A.10.2]{logDM}, 
\eqref{ML.6.2} extends to a cofiber sequence for all $X\in \lSm/k$.
Hence we have an equivalence in $\inflogDMeff(k)$
\begin{equation}
\Omega_{\Gmm}\Omega^j
\simeq
\Omega^{j-1}.
\end{equation}
Use this as bonding maps to define the $\Gmm$-spectrum
\begin{equation}
\bOmega^j
:=
(\Omega^j,\Omega^{j+1},\ldots)\in \inflogDM(k).
\end{equation}
We have a canonical isomorphism
\begin{equation}
\hom_{\inflogDM(k)}(M(X)(i),\bOmega^j[n])
\cong
H_{Zar}^n(X,\Omega^{j-i})
\end{equation}
for all $X\in \lSm/k$ and integers $i$, $j$, $n$.
This shows that $\bOmega^j$ is not $\A^1$-local.
Hence $\bOmega^j$ is not in the essential image of $\omega^*$.
\end{exm}

\begin{prop}
\label{ML.1}
Suppose $k$ is a perfect field admitting resolution of singularities.
For all $X\in \lSm/k$ with $X$ proper over $k$, the naturally induced map
\[
M(X)\to \omega^*\omega_\sharp M(X)
\]
in $\inflogDM(k)$ is an equivalence.
\end{prop}
\begin{proof}
We need to show that the naturally induced map
\begin{equation}
\Map_{\inflogDM(k,\Lambda)}(M(Y)(m),M(X))
\to
\Map_{\infDM(k,\Lambda)}(M(Y-\partial Y)(m),M(X-\partial X))
\end{equation}
is an equivalence of spaces for all $Y\in \lSm/k$.
Combine \cite[Theorem 8.2.11]{logDM} and \eqref{ML.0.1} to show this.
\end{proof}

\begin{rmk}
One can ask whether or not Proposition \ref{ML.1} holds for $\inflogSH$.
This is a question about expressing
\[
\Map_{\inflogSH(k)}(M(Y)(m),M(X))
\]
in terms of mapping spaces in $\infSH(k)$ for all $X,Y\in \lSm/k$.
\end{rmk}

\begin{prop}
\label{ML.2}
Suppose $k$ is a perfect field admitting resolution of singularities.
Then the functor $\omega^*\colon \infDM(k)\to \inflogDM(k)$ is monoidal.
\end{prop}
\begin{proof}
Note that $\omega^*$ is lax monoidal since it is right adjoint to the monoidal functor $\omega_\sharp$.
To show $\omega^*$ is monoidal, it suffices to show that the map
\begin{equation}
\label{ML.2.1}
\omega^*M(X)\otimes \omega^*M(Y) \to \omega^*M(X\times Y)
\end{equation}
is an equivalence in $\infDM(k)$ for all $X,Y\in \Sm/k$.

By resolution of singularities, 
there exist $\ol{X},\ol{Y}\in \SmlSm/k$ such that $\ol{X}-\partial\ol{X}\cong X$, 
$\ol{Y}-\partial\ol{Y}\cong Y$, 
where $\ol{X}$ and $\ol{Y}$ are proper over $k$.
Proposition \ref{ML.1} gives equivalences in $\inflogDM(k)$
\[
\omega^*M(X)
\simeq
M(\ol{X}),
\;
\omega^*M(Y)
\simeq
M(\ol{Y}),
\text{ and }
\omega^*M(X\times Y)
\simeq
M(\ol{X}\times \ol{Y}).
\]
Combine these to deduce that \eqref{ML.2.1} is an equivalence.
\end{proof}

The right adjoints $\beta^*$ and $\gamma^*$ of the functors $\beta_\sharp$ and $\gamma_\sharp$ 
in \eqref{logDA.3.3} and \eqref{logDA.5.2} form part of the commutative diagram
\[
\begin{tikzcd}
\infSH(k)\ar[r,"\beta^*",leftarrow]\ar[d,"\omega^*"',leftarrow]&
\infDA(k,\Lambda)\ar[r,"\gamma^*",leftarrow]\ar[d,"\omega^*",leftarrow]&
\infDM(k,\Lambda)\ar[d,"\omega^*",leftarrow]
\\
\inflogSH(k)\ar[r,"\beta^*",leftarrow]&
\inflogDA(k,\Lambda)\ar[r,"\gamma^*",leftarrow]&
\inflogDM(k,\Lambda),
\end{tikzcd}
\]
where the functors in the upper row are defined similarly to those in the lower row.

\begin{df}
The \emph{logarithmic motivic Eilenberg-MacLane spectrum over $k$ with 
$\Lambda$-coefficients}\index{logarithmic Motivic Eilenberg-MacLane spectrum}\index[notation]{logML @ $\spectlogML$} 
is defined as 
\[
\spectlogML
:=
\beta^*\gamma^*M(k)\in \inflogSH(k)
\]
\end{df}

Similarly, 
$\spectML:=\beta^*\gamma^*M(k)\in \infSH(k)$, 
which is equivalent to Voevodsky's motivic Eilenberg-MacLane spectrum in 
\cite[\S 6.1]{zbMATH01194164} by \cite[Example 3.4, Lemma 4.6]{DRO}.

\begin{prop}
\label{ML.5}
If $k$ is a perfect field that admits resolution of singularities, then there is an equivalence in $\inflogSH(k)$
\[
\spectlogML
\simeq
\omega^*\spectML.
\]
\end{prop}
\begin{proof}
We have equivalences in $\inflogSH(k)$
\[
\omega^*\spectML
\simeq
\omega^*\beta^*\gamma^*M(k)
\simeq
\beta^*\gamma^*\omega^*M(k).
\]
Proposition \ref{ML.1} finishes the proof.
\end{proof}

\begin{prop}
\label{ML.3}
Suppose $\ringE$ is an oriented $\E_\infty$-ring in $\infSH(k)$. 
Then $\omega^*\ringE$ is an oriented $\E_\infty$-ring in $\inflogSH(k)$.
\end{prop}
\begin{proof}
The functor $\omega_\sharp\colon \infSH(k)\to \inflogSH(k)$ is monoidal, so $\omega^*$ is lax monoidal.
Hence $\omega^*$ preserves $\E_\infty$-rings, so $\omega^*\ringE$ is an $\E_\infty$-ring.

Since $\ringE$ is oriented, 
there exists a morphism $c_\infty \colon \Sigma_T^\infty \P^\infty/\pt \to \Sigma^{2,1}\ringE$ in $\infSH(k)$ 
such that the composite morphism 
$\Sigma_T^\infty \P^1/\pt \to \Sigma_T^\infty \P^\infty/\pt \to \Sigma^{2,1}\ringE$ 
is equivalent to $\Sigma^{2,1}\unit$, 
where $\unit\colon \Sigma_T^\infty \pt \to \ringE$ is the unit.
Use $\omega_\sharp (\Sigma_T^\infty \P^\infty/\pt)\simeq \Sigma_T^\infty \P^\infty/\pt$ to obtain an orientation 
on $\omega^*\ringE$ from the orientation on $\ringE$.
\end{proof}

\begin{prop}
The spectrum $\spectlogML$ is an $\E_\infty$-ring in $\inflogSH(k)$.
\end{prop}
\begin{proof}
Since $\gamma_\sharp \beta_\sharp$ is monoidal, 
it follows that $\beta^*\gamma^*$ is lax monoidal and hence preserves $\E_\infty$-rings.
Moreover, 
the same holds for $\beta^*\gamma^*\gamma_\sharp \beta_\sharp$.
To conclude, 
observe that $\Sigma_T^\infty \pt\in \inflogSH(k)$ is an $\E_\infty$-ring.
\end{proof}

\begin{prop}
\label{ML.4}
If $k$ is a perfect field that admits resolution of singularities, 
then $\spectlogML$ is an oriented $\E_\infty$-ring in $\inflogSH(k)$.
\end{prop}
\begin{proof}
Recall that $\spectML$ is an oriented $\E_\infty$-ring in $\infSH(k)$.
This is a consequence of \cite[Theorem 4.1]{MVW} and \cite{zbMATH05827456}.
By Proposition \ref{ML.3}, $\omega^*\spectML$ is an oriented $\E_\infty$-ring in $\inflogSH(k)$.
Proposition \ref{ML.5} finishes the proof.
\end{proof}

\newpage

\section{Logarithmic topological Hochschild homology}
\label{section:lthh}
The objective of this section is to construct the logarithmic motivic spectrum
\[
\spectlogTHH\in \inflogSH(S)
\]
for all $S\in \Sch$, 
and show that the infinite $\Gmm$-loops spectrum $\Omega_{\Gmm}^\infty \spectlogTHH(X)\in \infSpt$ recovers 
for all $X\in \SmlSm/S$ the logarithmic topological Hochschild homology $\logTHH(X)$, 
see \eqref{Hoch.22.1}.
If $Y\in \Sch$, 
then $\logTHH(Y)$ is equivalent to the topological Hochschild homology $\THH(Y)$ by Proposition \ref{Hoch.36},
i.e., 
$\logTHH$ extends $\THH$ to the logarithmic setting. The idea of logarithmic topological Hochschild homology is due to Rognes \cite{RognesLogHH}: here we revisit his construction and show that it gives rise to a motivic invariant, as explained above. 

\begin{rmk}
It is a well known fact that topological Hochschild homology is non-$\A^1$-invariant,
see e.g., Remark \ref{rmk:nonA1} or \cite{EEnonA1}. 
In fact, it is \emph{very far} from being $\A^1$-invariant, 
since $L_{\A^1}\THH \simeq L_{\A^1} \mathrm{HH} \simeq 0$. 
It is therefore impossible to represent $\THH$ in $\infSH(S)$.
Hence, the representability of $\logTHH$ in $\inflogSH(S)$ shows an important difference between $\inflogSH(S)$ and $\infSH(S)$.
\end{rmk}

\begin{rmk} The cyclotomic structure on $\logTHH$ can be used to  define the refined invariants of 
logarithmic negative cyclic homology $\logTC^-$, logarithmic periodic homology $\logTP$, and logarithmic cyclic homology $\logTC$.
\end{rmk}

\begin{rmk}
For notational clarity, 
we shall often add the prefix ``$\mathrm{log}$'' to cohomology theories extended from schemes to log schemes, 
e.g., 
as for $\logPic$, $\Kthlog$, and $\logTHH$. This also agrees with the convention followed in \cite{BLPO2}. 
\end{rmk}

\subsection{\texorpdfstring{$\THH$}{THH} of rings}
\label{subsection:thhrings}

We briefly review the theory of topological Hochschild homology and topological cyclic homology in the 
$\infty$-categorical framework developed by Nikolaus and Scholze in \cite{NikolausScholze}.
\begin{df}
\label{Hoch.38}
Suppose $\infC$ is an $\infty$-category and $G$ is a group with classifying space $\clspace G$.
An object of the $\infty$-category $\Fun(\clspace G,\infC)$ is called a \emph{$G$-equivariant object of $\cC$}.
A map in $\Fun(\clspace G,\infC)$ is called a \emph{$G$-equivariant map in $\cC$}.

Suppose that $\cC$ admits small limits.
There is a limit preserving functor of $\infty$-categories
\[
(-)^{hG}
\colon
\Fun(\clspace G,\infC)\to \infC
\]
sending $F\colon \clspace G\to \infC$ to $\limit F$.
This is called the \emph{homotopy fixed point functor}\index{homotopy fixed point functor}.
The definition of limits gives a naturally induced map
\begin{equation}
\label{Hoch.38.1}
F^{hG}
=
\lim F
\to
F(*),
\end{equation}
where $*$ is a unique object of $\clspace G$.

If $G$ is a finite group  and $\infC$ is a stable $\infty$-category, there is a functor
\[
(-)^{tG}
\colon
\Fun(\clspace G,\infC)\to \infC.
\]
This is called the \emph{Tate construction}.
We refer to \cite[Definition I.1.13]{NikolausScholze} for more details.
There is a natural transformation $(-)^{hG}\to(-)^{tG}$,
see \cite[Theorem I.3.1]{NikolausScholze}. Following \cite[1.4]{NikolausScholze}, this can be generalized to the case where $G=\mathbb{T}$ is the circle group (as in \cite{NikolausScholze}, we keep the notation $S^1$ and $\mathbb{T}$ separate to stress on the group structure). \index[notation]{T @ $\mathbb{T}$}
\end{df}

For any integer $n\neq 0$, 
we let $C_n$ denote the cyclic group of order $\lvert n \rvert$.

\begin{df}
A \emph{cyclotomic spectrum}\index{cyclotomic spectrum} is a spectrum $X$ with a 
$\mathbb{T}$-action and 
$\mathbb{T}$-equivariant maps $X\to X^{tC_p}$ for every prime $p$.
\end{df}

\begin{rmk}
Nikolaus and Scholze constructed  
the \emph{$\infty$-category of cyclotomic spectra} $\infCycSpt$ \index[notation]{CycSpt @ $\infCycSpt$} 
in \cite[Definition II.1.6]{NikolausScholze}.
It underlies a symmetric monoidal $\infty$-category $\infCycSpt^\otimes$,
see \cite[Construction IV.2.1]{NikolausScholze} for the details.
\end{rmk}

\begin{df}Let $A$ be an $\Einfty$-algebra in $\infSpt$.
The \emph{topological Hochschild homology} \index{topological Hochschild homology}\index[notation]{THH @ $\THH$}of $A$, 
denoted $\THH(A)$, is the colimit of the constant diagram
\[
S^1
\to
\CAlg(\infSpt)
\]
with value $A$. 
There exists a cyclotomic structure on $\THH(A)$ exhibited in 
\cite[\S IV.2]{NikolausScholze}. 
Henceforth, we consider $\THH(A)$ as an object of $\infCycSpt$.\end{df}
\begin{rmk}
According to \cite[Proposition IV.4.14]{NikolausScholze}, there is an adjunction
\[
(-)^{\mathrm{triv}}
:
\infSpt \rightleftarrows \infCycSpt
:
\TC
\]
with $\TC(X) := \map_{\infCycSpt}(\Striv,X)$ \cite[Definition II.1.8(i)]{NikolausScholze}.
As noted in \cite[Example II.1.2(ii)]{NikolausScholze}, 
there is an equivalence of cyclotomic spectra $\Striv \simeq \THH(\Sphere)$.
By \cite[Theorem I.4.1, Construction IV.2.1]{NikolausScholze},
there is a lax symmetric monoidal natural transformation of lax symmetric monoidal functors $\mathrm{can}\colon (-)^{h\mathbb{T}}\to (-)^{t\mathbb{T}}$,
and the functor $\TC$ is lax symmetric monoidal.
\end{rmk}
\begin{df}
\label{Hoch.37}
Let $A$ be a ring or an $\Einfty$-algebra in $\infSpt$.
Following Nikolaus and Scholze \cite{NikolausScholze},
the \emph{negative topological cyclic homology of $A$}\index{negative topological cyclic homology}\index[notation]{TC- @ $\TC^{-}$},
the \emph{topological periodic homology of $A$}\index{topological periodic homology}\index[notation]{TP @ $\TP$},
and the \emph{topological cyclic homology of $A$}\index{topological cyclic homology}\index[notation]{TC @ $\TC$}
are defined as 
\begin{equation}
\label{Hoch.37.1}
\TC^{-}(A)
:=
\THH(A)^{h\mathbb{T}},
\text{ }
\TP(A)
:=
\THH(A)^{t\mathbb{T}},
\text{ }
\TC(A)
:=
\TC(\THH(A)).
\end{equation}
\end{df}

There are canonical maps
\[
\TC(A)\to \TC^{-}(A)\to \THH(A)
\text{ and }
\TC^-(A)\to \TP(A).
\]
Here the first map is obtained by the functor from $\infCycSpt$ to the $\infty$-category of $\mathbb{T}$-equivariant spectra, 
the second map is obtained from \eqref{Hoch.38.1},
and the third map is obtained from $\mathrm{can}\colon (-)^{h\mathbb{T}}\to (-)^{t\mathbb{T}}$.

By \cite[Corollary 3.4]{BMS19},
$\THH$, $\TC^-$, and $\TP$ are fpqc sheaves of spectra on the category of rings,
which implies that $\TC$ is an fpqc sheaf too.
By descent,
we obtain
\[
\THH,\TC^-,\TP,\TC\in \infShv_{fpqc}(\Sch,\infSpt)
\]

\subsection{Cyclic bar and replete bar constructions} We revisit the cyclotomic structure on logarithmic topological Hochschild homology as discussed in \cite{BLPO2}. 
\begin{df}
\label{Hoch.7}
Let $M$ be a (discrete) commutative monoid.
The \emph{cyclic bar construction}\index{cyclic bar construction} of $M$, 
denoted $\Bcy M$\index[notation]{Bcy @ $\Bcy M$}, 
is the colimit of the constant diagram
\[
S^1\to \CAlg(\infSpc)
\]
with value $M$.
The base point inclusion $*\to S^1$ and the collapse map $S^1\to *$ naturally induce maps
\begin{equation}
\label{Hoch.7.1}
M
\xrightarrow{\eta}
\Bcy M
\xrightarrow{\epsilon}
M,
\end{equation}
such that $\epsilon$ is $\T$-equivariant,
where $\mathbb{T}$ acts trivially on $M$.
Since $\Bcy M$ is defined as a colimit, 
$\eta$ is initial among all maps from $M$ to a commutative monoid in $\infSpc$ with an $\mathbb{T}$-action. See \cite[Proposition B.5]{NikolausScholze}.
\end{df}
\begin{rmk}
An explicit description of $M$ as a simplicial set is given as follows.
In simplicial degree $q$, $\Bcy M$ is given by $M^{\times (q+1)}$.
The face maps are
\[
d_i(f_0,\ldots,f_q)
=
\left\{
\begin{array}{ll}
(f_0,\ldots,f_{i-1},f_if_{i+1},f_{i+2},\ldots,f_q)&
\text{if }0\leq i\leq q-1,
\\
(f_qf_0,f_1,\ldots,f_{q-1})&
\text{if }i=q,
\end{array}
\right.
\]
and the degeneracy maps are
\[
s_i(f_0,\ldots,f_q)
=
(f_0,\ldots,f_i,0,f_{i+1},\ldots,f_q).
\]
In \cite[\S 2.3]{zbMATH03671343}, 
the cyclic bar construction is denoted by $\mathrm{N}^{cy}$ instead of $\Bcy$.

Observe that $\eta$ is simply the inclusion of zero simplices, 
and $\epsilon$ maps $(f_0,\ldots,f_q)$ to $f_0+\cdots +f_q$ in simplicial degree $q$.
There is a canonical decomposition
\begin{equation}
\label{Hoch.7.2}
\Bcy M
\simeq
\coprod_{j\in M}\Bcy (M;j),
\end{equation}
where $\Bcy (M;j) := \epsilon^{-1}(j)$.
\end{rmk}
\begin{df}
\label{Hoch.8}
The \emph{replete bar construction}\index{replete bar construction} of a monoid $M$ is 
\index[notation]{Brep @ $\Brep M$}
\begin{equation}
\label{Hoch.8.3}
\Brep M
:=
\Bcy M^{\gp}\times_{M^{\gp}} M.
\end{equation}
\end{df}

Rognes discussed this notion in \cite[Definition 3.16]{RognesLogHH}, 
and used it as a key input to define logarithmic topological Hochschild homology of pre-log rings.

\ 

The maps \eqref{Hoch.7.1} for $M^{\gp}$ restricts to maps
\begin{equation}
\label{Hoch.8.1}
M
\xrightarrow{\eta}
\Brep M
\xrightarrow{\epsilon}
M
\end{equation}
such that $\epsilon$ is $\T$-equivariant.
Moreover, 
there is a canonical decomposition
\begin{equation}
\label{Hoch.8.2}
\Brep M
\simeq
\coprod_{j\in M}\Brep (M;j),
\end{equation}
where $\Brep (M;j) := \epsilon^{-1}(j)$.

\

We will use the next computation to show $(\P^{\bullet},\P^{\bullet-1})$-invariance for logarithmic 
topological Hochschild homology.

\begin{prop}
\label{Hoch.9}
There are $\mathbb{T}$-equivariant equivalences
\[
\Bcy (\N,j)
\simeq
\left\{
\begin{array}{ll}
* & \text{if }j=0
\\
\mathbb{T}/C_j & \text{if }j>1
\end{array}
\right.
\]
and
\[
\Bcy (\Z,j)
\simeq
\mathbb{T}/C_j
\;
(j\in \Z),
\;
\Brep(\N,j)
\simeq
\mathbb{T}/C_j
\;
(j\geq 0),
\]
where $C_0$ denotes the trivial group for convenience.
The naturally induced maps
\[
\Bcy (\N,j)
\to
\Bcy(\Z,j)
\;
(j\geq 1)
\text{ and }
\Brep (\N,j)
\to
\Bcy(\Z,j)
\;
(j\geq 0)
\]
are $\mathbb{T}$-equivariant equivalences.
\end{prop}
\begin{proof}
We refer to \cite[Propositions 3.20, 3.21]{RognesLogHH}.
\end{proof}

\begin{df}
\label{Hoch.5}
Suppose $\theta\colon M\to P$ is a homomorphism of monoids.
Let $\Brep(P,M)$\index[notation]{BrepPM @ $\Brep(P,M)$} be the monoid space defined by the 
pushout square
\[
\begin{tikzcd}
\Bcy M\ar[r]\ar[d,"\Bcy \theta"']&
\Brep M\ar[d]
\\
\Bcy P\ar[r]&
\Brep (P,M)
\end{tikzcd}
\] 
of commutative monoid spaces.
\end{df}

\begin{lem}
\label{Hoch.2}
For all monoids $M$ and $N$, there is a canonical $\mathbb{T}$-equivariant equivalence
\begin{equation}
\label{Hoch.2.1}
\Bcy (M\times N) \cong \Bcy M \times \Bcy N.
\end{equation}
\end{lem}
\begin{proof}
The two projections $M\times N\to M,N$ naturally induces \eqref{Hoch.2.1}.
In simplicial degree $q$, \eqref{Hoch.2.1} is given by the shuffle map
\[
(M\times N)^{\times (q+1)}
\xrightarrow{\simeq}
M^{\times (q+1)} \times N^{\times (q+1)}.
\]
This shows that \eqref{Hoch.2.1} is an equivalence.
\end{proof}

\begin{lem}
\label{Hoch.3}
For all monoids $M$ and $N$, there is a canonical  $\mathbb{T}$-equivariant equivalence
\[
 \Brep (M\times N) \cong \Brep M \times \Brep N.
\]
\end{lem}
\begin{proof}
From \eqref{Hoch.8.3}, we have cartesian squares
\[
\begin{tikzcd}
\Brep (M\times N) \arrow[d]\arrow[r]&
\Bcy (M\times N)^{\gp}\arrow[d]
\\
M\times N\arrow[r]&
(M\times N)^{\gp},
\end{tikzcd}
\quad
\begin{tikzcd}
\Brep M\times \Brep N \arrow[d]\arrow[r]&
\Bcy M^{\gp}\times \Bcy N^{\gp}\arrow[d]
\\
M\times N\arrow[r]&
M^{\gp}\times N^{\gp}.
\end{tikzcd}
\]
Apply Lemma \ref{Hoch.2} and compare these two diagrams.
\end{proof}

\begin{lem}
\label{Hoch.14}
For $Q:=\N^r\oplus \Z^s$, $P:=\N^r \oplus \N \oplus \Z^{s-1}$, and $M:=0\oplus \N \oplus \Z^{s-1}$, 
there is a canonical $\mathbb{T}$-equivariant equivalence
\[
\Brep(P,M)
\simeq
\Bcy Q\times_Q P.
\]
\end{lem}
\begin{proof}
Lemmas \ref{Hoch.2} and \ref{Hoch.3} yield the $\mathbb{T}$-equivariant equivalences
\[
\Brep (P,M)
\simeq
(\Bcy \N)^r \times \Brep \N \times (\Bcy \Z)^{s-1}
\]
and
\[
\Bcy(Q)
\simeq
(\Bcy \N)^r \times \Bcy \Z \times (\Bcy \Z)^{s-1}.
\]
Using the definition $\Brep \N = \Bcy \Z \times_{\Z} \N$, we deduce the desired equivalence.
\end{proof}

\subsection{\texorpdfstring{$\logTHH$}{logTHH} of log schemes}
For a pre-log ring $(A,M)$, 
Rognes defined the logarithmic topological Hochschild homology of $(A,M)$ in \cite[Definition 8.11]{RognesLogHH} 
as a symmetric spectrum without a cyclotomic structure.
The cyclotomic structure on the topological Hochschild homology of $(A,M)$ was first discussed by 
Hesselholt-Madsen \cite{zbMATH00938856};
we will make use of the $\infty$-categorical formulation in \cite{zbMATH07063999}.

For a simplicial set $X$, we adopt the suggestive notation
\[
\Sphere[X]
:=
\Sigma_+^\infty X.
\]
For a monoid $M$,
\cite[Construction 3.7]{BLPO} yields a natural morphism
\[
\THH(\Sphere[M])
\simeq
\Sphere[\Bcy M]
\to
\Sphere[\Brep M]
\]
of $\E_\infty$-rings in $\infCycSpt$.

\begin{df}
\label{Hoch.6}
Let $(A,M)$ be a pre-log ring.
The \emph{logarithmic topological Hochschild homology of $(A,M)$}
\index{logarithmic topological Hochschild homology}\index[notation]{logTHH @ $\logTHH$} 
is the $\Einfty$-ring in $\infCycSpt$ given by
\begin{equation}
\label{Hoch.6.2}
\logTHH(A,M)
:=
\THH(A)\otimes_{\Sphere[\Bcy M]} \Sphere[\Brep M].
\end{equation}
The notation $\otimes$ means the coproduct in the $\infty$-category of $\E_\infty$-rings in $\infCycSpt$.
\end{df}

See again \cite[Construction 3.7, Definition 3.8]{BLPO2} for a more explicit description of the Frobenius maps $\varphi_p$ on $\logTHH(A,M)$, and \cite[Proposition 3.6]{BLPO2} (which is in fact \cite[Proposition 5.8]{lundemo}) for an equivalent definition in terms of the (derived) self intersection of the log diagonal. This latter point of view was used in \cite{BLPO}.

\begin{exm}
\label{Hoch.36}
Suppose $A$ is a ring.
Then we have $\Bcy A^*\simeq \Brep A^*$ since $A^*$ is a group.
Hence there is a canonical equivalence of cyclotomic spectra
\[
\logTHH(A,A^*)
\simeq
\THH(A).
\]
Thus we may view $\logTHH$ as an extension of $\THH$ as a cyclotomic spectrum.
\end{exm}

For a pre-log ring $(A,M)$ given by a homomorphism $\theta\colon M\to A$,
we set $M^a:=M\oplus_{\theta^{-1}(A^*)}A^*$ (this is the logification of $(A,M)$).

\begin{prop}
\label{Hoch.12}
Let $(A,M)$ be an integral pre-log ring.
Then the naturally induced map of cyclotomic spectra
\begin{equation}
\label{Hoch.12.2}
\logTHH(A,M)
\to
\logTHH(A,M^a)
\end{equation}
is an equivalence.
\end{prop}
\begin{proof}
Since the forgetful functor $\infCycSpt\to \infSpt$ is conservative by 
\cite[Corollary II.1.7]{NikolausScholze},
it suffices to show that \eqref{Hoch.12.2} is an equivalence of spectra.
Consider the derived logification $M^{La}:=M\oplus_{\theta^{-1}(A^*)}^{\bbL} A^*$.
By \cite[Theorem 4.24]{RSS},
we have an equivalence of spectra
\[
\logTHH(A,M)
\xrightarrow{\simeq}
\logTHH(A,M^{La}).
\]
Hence it suffices to show that $M^{La}$ is discrete.
We set $G:=\theta^{-1}(A^*)$ for simplicity.
Then we have an equivalence
\[
M^{La}
\simeq
M\oplus_G^\bbL G^\gp \oplus_{G^\gp}^\bbL A^*.
\]
The homomorphism $G\to G^\gp$ is integral by \cite[Proposition I.4.6.3(3)]{Ogu} and injective since $G$ is integral,
so $N:=M\oplus_G^\bbL G^\gp$ is a discrete monoid.
The homomorphism $G^\gp\to N$ is integral by \cite[Proposition I.4.6.3(2)]{Ogu} and injective.
It follows that $M^{La}$ is discrete.
\end{proof}

\begin{prop}
\label{THHsheaf.6}
Let $R$ be a ring.
The functor $\bbL_{-/R}$ on the category of $R$-algebras preserves filtered colimits.
\end{prop}
\begin{proof}
This is a consequence of the construction of $\bbL_{-/R}$.
More generally,
$\bbL_{-/R}$ preserves sifted colimits at the simplicial level,
see \cite[Construction 2.1]{BMS19} and \cite[\S 2.5]{BLPO}.
\end{proof}

\begin{prop}
\label{THHsheaf.4}
The functor $\logTHH$ from the category of pre-log rings to $\infCycSpt$ preserves filtered colimits.
\end{prop}
\begin{proof}
Observe that the forgetful functors $\CAlg(\infSpt)\to \infSpt$ and $\infCycSpt\to \infSpt$ are conservative and preserve small sifted (and hence filtered) colimits by \cite[Corollary 3.2.3.2]{HA} and \cite[Corollary II.1.7]{NikolausScholze}.
Hence it suffices to show that the functor
\[
(A,M)
\mapsto 
\THH(A,M)
\simeq
(A\otimes_{A\otimes_{\Sphere} A} A)\otimes_{\Sphere[M]\otimes_{\Sphere[M]\otimes_{\Sphere}\Sphere[M]} \Sphere[M]}\Sphere[M\times M^\gp]
\]
to $\CAlg(\infSpt)$ preserve filtered colimits.
This can be checked componentwise.
\end{proof}

Recall from \cite[Definition 2.44]{Koshikawa-Yao} that a homomorphism of pre-log rings $(A,M)\to (B,N)$ is \emph{a homologically log flat covering} if $A\to B$ is faithfully flat and
the derived pushout $N\oplus_M^\bbL M'$ is discrete for every homomorphism of monoids $M\to M'$,
e.g., $M\to N$ is injective and integral.
Using this,
we can define a homologically log flat topology on the opposite category of pre-log rings.

\begin{prop}
\label{THHsheaf.1}
The presheaf $\logTHH$ on the opposite category of pre-log rings is a homologically log flat sheaf of cyclotomic spectra.
\end{prop}
\begin{proof}
By \cite[Proposition II.1.5(v)]{NikolausScholze},
it suffices to show that $\logTHH(-)$ and $\logTHH(-)^{tC_p}$ are sheaves of spectra.
The first one is a sheaf by \cite[Proposition 3.10]{BLPO2}.
For the second one,
it suffices to show that $(A,M)\mapsto \wedge^i \bbL_{(A,M)/\Z} \otimes \tau_{\leq n}\Z_{hC_p}$ is a sheaf for every integer $n$ as in the proof of \cite[Corollary 3.4]{BMS19}.
This follows from \cite[Theorem 2.3]{BLPO2} and the fact that $\tau_{\leq n} \Z_{hC_p}$ is a perfect $\Z$-complex.
\end{proof}

\begin{df}
\label{THHsheaf.2}
For an integral log scheme $X$,
consider the presheaf of cyclotomic spectra
\[
\logTHH|_X(U)
:=
\logTHH(\Gamma(U,\cO_U),\Gamma(U,\cM_U))
\]
on the category of integral log schemes over $X$.
The \emph{topological Hochschild homology of $X$} is defined to be
\[
\logTHH(X)
:=
R\Gamma_{Zar}(X,L_{Zar}\logTHH|_X),
\]
where $L_{Zar}$ denotes the Zariski sheafification functor.
The \emph{logarithmic negative topological cyclic homology of $X$}\index{logarithmic negative topological cyclic homology}\index[notation]{logTC- @ $\logTC^{-}$},
the \emph{logarithmic topological periodic homology of $X$}\index{logarithmic topological periodic homology}\index[notation]{logTP @ $\logTP$},
and the \emph{logarithmic topological cyclic homology of $X$}\index{logarithmic topological cyclic homology}\index[notation]{logTC @ $\logTC$}
are defined as 
\begin{equation}
\logTC^{-}(X)
:=
\logTHH(X)^{h\mathbb{T}},
\text{ }
\logTP(X)
:=
\logTHH(X)^{t\mathbb{T}},
\text{ }
\logTC(A)
:=
\TC(\logTHH(A)).
\end{equation}
We note that the logarithmic version $\logTHH(X)$ takes into account the log structure on $X$, 
while the functors $(-)^{t\mathbb{T}},\TC\colon \infCycSpt\to \infSpt$ are as in \S\ref{subsection:thhrings}.
\end{df}

If $X$ is an integral log scheme over $\Spec{R}$ for some ring $R$,
consider the Zariski sheaf $\bbL\Omega^n_{X/R}$ of complexes of $\cO_X$-modules on $X$ associated with the presheaf on $X$ given by
\[
U\mapsto \wedge_{\Gamma(U,\cO_U)}^{n}\bbL_{(\Gamma(U,\cO_U),\Gamma(U,\cM_U))/R}.
\]

\begin{prop}
\label{THHsheaf.3}
Let $(A,M)$ be an integral pre-log ring.
Then the naturally induced map of cyclotomic spectra
\begin{equation}
\label{THHsheaf.3.2}
\logTHH(A,M)
\to
\logTHH(\Spec{A,M})
\end{equation}
is an equivalence.
If $R\to A$ is a homomorphism of rings,
then the naturally induced map of complexes
\begin{equation}
\label{THHsheaf.3.3}
\wedge_A^n
\bbL_{(A,M)/R}
\to
\Gamma_{Zar}(\Spec{A,M},\bbL\Omega_{\Spec{A,M}/R}^n)
\end{equation}
is an equivalence.
\end{prop}
\begin{proof}
Let us first show that \eqref{THHsheaf.3.2} is an equivalence.
Consider the presheaf of cyclotomic spectra
\[
\logTHH|_{(A,M)}(B)
:=
\logTHH(B,M)
\]
on the opposite category of $A$-algebras.
We have the naturally induced map of presheaves of cyclotomic spectra on the opposite category of $A$-algebras
\begin{equation}
\label{THHsheaf.3.1}
\logTHH|_{(A,M)}
\to
\logTHH|_{\Spec{A,M}}.
\end{equation}
If we take the global sections after Zariski sheafification,
then we get $\logTHH(A,M)$ for the left-hand side by Proposition \ref{THHsheaf.1} and $\logTHH(\Spec{A,M})$ for the right-hand side by definition.
Hence it suffices to show that \eqref{THHsheaf.3.1} becomes an equivalence of sheaves of cyclotomic spectra after Zariski sheafification.
For this,
it suffices to show that every stalk of \eqref{THHsheaf.3.1} is an equivalence.

A stalk of a presheaf is obtained by a filtered colimit.
Using Proposition \ref{THHsheaf.4},
we reduce to the case when $A$ is a local ring.
In this case, $\logTHH(\Spec{A,M})$ becomes $\logTHH(A,M^a)$.
Together with Proposition \ref{Hoch.12},
we deduce that \eqref{THHsheaf.3.2} is an equivalence.

To show that \eqref{THHsheaf.3.3} is an equivalence,
use \cite[Lemma 3.11]{BLPO} (resp.\ Proposition \ref{THHsheaf.6}, resp.\ \cite[Theorem 2.3]{BLPO}) instead of Proposition \ref{Hoch.12} (resp.\ Proposition \ref{THHsheaf.4}, resp.\ Proposition \ref{THHsheaf.1}).
\end{proof}

\begin{df}
Let $X$ be an integral log scheme over $R$,
where $R$ is a ring.
The \emph{(relative) logarithmic Hochschild homology of $X$} is
\[
\logHH(X/R)
:=
\logTHH(X)\otimes_{\THH(R)}R.
\]

Observe that if $(A,M)$ is a pre-log ring, then $\logHH(\Spec{A,M}/R)$ is equivalent to $\logHH((A,M)/R)$ in \cite[Definition 5.3]{BLPO}.
\end{df}

\begin{prop}
\label{THHsheaf.7}
Let $X$ be an integral log scheme over $R$,
where $R$ is a ring.
Then $\logHH(X/R)$ admits a natural complete exhaustive filtration with the graded pieces
\[
\mathrm{gr}^n\logHH(X/R)
\simeq
R\Gamma_{Zar}(X,\bbL\Omega^n_{X/R})[n]
\]
for all integers $n$.
\end{prop}
\begin{proof}
By Proposition \ref{THHsheaf.3},
this follows from \cite[Theorem 5.15]{BLPO} and Zariski descent.
\end{proof}

\begin{prop}
\label{THHsheaf.5}
The presheaf $\logTHH$ on the category of integral log schemes is a strict fpqc sheaf of cyclotomic spectra.
\end{prop}
\begin{proof}
Consider the strict fpqc sheaf given by
\[
\logTHH_{sfpqc}(X)
:=
R\Gamma_{sfpqc}(X,L_{sfpqc}\logTHH|_X(U))
\]
for integral log schemes $X$,
where $sfpqc$ is a shorthand for the strict fpqc topology.
Argue as in Proposition \ref{THHsheaf.3} to show that the naturally induced map of cyclotomic spectra
\[
\logTHH(A,M)
\to
\logTHH_{sfpqc}(\Spec{A,M})
\]
is an equivalence for every pre-log ring $(A,M)$.
This means that $\logTHH$ and $\logTHH_{sfpqc}$ are equivalent for the affine case and hence are equivalent for the general case by Zariski descent.
It follows that $\logTHH$ is a strict fpqc sheaf.
\end{proof}

\begin{prop}
\label{THHsheaf.8}
For a ring $R$ and integer $n$,
the presheaf $\bbL\Omega^n_{-/R}$ on the category of integral log schemes is a strict fpqc sheaf of chain complexes.
\end{prop}
\begin{proof}
Argue as in Proposition \ref{THHsheaf.5}.
\end{proof}

\begin{prop}
\label{THHsheaf.9}
The presheaves $\logTHH^{h\mathbb{T}}$, $\logTC^-$, $\logTP$, and $\logTC$ on the category of integral log schemes are strict fpqc sheaves of spectra.
\end{prop}
\begin{proof}
Argue as in \cite[Corollary 3.4]{BMS19} to deduce that $\logTHH_{h\mathbb{T}}$ is a sheaf from Propositions \ref{THHsheaf.7} and \ref{THHsheaf.8}.
Since the functors $(-)^{h\mathbb{T}}$ and $\TC$ preserve limits,
Proposition \ref{THHsheaf.5} implies that $\logTC^-$ and $\logTC$ are sheaves.
Using the cofiber sequence $\Sigma^1 (-)_{h\mathbb{T}}\to (-)^{h\mathbb{T}}\to (-)^{t\mathbb{T}}$,
we see that $\logTP$ is a sheaf too.
\end{proof}

The functor $\lambda\colon \Sch \to \lSch$ sending $X\in \Sch$ to $X$ naturally induces an adjunction
\begin{equation}
\label{Hoch.21.3}
\lambda_\sharpp : \infShv_{fpqc}(\Sch,\infSpt) \rightleftarrows \infShv_{sfpqc}(\lSch,\infSpt) : \lambda^*.
\end{equation}
As observed in Example \ref{Hoch.36}, we have equivalences in $\infShv_{fpqc}(\Sch,\infSpt)$
\begin{equation}
\label{Hoch.21.4}
\lambda^* \logTHH \simeq \THH,
\text{ }
\lambda^* \logTC^- \simeq \TC^-,
\text{ }
\lambda^* \logTP \simeq \TP,
\text{ and }
\lambda^* \logTC \simeq \TC.
\end{equation}

\begin{const}
\label{Hoch.23}
For $\E_\infty$-ring spectra $A$ and $B$,
we have an equivalence of spectra $\THH(A\otimes_{\Sphere} B)\simeq \THH(A)\otimes \THH(B)$.
Using Propositions \ref{Hoch.2} and \ref{Hoch.3},
we see that the functor $\logTHH$ from the category of pre-log rings to $\infCycSpt$ is lax symmetric monoidal.
It follows that the functor $\logTHH$ from the category of log schemes to $\infCycSpt$ is lax symmetric monoidal.
Hence we obtain
\[
\logTHH\in \CAlg(\infPsh(\lSch,\infSpt))
\]
by \cite[Example 2.2.6.9]{HA},
where the symmetric monoidal structure on $\infPsh(\lSch,\infSpt)$ is the Day convolution.
By applying the lax symmetric monoidal functors $(-)^{h\mathbb{T}}$, $(-)^{t\mathbb{T}}$ and $\TC$,
we obtain
\[
\logTC^-,\logTP,\logTC\in \CAlg(\infPsh(\lSch,\infSpt)).
\]
\end{const}

\subsection{\texorpdfstring{$(\P^n,\P^{n-1})$}{(Pn,Pn-1)}-invariance of \texorpdfstring{$\logTHH$}{logTHH}}
The purpose of this section is to prove that $\logTHH$ is $(\P^n,\P^{n-1})$-invariant for all integers $n\geq 1$.
As a consequence, 
we will represent $\logTHH(X)$, $\logTC^-(X)$, $\logTP(X)$, 
and $\logTC^-(X)$ in $\inflogSH^\eff(S)$ whenever $S\in \Sch$ and $X\in \SmlSm/S$.

\begin{prop}
\label{Hoch.10}
Let $R$ be a ring, and let $M$ be a monoid.
Then there is an equivalence of $\E_\infty$-rings in cyclotomic spectra
\begin{equation}
\label{Hoch.10.2}
\THH(R)
\otimes_{\Striv}
\Sphere[\Bcy(M)]
\xrightarrow{\simeq}
\THH(R[M]).
\end{equation}
\end{prop}
\begin{proof}
The naturally induced square
\[
\begin{tikzcd}
\THH(\Sphere)\ar[d]\ar[r]&
\THH(R)\ar[d]
\\
\THH(\Sphere[M])\ar[r]&
\THH(R[M])
\end{tikzcd}
\]
constructs \eqref{Hoch.10.2} as a map between $\E_\infty$-rings in cyclotomic spectra.
Owing to \cite[Corollary II.1.7]{NikolausScholze} we only need to show that \eqref{Hoch.10.2} is an equivalence of spectra.
If $M$ is a group, 
then by \cite[Proposition 3.2]{zbMATH05688706} there is a canonical equivalence of spectra 
\begin{equation}
\label{Hoch.10.1}
\Eilenberg (R[M])
\simeq
\Eilenberg R
\wedge
\Sphere[M].
\end{equation}
Furthermore, by inspection,
the same proof applies to the case when $M$ is a monoid.
Apply \cite[Theorem 7.1]{zbMATH00938856} to deduce the desired equivalence.
\end{proof}

\begin{rmk}
\label{rmk:nonA1}
By Proposition \ref{Hoch.10}, there is an equivalence of cyclotomic spectra
\[
\THH(X\times \A^1)
\simeq
\THH(X) \otimes_{\Striv}\Sphere[\Bcy(\N)]
\]
for $X\in \Sch$.
Together with Proposition \ref{Hoch.9}, we deduce that $\THH(X\times \A^1)$ is not equivalent to $\THH(X)$ if $\THH(X)$ is nontrivial.
In other words, $\THH$ is not $\A^1$-invariant.
\end{rmk}

For pre-log rings $(R,M,\theta)$ and $(R',M',\theta')$, 
let $(R,M)\otimes (R',M')$ be the pre-log ring $(R\otimes R',M\oplus M',\theta'')$, 
where $\theta''$ maps $(p,p')\in M\oplus M'$ to $p\otimes p'$.

\begin{prop}
\label{Hoch.11}
Let $(R,M)$ be a pre-log ring, and let $N\to P$ be a homomorphism of monoids.
Then there is a canonical equivalence of $\E_\infty$-rings in cyclotomic spectra
\begin{equation}
\label{Hoch.11.1}
\logTHH((R,M)\otimes (\Z[P],N))
\simeq
\logTHH(R,M)
\otimes_{\Striv}
\Sphere[\Brep(P,N)].
\end{equation}
\end{prop}
\begin{proof}
By definition, 
we have
\[
\logTHH((R,M)\times (\Z[P],N))
=
\THH(R[P])\otimes_{\Sphere[\Bcy(M\oplus N)]}\Sphere[\Brep(M\oplus N)].
\]
Owing to Propositions \ref{Hoch.2}, \ref{Hoch.3}, and \ref{Hoch.10} 
we have equivalences of $\E_\infty$-rings in cyclotomic spectra
\begin{align*}
& \THH(R[P])\otimes_{\Sphere[\Bcy(M\oplus N)]}\Sphere[\Brep(M\oplus N)]
\\
\simeq &
(\THH(R)\otimes_{\Striv} \Sphere[\Bcy P])
\otimes_{\Sphere[\Bcy M]\otimes_{\Striv} \Sphere[\Bcy N]}
(\Sphere[\Brep M]\otimes_{\Striv} \Sphere[\Brep N])
\\
\simeq &
(\THH(R)\otimes_{\Sphere[\Bcy M]}\Sphere[\Brep M])
\otimes_{\Striv}
(\Sphere[\Bcy P]\otimes_{\Sphere[\Bcy N]}\Sphere[\Brep N]).
\end{align*}
By definition, 
the last object is $\logTHH(R,M)\otimes_{\Striv} \Sphere[\Brep(P,N)]$.
\end{proof}

\begin{thm}
\label{Hoch.13}
If $X\in \lSch$, then the naturally induced morphism of spectra
\[
\logTHH(X)
\to
\logTHH(X\times (\P^n,\P^{n-1}))
\]
is an equivalence for all $n\geq 1$.
The same holds for $\logTC^-$, $\logTP$, and $\logTC$.
\end{thm}
\begin{proof}
Owing to Propositions \ref{Hoch.12} and \ref{Hoch.11}, we are reduced to consider $X=\Spec{\Z}$.
The fs log scheme $(\P^n,\P^{n-1})$ admits the Zariski cover $\{U_0,\ldots,U_n\}$,
where
\[
U_0:=\Spec{\Z[x_1/x_0,\ldots,x_n/x_0]}
\]
and
\[
U_i:=\Spec{\Z[x_0/x_i,\ldots,x_n/x_i],\N(x_0/x_i)}
\]
for $1\leq i\leq n$.
For every nonempty subset $I=\{i_1,\ldots,i_r\}\subset \{0,\ldots,n\}$, we set
\[
U_I:=U_{i_1}\cap \cdots \cap U_{i_r}.
\]
Let $\{e_0,\ldots,e_n\}$ be the standard basis in $\Z^{n+1}$, 
and let $P_I$ be the submonoid of $\Z^{n+1}$ generated by $e_i-e_j$ for all $i\in \{0,\ldots,n\}$ and $j\in I$.
We set $P_\emptyset := 0$ and let 
$F_I$ be the face of $P_I$ generated by $e_0-e_j$ for all $j\in I$ if $0\in I$.
By convention, 
$F_I=0$ if $0\notin I$.
There is an isomorphism of fs log schemes
\[
U_I
\cong
\Spec{F_I\to \Z[P_I]}
\]
if $I\neq \emptyset$, 
and we also have
\[
\Spec{\Z}
\cong
\Spec{F_\emptyset \to \Z[P_\emptyset]}.
\]
Let $C_n$ be the $(n+1)$-cube in $\infSpc$ naturally defined by $\Brep(P_I,F_I)$ for all $I\subset \{0,\ldots,n\}$,
which naturally induces an $(n+1)$-cube $\Sphere[C_n]$ in $\infSpt$.
We need to show that there is an equivalence of spectra
\[
\tfib(\Sphere[C_n])
\simeq
0,
\]
which implies the claim by Proposition \ref{logHprop.3}.
Owing to Proposition \ref{Thomdf.15}, it suffices to show there is an equivalence of spectra
\[
\tcofib(\Sphere[C_n])
\simeq
0.
\]
Since $\Sphere[-]$ is a left adjoint,
it preserves small colimits.
Hence it suffices to show there is an equivalence of spaces
\begin{equation}
\label{Hoch.13.6}
\tcofib(C_n)
\simeq
0.
\end{equation}

\

It is convenient to consider $Q_J:=P_{\{0,\ldots,n\}-J}$ and $G_J:=F_{\{0,\ldots,n\}-J}$ for subsets 
$J\subset \{0,\ldots,n\}$.
We set $Q_j:=Q_{\{j\}}$ for $j=0,\ldots,n$.
Let $\{f_1,\ldots,f_n\}$ be the standard basis in $\Z^n$.
By identifying $e_j-e_0$ with $f_j$ for $j=1,\ldots,n$, we have isomorphisms of monoids
\[
Q_j
\cong
\{ (x_1,\ldots,x_n)\in \Z^n : x_j\geq 0 \}
\]
for $j=1,\ldots,n$ and
\[
Q_0
\cong
\{ (x_1,\ldots,x_n)\in \Z^n : x_1+\cdots+x_n \leq 0\}.
\]
Furthermore, we have an isomorphism of monoids
\[
Q_J
\cong
\bigcap_{j\in J} Q_j.
\]

Suppose $0\notin J$.
We set $M_{j,J}:=\N f_j$ if $j\in J$ and $M_{j,J}:=\Z f_j$ if $j\notin J$.
Then there is a canonical isomorphism of monoids
\[
Q_J\cong M_{1,J}\times \cdots \times M_{n,J}.
\]
Furthermore, we have $G_J=0$.
Lemma \ref{Hoch.2} gives a canonical equivalence of spaces
\begin{equation}
\label{Hoch.13.1}
\Brep (Q_J,G_J)
\simeq
\Bcy M_{1,J}
\times
\cdots
\times
\Bcy M_{n,J}.
\end{equation}

On the other hand, suppose $0\in J$.
Let $i$ be any element in $\{0,\ldots,n\}-J$.
We set $M_{j,J}':=\N(f_j-f_i)$ if $j\in J$ and $M_{j,J}':=\Z(f_j-f_i)$ if $j\notin J$.
Then there are isomorphisms of monoids
\[
Q_J
\cong
M_{0,J}'\times M_{1,J}'\times \cdots \times M_{i-1,J}'\times M_{i+1,J}'\times \cdots \times M_{n,J}'
\]
and
\[
Q_{J-\{0\}}
\cong
\Z\times M_{1,J}'\times \cdots \times M_{i-1,J}'\times M_{i+1,J}'\times \cdots \times M_{n,J}'.
\]
Furthermore, $G_J$ is the face of $Q_J$ generated by $f_0-f_i\in M_{0,J}'$.
From Lemma \ref{Hoch.14}, we have a canonical equivalence of spaces
\begin{equation}
\label{Hoch.13.3}
\Brep (Q_J,G_J)
\simeq
\Brep (Q_{J-\{0\}},G_{J-\{0\}})\times_{Q_{J-\{0\}}}Q_J.
\end{equation}
Together with $Q_J\cong Q_{J-\{0\}}\times_{\Z^n} Q_0$, we have a canonical equivalence of spaces
\begin{equation}
\label{Hoch.13.5}
\Brep (Q_J,G_J)
\simeq
\Brep (Q_{J-\{0\}},G_{J-\{0\}}) \times_{\Z^n}Q_0.
\end{equation}

\

Let $V$ be the cofiber of $\Bcy \N\to \Bcy \Z$.
Owing to Proposition \ref{Hoch.9}, there is a decomposition
\[
V
\simeq
\coprod_{j\leq 0} V_j
\]
with $V_0:=\cofib(*\to S_1/C_0)$ and $V_j:=S_1/C_j$ for $j<0$.
In particular, the canonical map $V\to \Z$ factors through $-\N\subset \Z$.
Let $D_n$ (resp.\ $D_n'$) be the $n$-cube naturally associated with $\Brep (Q_J,G_J)$ for all $J$ with $0\notin J$ (resp.\ $0\in J$).
From \eqref{Hoch.13.1}, we see that $D_n$ is the $n$-cube formulated in Proposition \ref{Thomdf.6} associated with the maps
\[
\Bcy (\N f_i)
\to
\Bcy (\Z f_i)
\]
for $i=1,\ldots,n$.
Proposition \ref{Thomdf.6} gives an equivalence of spaces
\[
\tcofib D_n
\simeq
V\times \cdots \times V
\;
(n\text{ copies of }V).
\]
Together with \eqref{Hoch.13.5} we have an equivalence of spaces
\[
\tcofib D_n'
\simeq
(V\times \cdots \times V)\times_{\Z^n}Q_0
\;
(n\text{ copies of }V).
\]
The map $(V\times \cdots \times V)\to \Z^n$ factors through $(-\N)^n$, and we have $(-\N)^n\subset Q_0$.
Hence we have an equivalence of spaces
\begin{equation}
\label{Hoch.13.4}
\tcofib D_n'
\simeq
\tcofib D_n.
\end{equation}
From \eqref{Thomdf.5.3}, we have a cofiber sequence
\[
\tcofib D_n\to \tcofib D_n'\to \tcofib C_n.
\]
Together with \eqref{Hoch.13.4} we deduce $\tcofib C_n\simeq 0$.

\

Finally, we use \eqref{Hoch.37.1} to deduce the claims for $\logTC$ and $\logTC^-$.
\end{proof}

\begin{rmk}
To illustrate the previous proof in the case  $n=1$,  
we note that
\begin{gather*}
\Brep(P_\emptyset,F_\emptyset) \simeq *,
\;
\Brep(P_{\{0\}},F_{\{0\}})\simeq \Brep(-\N) \simeq \coprod_{j\geq 0}S^1/C_j,
\\
\Brep(P_{\{1\}},F_{\{1\}})\simeq \Bcy(\N)\simeq *\amalg \coprod_{j<0}S^1/C_j,
\\
\Brep(P_{\{0,1\}},F_{\{0,1\}}) \simeq \Bcy(\Z) \simeq \coprod_{j\in \Z} S^1/C_j.
\end{gather*}
The equivalence \eqref{Hoch.13.6} follows from the cocartesian square
\[
\begin{tikzcd}
*\ar[d]\ar[r]&
\coprod_{j\geq 0}S^1/C_j\ar[d]
\\
*\amalg \coprod_{j<0}S^1/C_j\ar[r]&
\coprod_{j\in \Z} S^1/C_j.
\end{tikzcd}
\]
\end{rmk}

\begin{const}
\label{Hoch.24}
Suppose $S\in \Sch$.
As a consequence of Remark \ref{Hoch.23} and Theorem \ref{Hoch.13}, we obtain
\begin{equation}
\label{Hoch.31.1}
\logTHH,\logTC^-,\logTP,\logTC
\in
\CAlg(\inflogSH^\eff(S))
\end{equation}
by using the $(\P^\bullet,\P^{\bullet-1})$-localized model in \S \ref{models}.
Moreover, there is an equivalence of spectra
\begin{equation}
\label{Hoch.31.2}
\logTHH(X)
\simeq
\map_{\inflogSHS(S)}(\Sigma_+^\infty X,\logTHH)
\end{equation}
for every $X\in \SmlSm/S$.
The same holds for $\logTC^-$, $\logTP$, and $\logTC$.
\end{const}

\begin{rmk}
We do not know whether \eqref{Hoch.31.2} holds for every $X\in \lSm/S$.
\end{rmk}

\subsection{A \texorpdfstring{$\P^1$}{P1}-bundle formula for \texorpdfstring{$\logTHH$}{logTHH}}\label{ssec:P1bundleTHH}

Blumberg and Mandell established projective bundle formulas for $\THH$ and $\TC$ of schemes in \cite[Theorem 1.5]{BM12}.
They also provided a simpler proof of the loops $\P^1$-bundle formulas for $\THH$ and $\TC$ in \cite[pp.\ 1102-1103]{BM12}.
Dundas proved the loops $\P^1$-bundle formula for rings in \cite[Proposition 3.3]{zbMATH01308186}.
In this subsection, we will prove the loops $\P^1$-bundle formula for $\logTHH$.
This will give the loops $\P^1$-bundle formulas for $\logTC^-$, $\logTP$, and $\logTC$.

\begin{thm}
\label{Hoch.34}
For every $X\in \lSch$, there is an equivalence of spectra
\begin{equation}
\label{Hoch.34.3}
\logTHH(X)
\simeq
\Omega_{\P^1}\logTHH(X).
\end{equation}
The same holds for $\logTC^-$, $\logTP$, and $\logTC$.
\end{thm}
\begin{proof}
Let us construct a cocartesian square
\begin{equation}
\begin{tikzcd}
\Striv\oplus \Striv\ar[d]\ar[r]&
\Sphere[\Bcy \N]\ar[d]
\\
\Sphere[\Bcy (-\N)]\ar[r]&
\Sphere[\Bcy \Z].
\end{tikzcd}
\end{equation}
Using the descriptions of $\Bcy \N$ and $\Bcy \Z$ in Proposition \ref{Hoch.9}, 
we may cancel the terms $S^1/C_j$ for all $j\neq 0$.
Hence our claim follows from the cocartesian square
\begin{equation}
\label{Hoch.34.1}
\begin{tikzcd}
\Striv \oplus \Striv\ar[d]\ar[r]&
\Striv\ar[d,"{(\id,0)}"]\\
\Striv\ar[r,"{(\id,0)}"]&
\Striv \oplus \Sigma_{S^1} \Striv.
\end{tikzcd}
\end{equation}

Proposition \ref{Hoch.11} yields a cocartesian square of cyclotomic spectra
\begin{equation}
\label{Hoch.34.4}
\begin{tikzcd}
\logTHH(R,M)\oplus \logTHH(R,M)\ar[d]\ar[r]&
\logTHH((R,M)\otimes \Z[\N])\ar[d]
\\
\logTHH((R,M)\otimes \Z[-\N])\ar[r]&
\logTHH((R,M)\otimes \Z[\Z]).
\end{tikzcd}
\end{equation}
for every pre-log ring $(R,M)$.
Thus there is a canonical equivalence of spectra
\begin{equation}
\label{Hoch.34.5}
\logTHH(R,M)
\simeq
\Omega_{\P^1} \logTHH(R,M).
\end{equation}
Applying \eqref{Hoch.37.1} to \eqref{Hoch.34.4} yields similar equivalences for $\logTC^-$, $\logTP$, and $\logTC$.

Since \eqref{Hoch.34.5} is canonical, we can globalize it to obtain \eqref{Hoch.34.3}.
\end{proof}

\begin{rmk}
\label{Hoch.35}
From \eqref{Hoch.34.4}, we deduce the map
\[
\partial \colon \Omega_{\G_m}\logTHH \to \Sigma_{S^1}\logTHH.
\]
It is induced by the nontrivial map $d\colon \Striv \oplus \Sigma_{S^1}\Striv \to \Sigma_{S^1}\Striv$ 
obtained from the upper left and lower right corners in \eqref{Hoch.34.1}.
Since $d$ admits a section, $\partial$ admits a section
\[
\tau\colon \Sigma_{S^1}\logTHH \to \Omega_{\G_m}\logTHH.
\]
When restricting to schemes, the former coincides with the map
\begin{equation}
\label{Hoch.35.1}
\tau\colon \Sigma_{S^1}\THH \to \Omega_{\G_m}\THH
\end{equation}
constructed by Blumberg and Mandell \cite[pp.\ 1102-1103]{BM12}.
\end{rmk}

As a consequence of Theorem \ref{Hoch.34}, we have an equivalence in $\infShv_{sNis}(\lSch,\infSpt)$
\begin{equation}
\label{Hoch.28.1}
\logTHH
\xrightarrow{\simeq}
\Omega_T \logTHH.
\end{equation}
By adjunction, we obtain the map
\begin{equation}
\label{Hoch.28.2}
T\wedge \logTHH \to \logTHH.
\end{equation}
There are similar maps for $\logTC^-$, $\logTP$, and $\logTC$ too.

\begin{df}
\label{Hoch.28}
We denote the logarithmic motivic Bott spectrum associated with $\logTHH$ by\index[notation]{logTHHspect @ $\spectlogTHH$}  
\[
\spectlogTHH
:=
\Bott(\logTHH)
=
(\logTHH,\logTHH,\ldots)
\in
\infSpt_T(\infShv_{sNis}(\lSch,\infSpt)).
\]
Here we use Construction \ref{K-theory.9} and the bonding maps $T\wedge \logTHH\to \logTHH$ are given by \eqref{Hoch.28.2}.

Similarly, 
we define\index[notation]{logTCmspect @ $\spectlogTC^{-}$}\index[notation]{logTPspect @ $\spectlogTP$}\index[notation]{logTCspect @ $\spectlogTC$}
\[
\spectlogTC^-
:=
\Bott(\logTC^-),
\text{ }
\spectlogTP
:=
\Bott(\logTP),
\text{ and }
\spectlogTC
:=
\Bott(\logTC).
\]
\end{df}

If $S\in \Sch$, 
then due to \eqref{Hoch.31.1}, 
we obtain
\begin{equation}
\label{Hoch.22.3}
\spectlogTHH,\spectlogTC^-,\spectlogTP,\spectlogTC
\in
\inflogSH(S).
\end{equation}
According to Construction \ref{K-theory.19}, 
these are homotopy commutative monoids in $\inflogSH(S)$.
Moreover, 
there is an equivalence of spectra
\begin{equation}
\label{Hoch.22.1}
\logTHH(X)
\simeq
\map_{\inflogSHS(S)}(\Sigma_{S^1}^\infty X_+,\Omega_{\Gmm}^\infty \spectlogTHH)
\end{equation}
for all $X\in \SmlSm/S$ due to \eqref{K-theory.9.2}.
By adjunction, 
we have an equivalence of spectra
\begin{equation}
\label{Hoch.22.2}
\logTHH(X)
\simeq
\map_{\inflogSH(S)}(\Sigma_T^\infty X_+,\spectlogTHH).
\end{equation}
By construction, we have an equivalence in $\inflogSH(S)$
\begin{equation}
\label{Hoch.22.4}
\spectlogTHH
\simeq
\Sigma^{2,1}\spectlogTHH.
\end{equation}
Similar results hold for $\spectlogTC^-$, $\spectlogTP$, and $\spectlogTC$ too.

\begin{rmk}
We do not know whether \eqref{Hoch.22.1} or \eqref{Hoch.22.2} holds for every $X\in \lSm/S$.
\end{rmk}

\subsection{The logarithmic cyclotomic trace}
\label{subsection:logtrace}

In the article \cite{zbMATH00203552}, 
B\"okstedt, Hsiang, and Madsen constructed the 
cyclotomic trace\index{cyclotomic trace}\index[notation]{Tr @ $\Tr$}
\[
\Tr(R)
\colon
\Kth(R) \to \TC(R)
\]
for every ring $R$.
Since the functor $\Kth$ (resp.\ $\TC$) satisfies Nisnevich (resp.\ \'etale) descent by 
\cite[Theorem 10.8]{TT} (resp.\ \cite[Corollary 3.3.3]{zbMATH01421292}), 
the cyclotomic trace can be extended to schemes.
In other words, we have a map
\begin{equation}
\label{trace.0.1}
\Tr
\colon
\Kth\to \TC
\end{equation}
in $\infShv_{Nis}(\Sch,\infSpt)$.
The purpose of this subsection is to construct the cyclotomic trace $\logTr\colon \Kthlog \to \logTC$ for $\lSm/k$, 
where $k$ is a perfect field admitting resolution of singularities.
We will also lift $\logTr$ to a map of log $\P^1$-spectra. In our approach, the existence of the cyclotomic trace will actually be a consequence of the fact that $\logTC$ is orientable, with multiplicative group law. We thank the referee for suggesting this perspective.

We begin by observing the following fact. 
\begin{prop}
\label{trace.7}
There is a commutative square in $\infShv_{Zar}(\Sch,\infSpt)$
\begin{equation}
\label{trace.7.1}
\begin{tikzcd}
\Kth\ar[d]\ar[r,"\Tr"]&
\TC\ar[d]
\\
\Omega_{\P^1}\Kth\ar[r,"\Tr"]&
\Omega_{\P^1}\TC.
\end{tikzcd}
\end{equation}
Here the vertical maps are given by the loops $\P^1$-bundle formulas for $\Kth$ and $\TC$.
\end{prop}
\begin{proof}
Blumberg and Mandell showed in \cite[p.\ 1102]{BM12} that the square
\begin{equation}
\label{trace.7.2}
\begin{tikzcd}
\Sigma_{S^1}\Kth\ar[r,"\Tr"]\ar[d,"\tau"']&
\Sigma_{S^1}\TC\ar[d,"\tau"]
\\
\Omega_{\P^1}\Kth\ar[r,"\Tr"]&
\Omega_{\P^1}\TC
\end{tikzcd}
\end{equation}
commutes, where the left vertical map is the map appearing in the construction of Bass $K$-theory, 
and the right vertical map is \eqref{Hoch.35.1}.

Suppose $\cF\in \infShv_{Zar}(\Sch,\infSpt)$ and there is a map $\tau\colon \Sigma_{S^1}\cF\to \Omega_{\G_m}\cF$.
The cartesian square
\[
\begin{tikzcd}
\Omega_{\P^1}\cF\ar[d]\ar[r]&
\Omega_{\A^1}\cF\ar[d]
\\
\Omega_{\A^1}\cF\ar[r]&\Omega_{\G_m}\cF
\end{tikzcd}
\]
yields a map $\Omega_{\G_m}\cF\to \Sigma_{S^1}\Omega_{\P^1}\cF$.
Compose with $\tau$ to get a map $\cF\to \Omega_{\P^1}\cF$.
It is a classical fact that this coincides with the loops $\P^1$-bundle formula when $\cF=\Kth$.
The same holds when $\cF=\TC$ by Remark \ref{Hoch.35}.
Hence we can obtain \eqref{trace.7.1} from \eqref{trace.7.2}.
\end{proof}

Suppose $S\in \Sch$ is a scheme.
Since $\spectlogTHH$, $\spectlogTC^{-}$, $\spectlogTP$,
and $\spectlogTC$ are homotopy commutative monoids in $\inflogSH(S)$ by Construction \ref{K-theory.19}, 
we can define the graded rings
\[
\spectlogTHH^{**}(X),
\text{ }
(\spectlogTC^{-})^{**}(X),
\text{ },
\spectlogTP^{**}(X),
\text{ and }
\spectlogTC^{**}(X)
\]
for all $X\in \lSm/S$ according to Definition \ref{Ori.12}.

\begin{thm}
\label{trace.8}
If $S\in \Sch$ is a scheme, 
then
\[
\spectlogTHH,\spectlogTC^{-},\spectlogTP,\spectlogTC\in \inflogSH(S)
\]
are multiplicatively oriented.
\end{thm}
\begin{proof}
As a consequence of the projective bundle formula for $K$-theory, we have an isomorphism of rings
\[
\Kth_0(\P^\infty) \cong \Kth_0(S)[[x]].
\]
Consider the naturally induced commutative diagram
\[
\begin{tikzcd}
\Kth_0(S)\ar[r,"\cong"]\ar[d,"\Tr"']&
\Kth_0(\P^1/\pt)\ar[d,"\Tr"]\ar[r,leftarrow]&
\Kth_0(\P^\infty/\pt)\ar[d,"\Tr"]
\\
\logTC_0(S)\ar[d,"\cong"']\ar[r,"\cong"]&
\logTC_0(\P^1/\pt)\ar[d,"\cong"]\ar[r,leftarrow]&
\logTC_0(\P^\infty/\pt)\ar[d,"\cong"]
\\
\spectlogTC^{0,0}\ar[rd,"\cong","\Sigma^{2,1}"']&
\spectlogTC^{0,0}(\P^1/\pt)\ar[d,"\cong"]\ar[r,leftarrow]&
\spectlogTC^{0,0}(\P^\infty/\pt)\ar[d,"\cong"]
\\
&
\spectlogTC^{2,1}(\P^1/\pt)\ar[r,leftarrow]&
\spectlogTC^{2,1}(\P^\infty/\pt),
\end{tikzcd}
\]
where the upper left square comes from Proposition \ref{trace.7}.
Let $c_\infty\in \spectlogTC^{2,1}(\P^\infty/\pt)$ and $c_1\in \spectlogTC^{2,1}(\P^1/\pt)$ be the images of $x\in \Kth_0(\P^\infty/\pt)$.
Since the image of the unit $\unit\in \Kth_0(X)$ in $K_0(\P^1/\pt)$ coincides with the image of $x\in \Kth_0(\P^\infty/\pt)$ in $K_0(\P^1/\pt)$, $c_1$ is the image of $\unit\in \Kth_0(X)$.
It follows that $c_\infty$ gives an orientation of $\spectlogTC$.

The homomorphism
\[
\Kth_0(\P^\infty) \to \Kth_0(\P^\infty \times \P^\infty)
\]
naturally induced by \eqref{Ori.61.1} is isomorphic to the homomorphism
\[
\Kth_0(S)[[x]]\to \Kth_0(S)[[x,y]]
\]
sending $x$ to $x+y-xy$.
Consider the periodicity isomorphism $\spectlogTC^{0,0}\cong \spectlogTC^{2,1}$, 
and let $\gamma$ be the image of the unit $1\in \spectlogTC^{0,0}$ in $\spectlogTC^{2,1}$.
Since $\spectlogTC$ is oriented, 
the naturally induced commutative square
\[
\begin{tikzcd}
\Kth_0(\P^\infty)\ar[r]\ar[d]&
\Kth_0(\P^\infty\times \P^\infty)\ar[d]
\\
\spectlogTC^{0,0}(\P^\infty)\ar[r]&
\spectlogTC^{0,0}(\P^\infty\times \P^\infty)
\end{tikzcd}
\]
can be written as
\[
\begin{tikzcd}
\Kth_0(S)[[x]]\ar[r]\ar[d]&
\Kth_0(S)[[x,y]]\ar[d]
\\
\spectlogTC^{0,0}[[x/\gamma]]\ar[r]&
\spectlogTC^{0,0}[[x/\gamma,y/\gamma]].
\end{tikzcd}
\]
The left (resp.\ right) vertical homomorphism sends $x$ to $x/\gamma$ (resp.\ $x$ to $x/\gamma$ and $y$ to $y/\gamma$).
Thus the lower horizontal homomorphism sends $x/\gamma$ to $x/\gamma+y/\gamma-xy/\gamma^2$.
Setting $\beta=\gamma^{-1}$ it follows that the formal group law on $\spectlogTC$ is given by $x+y-\beta xy$.

For $\spectlogTHH$, $\spectlogTC^{-}$ and $\spectlogTP$, use Proposition \ref{trace.6}.
\end{proof}

\begin{thm}
\label{trace.4}
Assume that $k$ is a perfect field admitting resolution of singularities.
There is a unique homomorphism\index[notation]{logTrspect @ $\spectlogTr$}
\[
\spectlogTr\colon \logKGL \to \spectlogTC
\]
of homotopy commutative monoids in $\inflogSH(k)$ preserving orientations.
\end{thm}
\begin{proof}
This is an immediate consequence of Theorems \ref{Ori.71} and \ref{trace.8}.
\end{proof}

\begin{rmk}
\label{trace.12}
Apply $\Omega_{\Gmm}^\infty$ to $\spectlogTr$ in Theorem \ref{trace.4} to obtain
\[
\logTr
\colon
\Kthlog \to \logTC
\]
in $\inflogSH^\eff(k)$.
This extends the cyclotomic trace $\Tr\colon \Kth\to \TC$,
i.e., $\logTr\colon \Kthlog(X)\to \logTC(X)$ can be identified with $\Tr\colon \Kth(X)\to \TC(X)$ for $X\in \Sm/k$,
and its proof relies on \cite{logSHF1} as follows.
We begin with the log cyclotomic trace
\[
f\colon 
\Kthlog
\to
\logTC
\]
obtained by \cite[Theorem F]{logSHF1},
which is a homomorphism of $\E_\infty$-rings in $\inflogSH^\eff(k)$ and extends the cyclotomic trace for $\Sm/k$.
This induces a map $f$ from
\[
\logKGL\simeq (\omega^*\Omega_{S^1}^\infty \Kth,\omega^*\Omega_{S^1}^\infty \Kth,\ldots)
\]
to
\[
\spectlogTC\simeq (\Omega_{S^1}^\infty \logTC,\Omega_{S^1}^\infty \logTC,\ldots).
\]
To show that this is a map of homotopy commutative monoids in $\inflogSH(k)$,
it suffices to show that the diagram of $(\P^1)^{\wedge 2}$-spectra
\[
\begin{tikzcd}
(\omega^*\Omega_{S^1}^\infty \Kth\wedge \omega^*\Omega_{S^1}^\infty \Kth,\ldots)\ar[d]\ar[r]&
(\omega^*\Omega_{S^1}^\infty \Kth,\ldots)\ar[d]
\\
(\Omega_{S^1}\logTC\wedge \Omega_{S^1}\logTC,\ldots)\ar[r]&
(\Omega_{S^1}\logTC,\ldots)
\end{tikzcd}
\]
commutes.
This is a consequence of the fact that $f\colon \Kthlog\to \logTC$ is a homomorphism of $\E_\infty$-rings.
Argue as in the proof of Theorem \ref{trace.8} to show that $f\colon \logKGL \to \spectlogTC$ preserves orientations.
The uniqueness in Theorem \ref{trace.4} finishes the proof.
\end{rmk}

\begin{prop}
\label{trace.9}
If $S\in \Sch$ is a scheme, 
there are isomorphisms of graded rings
\[
\spectlogTHH^{**}(\Gr(r))
\cong
\spectlogTHH^{**}[[c_1,\ldots,c_r]]
\]
and
\[
\spectlogTHH^{**}(\Gr)
\cong
\spectlogTHH^{**}[[c_1,c_2,\ldots]].
\]
Similar results hold for $\spectlogTC^{-}$, $\spectlogTP$, and $\spectlogTC$.
\end{prop}
\begin{proof}
Combine Corollary \ref{Ori.10} and Theorem \ref{trace.8}.
\end{proof}

\begin{prop}[Thom equivalence]
\label{trace.10}
Suppose $S\in \Sch$, $X\in \SmlSm/S$, and $\cE\to X$ is a rank $d$ vector bundle.
Then there exists a canonical equivalence of spectra
\[
\logTHH(X)
\xrightarrow{\simeq}
\logTHH(\Thom(\cE)).
\]
The same holds for $\logTC^-$, $\logTP$, and $\logTC$.
\end{prop}
\begin{proof}
Combine Proposition \ref{Ori.70}, \eqref{Hoch.22.2}, and \eqref{Hoch.22.4}.
\end{proof}

\begin{prop}
\label{trace.11}
Suppose $S\in \Sch$ and $X\in\SmlSm/S$ is of the form $(\ul{X},Z_1+\cdots+Z_r)$ and $Z$ is a strict smooth closed subscheme of $X$ strict normal crossing with $Z_1+\cdots+Z_r$.
Then there exists a cofiber sequence of spectra
\[
\logTHH(Z)
\to
\logTHH(X)
\to
\logTHH(\Blow_Z X,E),
\]
where $E$ is the exceptional divisor on $\Blow_Z X$.

The same holds for $\logTC^-$, $\logTP$, and $\logTC$.
\end{prop}
\begin{proof}
Immediate from Theorem \ref{Thom.1} and Proposition \ref{trace.10}.
\end{proof}

\subsection{Representability of Nygaard completed prismatic cohomology}
\label{subsection:rncpc}

In \cite[Theorem 1.12(2)]{BMS19}, 
Bhatt-Morrow-Scholze uses \v{C}ech descent for the quasi-syntomic topology to construct their 
``motivic filtration''
\[
\mathrm{Fil}^\ast \TC^-(A;\Z_p)
\]
on negative topological cyclic homology $\TC^-(A;\Z_p)$ 
--- here $A$ is any quasi-syntomic $\Z_p$-algebra in the sense of \cite[Definition 4.10]{BMS19}. 
The $0$-th graded piece is denoted by
\[
\widehat{\mathbf{\Delta}}_A
:=
\mathrm{gr}^0\TC^-(A;\Z_p).
\]
The Nygaard filtration $\cN^{\geq \ast}\widehat{\mathbf{\Delta}}$ on $\widehat{\mathbf{\Delta}}_A$ is constructed 
in \cite[Theorem 7.2(2)]{BMS19}.
Bhatt-Scholze's \cite[Theorem 13.1]{BS22} shows that $\widehat{\mathbf{\Delta}}_A$ is equivalent to the 
completion of the prismatic cohomology $\mathbf{\Delta}_A$ with respect to the Nygaard filtration.
\begin{rmk}
Suppose $A$ is a smooth $k$-algebra and $k$ is a perfect field of characteristic $p$.
Then, by \cite[Theorem 1.10]{BMS19}, 
there is an equivalence between  $\widehat{\mathbf{\Delta}}_A$ and the crystalline cohomology of $A$.
In particular, 
this implies that $\widehat{\mathbf{\Delta}}_A$ is not $\A^1$-invariant, 
and hence there is no ($\mathbb{A}^1$-invariant) motivic $S^1$-spectrum $\widehat{\mathbf{\Delta}}_A\in \infSH_{S^1}(A)$ 
providing an equivalence of spectra
\[
\map_{\infSHS(A)}
(
\Sigma_{S^1}^\infty \Spec{B}_+,\widehat{\mathbf{\Delta}}_A
)
\simeq
\widehat{\mathbf{\Delta}}_B
\]
for every smooth $A$-algebra $B$.
\end{rmk}
According to \cite{BLPO2},
logarithmic negative topological cyclic homology 
$\logTC^-$ also admits a motivic filtration --- 
constructed using \v{C}ech descent for a log variant of the quasi-syntomic topology.
The corresponding $0$-th graded piece furnishes a log extension $\widehat{\mathbf{\Delta}}_{(A,M)}$ 
of $\widehat{\mathbf{\Delta}}_{A}$ for every log quasi-syntomic integral pre-log ring $(A,M)$.
This can be further extended to log schemes, giving rise to the Nygaard completed log prismatic cohomology $R\Gamma_{\widehat{\mathbf{\Delta}}}(\mathfrak{X})$.
With a log variant of \cite[Proposition 7.8]{BMS19},
it is possible to show that $R\Gamma_{\widehat{\mathbf{\Delta}}}(-)$ is 
$(\P^n,\P^{n-1})$-invariant for every integer $n\geq 1$.

We conclude this section by stating a result from \cite{BLMP}. The following definition is a straightforward adaptation of our constructions to the setting of formal schemes.
\begin{df}
Let $S$ be a quasi-compact and quasi-separated $p$-adic formal scheme. 
We denote by $\inflogFDA^\eff(S,\Z_p)$ be the monoidal stable $\infty$-category of formal log motives over $S$, that is the full $\infty$-subcategory of $\infShv_{s\et}(\mathbf{FlSm}/S, \Z_p)$ spanned by the objects $M$ that are local with respect to the maps $\Z_p(\mathfrak{X}\times ((\P^n)^\wedge_p, (\P^{n-1})^\wedge_p)) \to \Z_p(\mathfrak{X})$ for all $\mathfrak{X}\in \mathbf{FlSm}/S$ and $n\geq 1$. We denote by $\inflogFDA(S,\Z_p)$ the category of $T$-spectra in  $\inflogFDA^\eff(S,\Z_p)$, defined in the obvious way.
\end{df}
\begin{thm}[{\cite[Theorem 3.5]{BLMP}}] 
    Let $S$ be a quasi-syntomic $p$-adic formal scheme. There are oriented ring spectra $\mathbf{E}^{\widehat{\mathbf{\Delta}}}$ and  $\mathbf{E}^{\mathrm{Fil}\widehat{\mathbf{\Delta}}}$ in $\CAlg(\inflogFDA^\eff(S,\Z_p))$ such that for all $\mathfrak{X}\in \mathbf{FlSm}/S$ we have
    \[
\Map_{\inflogFDA(S,\Z_p)}(\Sigma^{\infty}(\mathfrak{X}),\Sigma^{r,s}\mathbf{E}^{\widehat{\mathbf{\Delta}}}) \simeq R\Gamma_{\widehat{\mathbf{\Delta}}}(\mathfrak{X})\{s\}[r]
\]
and
\[
\Map_{\inflogFDA(S,\Z_p)}(\Sigma^{\infty}(\mathfrak{X}),\Sigma^{r,s}\mathbf{E}^{\mathrm{Fil}\widehat{\mathbf{\Delta}}}) \simeq \mathrm{Fil}_N^{\geq s}R\Gamma_{\widehat{\mathbf{\Delta}}}(\mathfrak{X})\{s\}[r],
\]
where $\{s\}$ denotes the Breuil--Kisin twist. 
If $S\in \mathrm{QSyn}/R$ for $R$ perfectoid, $\mathbf{E}^{\widehat{\mathbf{\Delta}}}$ and $\mathbf{E}^{\mathrm{Fil}\widehat{\mathbf{\Delta}}}$ live in $\inflogFDA(S,A_{\inf}(R))$, and there are oriented ring spectra $\mathbf{E}^{\widehat{\mathbf{\Delta}}^{\rm nc}}$ and $\mathbf{E}^{\mathrm{Fil}\widehat{\mathbf{\Delta}}^{\rm nc}}$ in $\CAlg(\inflogFDA(S,A_{\inf}(R)))$, together with equivalences 
\begin{equation}\label{eq:prism-deRham}
\mathbf{E}^{\widehat{\mathbf{\Delta}}}\otimes_{\widehat{\mathbf{\Delta}}_R,\theta}R \simeq \E^{\rm \widehat{dR}}\qquad \mathbf{E}^{\widehat{\mathbf{\Delta}}^{\rm nc}}\otimes_{\widehat{\mathbf{\Delta}}_R,\theta}R \simeq \E^{\rm dR}        
\end{equation}
of oriented ring spectra in $\CAlg(\inflogFDA(S,R))$, where $\theta \colon A_{\rm inf}(R) \to R$ is Fontaine's period map, and $\E^{\rm \widehat{dR}}$ is the motivic spectrum representing (completed) derived de Rham cohomology.
\end{thm}
Among the applications of the results constructed in this work, we deduce from the above representabiliy results Gysin sequences and blow-up formulas,
and we compute the prismatic cohomology of Grassmannians. See \cite{BLMP} for a full discussion.

\newpage

\section{Realization functors}
\label{section:rf}
Realization functors are very useful for the purpose of comparing $\A^1$-homotopy theory with 
classical homotopy theories.
For example, 
Levine \cite{zbMATH06366499} used the Betti realization to prove that the naturally induced functor
\[
\infSH\to \infSH(k)
\]
is fully faithful when $k$ is an algebraically closed field of characteristic $0$.

\

The Betti realization is not available if the characteristic of $k$ is nonzero;
however, the $\ell$-adic \'etale realization is a useful replacement in positive characteristic.
Zargar \cite{zbMATH07103864} used the $\ell$-adic \'etale realization to prove the naturally induced  functor
\[
\infSH[1/p] \to \infSH(k)[1/p]
\]
is fully faithful when $k$ is an algebraically closed field of exponential characteristic $p$.
A closely related result due to Wilson and {\O}stv{\ae}r \cite{zbMATH06698209} shows the naturally induced homomorphism
\[
\hom_{\infSH}(\Sphere,\Sigma^n \Sphere)
\to
\hom_{\infSH(k)}(\unit,\Sigma^{n,0}\unit)
\]
is an isomorphism for every integer $n$; 
their proof employs Witt vectors to reduce to algebraically closed fields of characteristic zero 
($\Sphere$ and $\unit$ denote the topological and motivic sphere, respectively).

\begin{rmk}In order to construct the $p$-adic \'etale realization within the framework of $\A^1$-homotopy theory,
one needs to invert the exponential characteristic $p$ because \'etale cohomology with 
$\Z/p^n$-coefficient is not $\A^1$-invariant.
In the setting of logarithmic motives, 
Merici \cite{2211.14303} constructed a $p$-adic realization functor
\[
\inflogDMeff_{\leta}(k)\to \infD(W(k))^\mathrm{op}
\]
to the opposite derived $\infty$-category of $W(k)$-modules for all field $k$, 
by showing that log de Rham-Witt sheaves admit a log transfer structure, and that they are $(\P^n, \P^{n-1})$-invariant for every $n\geq1$.
\end{rmk}
In this section, 
we will explain how the log \'etale realization can be used to somehow incorporate $\Z/p^n$-coefficients.
Moreover, 
we will also construct
Kato-Nakayama realization functors
\[
\inflogSH(S)\to \infSpt
\]
for every subfield $k$ of $\C$ and $S\in \lSch/k$, 
which is a logarithmic analogue of the Betti realization in $\A^1$-homotopy theory.

\subsection{\texorpdfstring{$\boxx$-invariance}{Invariance} of Kummer \'etale cohomology}

In this subsection, 
we show that log \'etale cohomology groups are $\boxx$-invariant for locally constant sheaves.
This will be used to construct log \'etale realization functors in \S\ref{subsection:ler}.

\begin{df}
Recall from \cite[Definition 2.5]{MR1457738} that a \emph{log geometric point}\index{log geometric point} 
of an fs log scheme $X$ is a morphism $f\colon Z\to X$ of saturated log schemes such that $\ul{Z}$ is the 
spectrum of a separably closed field and the multiplication homomorphism of abelian groups
$\cdot n \colon \ol{\cM}^{\gp}(Z)\to \ol{\cM}^{\gp}(Z)$ is an isomorphism for every integer $n\geq 2$ 
prime to the characteristic of $k$.
The log geometric points of $X$ form enough points for the small Kummer \'etale site $X_{\ket}$.
\end{df}

\begin{lem}
\label{ketreal.4}
Let $X$ be a log geometric point.
If $h\colon X\times \A^1\to X$ denotes the projection, then we have
\[
\R_{k\et}^i h_!( \Z/n)
\simeq
\left\{
\begin{array}{ll}
\Z/n &
\text{if }i=2,
\\
0 & \text{otherwise}
\end{array}
\right.
\]
for all integer $n\geq 2$ prime to the characteristic of $X$.
\end{lem}
\begin{proof}
By \cite[Paragraph 4.7(ii)]{MR1457738}, we reduce to the case when $X$ has trivial log structure.
There is a commutative diagram
\[
\begin{tikzcd}
X\ar[r,"i"]\ar[rd,"\id"']&
X\times \P^1\ar[d,"f"]&
X\times \A^1\ar[l,"j"']\arrow[ld,"h"]
\\
&
X
\end{tikzcd}
\]
where $f$ is the projection, $i$ is the inclusion at $\infty$, and $j$ is the complement of $i$.
From the localization sequence \cite[\S 2.8(4)]{MR1457738}
\[
0\to j_! j^* \to \id \to i_*i^* \to 0,
\]
we have a distinguished triangle
\[
\R_{\et} h_! (\Z/n) \to \R_{\et} f_! (\Z/n) \to \Z/n \to \R_{\et} h_! (\Z/n)[1].
\]
To conclude, 
observe that
\cite[Corollaire IX.4.7]{SGA4} gives $\R_{\et}^i f_!( \Z/n)=\Z/n$ for $i=0,2$ and $\R_{\et}^i f_!( \Z/n)=0$ otherwise.
\end{proof}

\begin{lem}
\label{ketreal.3}
Let $X$ be a log geometric point.
Then we have
\[
H_{k\et}^i(X\times \pt_\N, \Z/n)
\simeq
\left\{
\begin{array}{ll}
\Z/n &
\text{if }i=0,1
\\
0 & \text{otherwise}
\end{array}
\right.
\]
for all integer $n\geq 2$ prime to the characteristic of $X$.
\end{lem}
\begin{proof}
Suppose $X$ has characteristic $p$ and set $G:=\ol{\cM(X)}^\gp$.
For every prime $\ell\neq p$, 
the homomorphism $\cdot \ell \colon G\to G$ is an isomorphism of abelian groups.
This implies
\[
\pi(G):=\hom\big( G,\coprod_{\ell\neq p}\Z_{\ell}\big) =0.
\]
Hence $\pi(G\oplus \Z) \cong \pi(\Z)$.
Owing to \cite[Proposition 1.2]{NakayamaII}, there are isomorphisms of abelian groups
\[
H_{k\et}^i(X\times \pt_\N, \Z/n)
\cong
H^i(\pi(G\oplus \Z),\Z/n)
\cong
H^i(\pi(\Z),\Z/n)
\cong
H_{k\et}^i(\ul{X}\times \pt_\N, \Z/n),
\]
where the second and third groups are the profinite cohomology groups.
Hence we reduce to the case when $X$ has trivial log structure.

\ 

There is a commutative diagram
\[
\begin{tikzcd}
X\times \pt_\N\ar[r,"i"]\ar[rd,"g"']&
X\times \boxx\ar[d,"f"]&
X\times \A^1\ar[l,"j"']\arrow[ld,"h"]
\\
&
X,
\end{tikzcd}
\]
where $f$, $g$, and $h$ are the projections, $i$ is the inclusion at $\infty$, and $j$ is the complement of $i$.
From the localization sequence \cite[\S 2.8(4)]{MR1457738}
\[
0\to j_! j^* \to \id \to i_*i^* \to 0,
\]
we deduce the distinguished triangle
\[
\R_{k\et} h_! (\Z/n) \to \R_{k\et} f_! (\Z/n) \to \R_{k\et} g_!\Z/n \to \R_{k\et} h_! (\Z/n)[1].
\]
Hence due to Lemma \ref{ketreal.4}, there is an exact sequence
\begin{equation}
\label{ketreal.3.1}
\begin{split}
0 &\to H_{k\et}^1(X\times \boxx,\Z/n) \to  H_{k\et}^1(X\times \pt_\N,\Z/n)
\\
 &\to  \Z/n \to  H_{k\et}^2(X\times \boxx,\Z/n) \to  H_{k\et}^2(X\times \pt_\N,\Z/n).
\end{split}
\end{equation}

Owing to \cite[Proposition 5.3]{MR1457738}, we have the vanishing
\[
H_{k\et}^i(X\times \pt_\N, \Z/n)=0
\]
for every integer $i\geq 2$.
Apply \cite[Theorem 7.4]{MR1922832} to $X\times \boxx$ to deduce 
\[
H_{k\et}^i(X\times \boxx,\Z/n)
\cong
H_{\et}^i(X\times \A^1,\Z/n)
=
0
\]
for every integer $i\geq 1$.
Together with \eqref{ketreal.3.1}, this finishes the computation.
\end{proof}

\begin{lem}
\label{ketreal.2}
Let $X$ be a log geometric point.
Then we have
\[
H_{k\et}^i(X\times \boxx, \Z/n)
\cong
\left\{
\begin{array}{ll}
\Z/n &
\text{if }i=0,
\\
0 & \text{otherwise}
\end{array}
\right.
\]
for all integers $n\geq 2$.
\end{lem}
\begin{proof}
Suppose $X$ has characteristic $p$.
It suffices to show the claim when $n$ is prime.
If $n=p$,
the Artin-Schreier sequence
\[
0\to \Z/p \to \cO \xrightarrow{F-\id} \cO \to 0
\]
gives a long exact sequence
\begin{equation}
\label{ketreal.2.1}
\begin{split}
\cdots \to & H_{k\et}^i(X\times \boxx,\Z/p) \to H_{k\et}^i(X\times \boxx,\cO) \to H_{k\et}^i(X\times \boxx,\cO)\\
\to & H_{k\et}^{i+1}(X\times \boxx,\Z/p) \to \cdots.
\end{split}
\end{equation}
Due to \cite[Proposition 6.5]{KatoLog2} (see also \cite[Proposition 3.27]{MR2452875}), there is an isomorphism of abelian groups
\[
H_{k\et}^i(X\times \boxx,\cO)
\cong
H_{\et}^i(X\times \P^1,\cO).
\]
This group is trivial for $i\geq 1$.
Hence we conclude from \eqref{ketreal.2.1}.

\

Suppose $n\neq p$.
Due to \cite[Proposition 5.3]{MR1457738}, we have the vanishing
\[
H_{k\et}^i(X\times \boxx, \Z/n)=0
\]
for every $i\geq 3$.
The exact sequence \eqref{ketreal.3.1} and Lemma \ref{ketreal.3} yield the exact sequence
\[
0 \to H_{k\et}^1(X\times \boxx,\Z/n) \to  \Z/n
\to  \Z/n \to  H_{k\et}^2(X\times \boxx,\Z/n) \to  0.
\]
Hence, to finish the proof, it suffices to show the vanishing
\[
H_{k\et}^1(X\times \boxx, \Z/n)=0.
\]
The Kummer sequence
\[
0 \to \mu_n \to \cM^{\gp} \xrightarrow{\cdot n} \cM^{\gp} \to 0
\]
induces the exact sequence
\[
H_{k\et}^0(X\times \boxx,\cM^{\gp})
\to
H_{k\et}^0(X\times \boxx,\cM^{\gp})
\to
H_{k\et}^1(X\times \boxx,\Z/n)
\to
H_{k\et}^1(X\times \boxx,\cM^{\gp}).
\]
Apply Proposition \ref{logPic.6} to conclude.
\end{proof}

\begin{thm}
\label{ketreal.1}
Let $X$ be an fs log scheme.
For every locally constant Kummer \'etale sheaf $\cF$ on $X_{k\et}$, 
the projection $p\colon X\times \boxx\to X$ naturally induces an isomorphism of graded rings
\[
p^*\colon H_{k\et}^*(X,\cF)\xrightarrow{\cong} H_{k\et}^*(X\times \boxx,p^*\cF).
\]
\end{thm}
\begin{proof}
There is a Leray spectral sequence
\[
E_2^{pq}=H^q_{k\et}(X,\R^p_{k\et}p_*p^*\cF) \Rightarrow H^{p+q}(X\times \boxx,p^*\cF).
\]
Hence it suffices to show the vanishing
\[
\R^i_{k\et}p_*p^*\cF=0
\]
for every integer $i\geq 1$.
Let $x$ be a log geometric point of $X$.
We reduce to showing the vanishing
\[
(\R^i_{k\et}p_*p^*\cF)_x=0.
\]

\ 

There is a naturally induced cartesian square
\[
\begin{tikzcd}
x\times \boxx\arrow[d,"p'"']\arrow[r,"i'"]&
X\times \boxx\arrow[d,"p"]
\\
x\arrow[r,"i"]&
X.
\end{tikzcd}
\]
Applying the proper base change theorem \cite[Theorem 5.1]{MR1457738}
\[
i^*\R p_* \xrightarrow{\cong} \R p'^* i_*'
\]
we reduce to the case when $X$ is a log geometric point.
In this case, $\cF$ is a constant sheaf because it is locally constant.
Lemma \ref{ketreal.2} finishes the proof.
\end{proof}

Thanks to the following, 
several results in this section also hold for log \'etale cohomology.

\begin{prop}
\label{ketreal.15}
Let $X$ be an fs log scheme.
Then for every torsion abelian group $A$, 
there is a canonical isomorphism of graded rings
\[
H_{\ket}^*(X,A)
\cong
H_{\leta}^*(X,A).
\]
\end{prop}
\begin{proof}
Follows from \cite[Proposition 5.4(2)]{NakayamaII}.
\end{proof}

\subsection{The log \'etale realization}
\label{subsection:ler}
Artin and Mazur \cite{zbMATH03290088} defined the \'etale homotopy type of a scheme $X$ as a pro-object 
in the homotopy category of simplicial sets.
Friedlander \cite{zbMATH03855996} promoted this to a pro-object of simplicial sets.
More recently, 
Lurie established that the \'etale homotopy type can given by the shape of an $\infty$-topos, 
see \cite[Example 7.1.6.9]{HTT}.
Suppose $k$ is a field of characteristic $p$.
For every prime $\ell\neq p$, 
Isaksen \cite[Theorem 6]{zbMATH02081955} constructed a left Quillen functor, 
whose underlying functor of $\infty$-categories is the \'etale realization functor
\[
\infH(k)
\to
\infPro(\infSpc)^{\Z/\ell}
\]
sending $X\in \Sm/k$ to its \'etale homotopy type for every prime $\ell\neq p$,
where $\infPro(\infSpc)^{\Z/\ell}$,
reviewed in Definition \ref{ketreal.7} below,
is a certain localization of the $\infty$-category of pro-spaces $\infPro(\infSpc)$.
Here, $\infPro(-)$ denotes the $\infty$-category of pro-objects.
An equivalent construction in terms of shape theory was given by Hoyois in \cite{Hoyoisetalesymmetric}.
Based on this, 
Zargar \cite{zbMATH07103864} constructed the stable \'etale realization functor.
To construct the log \'etale realization functor one may proceed as Isaksen in \cite{zbMATH02081955};
however, 
we will follow the method developed by Hoyois and Zargar using $\infty$-categories.

\begin{df}
\label{ketreal.7}
Let $X=\{X_\alpha\}_{\alpha\in I}$ be a pro-space, $A$ an abelian group, and $i\in\mathbb{Z}$.
The $i$th singular cohomology of $X$ with $A$-coefficients is the abelian group
\[
H^i(X,A)
:=
\pi_i(\Map_{\infPro(\infSpc)}(X,A))
\cong
\colimit_{\alpha\in I} H^i(X_\alpha,A).
\]
Let $\Eilenberg \Z/\ell$ denote the Eilenberg-MacLane spectrum associated with the abelian group $\Z/\ell$.
A map $f\colon Z\to Y$ in $\infPro(\infSpc)$ is an \emph{$(\Eilenberg \Z/\ell)^*$-local equivalence} 
if the naturally induced homomorphism of graded rings
\[
f^*
\colon
H^*(Y,\Z/\ell)
\to
H^*(Z,\Z/\ell)
\]
is an isomorphism.
A pro-space $X\in \infPro(\infSpc)$ is \emph{$(\Eilenberg \Z/\ell)^*$-local} if the map
\[
\Map_{\infPro(\infSpc)}(Y,X)
\to
\Map_{\infPro(\infSpc)}(Z,X)
\]
is an equivalence of spaces for every $(\Eilenberg \Z/\ell)^*$-local equivalence $f\colon Z\to Y$ in $\infPro(\infSpc)$.

Let $\infPro(\infSpc)^{\Z/\ell}$ denote the full subcategory of $\infPro(\infSpc)$ spanned by the 
$(\Eilenberg \Z/\ell)^*$-local pro-spaces.
\end{df}

There is an analogous definition for spectra as follows.

\begin{df}
\label{ketreal.8}
A map $f\colon Z\to Y$ in $\infPro(\infSpt)$ is an 
\emph{$(\Eilenberg \Z/\ell)^*$-local equivalence}\index{HZl-local equivalence@ $(\Eilenberg \Z/\ell)^*$-local equivalence} 
if the naturally induced map of spaces
\[
f^*
\colon
\Map_{\infPro(\infSpt)}(Y,\Eilenberg \Z/\ell)
\to
\Map_{\infPro(\infSpt)}(Z,\Eilenberg \Z/\ell)
\]
is an equivalence.
A pro-spectrum $X\in \infPro(\infSpt)$ is \emph{$(\Eilenberg \Z/\ell)^*$-local} if the map of spaces
\[
\Map_{\infPro(\infSpt)}(Y,X)
\to
\Map_{\infPro(\infSpt)}(Z,X)
\]
is an equivalence of spaces for every $(\Eilenberg \Z/\ell)^*$-local equivalence $f\colon Z\to Y$ in $\infPro(\infSpt)$.

Let $\infPro(\infSpt)^{\Z/\ell}$\index[notation]{ProSptZl @ $\infPro(\infSpt)^{\Z/\ell}$} denote the 
full subcategory of $\infPro(\infSpt)$ spanned by the $(\Eilenberg \Z/\ell)^*$-local pro-spectra.
\end{df}

Next we recall the definition of shapes, 
see \cite[\S 7.1.6]{HTT} and \cite[\S 4]{zbMATH07103864}.

\begin{df}
\label{ketreal.9}
A \emph{geometric morphism} of $\infty$-topoi is a functor $f_*\colon \cX\to \cY$ of $\infty$-topoi that admits 
a left exact left adjoint $f^*\colon \cY\to \cX$.
The functor $f^*$ admits a pro-left adjoint
\[
f_!\colon \cX\to \infPro(\cY)
\]
satisfying $f_!(X)(Y):=\Map_{\cX}(X,f^*Y)$.
By \cite[Proposition 6.3.4.1]{HTT}, 
there exists a unique geometric morphism $f_*\colon \cX\to \infSpc$ up to equivalence for every $\infty$-topoi $\cX$.
Furthermore, 
its left adjoint $f^*\colon \infSpc\to \cX$ is the constant sheaf functor.

The \emph{shape of $\cX$}\index{shape} is defined to be\index[notation]{PiX @ $\Pi(\cX)$}
\begin{equation}
\label{ketreal.9.1}
\Pi(\cX)
:=
f_!\one_{\cX}\in \infPro(\infSpc),
\end{equation}
where $\one_{\cX}$ denotes the final object of $\cX$.

Let $\infTopos^R$\index[notation]{ToposR @ $\infTopos^R$} denote the $\infty$-category of $\infty$-topoi, where $1$-morphisms are geometric morphisms.
According to \cite[Remark 7.1.6.15]{HTT}, there is a functor
\begin{equation}
\label{ketreal.9.2}
\Pi
\colon
\infTopos^R\to \infPro(\infSpc)
\end{equation}
sending $\cX$ to $\Pi(\cX)$, which preserves small colimits.
\end{df}

\begin{df}
\label{ketreal.10}
For $X\in \lSch$, we set\index[notation]{Pilet @ $\Pi^{\leta}$}
\[
\Pi^{\leta}(X)
:=
\Pi(\infShv(X_{\leta},\infSpc)).
\]
For any prime $\ell$, let $\Pi_{\ell}^{\leta}(X)$ be its image in $\infPro(\infSpc)^{\Z/\ell}$.
\end{df}

For every abelian group $A$, 
\eqref{ketreal.9.1} gives an isomorphism of abelian groups
\[
\pi_*(\Map_{\infShv(X_{\leta},\infSpc)}(X,A))
\cong
\pi_*(\Map_{\infPro(\infSpc)}(\Pi^{\leta}(X),A)).
\]
Hence we obtain an isomorphism of abelian groups
\begin{equation}
\label{ketreal.10.1}
H_{\leta}^*(X,A)
\cong
H^*(\Pi^{\leta}(X),A).
\end{equation}
Let $f\colon X\to Y$ be a morphism in $\lSch$.
Then the 
functor $f_*\colon\Shv(X_{\leta})\to \Shv(Y_{\leta})$ is a geometric morphism of topoi, 
and hence $f_*\colon \infShv(X_{\leta},\infSpc)\to \infShv(Y_{\leta},\infSpc)$ is a geometric morphism of $\infty$-topoi.
Combined with \eqref{ketreal.9.2}, 
we define the functor 
\begin{equation}
\label{ketreal.10.2}
\Pi^{\leta}
\colon
\lSch \to \infPro(\infSpc)
\end{equation}
by sending $X\in \lSch$ to the shape $\Pi^{\leta}(X_{\leta})$.
Composing with the localization functor we obtain the functor
\begin{equation}
\label{ketreal.10.3}
\Pi_{\ell}^\leta
\colon
\lSch\to \infPro(\infSpc)^{\Z/\ell}.
\end{equation}

\begin{lem}
\label{ketreal.17}
Let $U_\bullet \to X$ be a log \'etale hypercover in $\lSch$.
Then there is an equivalence of pro-spaces
\[
\Pi^{\leta}(X)
\simeq
\colimit_i(\Pi^{\leta}(U_i)).
\]
\end{lem}
\begin{proof}
This is a special case of the claim in \cite[Remark 1.6]{Hoyoisetalesymmetric}.
The proof uses descent and the fact that \eqref{ketreal.9.2} preserves small colimits.
\end{proof}

The functor $\Pi_{\ell}^\leta$ in \eqref{ketreal.10.3} satisfies $\boxx$-invariance by Theorem \ref{ketreal.1}, 
Proposition \ref{ketreal.15}, and \eqref{ketreal.10.1}.
Moreover, 
$\Pi_{\ell}^\leta$ satisfies log \'etale descent due to Lemma \ref{ketreal.17}.
Hence by Proposition \ref{univlogSH.1}, 
for every $S\in \lSch$, 
we obtain a functor\index{Rellet @ $\Real_{\ell}^{\leta}$}
\begin{equation}
\label{ketreal.17.1}
\Real_{\ell}^{\leta}
\colon
\infShv_{\leta}(\lSm/S)[\square^{-1}]
\to
\infPro(\infSpc)^{\Z/\ell}
\end{equation}
sending $X\in \lSm/S$ to $\Pi_{\ell}^\leta(X)$.

\begin{lem}
\label{ketreal.19}
Suppose $k$ is a separably closed field of characteristic $p$ and $\ell\neq p$.
For every $X\in \lSm/k$ and integer $i$, there is a canonical isomorphism of abelian groups
\[
H_{\leta}^i(X,\Z/\ell)
\cong
H_{\et}^i(X-\partial X,\Z/\ell).
\]
\end{lem}
\begin{proof}
Combine Proposition \ref{ketreal.15} and \cite[Corollary 7.5]{MR1922832}.
\end{proof}

\begin{lem}
\label{ketreal.18}
Suppose $k$ is a separably closed field of characteristic $p$ and $\ell\neq p$.
For every $X,Y\in \lSm/k$, there exists a canonical
equivalence in $\infPro(\infSpc)^{\Z/\ell}$
\[
\Pi_{\ell}^{\leta}(X\times_k Y)
\simeq
\Pi_{\ell}^{\leta}(X)
\times
\Pi_{\ell}^{\leta}(Y).
\]
\end{lem}
\begin{proof}
We have isomorphisms of graded rings
\begin{equation}
\label{ketreal.18.1}
\begin{split}
H^*(\Pi_{\ell}^{\leta}(X)\times \Pi_\ell^{\leta}(Y),\Z/\ell)
&\cong
H^*(\Pi_{\ell}^{\leta}(X),\Z/\ell)
\otimes_{\Z/\ell}
H^*(\Pi_{\ell}^{\leta}(Y),\Z/\ell)
\\
&\cong
H_{\leta}^*(X,\Z/\ell)
\otimes_{\Z/\ell}
H_{\leta}^*(Y,\Z/\ell)
\\
&\cong
H_{\et}^*(X-\partial X,\Z/\ell)
\otimes_{\Z/\ell}
H_{\et}^*(Y-\partial Y,\Z/\ell),
\end{split}
\end{equation}
where the first isomorphism comes from the K\"unneth formula in algebraic topology, the second isomorphism comes from \eqref{ketreal.10.1}, and the third isomorphism comes from Lemma \ref{ketreal.19}.

We also have isomorphisms of graded rings
\begin{equation}
\label{ketreal.18.2}
\begin{split}
H^*(\Pi_{\ell}^{\leta}(X\times_k Y),\Z/\ell)
&\cong
H_{\leta}^*(X\times_k Y,\Z/\ell)
\\
&\cong
H_{\et}^*(X-\partial X\times_k Y-\partial Y,\Z/\ell)
\end{split}
\end{equation}
where the first isomorphism comes from \eqref{ketreal.10.1}, and the second isomorphism comes Lemma \ref{ketreal.19}.
The K\"unneth formula
\[
\R\Gamma_{\et}(X-\partial X\times_k Y-\partial Y,\Z/\ell)
\simeq
\R\Gamma_{\et}(X-\partial X,\Z/\ell)
\otimes^{\bbL}
\R\Gamma_{\et}(Y-\partial Y,\Z/\ell)
\]
in \cite[{Th.\ finitude, Corollaire 1.11}]{SGA4half} gives an isomorphism of graded rings
\begin{equation}
\label{ketreal.18.3}
H_{\et}^*(X-\partial X,\Z/\ell)
\otimes_{\Z/\ell}
H_{\et}^*(Y-\partial Y,\Z/\ell)
\cong H_{\et}^*(X-\partial X\times_k Y-\partial Y,\Z/\ell)
\end{equation}
since $\Z/\ell$ is a field.

Combine \eqref{ketreal.18.1}, \eqref{ketreal.18.2}, and \eqref{ketreal.18.3} to conclude the desired equivalence.
\end{proof}

If $k$ is a separably closed field of characteristic $p$ and $\ell\neq p$, 
we have the $\ell$-adic Tate module $T_{\ell}\mu:=\limit_r \mu_{\ell^r}$.
There is an isomorphism of pro-groups $T_{\ell}\mu \cong \Z_\ell$.
In $\infPro(\infSpc)^{\Z/\ell}$,
we have the equivalent mixed spheres
\[
S_{\ell}^{1,0}:=K(\Z_\ell,1)
\text{ and }
S_{\ell}^{1,1}:=K(T_\ell \mu,1),
\]
where $K(-,1)$ denotes the Eilenberg-MacLane space. 
The naturally induced functor $\infSpc\to \infPro(\infSpc)^{\Z/\ell}$ sends $S^1$ to $S_\ell^{1,0}$.
Hence $S_\ell^{1,0}$ is invertible in $\infPro(\infSpt)^{\Eilenberg \Z/\ell}$.
It follows that $S_{\ell}^{1,1}$ is invertible in 
$\infPro(\infSpt)^{\Eilenberg \Z/\ell}$ too.

\begin{lem}
\label{ketreal.12}
Suppose $k$ is a separably closed field of characteristic $p$ and $\ell\neq p$.
In $\infPro(\infSpc)^{\Z/\ell}$, 
there is an equivalence
\[
\Pi_{\ell}^{\leta}(\Gmm)/\Pi_p^{\leta}(\pt)
\simeq
S_{\ell}^{1,1}.
\]
\end{lem}
\begin{proof}
By combining \cite[Corollary 7.5]{MR1922832} and Proposition \ref{ketreal.15} we obtain the computation of 
log \'etale cohomology groups
\begin{equation}
\label{ketreal.12.2}
H_{\leta}^i(\Gmm,\Z/\ell)
\cong
\left\{
\begin{array}{ll}
\Z/\ell & \text{if }i=0,1,
\\
0 & \text{if }i>1.
\end{array}
\right.
\end{equation}

The fs log scheme $\Gmm$ has a Zariski cover consisting of the two patches
\[
U_1:=\Spec{\N x\to \Z[x]}
\text{ and }
U_2:=\Spec{\N (1/x) \to \Z[1/x]}.
\]
For every integer $r\geq 1$, there is a morphism
\[
m_r
\colon
\Gmm\to \Gmm
\]
given by the formula $x\mapsto x^{\ell^r}$.
The assumption $\ell\neq p$ implies that $m_{\ell^r}$ is Kummer \'etale for every integer $r\geq 1$.
The connected component of the \v{C}ech nerve associated with $m_r$ is $\clspace \mu_{\ell^r}$.
Hence we have a morphism of pro-spaces
\[
\alpha_{\ell^r}
\colon
\Pi_{\ell}^{\leta}(\Gmm) \to \clspace \mu_{\ell^r}.
\]
One can readily check that the naturally induced homomorphism of abelian groups
\begin{equation}
\label{ketreal.12.1}
\alpha_{\ell^r}^*
\colon
H^i(\clspace \mu_{\ell^r},\Z/\ell)
\to
H^i(\Pi_{\ell}^{\leta}(\Gmm),\Z/\ell)
\cong
H_{\leta}^i(\Gmm,\Z/\ell)
\end{equation}
is an isomorphism for $i=0,1$.
Collecting $\alpha_{\ell^r}$ for all $r\geq 1$ we produce a morphism of pro-spaces
\[
\beta_{\ell}
\colon
\Pi_{\ell}^{\leta}(\Gmm) \to \limit_r \clspace \mu_{\ell^r}.
\]
The quotient homomorphism $\mu_{\ell^{r+1}}\to \mu_{\ell^r}$ naturally induces a homomorphism
\[
H^i(\clspace\mu_{\ell^r},\Z/\ell)
\to
H^i(\clspace\mu_{\ell^{r+1}},\Z/\ell)
\]
for every integer $i\geq 0$.
This is an isomorphism for $i=0,1$ and the zero map for every integer $i>1$, see \cite[Exercise 1, \S XIII.1]{zbMATH03657912}.
Since \eqref{ketreal.12.1} is an isomorphism for $i=0,1$, we deduce that $\beta_\ell$ is an equivalence in $\infPro(\infSpc)^{\Z/\ell}$ together with the computation \eqref{ketreal.12.2}.
\end{proof}

Suppose $k$ is a separably closed field of characteristic $p$ and $\ell\neq p$.
The composite
\[
\lSm/k
\xrightarrow{\Pi_{\ell}^{\leta}}
\infPro(\infSpc)^{\Z/\ell}
\to
\infPro(\infSpt)^{\Eilenberg \Z/\ell}
\]
is symmetric monoidal by Lemma \ref{ketreal.18}, 
and it sends $1\to \P^1$ to a map in $\infPro(\infSpt)^{\Eilenberg \Z/\ell}$ whose cofiber is an invertible object by 
$\boxx$-invariance of $\Pi_{\ell}^{\leta}$ and Lemma \ref{ketreal.12}.
Hence Proposition \ref{univlogSH.5} gives a functor\index{Rellet @ $\Real_{\ell}^{\leta}$}
\begin{equation}
\label{ketreal.12.3}
\Real_{\ell}^{\leta}
\colon
\inflogSH_{\leta}(k)
\to
\infPro(\infSpt)^{\Eilenberg \Z/\ell}
\end{equation}
of $\infty$-categories, which we call the \emph{stable $\ell$-adic realization functor}.

\begin{lem}
\label{ketreal.16}
Let $f\colon Y\to X$ be a log \'etale morphism of fs log schemes.
If there is a strict morphism $X\to \A_\N$,
then $f$ is Kummer \'etale.
\end{lem}
\begin{proof}
The question is \'etale local on $Y$, 
so owing to \cite[Theorem IV.3.3.1]{Ogu} we may assume $f$ admits a chart $\theta\colon \N\to P$ 
with the following properties: 
$\theta$ is injective, 
the cokernel of $\theta^{\gp}$ is finite, 
and the naturally induced morphism $g\colon Y\to X\times_{\A_\N}\A_P$ is strict \'etale.

An elementary argument shows that for every $x\in P$, 
there exists an integer $n>0$ such that $nx$ is in the image of $\theta$.
In other words, $\theta$ is Kummer.
Hence the projection $p\colon X\times_{\A_\N}\A_P\to X$ is Kummer \'etale.
Since $g$ is strict \'etale, $f=pg$ is Kummer \'etale.
\end{proof}

\begin{lem}
\label{ketreal.14}
Suppose $k$ is a separably closed field of characteristic $p>0$.
In $\infPro(\infSpc)^{\Z/p}$, there is an equivalence
\[
\Pi_{p}^{\leta}(\Gmm)/\Pi_p^{\leta}(\pt)
\simeq
0.
\]
\end{lem}
\begin{proof}
We need to check that 
\[
H_{\leta}^i(\Gmm,\Z/p)
\cong
\left\{
\begin{array}{ll}
\Z/p & \text{if }i=0,
\\
0 & \text{if }i>0.
\end{array}
\right.
\]
The Artin-Schreier sequence
\[
0\to \Z/p\to \mathbb{G}_a\to \mathbb{G}_a\to 0
\]
gives a long exact sequence
\[
\cdots
\to
H_{\leta}^i(\Gmm,\Z/p)
\to
H_{\leta}^i(\Gmm,\cO_{\Gmm})
\to
H_{\leta}^i(\Gmm,\cO_{\Gmm})
\to
H_{\leta}^{i+1}(\Gmm,\Z/p)
\to
\cdots.
\]
Hence it remains to check the vanishing $H_{\leta}^i(\Gmm,\cO_{\Gmm})=0$ for $i>0$.
Lemma \ref{ketreal.16} implies an isomorphism of cohomology groups
\[
H_{\leta}^i(\Gmm,\cO_{\Gmm})
\cong
H_{\ket}^i(\Gmm,\cO_{\Gmm})
\]
for every integer $i\geq 0$.
Furthermore, 
\cite[Proposition 6.5]{KatoLog2} (see also \cite[Proposition 3.27]{MR2452875}) implies an 
isomorphism of abelian groups
\[
H_{\ket}^i(\Gmm,\cO_{\Gmm})
\cong
H_{\et}^i(\P^1,\cO_{\P^1})
\]
for every integer $i\geq 0$.
A classical computation of the Zariski cohomology of $\P^1$ and the comparison between Zariski 
and \'etale cohomology show that the last group vanishes for $i>0$.
\end{proof}

Suppose $k$ is a separably closed field of characteristic $p>0$.
Due to Lemma \ref{ketreal.14}, 
we \emph{cannot} apply Proposition \ref{univlogSH.5} to the composite
\[
\lSm/k
\xrightarrow{\Pi_{p}^{\leta}}
\infPro(\infSpc)^{\Z/p}
\to
\infPro(\infSpt)^{\Eilenberg \Z/p}.
\]

This suggests that one may need to enlarge $\infPro(\infSpt)^{\Eilenberg \Z/p}$ in order to construct a 
stable $p$-adic realization from $\inflogSH_{\leta}(k)$.

\subsection{Kato-Nakayama realizations}

Throughout this subsection, 
we fix a field $k$ of characteristic $0$ and an embedding $k\hookrightarrow \C$.

There is a functor to the category of  topological spaces\index[notation]{log @ $(-)^{\log}$}
\begin{equation}
(-)^{\log}
\colon
\lSch/k
\to
\Spc
\end{equation}
sending $X\in \lSch/k$ to its Kato-Nakayama space $X^{\log}:=(X\times_k \Spec{\C})^{\log}$.
We refer to \cite[\S 1]{MR1700591} for details.
This naturally induces a functor to the $\infty$-category of spaces
\begin{equation}
\label{KNreal.1.1}
(-)^{\log}
\colon
\lSch/k
\to
\infSpc
\end{equation}

If $P$ is an fs monoid, 
we set\index[notation]{RP @ $\rR_P$}\index[notation]{TP @ $\rT_P$}
\[
\rR_P:=\hom(P,\R_{\geq 0})
\text{ and }
\rT_P:=\hom(P,S^1),
\]
where the monoid operations on $\R_{\geq 0}$ and $S^1$ are the multiplications.
We can view $\rR_P$ and $\rT_P$ as topological spaces.
According to \cite[Example 1.2.1.1]{MR1700591}, 
there is a canonical homeomorphism
\begin{equation}
\label{KNreal.1.2}
(\A_P)^{\log}
\cong
\rR_P \times \rT_P.
\end{equation}

\begin{prop}
\label{KNreal.1}
Suppose $f\colon X\to Z$ and $g\colon Y\to Z$ are morphisms in $\lSch/k$.
If either $f$ or $g$ is exact, then there is a canonical homeomorphism
\[
(X\times_Z Y)^{\log}
\cong
X^{\log} \times_{Z^{\log}} Y^{\log}.
\]
\end{prop}
\begin{proof}
This is a special case of \cite[Proposition V.1.2.11]{Ogu}.
\end{proof}

\begin{prop}
\label{KNreal.2}
Suppose $X\in \lSch/k$.
Then there are canonical homotopy equivalences
\[
(X\times \boxx)^{\log}
\simeq
(X\times \A^1)^{\log}
\simeq
X^{\log}
\]
\end{prop}
\begin{proof}
Owing to Proposition \ref{KNreal.1},
we reduce to the case when $X=k$.
By \eqref{KNreal.1.2},
there are homeomorphisms
\[
(\P^1-0,\infty)^{\log}\cong \R_{\geq 0}\times S^1,
\;
(\A^1)^{\log} \cong \R^2,
\text{ and }
(\mathbb{G}_m)^{\log} \cong \R_{>0}\times S^1.
\]
In particular, $(\A^1)^{\log}$ is contractible.
From the cartesian square
\[
\begin{tikzcd}
(\G_m)^{\log}\ar[d]\ar[r]&
(\P^1-0,\infty)^{\log}\ar[d]
\\
(\A^1)^{\log}\ar[r]&
\boxx^{\log}
\end{tikzcd}
\]
we see that $\boxx^{\log}$ is homeomorphic to $[0,1]^2$,
where we regard $[0,1]$ as a compactification of $(0,1)\cong \R$.
This shows that $\boxx^{\log}$ is contractible too.
\end{proof}

A continuous map $f\colon X\to Y$ of spaces is called \emph{\'etale} (resp.\ an \'etale cover) 
if $f$ is a local homeomorphism (resp.\ surjective local homeomorphism).

\begin{prop}
\label{KNreal.3}
Suppose $f\colon \sX\to X$ is a Kummer \'etale hypercover.
Then the naturally induced map $f^{\log}\colon \sX^{\log}\to X^{\log}$ is an \'etale hypercover.
\end{prop}
\begin{proof}
Combine \cite[Lemma 2.2]{MR1700591} and Proposition \ref{KNreal.1}.
\end{proof}

\begin{lem}
\label{logtop.7}
Let $f\colon Y\to X$ be a dividing cover of fs log schemes.
If $X$ is quasi-compact and has a chart $P$ such that $P$ is sharp, then there exists a subdivision $\Sigma\to \Spec{P}$ of toric fans such that the naturally induced morphism
\[
Y\times_{\A_P}\T_\Sigma\to X\times_{\A_P}\T_{\Sigma}
\]
is an isomorphism of fs log schemes.
\end{lem}
\begin{proof}
This is a special case of \cite[Proposition A.11.5]{logDM}.
\end{proof}

\begin{prop}
\label{KNreal.4}
Suppose $f\colon Y\to X$ is a dividing cover in $\lSch/k$.
Then the naturally induced map
\[
f^{\log}
\colon
Y^{\log}
\to
X^{\log}
\]
is an equivalence of spaces.
\end{prop}
\begin{proof}
We only need to consider the case when $X$ is connected.
Then $f^{\log}$ is a proper surjective map of connected finite CW-complexes.
By \cite{zbMATH03146521}, 
it suffices to show that $(f^{\log})^{-1}(x)$ is contractible for every $x\in X^{\log}$.
This question is Zariski local on $X$, 
so we may assume that $X$ admits a chart $X\to \A_P$, 
where $P$ is sharp, 
see \cite[Proposition II.2.3.7]{Ogu}.
According to Lemma \ref{logtop.7}, 
there exists a subdivision $\Sigma\to \Spec{P}$ of toric fans such that the naturally induced morphism
\[
Y\times_{\A_P}\T_\Sigma
\to
X\times_{\A_P}\T_\Sigma
\]
is an isomorphism.
Hence it suffices to show that both maps
\[
(Y\times_{\A_P}\T_\Sigma)^{\log}\to Y^{\log}
\text{ and }
(X\times_{\A_P}\T_\Sigma)^{\log}\to X^{\log}
\]
are equivalences of spaces.
This means that we only need to consider the case when the morphism $f$ is the projection $X\times_{\A_P}\T_\Sigma\to X$ 
for some morphism $u\colon X\to \A_P$ and subdivision $\Sigma\to \Spec{P}$.

\ 

Since this question is Zariski local on $X$, 
we may assume that $u$ factors through a chart $X\to \A_Q$ such that $Q$ is sharp and $\ol{\cM}_{X,x}\cong Q$, 
see \cite[Proposition II.2.3.7]{Ogu}.
Since $\Spec{Q}\times_{\Spec{P}}\Sigma\to \Spec{Q}$ is again a subdivision, 
we may reduce to the following case:
The morphism $f$ is the projection $X\times_{\A_P}\T_\Sigma\to X$ for some chart $X\to \A_P$ and subdivision 
$\Sigma\to \Spec{P}$, where $P$ is sharp and $\ol{\cM}_{X,x}\cong P$.
Now, 
by \cite[Lemma 1.3]{MR1700591}, 
there is a cartesian square
\[
\begin{tikzcd}
Y^{\log}\ar[d,"f^{\log}"']\ar[r]&
(\T_\Sigma)^{\log}\ar[d,"g^{\log}"]
\\
X^{\log}\ar[r]&
(\A_P)^{\log}.
\end{tikzcd}
\]
Here $g\colon \T_\Sigma\to \A_P$ is the naturally induced morphism.
Hence we are reduced to the case when $f$ is the morphism $\T_\Sigma\to \A_P$ and $x$ is the origin of $(\A_P)^{\log}$.
In this case, 
let $\sigma_1,\ldots,\sigma_n$ be the cones of $\Sigma$.
Then $f$ is locally given by $f_i\colon \T_{\sigma_i^\wedge}\to \A_P$.
Hence the morphism $f^{\log}$ is locally given by the product of the naturally induced maps 
$\rR_{f_i}\colon \rR_{\sigma_i^\vee}\to \rR_P$ and $\rT_{f_i}\colon \rT_{\sigma_i^\vee}\to \rT_P$.

\ 

The multiplication map $\R_{\geq 0}\times [0,1]\to \R_{\geq 0}$ given by $(x,y)\mapsto xy$ 
naturally induces a commutative square
\begin{equation}
\label{KNreal.4.4}
\begin{tikzcd}
\rR_{\sigma_i^\vee}\times [0,1]\ar[r]\ar[d,"\rR_{f_i}\times \id"']&
\rR_{\sigma_i^\vee}\ar[d,"\rR_{f_i}"]
\\
\rR_P\times [0,1]\ar[r]&
\rR_P.
\end{tikzcd}
\end{equation}
Since $\rT_{f_i}$ is a homeomorphism, \eqref{KNreal.4.4} gives a commutative square
\[
\begin{tikzcd}
(\T_{\sigma_i^\vee})^{\log}\times [0,1]\ar[d]\ar[r]&
(\T_{\sigma_i^\vee})^{\log}\ar[d]
\\
(\A_P)^{\log}\times [0,1]\ar[r]&
(\A_P)^{\log}.
\end{tikzcd}
\]
By gluing these squares together we obtain the commutative square
\[
\begin{tikzcd}
(\T_\Sigma)^{\log}\times [0,1]\ar[d]\ar[r]&
(\T_\Sigma)^{\log}\ar[d,"g^{\log}"]
\\
(\A_P)^{\log}\times [0,1]\ar[r]&
(\A_P)^{\log}.
\end{tikzcd}
\]
As a consequence, 
we obtain a map $h\colon (g^{\log})^{-1}(0)\times [0,1]\to (g^{\log})^{-1}(0)$ such that $h(u,0)=0$ 
and $h(u,1)=u$ for all $u$.
This shows that $(g^{\log})^{-1}(0)$ is contractible.
\end{proof}

Propositions \ref{KNreal.2}, \ref{KNreal.3}, and \ref{KNreal.4} show  
the Kato-Nakayama realization $(-)^{\log}$ is $\boxx$-invariant and satisfies $dNis$-descent.
Hence by Proposition \ref{univlogSH.1}, we obtain a functor\index[notation]{ReKN @ $\Real_{\mathrm{KN}}$}
\begin{equation}
\label{KNreal.4.1}
\Real_{\mathrm{KN}}
\colon
\infShv_{dNis}(\lSm/S)[\square^{-1}]
\to
\infSpc
\end{equation}
sending $X\in \lSm/S$ to $X^{\log}$ for every $S\in \lSch/k$ and $X\in \lSm/S$.
The constant presheaf functor $\infSpc \to \Psh(\lSm/S)$ naturally induces a functor $\infSpc\to \inflogH(S)$.
This is a section of \eqref{KNreal.4.1}.
Similarly, 
in the spectral setting, 
have a functor
\begin{equation}
\label{KNreal.4.2}
\Real_{\mathrm{KN}}
\colon
\inflogSH^\eff(S)
\to
\infSpt.
\end{equation}
It also admits a section.

\ 

Due to Proposition \ref{KNreal.1}, the square
\begin{equation}
\label{KNreal.4.5}
\begin{tikzcd}
\inflogSH^\eff(S)\ar[d,"-\otimes T"']\ar[r]&
\infSpt\ar[d,"-\otimes S^2"]
\\
\inflogSH^\eff(S)\ar[r]&
\infSpt
\end{tikzcd}
\end{equation}
commutes.
Taking colimits in $\LPr$ we obtain a functor
\begin{equation}
\label{KNreal.4.3}
\Real_{\mathrm{KN}}
\colon
\inflogSH(S)
\to
\infSpt.
\end{equation}
We refer to $\Real_{\mathrm{KN}}$ as the \emph{stable Kato-Nakayama realization functor}.
It also admits a section.
It sends the suspension spectrum $\Sigma_T^\infty X_+$ to $\Sigma_{S^2}^\infty (X^{\log})_+$ for every $X\in \lSm/S$.

Suppose $S\in \lSch/k$ has the chart $\N$ or $0$.
Then every morphism $f\colon X\to S$ in $\lSch/k$ is exact, see \cite[Proposition I.4.2.1(4)]{Ogu}.
Hence, by Proposition \ref{KNreal.1}, the functor
\[
(-)^{\log}
\colon
\lSm/S
\to
\infSpc
\]
is symmetric monoidal.
This implies that the enhanced Kato-Nakayama realization functors \eqref{KNreal.4.1}, \eqref{KNreal.4.2}, 
and \eqref{KNreal.4.3} are symmetric monoidal.

\begin{prop}
\label{KNreal.5}
There exists a functor
\[
\Real_\mathrm{KN}
\colon
\inflogH(k)
\to
\infSpc
\]
such that the triangle
\begin{equation}
\label{KNreal.5.1}
\begin{tikzcd}
\inflogH(k)\ar[rr,"\omega_\sharp"]\ar[rd,"\Real_{\mathrm{KN}}"']&
&
\infH(k)\ar[ld,"\Real_\Betti"]
\\
&
\infSpc
\end{tikzcd}
\end{equation}
commutes,
where $\Real_{\Betti}$\index[notation]{ReBetti @ $\Real_{\Betti}$} denotes the Betti realization functor 
$\infH(k)\to \infSpc$.
There is a similar commutative triangle for $\inflogSH(k)$, $\infSH(k)$, and $\infSpt$.
\end{prop}
\begin{proof}
According to \cite[Theorem V.1.3.1]{Ogu}, 
$X^{\log}$ is a manifold with boundary, 
and its interior is the Betti realization $(X-\partial X)^{\Betti}$ of $X-\partial X$.
The collar neighborhood theorem \cite[Theorem 2]{zbMATH03320239} implies that the inclusion
\[
(X-\partial X)^{\Betti}
\to
X^{\log}
\]
is a homotopy equivalence.
This means that the triangle
\[
\begin{tikzcd}
\lSm/S\ar[rr,"X\mapsto X-\partial X"]\ar[rd,"X\mapsto X^{\log}"']&
&
\infH(k)\ar[ld,"\Real_\Betti"]
\\
&
\infSpc
\end{tikzcd}
\]
commutes.
Together with Proposition \ref{univlogSH.1}, we obtain \eqref{KNreal.5.1}.

We can similarly show that the $S^1$-spectral version of \eqref{KNreal.5.1} commutes.
For the $\P^1$-stable version, 
we use \eqref{KNreal.4.5}.
\end{proof}

\newpage

\appendix

\section{Cubes and total cofibers}
\label{cubes}
To efficiently deal with iterated cofibers, 
we will introduce the notion of the total cofibers of $n$-cubes in this subsection.
We will use this notion in Theorem \ref{Hoch.13}.

\begin{df}
\label{Thomdf.3}
Let $\cC$ be a pointed $\infty$-category with small colimits.
For an integer $n\geq 0$, an \emph{$n$-cube}\index{cube} in $\cC$ is a map
\[
Q\colon (\Delta^1)^n \to \cC,
\]
compare \cite[Definition 6.1.1.2]{HA}.
An \emph{$n$-cubical horn} \index{cubical horn} in $\cC$ is a map
\[
Q\colon (\Delta^1)^n-\{(1,\ldots,1)\} \to \cC.
\]
For every $n$-cube $Q$, we can naturally associate an $n$-cubical horn denoted by 
$\horn(Q)$. \index[notation]{horn @ $\horn(Q)$}
Moreover, 
we set $\vertex(Q):=Q(1,\ldots,1)$. \index[notation]{vertex @ $\vertex(Q)$}
\end{df}

The cofiber construction applies to $1$-cubes.
We give a generalization to $n$-cubes as follows.

\begin{df}
\label{Thomdf.10}
Let $\cC$ be a pointed $\infty$-category with small colimits.
An $n$-cube $Q$ in $\cC$ naturally induces a map
\[
\colimit(\horn(Q))
\to
\vertex(Q).
\]
The \emph{total cofiber}\index{total cofiber}\index[notation]{tcofib @ $\tcofib(Q)$} of $Q$, 
denoted by $\tcofib(Q)$, 
is the cofiber of this map.
Thus there is a cofiber sequence
\begin{equation}
\label{Thomdf.10.1}
\colimit(\horn(Q))
\to
\vertex(Q)
\to
\tcofib(Q).
\end{equation}
If $Q$ is a $1$-cube, then $\tcofib(Q)=\cofib(Q)$.
We say that $Q$ is \emph{cocartesian} \index{cocartesian cube} if there is an equivalence $\tcofib(Q)\simeq 0$ in $\cC$.
\end{df}

\begin{df}
\label{Thomdf.16}
Let $\cC$ be a pointed $\infty$-category with limits.
An $n$-cube in $\cC$ naturally induces a map
\[
Q(0,\ldots,0)\to \limit(Q|_{(\Delta^1)^n-\{0,\ldots,0\}}).
\]
The \emph{total fiber}\index{total fiber} of $Q$, 
denoted by $\tfib(Q)$, \index[notation]{tfib @ $\tfib(Q)$} 
is the fiber of this map.
We say that $Q$ is \emph{cartesian}\index{cartesian cube} if there is an equivalence $\tfib(Q)\simeq 0$ in $\cC$.
\end{df}

Note that if $Q$ is a $1$-cube, then $\tfib(Q)=\fib(Q)$.

\begin{prop}
\label{Thomdf.5}
Let $\cC$ be a pointed $\infty$-category with small colimits.
For every $n$-cube $Q\colon (\Delta^1)^n\to \cC$ and integer $1\leq i\leq n$, 
there is a canonically induced cocartesian square
\begin{equation}
\label{Thomdf.5.2}
\begin{tikzcd}
\colimit(\horn(Q|_{(\Delta^1)^{i-1}\times \{0\}\times (\Delta^1)^{n-i}}))\ar[d]\ar[r]&
\colimit(\horn(Q|_{(\Delta^1)^{i-1}\times \{1\}\times (\Delta^1)^{n-i}}))\ar[d]
\\
\vertex(Q|_{(\Delta^1)^{i-1}\times \{0\}\times (\Delta^1)^{n-i}})\ar[r]&
\colimit(\horn(Q)).
\end{tikzcd}
\end{equation}
\end{prop}
\begin{proof}
For simplicity of notation, we may assume $i=1$.
We set
\[
A_i:=\{i\}\times \Delta^{n-1},
\;
B_i:=\{i\}\times (\Delta^{n-1}-\{(1,\ldots,1)\})
\]
for $i=0,1$ and
\[
C:=\Delta^1\times (\Delta^{n-1}-\{(1,\ldots,1)\})
\;
D:=\Delta^{n}-\{(1,\ldots,1)\}.
\]

Since $A_0$ has a final object, we have an equivalence in $\cC$
\begin{equation}
\label{Thomdf.5.5}
\vertex(Q|_{A_0})
\simeq
\colimit(Q|_{A_0}).
\end{equation}
The inclusion $\{1\} \to \Delta^1$ is right anodyne, so it is final due to \cite[Corollary 4.1.1.12]{HTT}.
Together with \cite[Corollary 4.1.2.7]{HTT}, we deduce that the inclusion $B_1\to C$ is final.
It follows that we have an equivalence in $\cC$
\begin{equation}
\label{Thomdf.5.6}
\colimit(Q|_{B_1})
\simeq
\colimit(Q|_{C}).
\end{equation}

The square
\[
\begin{tikzcd}
B_0\ar[r]\ar[d]&
C\ar[d]
\\
A_0\ar[r]&
D
\end{tikzcd}
\]
is cocartesian.
Apply \cite[Proposition 4.4.2.2]{HTT} to deduce that the square
\begin{equation}
\label{Thomdf.5.4}
\begin{tikzcd}
\colimit(Q|_{B_0})\ar[d]\ar[r]&
\colimit(Q|_{C})\ar[d]
\\
\colimit(Q|_{A_0})\ar[r]&
\colimit(Q|_D)
\end{tikzcd}
\end{equation}
is cocartesian.
Together with \eqref{Thomdf.5.5} and \eqref{Thomdf.5.6}, we see that \eqref{Thomdf.5.2} is cocartesian.
\end{proof}

\begin{prop}
\label{Thomdf.13}
Let $\cC$ be an $\infty$-category with small colimits.
For every $n$-cube $Q\colon (\Delta^1)^n\to \cC$ and integer $1\leq i\leq n$, 
there is a canonically induced cofiber sequence
\begin{equation}
\label{Thomdf.5.3}
\tcofib (Q|_{(\Delta^1)^{i-1}\times \{0\}\times (\Delta^1)^{n-i}})
\to
\tcofib (Q|_{(\Delta^1)^{i-1}\times \{1\}\times (\Delta^1)^{n-i}})
\to
\tcofib(Q).
\end{equation}
\end{prop}
\begin{proof}
We set $Q_j:=Q|_{(\Delta^1)^{i-1}\times \{j\}\times (\Delta^1)^{n-i}}$ for $j=0,1$ and
\[
\cF:=\cofib(\colimit(\horn(Q_0))\to \colimit(\horn(Q))).
\]
Consider the commutative diagram
\[
\begin{tikzcd}[column sep=small, row sep=small]
\colimit(\horn(Q_0))\ar[r]\ar[d]\ar[rd,phantom,"(a)" description]&
\colimit(\horn(Q_1))\ar[d]\ar[r]\ar[rd,phantom,"(b)" description]&
0\ar[d]
\\
\vertex(Q_0)\ar[r]&
\colimit(\horn(Q))\ar[d]\ar[r]\ar[rd,phantom,"(c)" description]&
\cF\ar[d]\ar[r]\ar[rd,phantom,"(d)" description]&
0\ar[d]
\\
&
\vertex(Q)\ar[r]&
\tcofib(Q_1)\ar[r]&
\tcofib(Q).
\end{tikzcd}
\]
By definition, $(b)$, $(b)\cup (c)$, and $(c)\cup (d)$ are cocartesian.
Proposition \ref{Thomdf.5} shows that $(a)$ is cocartesian.
Use the pasting law to deduce that $(a)\cup (b)$ is cocartesian.
This implies there is an equivalence in $\cC$
\begin{equation}
\label{Thomdf.13.1}
\cF
\simeq
\tcofib(Q_0).
\end{equation}
Use the pasting law two times to deduce that $(d)$ is cocartesian.
Together with \eqref{Thomdf.13.1}, we have the desired cofiber sequence.
\end{proof}

\begin{prop}
\label{Thomdf.12}
Let $Q\colon (\Delta^1)^n \to \cC$ be an $n$-cube in a pointed symmetric monoidal $\infty$-category $\cC$ with small colimits, 
and let $X$ be an object of $\cC$.
If the tensor product operation on $\cC$ preserves small colimits in each variable, then there is a canonical equivalence in $\cC$
\[
\tcofib(Q)\otimes X
\simeq
\tcofib(Q'),
\]
where $Q'$ is the naturally associated $n$-cube sending any point $(a_1,\ldots,a_n)$ to $Q(a_1,\ldots,a_n)\otimes X$.
\end{prop}
\begin{proof}
Apply the assumption that the tensor product operation preserves small colimits in each variable to \eqref{Thomdf.10.1}.
\end{proof}

\begin{prop}
\label{Thomdf.11}
Let $Q\colon (\Delta^1)^n \to \cC$ be an $n$-cube in a pointed symmetric monoidal $\infty$-category $\cC$ with small colimits, 
and let $i\colon X_1\to X_0$ be a map in $\cC$.
If the tensor product operation on $\cC$ preserves small colimits in each variable, then there is a canonical equivalence in $\cC$
\[
\tcofib(Q) \otimes \cofib(f)
\simeq
\tcofib(Q'),
\]
where $Q'$ is the naturally associated $(n+1)$-cube sending any point $(a_1,\ldots,a_{n+1})$ to 
$Q(a_1,\ldots,a_n)\otimes X_{a_{n+1}}$.
\end{prop}
\begin{proof}
We set
\[
Q_i':=Q'|_{(\Delta^1)^n \times \{i\}}
\]
for $i=0,1$.
Owing to Proposition \ref{Thomdf.12}, there is a commutative square in $\cC$ with vertical equivalences
\[
\begin{tikzcd}
\tcofib(Q)\otimes X_0\ar[d,"\simeq"']\ar[r]&
\tcofib(Q)\otimes X_1\ar[d,"\simeq"]
\\
\tcofib(Q_0')\ar[r]&
\tcofib(Q_1').
\end{tikzcd}
\]
Since the tensor product operation preserves small colimits, we have a cofiber sequence
\[
\tcofib(Q)\otimes X_0\to \tcofib(Q)\otimes X_1\to \tcofib(Q)\otimes \cofib(f).
\]
Combine with \eqref{Thomdf.5.3} to conclude.
\end{proof}

\begin{prop}
\label{Thomdf.6}
Let $i_1\colon X_{01}\to X_{11}$, $\ldots$, $i_n\colon X_{0n}\to X_{1n}$ be maps in a pointed symmetric monoidal 
$\infty$-category $\cC$ with small colimits.
If the tensor product operation on $\cC$ preserves small colimits in each variable, 
then there is a canonical equivalence in $\cC$
\[
\cofib(i_1)\otimes \cdots \otimes \cofib(i_n)
\simeq
\tcofib(Q_{i_1\cdots i_n}).
\]
Here we write $Q$ for the naturally associated $n$-cube sending any point $(a_1,\ldots,a_n)$ to 
$X_{a_1 1}\otimes \cdots \otimes X_{a_n n}$.
\end{prop}
\begin{proof}
Apply Proposition \ref{Thomdf.11} repeatedly.
\end{proof}

\begin{df}
\label{Thomdf.9}
Let $\cC$ be an $\infty$-category with fiber products.
For any object $X\in \cC$, 
the \emph{cube associated with a family of maps $\{U_i\to X\}_{1\leq i\leq n}$} 
\index{cube associated with a family} is the $n$-cube
\[
(\Delta^1)^n\to \cC
\]
sending a vertex of the form $(a_1,\ldots,a_n)$ to $Y_{a_1 1}\times_X \cdots \times_X Y_{a_n n}$, 
where $Y_{0i}:=U_i$ and $Y_{1i}:=X$.
The \emph{cubical horn associated with $\{U_i\to X\}_{1\leq i\leq n}$} \index{cubical horn associated with a family} 
is the restriction of the above cube to $(\Delta^1)^n-\{(1,\ldots,1)\}$.
For example, if $n=2$, then the cube associated with $U_1\to X$ and $U_2\to X$ is the $2$-cube
\[
\begin{tikzcd}
U_1\times U_2\ar[r]\ar[d]&U_2\ar[d]
\\
U_1\ar[r]&X.
\end{tikzcd}
\]
\end{df}

Any stable $\infty$-category has small colimits and limits, see \cite[Proposition 1.1.3.4]{HA}.

\begin{prop}
\label{Thomdf.15}
Let $\cC$ be a stable $\infty$-category.
Then an $n$-cube $Q$ in $\cC$ is cartesian if and only if it is cocartesian.
\end{prop}
\begin{proof}
We proceed by induction on $n$.
The claim is clear if $n=1$.
Suppose $n>1$.
We set $Q_i:=Q|_{\{i\}\times (\Delta^1)^{n-1}}$ for $i=0,1$.
Owing to \cite[Proposition 1.1.3.1]{HA}, the $\infty$-category $\Fun((\Delta^1)^n,\cC)$ is stable.
Furthermore, the $(n-1)$-cubes $\cofib(Q_0\to Q_1)$ and $\fib(Q_0\to Q_1)$ can be computed pointwise.
Since $\cC$ is stable, we have an equivalence in $\cC$
\[
\cofib(Q_0\to Q_1)\simeq  \Sigma \fib(Q_0\to Q_1),
\]
where $\Sigma$ is the suspension functor.
In particular, $\cofib(Q_0\to Q_1)$ is cocartesian if and only if $\fib(Q_0\to Q_1)$ is cartesian by induction.

\ 

Owing to \eqref{Thomdf.5.3}, $Q$ is cocartesian if and only if $\cofib(\tcofib(Q_0)\to \tcofib(Q_1))\simeq 0$, 
which is equivalent to $\tcofib(\cofib(Q_0\to Q_1))\simeq 0$.
This is equivalent to saying that $\cofib(Q_0\to Q_1)$ is cocartesian.
We can similarly show that $Q$ is cartesian if and only if $\fib(Q_0\to Q_1)$ is cartesian.
This finishes the proof.
\end{proof}

\newpage

\bibliography{biblogSH}
\bibliographystyle{siam}

\newpage

\printindex
\printindex[notation]
\newpage
\end{document}